\documentclass[12pt]{article}

\usepackage{amsmath}
\usepackage{graphicx}
\graphicspath{{./}}
\usepackage{enumerate}
\usepackage{natbib}
\usepackage{url} 
\usepackage{bbm}
\usepackage{titletoc}

\newcommand{\Bf}{\boldsymbol{f}}
\newcommand{\Bg}{\boldsymbol{g}}
\newcommand{\Bh}{\boldsymbol{h}}
\newcommand{\Bpi}{\boldsymbol{\pi}}
\newcommand{\fstar}{\boldsymbol{f}^*}
\newcommand{\norm}[1]{\Vert #1 \Vert}
\newcommand{\fnk}{\boldsymbol{f}^{(n,k)}}
\newcommand{\fn}{\boldsymbol{f}^{(n)}}
\newcommand{\fnkj}{f_j^{(n,k)}}
\newcommand{\dd}{\,\mathrm{d}}
\newcommand{\vol}{\mathrm{vol}}

\newcommand{\M}{\mathfrak{M}}
\newcommand{\Lnorm}[1]{\norm{#1 }_{L^2(\M)}}

\newcommand{\inner}[2]{\left\langle#1, #2\right\rangle}
\newcommand{\Ric}{\operatorname{Ric}} 
\newcommand{\intsph}[1]{\int_{\mbS^2} #1 d\vol_g}

\newcommand{\R}{\mathbb{R}}
\newcommand{\II}{\boldsymbol{\mathrm{II}}}
\newcommand{\nablaf}{\nabla^{\boldsymbol{f}}}   

\newcommand{\dist}{\varrho}

\newcommand{\Deltaxf}[2]{\dist(\boldsymbol{#1}, \boldsymbol{#2})}
\newcommand{\Deltadotf}[1]{\dist(\cdot, \boldsymbol{#1})}

\newcommand{\Lnfunc}{\mcL_{N,\lambda}}
\newcommand{\Lfunc}{\mcL_{\lambda}}

\usepackage{amsmath,amssymb}
\usepackage{xcolor}
\usepackage{mathrsfs}
\definecolor{mygreen}{RGB}{0,140,110}

\newcommand{\mF}{\mathscr{F}}
\newcommand{\mG}{\mathscr{G}}

\newcommand{\mU}{\mathscr{U}}

\newcommand{\mcB}{\mathcal{B}}

\newcommand{\mcH}{\mathcal{H}}

\newcommand{\mcL}{\mathcal{L}}
\newcommand{\mcM}{\mathcal{M}}

\newcommand{\mcO}{\mathcal{O}}
\newcommand{\mcP}{\mathcal{P}}

\newcommand{\mbE}{\mathbb{E}}

\newcommand{\mbN}{\mathbb{N}}

\newcommand{\mbP}{\mathbb{P}}

\newcommand{\mbR}{\mathbb{R}}
\newcommand{\mbS}{\mathbb{S}}

\usepackage{amsmath, amssymb, amscd, amsthm, amsfonts}
\usepackage[perpage,symbol*]{footmisc}
\usepackage{float}
\usepackage{pgfplots}
\usepackage{color}  
\usepackage{tikz} 
\usetikzlibrary{decorations.pathreplacing,calc,angles,quotes,arrows.meta}
\usepackage{algorithm,algpseudocode,amsmath}
\DeclareMathOperator*{\argmin}{\arg\!\min}
\algnewcommand\INPUT{\item[\textbf{Input:}]}%
\algnewcommand\OUTPUT{\item[\textbf{Output:}]}%
\usepackage{authblk}
\usepackage{enumitem}
\usepackage[english]{babel}
\setcitestyle{authoryear,open={(},close={)}}
\usepackage{colortbl,dcolumn}
\usepackage{listings} 
\usepackage{xcolor} 
\usepackage{amsfonts,amssymb}
\usepackage{mathrsfs}
\usepackage{comment}

\usetikzlibrary{hobby}

\usepackage[section]{placeins}
\usepackage[colorlinks,citecolor=blue,urlcolor=blue,linkcolor=purple]{hyperref}

\newtheorem{theorem}{Theorem}

\newtheorem{remark}{Remark}

\newtheorem{corollary}{Corollary}

\newtheorem{definition}{Definition}

\newtheorem{lemma}{Lemma}
\newtheorem{assumption}{Assumption}
\newtheorem{proposition}{Proposition}

\numberwithin{equation}{section} 
\numberwithin{theorem}{section}
\numberwithin{lemma}{section} 
\numberwithin{corollary}{section}
\numberwithin{definition}{section}
\numberwithin{proposition}{section} 
\numberwithin{remark}{section}
\numberwithin{example}{section}
\DeclareMathOperator\supp{supp}
\numberwithin{table}{section}
\numberwithin{figure}{section}

\newcommand{\T}{\intercal}

\counterwithin{algorithm}{section}

\usepackage{fancyhdr}
\hypersetup{
  pdftitle={Principal Manifold Estimation with Riemannian Templates},
  pdfauthor={Kun Meng and Christopher Perez},
  pdfsubject={Combined arXiv version}
}

\newcommand{\spacingset}[1]{\renewcommand{\baselinestretch}{#1}\small\normalsize}

\title{\bfseries Principal Manifold Estimation\\with Riemannian Templates}
\author[1,*]{Kun Meng}
\author[1]{Christopher Perez}
\affil[1]{Department of Statistics, Florida State University}
\affil[*]{\small Correspondence: 117 N. Woodward Ave, Tallahassee, FL 32306, USA.
Email: \texttt{kmeng@fsu.edu}.}
\date{}

\begin{document}
\spacingset{1}
\maketitle
\bigskip

\begin{abstract}
Principal manifold estimation (PME) has been successfully applied to robotics and neuroimaging. However, the existing PME framework does not directly accommodate manifolds with non-Euclidean topology, and its theoretical foundations remain incomplete. In this article, we develop a complete theoretical framework for PME that summarizes high-dimensional point clouds using low-dimensional manifolds with general, potentially non-Euclidean topology. The choice of topology can be informed by topological data analysis. Drawing on the theory of Sobolev spaces on Riemannian manifolds, we establish that the proposed principal manifolds are well-defined, prove convergence of the iterative algorithm used to compute them, and establish the consistency of the finite-sample estimator. We further provide a Bayesian interpretation of PME that connects the framework to recent developments in Gaussian processes on Riemannian manifolds. Examples illustrate how the proposed PME framework can be applied to existing statistical methods, including manifold-valued regression. Numerical experiments illustrate the proposed methodology, its computational implementation, and its theoretical properties.
\end{abstract}


\spacingset{1.2}

\section{Introduction}\label{Introduction}

\textit{Manifold learning} refers to a broad class of methods for modeling high-dimensional data under the \textit{manifold hypothesis}---``high-dimensional data tend to lie in the vicinity of a low-dimensional manifold'' \citep{narayanan2010sample, fefferman2016testing, fefferman2025fitting}. In recent years, manifold learning has found applications across a wide range of scientific disciplines, including sleep stage assessment \citep{lederman2015alternating}, robotics \citep{gao2023k, gao2024bi}, neuroimaging \citep{yue2016parameterization, busch2023multi, zielinski2024longitudinal}, single-cell biology \citep{yao2024single, ding2023learning, ding2025kernel, liu2025assessing}, the human cell cycle \citep{luo2024spherical}, and metabolomics \citep{li2025manifold}. In addition, the manifold hypothesis has been suggested as an explanation for the success of deep learning methods \citep{brahma2015deep} and is central to recent advances in generative modeling \citep{debortoli2022convergence, he2024manifold}, suggesting that manifold learning may find broader applications in artificial intelligence in the near future.

Within statistical methodology, manifold learning also plays a foundational role in many established frameworks, e.g., functional data analysis \citep{chen2012nonlinear}, kernel two-sample tests \citep{cheng2024kernel}, and Fréchet Regression \citep{iao2025deep}. To further illustrate this point, we examine two examples in greater detail. (i) In the framework of \textit{extrinsic regression on manifold-valued data} \citep{lin2017extrinsic}, the responses are assumed to take values on a Riemannian manifold. In this framework, the manifold is assumed to be known, embedded in a Euclidean space, and equipped with an explicit embedding map that is easy to implement. Through this embedding, the responses can be represented as a point cloud in the ambient Euclidean space. However, when the manifold is unknown and the responses are contaminated by ambient noise, both the manifold and its embedding must first be learned before the extrinsic regression method is applied (see Section~\ref{section: Selection Based on Downstream Performance} for a detailed discussion). (ii) \textit{Shape data analysis} \citep{kurtek2011elastic, kurtek2012statistical, srivastava2016functional} concerns the statistical study of curves and surfaces, which are naturally modeled as manifolds. In many applications, a necessary first step for downstream analysis is to recover a functional representation of an underlying shape from an observed point cloud. This task can be formulated as a manifold learning problem and provides an alternative to widely used methods based on \textit{Poisson surface reconstruction} \citep[][see Appendix~\ref{appendix: Poisson Surface Reconstruction versus Principal Manifold Estimation} for a detailed discussion]{kazhdan2006poisson}.

Consider a point cloud (i.e., a set of data points) observed in a high-dimensional Euclidean space, i.e., in $\mathbb{R}^D$ with $D$ large (the \textit{ambient space} with the \textit{ambient dimension} $D$). As noted in the literature \citep[e.g.,][]{meng_principal_2021, yao2024single, dunson2021inferring}, there are two primary directions in manifold learning from a methodological viewpoint. The first focuses on \textit{nonlinear dimension reduction}, aiming to replace the original high-dimensional data with low-dimensional coordinates. Numerous widely adopted methods exemplify this direction, e.g., ISOMAP \citep{tenenbaum2000global} and locally linear embedding \citep{roweis2000nonlinear, wu2018think}. However, these methods do not estimate the underlying manifold $\mathcal{M}$ as a submanifold of the ambient space $\mathbb{R}^D$. In contrast, the second direction, \textit{manifold reconstruction}, seeks to reconstruct the underlying manifold $\mathcal{M}$ as a submanifold of $\mbR^D$. This direction will be the primary focus of this article. One family of methods in manifold reconstruction is known as \textit{principal manifold estimation} (PME), which can be understood as a nonlinear analogue of linear principal component analysis \citep[PCA;][]{pearson1901liii}. While many PME methods exhibit convincing empirical performance in numerical experiments and have found practical success in applications \citep[e.g.,][]{yue2016parameterization, gao2023k, gao2024bi, zielinski2024longitudinal}, their theoretical foundations remain underdeveloped (as detailed in Section~\ref{section: Contributions and Paper Organization}). Moreover, existing PME methods do not incorporate the topological information latent in data; as a result, they fail to account for the global topology of the underlying manifold and do not take advantage of the tools developed by the topological data analysis (TDA) community \citep[e.g.,][]{fasy2014confidence, bubenik2015statistical, fasy2021tda, roycraft2023bootstrapping, meng2025randomness}. In this article, we address these limitations by developing comprehensive and mathematically rigorous theoretical foundations for PME, while also providing a framework that naturally incorporates tools from TDA. Section~\ref{section: Contributions and Paper Organization} provides a list of our main contributions.

\subsection{Overview of Principal Manifold Estimation}

We first provide an overview of existing works under the PME umbrella that are closely related to our main contributions. These works model the underlying manifold $\mathcal{M}$ as the image of an $\mathbb{R}^D$-valued function $\Bf^*$ satisfying the conditions specified in each work.

The PME framework originated with \textit{principal curves} \citep{Hastie1984PrincipalCurves, hastie1989principal}, a methodology which fits a smooth curve that passes through the ``middle'' of a $D$-dimensional point cloud (e.g., Figure~\ref{fig:proj-curve}) and investigates the \textit{fitting-error functional}
\begin{align}\label{eq: def of the fitting-error functional}
    \mathcal{D}(\boldsymbol{f}) = \mathbb{E} \left\{\Vert \boldsymbol{X}-\Pi_{\Bf}(\boldsymbol{X})\Vert^2\right\},
\end{align}
where $\boldsymbol{X}$ represents data observed in the ambient space $\mathbb{R}^D$, the function $\Bf: \mathfrak{I}\rightarrow\mathbb{R}^D$ is defined on a prespecified interval $\mathfrak{I}\subseteq\mathbb{R}^1$, and $\Pi_{\Bf}(\boldsymbol{X})$ denotes the nearest-point projection of $\boldsymbol{X}$ on the image $\Bf(\mathfrak{I}):=\{\Bf(m): m\in\mathfrak{I}\}$, illustrated in Figure \ref{fig:proj-curve}. The manifold hypothesis posits that data $\boldsymbol{X}$ lie near an underlying manifold $\boldsymbol{f}^*(\mathfrak{I})$, resulting in a small fitting error $\mathcal{D}(\boldsymbol{f}^*)$. A tentative approach is to use a minimizer of $\mathcal{D}(\boldsymbol{f})$ to estimate the underlying manifold, which may correspond to a critical point of $\mathcal{D}(\boldsymbol{f})$ (i.e., a function-valued point where the Gâteaux derivative of $\mathcal{D}(\boldsymbol{f})$ vanishes). In particular, a principal curve may be equivalently defined as the image of a function that is a critical point of $\mathcal{D}(\boldsymbol{f})$ \citep[][Proposition 4]{hastie1989principal}. 

While the principal curve framework resembles a regression model that minimizes fitting error, it does not assign roles to ``predictors'' and ``responses.'' Therefore, this framework is an unsupervised learning method, like linear PCA. Notably, the initial work by \cite{hastie1989principal} opens the door to nonlinear extensions of linear PCA. Following this initial work, the principal curve framework has been further developed by numerous researchers from both theoretical and applied perspectives \citep[e.g.,][]{tibshirani1992principal, banfield1992ice, leblanc1994adaptive, de1999principal, delicado2001another, caffo2008case, ozertem2011locally, gerber2013regularization, yue2016parameterization, kirov2017multiple, lee2020spherical, lee2023robust, magana2026learning}.

However, as \cite{duchamp1996extremal} show by computing the second variation of $\mathcal{D}(\boldsymbol{f})$, this framework has a theoretical flaw: principal curves are never local minima of $\mathcal{D}(\boldsymbol{f})$, while they are critical points of $\mathcal{D}(\boldsymbol{f})$. Specifically, they are saddle points, and $\mathcal{D}(\boldsymbol{f})$ lacks a minimizer. This property stands in contrast to the regression setting and underscores a fundamental distinction between nonlinear supervised and unsupervised learning. To address the nonexistence of a minimizer, \cite{kegl2000learning} introduce a penalty term on curve length. Specifically, they define a principal curve as a minimizer of the following functional and establish its existence rigorously.
\begin{align}\label{eq: the loss utilized by Kegl et al.}
    \mathcal{D}(\boldsymbol{f}) + \lambda \cdot \int\Vert \Bf'(m) \Vert \, dm,
\end{align}
where the first term is the fitting-error functional defined in \eqref{eq: def of the fitting-error functional}, the integral of the derivative in the second term is equal to the length of the curve $\boldsymbol{f}(\mathfrak{I})$, and $\lambda>0$ is a tuning parameter. \cite{smola2001regularized} propose substituting the $L^1$ norm in \eqref{eq: the loss utilized by Kegl et al.} with a Hilbert space norm and replacing the first derivative operator with a projection operator, which leads to the application of reproducing kernel Hilbert spaces \citep[RKHSs;][]{wahba1990spline, chen2025probabilistic}.

Motivated by the approaches of \cite{kegl2000learning} and \cite{smola2001regularized}, \cite{meng_principal_2021} replace the penalty term $\int\Vert \Bf'(m) \Vert \, dm$ in \eqref{eq: the loss utilized by Kegl et al.} with a \textit{roughness} penalty $\int\Vert \Bf''(m) \Vert^2 \, dm$, defined as the integral of the squared second derivative, which is an unsupervised analogue of smoothing splines \citep{wahba1990spline, ma2015efficient, meng2020more}. More importantly, they extend the classical notion of principal curves (intrinsic dimension $d=1$) to \textit{principal manifolds} of intrinsic dimension $d<4$ by minimizing the following penalized functional in an RKHS, subject to appropriate constraints
\begin{align}\label{eq: def of the ME principal manifolds}
    \mathcal{L}_\lambda(\Bf)=\mathbb{E} \left\{\Vert \boldsymbol{X}-\Pi_{\Bf}(\boldsymbol{X})\Vert^2\right\} + \lambda \cdot \int \Vert \nabla^2\Bf(\boldsymbol{m})\Vert^2 \;d\boldsymbol{m},
\end{align}
where $\Bf: \mathbb{R}^d\rightarrow\mathbb{R}^D$ with $d<\min\{4,D\}$, the first term is a higher-dimensional analogue of the fitting error defined in \eqref{eq: def of the fitting-error functional}, and $\nabla^2=(\partial_i\partial_j)_{1\le i,j \le d}$ denotes the Hessian matrix defined using ordinary Euclidean coordinates. Although the framework introduced by \cite{meng_principal_2021} has been applied successfully to both robotics \citep{gao2023k, gao2024bi} and neuroimaging \citep{zielinski2024longitudinal}, its theoretical foundation remains underdeveloped, as will be elaborated in Section \ref{section: Contributions and Paper Organization}.

\begin{figure}[!htbp]
\centering
\begin{tikzpicture}[x=1cm,y=1cm,
  font=\small,
  >=Latex,
  curve/.style={black, line width=1.5pt, line cap=round, line join=round},
  seg/.style={blue!75!black, line width=1.5pt, line cap=round},
  data/.style={circle, fill=orange!85!black, inner sep=0pt, minimum size=7pt},
  proj/.style={circle, fill=red!80!black, inner sep=0pt, minimum size=8.5pt},
  lab/.style={inner sep=1pt}
]

\coordinate (pm) at (7.55,1.55);   
\coordinate (pi) at (8.2,2.1);   
\coordinate (pp) at (9.55,2.55);   

\def\dist{1.45}     
\def\tanlen{0.65}   

\path let \p1=(pp), \p2=(pm), \p3=(pi) in
  \pgfextra{
    \pgfmathsetmacro{\dx}{\x1-\x2}
    \pgfmathsetmacro{\dy}{\y1-\y2}
    \pgfmathsetmacro{\len}{sqrt(\dx*\dx+\dy*\dy)}
    \pgfmathsetmacro{\tx}{\dx/\len}
    \pgfmathsetmacro{\ty}{\dy/\len}
    \pgfmathsetmacro{\nx}{-\ty}
    \pgfmathsetmacro{\ny}{ \tx}
  }
  coordinate (X)  at ($ (pi) + (\nx*\dist,\ny*\dist) $)
  coordinate (Ta) at ($ (pi) + (\tx*\tanlen,\ty*\tanlen) $)
  coordinate (Tb) at ($ (pi) - (\tx*\tanlen,\ty*\tanlen) $);

\draw[curve]
  (0.2,2.0)
    .. controls (0.9,3.4) and (2.0,3.5) .. (2.6,3.4)
    .. controls (3.3,3.3) and (3.9,2.5) .. (4.05,2.2)
    .. controls (4.6,1.2) and (5.5,0.85) .. (5.6,0.8)
    .. controls (6.05,0.6) and (6.8,0.9) .. (7.0,1.07)
    .. controls (7.1,1.1) and (7.4,1.4) .. (pm)
    .. controls (7.80,1.90) and (8.15,2.10) .. (pi)
    .. controls (8.90,2.40) and (9.20,2.50) .. (pp)
    .. controls (10.1,2.63) and (11.2,3.1) .. (12.2,1.0);

\node[lab, black!100] at (2.0,1.95) {$\boldsymbol{f}(\mathfrak I)\subseteq \mathbb{R}^D$};

\foreach \p in {
  (0.35,2.10),
  (1.25,3.55),
  (2.25,3.05),
  (3.80,2.20),
  (5.10,0.7),
  (7.6,1.65),
  (7.10,1.1),
  (9.95,2.8),
  (11.85,1.5)
}{
  \node[data] at \p {};
}

\node[data] at (X) {};
\node[proj] at (pi) {};

\node[lab] at ($(X)+(0.40,0.15)$) {$\boldsymbol{X}$};
\node[lab] at ($(pi)+(1.6,-0.5)$) {$\Pi_{\boldsymbol{f}}(\boldsymbol{X})=\Bf(\Bpi_{\Bf}\left(\boldsymbol{X})\right)$};

\draw[seg] (X) -- (pi);

\draw[black!55, line width=1.0pt, line cap=round] (Tb) -- (Ta);
\pic[draw=black!55, line width=1.0pt, angle radius=8.5pt]
  {right angle = X--pi--Ta};

\draw[decorate, decoration={brace, amplitude=7pt, raise=2pt},
      line width=1.0pt, black!60]
  (pi) -- (X)
  node[midway, xshift=-1.7cm, yshift=-0.20cm, lab, black!100]
  {$\|\boldsymbol{X}-\Pi_{\boldsymbol{f}}(\boldsymbol{X})\|$};

\end{tikzpicture}
\caption{\footnotesize A curve $f(\mathfrak I)$ (black) fitted to a point cloud (orange), represented as the image of a function $\Bf:\mathfrak{I}\rightarrow\mathbb{R}^D$. For an observation \(\boldsymbol X\), the point \(\Pi_{\boldsymbol{f}}(\boldsymbol{X})\) (red) is its nearest-point projection onto the curve, and $\Bpi_{\Bf}(\boldsymbol{X})$ is the \textit{projection index} defined in Section \ref{section: Projection Indices}. The brace indicates the error \(\|\boldsymbol X-\Pi_{\Bf}(\boldsymbol X)\| = \|\boldsymbol X-\Bf(\Bpi_{\Bf}\left(\boldsymbol{X})\right)\|\).}
\label{fig:proj-curve}
\end{figure}

\subsection{Contributions}\label{section: Contributions and Paper Organization}

We use the studies by \cite{meng_principal_2021} and \cite{zielinski2024longitudinal} as illustrative examples to highlight underdeveloped aspects of the PME framework and to demonstrate how the present article addresses these limitations.

\subsubsection{Non-Euclidean Templates}\label{section: Non-Euclidean Template} 

In the PME framework, an underlying manifold is represented as the image of an $\mathbb{R}^D$-valued function $\Bf^*$ defined on a domain $\M$, called the \textit{template}. Many existing PME methods take the template to be a flat Euclidean domain. For example, \cite{hastie1989principal} use an interval such as $\M=[0,1]$ to define principal curves, whereas \cite{meng_principal_2021} use $\M=\mathbb{R}^d$ with $d<\min\{4,D\}$. Although convenient, these choices are restrictive when the underlying manifold has non-Euclidean topology \citep[e.g., the closed surface of a hippocampus,][]{zhang2023lesa}. Shape data analysis suggests choosing a template diffeomorphic to the underlying manifold \citep{srivastava2016functional}. For example, the unit sphere $\M=\mathbb{S}^2$ is a natural template for a closed surface of genus zero \citep[e.g.,][]{kurtek2011elastic}. Extending PME to such templates requires reformulating the penalty in \eqref{eq: def of the ME principal manifolds} to reflect the geometry of $\M$. In particular, the Euclidean Hessian $\nabla^2$ used by \cite{meng_principal_2021} must be replaced by its Riemannian counterpart, and the integration $\int(\cdot)\,d\boldsymbol{m}$ must use the Riemannian volume measure.

Therefore, we extend the PME framework to compact Riemannian templates to accommodate underlying manifolds with non-Euclidean topology. TDA can inform template selection by providing information about the underlying manifold's global topology \citep[e.g.,][]{fasy2014confidence, fasy2021tda}.

\subsubsection{Existence of a Minimizer}\label{section: Existence of Minimizer} 

Following the minimization approach of \cite{kegl2000learning}, \cite{meng_principal_2021} define a principal manifold as the image of a minimizer of the penalized functional in \eqref{eq: def of the ME principal manifolds}. Unlike \cite{kegl2000learning}, however, they do not establish the existence of a minimizer. This leaves a foundational question unresolved in the PME framework developed by \cite{meng_principal_2021} and \cite{zielinski2024longitudinal}.

In this article, we first generalize the functional in \eqref{eq: def of the ME principal manifolds} to accommodate functions defined on a general compact Riemannian manifold $\M$, which need not have Euclidean topology. Then, we define a principal manifold as a minimizer of this generalized functional. Using \textit{Sobolev space} theory \citep{hebey2000nonlinear}, we provide a rigorous proof that such a minimizer exists. This existence guarantee serves as the cornerstone of the generalized PME framework.

\subsubsection{Consistency}\label{section: Introduction, Consistency} 

In the studies of \cite{meng_principal_2021} and \cite{zielinski2024longitudinal}, a principal manifold is first defined at the population level with respect to an underlying distribution $\mathbb{P}$. However, in practice, when only a sample $\{\boldsymbol{X}_i\}_{i\in[N]}\subseteq\mathbb{R}^D$ is observed (where $[N]:=\{1,\ldots,N\}$), the principal manifold is computed at the empirical level using the empirical distribution $\mathbb{P}_N=\frac{1}{N}\sum_{i=1}^N \delta_{\boldsymbol{X}_i}$. This raises a consistency question that remains unanswered in the studies of \cite{meng_principal_2021} and \cite{zielinski2024longitudinal}: \textit{does the principal manifold computed at the empirical level converge to its population-level counterpart as the sample size $N\rightarrow\infty$?} In this article, we answer this question in the affirmative by establishing a rigorous proof of consistency using the \textit{argmax theorem} \citep{van1996weak} and the \textit{Rellich--Kondrachov compact embedding theorem} \citep[Theorem~\ref{thm: Rellich–Kondrachov embedding};][]{hebey2000nonlinear}.

\subsubsection{Convergence of the Iterative Estimation Algorithm}\label{section: Convergence of the Iterative Estimation Algorithm} 

Unlike in the regression setting, there is currently no direct method for computing a minimizer of the functional in \eqref{eq: def of the ME principal manifolds}. \cite{smola2001regularized} suggest an iterative procedure, known as the \textit{projection-adaptation (PA) algorithm}, to approximate such a minimizer numerically. Using this algorithm, \cite{meng_principal_2021} and \cite{zielinski2024longitudinal} conduct a range of numerical experiments and observe empirical convergence in their numerical experiments. However, a rigorous proof of the convergence of the PA algorithm remains unavailable. It is worth noting that the PA algorithm generalizes the \textit{principal curve algorithm} \citep{Hastie1984PrincipalCurves, hastie1989principal, meng_principal_2021}, for which a theoretical proof of convergence is likewise still unavailable.

As discussed in Sections~\ref{section: Non-Euclidean Template} and \ref{section: Existence of Minimizer}, this article extends the functional in \eqref{eq: def of the ME principal manifolds} to a general template manifold that need not have Euclidean topology. We then prove that, under suitable initialization and regularity conditions, the fitted manifolds produced by the PA algorithm converge to the image of a minimizer of the generalized functional.

\subsubsection{Invariance under Template Isometries}\label{section: Invariance under Template Isometries}

A fitted manifold, presented as the image $\Bf^*(\M)$ of a fitted function $\Bf^*$, may admit multiple reparameterizations. \textit{Riemannian isometries} of the template $\M$ \citep{myers1939group} preserve the geometry of $\M$ and provide a natural class of reparameterizations $\phi$ that leave the fitted manifold unchanged, i.e., $\Bf^*\circ\phi(\M)=\Bf^*(\M)$. However, \cite{meng_principal_2021} and \cite{zielinski2024longitudinal} do not account for these reparameterizations. We address this gap by treating fitted functions related by Riemannian isometries as equivalent. All results described in Sections~\ref{section: Non-Euclidean Template}--\ref{section: Convergence of the Iterative Estimation Algorithm} are invariant under these isometries.

In shape analysis, different reparameterizations of the same geometric object are treated as equivalent, and comparisons are made after optimal alignment \citep{srivastava2016functional}. Our framework applies this principle to reparameterizations induced by Riemannian isometries of the template $\M$, which leave the fitted manifold and its penalty unchanged.

\subsubsection{Bayesian Interpretation} 

\cite{meng_principal_2021} and \cite{zielinski2024longitudinal} examine the PA algorithm solely from a frequentist perspective. We provide a Bayesian interpretation under which PA can be viewed as a variant of the recent \textit{manifold reconstruction via Gaussian processes} (MrGap) algorithm \citep{dunson2021inferring}. This interpretation also provides a basis for future research on alternative forms of regularization within the PME framework.

\subsubsection{Model Complexity Selection}\label{section: Model Complexity Selection} 

\cite{meng_principal_2021} do not fully address the selection of the tuning parameter $\lambda$ in \eqref{eq: def of the ME principal manifolds}. Section~\ref{section: Numerical Nonidentifiability} shows that geometric features of the underlying manifold need not be \textit{numerically identifiable}: arbitrarily close data-generating distributions can correspond to manifolds with substantially different geometries. Consequently, recovery of the ``true'' manifold does not provide a stable criterion for selecting $\lambda$. Therefore, we select $\lambda$ according to the fitted manifold's performance in a prespecified downstream task. Section~\ref{section: Selection Based on Downstream Performance} illustrates this approach using extrinsic local regression on manifold-valued data \citep{lin2017extrinsic}.

\subsubsection{Geometric Interpretation of the Penalty} 

The penalty term in \eqref{eq: def of the ME principal manifolds}, used by \cite{meng_principal_2021} and \cite{zielinski2024longitudinal}, is motivated by thin-plate splines \citep{Duchon1977Splines, wahba1990spline, steinke2010nonparametric}. However, these studies describe this term simply as the ``roughness'' of a fitted manifold without providing a precise geometric interpretation. Consequently, it remains unclear which geometric features of the fitted manifold this term penalizes. In this article, we generalize this penalty to account for the non-Euclidean geometry of a general template $\M$. We interpret the resulting penalty in terms of the \textit{second fundamental form} and \textit{totally geodesic maps} \citep{vilms1970totally}. We also examine how it relates to the curve-length penalty introduced by \cite{kegl2000learning} and the Laplace--Beltrami operator on $\M$. These analyses clarify the geometric role of regularization in the PME framework.

\subsubsection{Notation and Paper Organization} Throughout this article, we adopt the following standard notation: (i) for any positive integer $I$, denote $[I]:=\{1,\ldots,I\}$; (ii) $\overset{\mathrm{iid}}{\sim}$ indicates that observations are independent and identically distributed (iid) draws from a distribution.

Section~\ref{section: Mathematical Preliminaries} introduces mathematical preliminaries. Section~\ref{section: Principal Manifold Estimation} formulates PME on Riemannian templates and establishes existence of minimizers, invariance under Riemannian isometries, and PA convergence. Section~\ref{section: Empirical Version} presents consistency, the empirical PA algorithm, and its Bayesian interpretation. Section~\ref{section: Tuning Parameter Selection} examines numerical nonidentifiability and selects $\lambda$ through downstream performance. Section~\ref{section: discussions and future research} concludes with future research directions. Appendices~\ref{appendix: Mathematical Preparations}--\ref{appendix: Data-Generating Mechanisms for Numerical Experiments} collect notation and supplementary material on Poisson surface reconstruction, penalty geometry, RKHS methods, PA implementation, nonidentifiability, latent-variable tuning, additional simulations, mathematical proofs, and data-generating mechanisms.

\section{Mathematical Preliminaries}\label{section: Mathematical Preliminaries}

This section collects the mathematical preliminaries used throughout the article. We also explain the role of TDA \citep{fasy2014confidence, fasy2021tda} in the PME framework and extend the notion of the \textit{projection index} previously introduced by \cite{Hastie1984PrincipalCurves}, \cite{hastie1989principal}, and \cite{meng_principal_2021}.

\subsection{Riemannian Manifolds and Their Topologies}\label{section: Riemannian Manifolds and Their Topology}

\subsubsection{Template Manifolds} In the PME framework, the underlying manifold of interest is modeled as the image of an $\mathbb{R}^D$-valued map defined on a ``template.'' For instance, \cite{hastie1989principal} primarily take the compact interval $[0,1]$ as the template (e.g., the curve in Figure~\ref{fig:proj-curve} is the image of such a map $\boldsymbol{f}:[0,1]\rightarrow\mathbb{R}^2$), and \cite{meng_principal_2021} use $\mathbb{R}^d$ with $d<\min\{4,D\}$ as the template. In this article, we treat a template as a general Riemannian manifold $(\mathfrak{M},g)$, which we call a \emph{template manifold}. Specifically, we assume the following throughout the article.
\begin{assumption}\label{assumption: assumption on the template manifold}
    The template manifold $(\mathfrak{M},g)$ is compact, orientable, connected, and smooth. Additionally, its boundary $\partial\M$ is either empty or smooth.
\end{assumption}
\noindent In this article, the Riemannian metric $g$ may be selected arbitrarily, provided that Assumption \ref{assumption: assumption on the template manifold} is satisfied. The book by \cite{lee2018riemannian} serves as a comprehensive reference for the foundational concepts related to Riemannian manifolds. The compactness requirement in Assumption \ref{assumption: assumption on the template manifold} is typically appropriate in applications \citep[e.g.,][]{srivastava2016functional}. The orientability assumption on the template manifold helps avoid pathological cases in which TDA descriptors \citep[e.g., estimated Betti numbers,][]{fasy2014confidence, fasy2021tda} fail to distinguish certain manifolds, as discussed later in Section \ref{section: Topology}. The smoothness assumption is standard in the literature \citep[e.g.,][]{hastie1989principal, wu2018think, fefferman2025fitting, dunson2021inferring}, and it allows us to study the geometry of the resulting fitted manifold. Lastly, the assumption of connectedness can be verified using TDA \citep{fasy2014confidence, fasy2021tda}, as discussed in Section \ref{section: Topology}.

In manifold learning tasks, the template manifold is specified by the user (e.g., the unit interval $[0,1]$ employed by \cite{hastie1989principal}) and may potentially be informed by TDA descriptors (as discussed later in Section \ref{section: Topology}). Beyond the unit interval $[0,1]$, common choices for template manifolds include the unit circle $\mathbb{S}^1$ and the unit sphere $\mathbb{S}^2$.

\subsubsection{Topology}\label{section: Topology} 
Suppose $\mathcal{M}$ is a submanifold of $\mathbb{R}^D$ representing the underlying manifold of interest. That is, the data points observed in $\mathbb{R}^D$ lie in a small neighborhood of $\mathcal{M}$. Throughout this article, we work under the following assumption:
\begin{assumption}\label{assumption: topological assumption}
    $\mathcal{M}$ is diffeomorphic to the prespecified template manifold $\mathfrak{M}$.
\end{assumption}
\noindent Assumption \ref{assumption: topological assumption} is often artificially imposed in the manifold learning literature rather than inferred from observed data \citep[e.g.,][]{hastie1989principal}. With advances in TDA over the past two decades \citep[e.g.,][]{fasy2014confidence}, in many cases, observed data can inform the choice of topology and the selection of a corresponding template manifold under Assumptions \ref{assumption: assumption on the template manifold} and \ref{assumption: topological assumption}. Specifically, the \texttt{R} package \texttt{TDA} \citep{fasy2021tda}, in conjunction with observed data, can be implemented to obtain data-based estimates of the Betti numbers $\{\beta_k(\mathcal{M})\}_{k\ge0}$ of the underlying manifold $\mathcal{M}$. Informally, \(\beta_k(\mathcal{M})\) is the number of \(k\)-dimensional homology features of \(\mathcal{M}\). Based on the classification of curves and surfaces \citep[e.g.,][]{lee2000introduction}, one- and two-dimensional template manifolds are selected according to the estimated Betti numbers, as described below.
\begin{enumerate}
    \item $\beta_0(\mathcal{M})=$ the number of connected components of $\mathcal{M}$, and $\beta_0(\mathcal{M})=1$ means that $\mathcal{M}$ is connected. When $\beta_0(\mathcal{M})>1$, we may apply clustering methods to identify each connected component of $\mathcal{M}$ and focus on estimating individual connected components. Hence, we assume $\beta_0(\mathcal{M})=1$ hereafter, which is consistent with Assumptions \ref{assumption: assumption on the template manifold} and \ref{assumption: topological assumption}. 

    \item When $\dim(\mathcal{M})=1$, the first Betti number $\beta_1(\mathcal{M})=$ the number of circle components. Hence, $\beta_1(\mathcal{M})=0$ indicates an interval-type template, which suggests the choice $\mathfrak{M}=[0,1]$, and $\beta_1(\mathcal{M})=1$ indicates a circle-type template, which suggests the choice $\mathfrak{M}=\mathbb{S}^1$. Figure \ref{fig:interval_S1_full} provides an example. For the point cloud (black dots) in panel (b), its persistence diagram \citep[PD;][]{edelsbrunner2010computational} is shown in panel (a). The PD indicates that the manifold $\mathcal{M}$ underlying the point cloud has no significant one-dimensional homology feature, suggesting that \(\beta_1(\mathcal{M})=0\) \citep[the notion of significance is explained by][]{fasy2014confidence}. Accordingly, we choose \([0,1]\) as the template for fitting the point cloud. By contrast, the PD of the point cloud in panel (e) exhibits a significant one-dimensional homology feature, indicating that \(\beta_1(\mathcal{M})=1\). Therefore, we choose \(\mathbb{S}^1\) as the template for fitting the points in panel (e).

    \item Suppose $\dim(\mathcal{M})=2$ and $\mathcal{M}$ is closed \citep[i.e., compact and without boundary, a common assumption in the literature, e.g.,][]{wu2018think, dunson2021spectral, dunson2022graph, dunson2021inferring}.  By Assumptions \ref{assumption: assumption on the template manifold} and \ref{assumption: topological assumption}, the orientability of $\mcM$ implies $\beta_2(\mathcal{M})=1$ and $\beta_1(\mathcal{M})=$ twice the genus of $\mathcal{M}$ (intuitively, the number of ``holes'' of $\mathcal{M}$). Hence, $\beta_1(\mathcal{M})=0$ suggests $\mathfrak{M}=$ the unit sphere $\mathbb{S}^2$ (no hole), and $\beta_1(\mathcal{M})=2$ suggests the torus (one hole). Figure \ref{fig:sph_results_panel_8} provides an illustrative example. Panels (iv) and (xii) of Figure \ref{fig:sph_results_panel_8} display the PDs for the star- and cashew-shaped point clouds, respectively, shown as black dots in the figure. In both PDs, no one-dimensional homological feature is significant (visually, there is no red triangle in each PD), suggesting that \(\beta_1(\mathcal{M})=0\). Accordingly, we choose \(\mathbb{S}^2\) as the template for fitting both point clouds.
\end{enumerate}
In general, Betti numbers alone are insufficient to determine the topology of a three-dimensional manifold. The classification of compact three-dimensional manifolds up to homeomorphism is a deep area of pure mathematics \citep[e.g.,][]{morgan2007ricci}.
\begin{figure}[ht]
    \centering
    \includegraphics[scale=0.62]{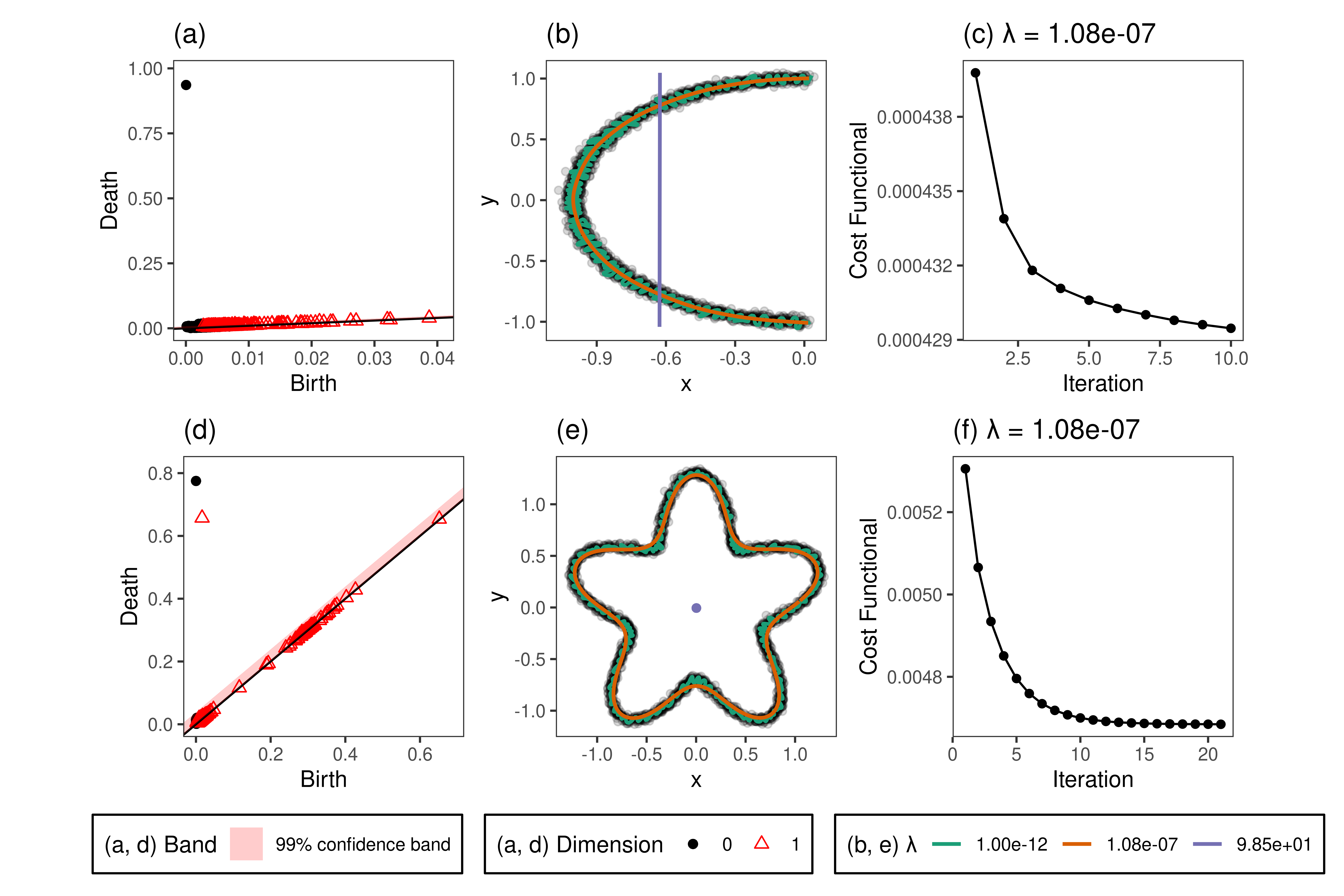}
    \vspace{-10pt}
    \caption{\footnotesize The point clouds in panels (b,e) are generated via the mechanisms described in Appendices \ref{appendix: Half Circle Point Cloud} and \ref{appendix: Boundary of a Flower/Star}, respectively. The persistence diagrams \citep[PDs;][]{fasy2014confidence} corresponding to these two point clouds are shown in panels (a,d), respectively. The significant homology features in the PDs support the intended template topologies in these simulated examples. The curves displayed in panels (b,e) are fitted using the PA algorithm (Algorithm~\ref{alg: empirical pme}) and correspond to excessively small, moderate, and excessively large $\lambda$ values. These numerical results validate Theorem \ref{thm: PME with lambda=infty}. The nearly point-like fit for the excessively large \(\lambda\) in panel (e) is further explained by the curve-length inequality in \eqref{eq: the dominating inequality related to Kegl et al.}; see Section~\ref{section: Curve Length Penalty}. Panels (c,f) display the nonincreasing values of the cost functional \(\{\mathcal{L}_{N,\lambda}(\Bf^{(n)}_{N,\lambda})\}_{n\in\mathbb{N}}\) versus the number of iterations $n$, where the penalty term $\Vert \nabla{}^2 \Bf^{(n)}_{N,\lambda} \Vert^2_{L^2(\M)}$ is computed using Lemma \ref{lemma:  closed form penalty circle}. Theorem~\ref{thm: the convergence theorem of the core iterative algorithm} implies that $\mathcal{L}_{N,\lambda}(\Bf^{(n)}_{N,\lambda})$ converges to $\mathcal{L}_{N,\lambda}(\Bf^{*}_{N,\lambda})$, as the number of iterations $n\rightarrow\infty$, under the regularity conditions specified in Theorem~\ref{thm: the convergence theorem of the core iterative algorithm}, where $\Bf^*_{N,\lambda}\in\argmin_{\Bf\in\mathscr{F}} \,\mathcal{L}_{N,\lambda}(\Bf)$.}
    \label{fig:interval_S1_full}
\end{figure}

\subsubsection{Norms} We introduce the pointwise and $L^2$ norms of tensor fields in preparation for defining Sobolev spaces. Let $T$ be a $(0,2)$-tensor field on the template manifold $(\M,g)$, i.e., $T \in \Gamma(T^*\M \otimes T^*\M)$, where $\Gamma(\cdot)$ denotes the space of smooth sections of a tensor bundle.
The pointwise norm of $T$ with respect to the Riemannian metric $g$ is defined by
\begin{align}\label{eq: vert g norm using bases}
    \lvert T(\boldsymbol{m})\rvert_g^2 := \sum_{i,k=1}^d \bigl(T(e_i,e_k)\bigr)^2 \ \ \text{ for any }\boldsymbol{m}\in\mathfrak{M},
\end{align}
where $\{e_k\}_{k\in[d]}$ is any $g$-orthonormal basis of the tangent space $T_{\boldsymbol{m}}\M$ at $\boldsymbol{m}$. Equivalently, in local coordinates with components $T_{ik}$ and inverse metric $(g^{ik})$, 
\begin{align}\label{eq: vert cdot vert g}
    \lvert T(\boldsymbol{m})\rvert_g^2 = \sum_{i, k, a, b} g^{ia}(\boldsymbol{m})\,g^{kb}(\boldsymbol{m})\,T_{ik}(\boldsymbol{m})\,T_{ab}(\boldsymbol{m}) \quad \text{ for all }\boldsymbol{m}\in\M.
\end{align}
More generally, let $\boldsymbol{J} = (J_1,\dots,J_D) \in \bigl(\Gamma(T^*\M \otimes T^*\M)\bigr)^D$, i.e., each component $J_j \in \Gamma(T^*\M \otimes T^*\M)$, where $D\in\mathbb{N}$. The pointwise norm of $\boldsymbol{J}$ is defined by $\lvert \boldsymbol{J}(\boldsymbol{m})\rvert_g^2
:= \sum_{j=1}^D \lvert J_j(\boldsymbol{m})\rvert_g^2$. The corresponding $L^2$-norm is
\begin{align}\label{eq: def of L2 norm of a (0,2) tensor}
    \lVert \boldsymbol{J}\rVert_{L^2(\M)}^2
:= \int_\M \lvert \boldsymbol{J}(\boldsymbol{m})\rvert_g^2 \,\,d\vol_g(\boldsymbol{m}),
\end{align}
where $d\vol_g$ denotes the Riemannian volume measure induced by the metric $g$ \citep[][Proposition 2.41]{lee2018riemannian}. Further information regarding the introduced norms is available in \cite{hebey2000nonlinear} and \cite{lee2018riemannian}.

\subsection{Sobolev Spaces on Riemannian Manifolds}\label{section: Sobolev Spaces on Riemannian Manifolds}

Sobolev spaces, initially defined on Euclidean spaces, are motivated by partial differential equations \citep[e.g.,][]{evans1998pde, li2016global}. Because many Sobolev spaces are also RKHSs, their theory has found applications in statistics \citep{wahba1990spline}. Advances in geometric analysis have spurred the development of Sobolev spaces on Riemannian manifolds \citep{hebey2000nonlinear}. In this section, we provide the preliminaries of Sobolev spaces on compact Riemannian manifolds that are necessary for our proposed theoretical foundations for PME.

Hereafter, $\nabla$ denotes the Levi--Civita connection of the metric $g$ on $\M$ \citep[][Section 1.2]{hebey2000nonlinear}. Let $\boldsymbol{f} = (f_1,\dots,f_D) : \M \to \mathbb{R}^D$ be a smooth map. Naturally, its $L^2$ norm is defined as 
\begin{align}\label{eq: L2 norm on M}
    \Vert \Bf \Vert_{L^2(\M)}^2 := \sum_{j=1}^D \int_{\M} \vert f_j(\boldsymbol{m}) \vert^2  \,\,d\vol_g(\boldsymbol{m}).
\end{align}
For each scalar component $f_j:\M\to\R$, its \textit{Hessian} is the symmetric $(0,2)$-tensor defined by
\begin{align}\label{eq: def of Hess}
    \nabla{}^2 f_j(u,v) := u(v f_j) - (\nabla_u v) f_j \ \ \text{ for all } u,v\in \Gamma(T\M).
\end{align}
Define the vector-valued \textit{Hessian} of $\boldsymbol{f}$ by stacking the components, i.e.,
\begin{align}\label{eq: def of nabla f}
\nabla{}^2 \boldsymbol{f} &:= \big(\nabla{}^2 f_1,\dots,\nabla{}^2 f_D\big) \in \bigl(\Gamma(T^*\M \otimes T^*\M)\bigr)^D,
\end{align}
By \eqref{eq: def of L2 norm of a (0,2) tensor}, its $L^2$ norm is defined by
\begin{align}\label{eq: L2 norm of Hf}
\Vert \nabla{}^2 \boldsymbol{f} \Vert^2_{L^2(\M)} &:= \int_{\M} | \nabla{}^2 \boldsymbol{f}(\boldsymbol{m}) |^2_g \,\, d\vol_g(\boldsymbol{m}) = \sum_{j=1}^D\int_{\M} | \nabla{}^2 f_j(\boldsymbol{m}) |^2_g \,\, d\vol_g(\boldsymbol{m}).
\end{align}
Then, we define the Sobolev norm $\Vert \cdot\Vert_{H^2(\M)}$ by $\Vert \Bf \Vert_{H^2(\M)}:= ( \Vert \Bf \Vert_{L^2(\M)}^2 + \Vert \nabla{}^2 \boldsymbol{f} \Vert_{L^2(\M)}^2 )^{1/2}$. Appendix~\ref{section: Equivalent Sobolev Norms} provides the justification for omitting the first-order derivative term $\Vert \nabla\Bf \Vert_{L^2(\M)}^2$.

The Sobolev space $H^2(\M;\,\mathbb{R}^D)$ is the completion of $C^{\infty}(\mathfrak{M};\,\mathbb{R}^D)$ with respect to $\Vert \cdot\Vert_{H^2(\M)}$, where $C^{\infty}(\mathfrak{M};\,\mathbb{R}^D)$ denotes the collection of all smooth $\M \to \mathbb{R}^D$ maps \citep[][Definition 2.1]{hebey2000nonlinear}. The Sobolev space $H^2(\M;\mathbb{R}^1)$ of scalar-valued functions is defined analogously; informally, it is the special case $D=1$. Throughout this article, by a slight abuse of notation, we use $H^2(\M)$ to denote either $H^2(\M;\mathbb{R}^D)$ or $H^2(\M;\mathbb{R}^1)$, whenever no confusion is likely to arise. The norm $\Vert \cdot\Vert_{H^2(\M)}$ is formally defined using the Riemannian metric $g$, as shown in \eqref{eq: L2 norm on M} and \eqref{eq: L2 norm of Hf}. Nevertheless, since $\M$ is compact, the Sobolev space $H^2(\M)$, considered as a set, does not depend on the choice of metric $g$ \citep[][Proposition 2.2]{hebey2000nonlinear}.

We next introduce the space of continuous functions and its norm. Let $C(\M;\mathbb{R}^D)$ denote the set of all continuous maps from $\M$ into $\mathbb{R}^D$. For each $\Bf\in C(\M;\mathbb{R}^D)$, define its supremum norm by $\|\Bf\|_{C(\M;\mathbb{R}^D)}:=\max_{\boldsymbol{m}\in\M}\|\Bf(\boldsymbol{m})\|$. Similarly, let $C(\M;\mathbb{R}^1)$ denote the space of all scalar-valued continuous functions on $\M$, equipped with the corresponding supremum norm. By a slight abuse of notation, we write $C(\M)$ to denote either $C(\M;\mathbb{R}^D)$ or $C(\M;\mathbb{R})$.

Hereafter, we assume that $d:=\dim(\M)$, the \textit{intrinsic dimension}, is less than four, i.e., $d<\min\{4,D\}$. This dimension restriction is motivated in part by the need to avoid exotic smooth structures \citep[e.g., Milnor’s exotic spheres,][]{milnor1956manifolds}. Specifically, according to Moise's theorem \citep{moise1952affine}, every topological manifold of dimension less than four admits a unique smooth structure up to diffeomorphism, thereby ruling out exotic smooth structures in these dimensions. More importantly, when the intrinsic dimension is less than four, we have the following theorem.
\begin{theorem}[Rellich–Kondrachov theorem]\label{thm: Rellich–Kondrachov embedding}
    Let $(\mathfrak{M},g)$ be a compact $d$-dimensional Riemannian manifold with $d<4$. Suppose its boundary $\partial \mathfrak{M}$ is either empty (i.e., $\mathfrak{M}$ is a closed manifold in this case) or smooth. Then, $H^2(\mathfrak{M}) \subseteq C(\mathfrak{M})$, and the inclusion map $\iota:\, H^2(\mathfrak{M}) \rightarrow C(\mathfrak{M})$ is compact with respect to the topologies of $H^2(\mathfrak{M})$ and $C(\mathfrak{M})$.
\end{theorem}
\noindent Theorem \ref{thm: Rellich–Kondrachov embedding} applies both to the scalar-valued space $H^2(\mathfrak{M})=H^2(\mathfrak{M};\mathbb{R}^1)$ and to the vector-valued space $H^2(\mathfrak{M})=H^2(\mathfrak{M};\mathbb{R}^D)$. It is implicitly stated and proved by \cite{hebey2000nonlinear}. For the reader’s convenience, we provide a complete and explicit proof of Theorem \ref{thm: Rellich–Kondrachov embedding} in Appendix \ref{Appendix: Proofs}. The compactness of the inclusion map $\iota$ in Theorem \ref{thm: Rellich–Kondrachov embedding} plays a critical role in our proposed theoretical foundations for PME and will be utilized in the proofs of several key theorems throughout this article. Furthermore, the compactness of the inclusion map (equivalently, Lemma~\ref{lemma: general Sobolev inequalities on manifolds}) leads to the inequality $\Vert \Bf \Vert_{C(\mathfrak{M})} \le C\cdot \Vert \Bf \Vert_{H^2(\M)}$ for all $\Bf\in H^2(\mathfrak{M})$, where $C$ is a constant, thereby demonstrating that $H^2(\mathfrak{M})$ is an RKHS. The RKHS structure of $H^2(\mathfrak{M})$ enables the application of the \textit{representer theorem} \citep[][Theorem 1.3.1]{wahba1990spline}, which is crucial for the estimation of principal manifolds (see Section \ref{section: Empirical Version} and Appendix \ref{appendix: Reproducing Kernel Hilbert Spaces}).

\subsection{Projection Indices}\label{section: Projection Indices}

In this section, we generalize the notion of the \textit{projection index}, first introduced by \cite{Hastie1984PrincipalCurves}. Given a point $\boldsymbol{x}\in\mathbb{R}^D$ and a continuous map $\Bf:\M\rightarrow\mathbb{R}^D$, the projection index $\Bpi_{\Bf}(\boldsymbol{x})$ may be viewed informally as a point $\tilde{\boldsymbol{m}}\in\mathfrak{M}$ for which $\Bf(\tilde{\boldsymbol{m}})$ is nearest to $\boldsymbol{x}$, as illustrated in Figure~\ref{fig:proj-curve}. Let $\Bf\in H^2(\M)$, which is continuous by Theorem \ref{thm: Rellich–Kondrachov embedding}. Then, the following set of arguments of the minima is nonempty and compact:
\begin{align*}
    \argmin_{\boldsymbol{m}\in\mathfrak{M}} \Vert \boldsymbol{x}-\Bf(\boldsymbol{m})\Vert :=\left\{ \boldsymbol{m}\in\mathfrak{M}:\, \Vert \boldsymbol{x}-\Bf(\boldsymbol{m})\Vert = \min_{\boldsymbol{m}'\in\mathfrak{M}} \Vert \boldsymbol{x}-\Bf(\boldsymbol{m}')\Vert \right\}.
\end{align*}
Note that this set may contain multiple elements (e.g., see Figure \ref{fig:ellipse-medial-axis}). When the template manifold $\M$ is the one-dimensional compact interval $[0,1]$, \cite{Hastie1984PrincipalCurves} and \cite{hastie1989principal} define the projection index $\Bpi_{\Bf}: \R^D\rightarrow [0,1]$ of $\Bf$ by 
\begin{align}\label{eq: projection index defined by HS}
    \Bpi_{\Bf}(\boldsymbol{x}):=\max\left\{\argmin_{m\in[0,1]} \Vert \boldsymbol{x}-\Bf(m)\Vert\right\}.
\end{align}
Importantly, \cite{Hastie1984PrincipalCurves} shows that $\Bpi_{\Bf}$ is a measurable function of $\boldsymbol{x}$ \citep[][Theorem~4.1]{Hastie1984PrincipalCurves}. \cite{meng_principal_2021} extend the projection index $\Bpi_{\Bf}$ to cases where the intrinsic dimension of the template manifold is greater than one and demonstrate the associated measurability. When a data point is represented as an $\mathbb{R}^D$-valued random variable $\boldsymbol{X}$, the measurability of the projection index $\Bpi_{\Bf}$ guarantees that $\Bpi_{\Bf}(\boldsymbol{X})$ is also a random variable in the measure-theoretic sense. This property provides a theoretical foundation for certain statistical summaries in applied contexts, such as the parameterization of manifold-like white matter tracts constructed by \cite{yue2016parameterization}.

In the PME literature, it is sufficient for the projection index $\Bpi_{\Bf}$ of $\Bf$ to satisfy only two properties: (i) $\Bpi_{\Bf}(\boldsymbol{x})\in \argmin_{\boldsymbol{m}\in\mathfrak{M}} \Vert \boldsymbol{x}-\Bf(\boldsymbol{m})\Vert$ and (ii) $\Bpi_{\Bf}:\mathbb{R}^D\rightarrow\M$ is measurable, regardless of its specific definition. Therefore, the technical definition of $\Bpi_{\Bf}$ via (sequential) maximization by \cite{Hastie1984PrincipalCurves}, \cite{hastie1989principal}, and \cite{meng_principal_2021} (e.g., \eqref{eq: projection index defined by HS}) is not necessary and may preclude other valid constructions of $\Bpi_{\Bf}$. Furthermore, their definitions rely heavily on the geometric structures of their selected template manifolds (e.g., $[0,1]$) and lack generalizability. Accordingly, we propose a generalized definition of the projection index $\Bpi_{\Bf}$ using the following lemma.
\begin{lemma}\label{lemma: Borel measurable selection of nearest-point}
    Let $\Bf\in C(\mathfrak{M};\,\mathbb{R}^D)$. Define a set-valued function $\boldsymbol{\Psi}:\,\boldsymbol{x} \mapsto \boldsymbol{\Psi}(\boldsymbol{x}) := \argmin_{\boldsymbol{m}\in\mathfrak{M}} \Vert \boldsymbol{x}-\Bf(\boldsymbol{m})\Vert$. Then, $\boldsymbol{\Psi}$ admits a Borel measurable selection
$\Bpi_{\Bf}:\mathbb{R}^D\to\mathfrak{M}$ with $\Bpi_{\Bf}(\boldsymbol{x})\in\boldsymbol{\Psi}(\boldsymbol{x})$
for all $\boldsymbol{x}\in\mathbb{R}^D$. That is, the Borel measurable function $\Bpi_{\Bf}$ satisfies
\begin{align}\label{eq: Borel measurable selection of nearest-point}
     \Vert \boldsymbol{x}-\Bf\left(\Bpi_{\Bf}(\boldsymbol{x})\right)\Vert = \min_{\boldsymbol{m}\in\mathfrak{M}} \Vert \boldsymbol{x}-\Bf(\boldsymbol{m})\Vert \ \ \text{ for all }\boldsymbol{x}\in\mathbb{R}^D.
\end{align}
\end{lemma}
\noindent The identity \eqref{eq: Borel measurable selection of nearest-point} is independent of the chosen measurable selection. Lemma \ref{lemma: Borel measurable selection of nearest-point} is derived from the \textit{Kuratowski--Ryll-Nardzewski measurable selection theorem} \citep[][Theorem 5.2.1]{srivastava1998course}. Building on Lemma \ref{lemma: Borel measurable selection of nearest-point}, we now define the projection indices as follows.
\begin{definition}\label{def: projection index}
    Any Borel measurable function $\Bpi_{\Bf}: \mathbb{R}^D\to\mathfrak{M}$ satisfying \eqref{eq: Borel measurable selection of nearest-point} is said to be a projection index of $\Bf$.
\end{definition}
\noindent Definition \ref{def: projection index} generalizes the projection indices constructed by \cite{Hastie1984PrincipalCurves}, \cite{hastie1989principal}, and \cite{meng_principal_2021}.

\begin{figure}[ht]
\centering
\begin{tikzpicture}[line cap=round,line join=round,scale=1]

\pgfmathsetmacro{\a}{3.0}   
\pgfmathsetmacro{\b}{1.5}   

\pgfmathsetmacro{\cusp}{(\a*\a-\b*\b)/\a}

\draw[blue!70!black, very thick] (0,0) ellipse [x radius=\a, y radius=\b];

\draw[red!70!black, very thick] (-\cusp,0) -- (\cusp,0);

\node[blue!70!black] at (0,\b+0.45) {ellipse};

\node[red!70!black] at (0,-0.35) {medial axis};

\draw[->, gray!70] (0,0) -- (\a,0) node[midway,below] {$a$};
\draw[->, gray!70] (0,0) -- (0,\b) node[midway,left] {$b$};

\draw[red, very thick, fill=white] (-\cusp,0) circle (1.8pt);
\draw[red, very thick, fill=white] (\cusp,0) circle (1.8pt);
\node[black] at (-\cusp,0.35) {\footnotesize $-\frac{a^2-b^2}{a}$};
\node[black] at (\cusp,0.35) {\footnotesize $\frac{a^2-b^2}{a}$};

\end{tikzpicture}
\caption{\footnotesize The blue curve is an ellipse with semi-axes $a>b$. The red open segment is its medial axis \citep{blum1967transformation}, i.e., the set of points that admit more than one nearest point on the ellipse; its excluded boundary points occur at $x=\pm (a^{2}-b^{2})/a$, where the nearest point is unique. Let $\M=\mathbb{S}^1$, and let $\Bf:\M\rightarrow\mathbb{R}^2$ injectively parametrize the ellipse. Outside the medial axis, $\argmin_{\boldsymbol{m}\in\M} \Vert \boldsymbol{x}-\Bf(\boldsymbol{m})\Vert$ is then a singleton, whose unique element is the projection index $\Bpi_{\Bf}(\boldsymbol{x})$. In contrast, each point $\boldsymbol{x}$ on the medial axis has at least two closest points on the ellipse.}\label{fig:ellipse-medial-axis}
\end{figure}

\section{Principal Manifold Estimation}\label{section: Principal Manifold Estimation}

This section presents the core theoretical foundations of the PME framework. Suppose the observed $\mathbb{R}^D$-valued data $\boldsymbol{X}$ are generated from a distribution $\mathbb{P}$ supported in the vicinity of an underlying manifold $\mathcal{M}$ (see the ``manifold hypothesis'' in Section \ref{Introduction}), where $\mathcal{M}$ is a submanifold of $\mathbb{R}^D$. Without loss of generality for most applications, we assume that the support $\mathrm{supp}(\mathbb{P})$ of the distribution $\mathbb{P}$ is compact. Let $(\M, g)$ denote a prespecified template manifold with $d=\dim(\M)<\min\{4, D\}$, which may be selected based on the TDA discussion in Section \ref{section: Topology}. The template and underlying manifolds are assumed to satisfy Assumptions \ref{assumption: assumption on the template manifold} and \ref{assumption: topological assumption}. The PME framework \citep[e.g.,][]{kegl2000learning, meng_principal_2021, zielinski2024longitudinal} represents the underlying manifold $\mathcal{M}$ by regularized fitting in two steps: first, it constructs a collection of manifolds indexed by a tuning parameter; second, it selects an appropriate manifold from this collection for the intended downstream task. Without additional identifying restrictions, the selected fit is interpreted as a regularized manifold representation of the data distribution $\mathbb{P}$ rather than as a uniquely recovered underlying manifold $\mathcal{M}$. This section addresses the first step. Section \ref{section: Tuning Parameter Selection} discusses tuning-parameter selection.

\subsection{Principal Manifolds via Minimization}\label{section: Manifolds via Minimization}

The PME framework constructs the collection of manifolds $\{\widehat{\mathcal{M}}_\lambda\}_{\lambda>0}$ from a family of functions satisfying some regularity properties. In this article, we construct the manifold collection using a family of Sobolev functions $\{\boldsymbol{f}^*_\lambda\}_{\lambda>0} \subseteq H^2(\mathfrak{M};\,\mathbb{R}^D)$. Specifically, for each $\lambda>0$, we define $\widehat{\mathcal{M}}_\lambda:=\Bf^*_\lambda(\M)=\{\boldsymbol{f}^*_\lambda(\boldsymbol{m}):\boldsymbol{m}\in\M\}$ to be a \textit{principal manifold}, where the function $\Bf^*_\lambda$ is given by
\begin{align}\label{eq: the core min problem}
\begin{aligned}
    & \Bf_{\lambda}^* \in \argmin_{\Bf\in\mathscr{F}} \,\mathcal{L}_\lambda(\Bf), \quad \text{where } \mathscr{F}:=\left\{\Bf\in H^2(\mathfrak{M};\,\mathbb{R}^D):\,\Vert \Bf\Vert_{C(\mathfrak{M};\,\mathbb{R}^D)}\le \Upsilon \right\},\\
    &\text{and } \mathcal{L}_{\lambda}(\Bf):= \mathbb{E} \left\{\Vert \boldsymbol{X}-\Bf\left(\Bpi_{\Bf}(\boldsymbol{X})\right)\Vert^2\right\} + \lambda\cdot \Vert \nabla{}^2 \boldsymbol{f} \Vert^2_{L^2(\M)}.
\end{aligned}
\end{align}
The expectation $\mathbb{E}(\cdot)=\int(\cdot)\,d\mathbb{P}$ is taken with respect to the data-generating distribution $\mathbb{P}$; $\Upsilon$ denotes any finite constant such that $\Vert \boldsymbol{x}\Vert \le \Upsilon$ for all $\boldsymbol{x}\in \operatorname{supp}(\mathbb{P})$. Such a constant exists because $\operatorname{supp}(\mathbb{P})$ is compact. It is introduced as a localization device for the theoretical development; the value of $\Upsilon$ is not needed for practical estimation, as explained in Section~\ref{section: The Empirical PA Algorithm}.

Let $\operatorname{Isom}(\M,g)$ denote the group of Riemannian isometries of $(\M,g)$ \citep{myers1939group, srivastava2016functional}. For every $\phi\in\operatorname{Isom}(\M,g)$, we have $\mathcal L_\lambda(\Bf\circ\phi)=\mathcal L_\lambda(\Bf)$ and $(\Bf\circ\phi)(\M)=\Bf(\M)$. Thus, the parametrizing map is identifiable only up to the orbit $[\Bf]:=\{\Bf\circ\phi:\phi\in\operatorname{Isom}(\M,g)\}$. Statistically, these isometries represent nuisance reparameterizations of the template manifold; hence, inference on $[\Bf]$ removes this ambiguity without changing the fitted manifold. From the viewpoint of image registration, composition with $\phi$ aligns signal values by transforming their domain locations. Here, registration is restricted to Riemannian isometries of the template, and the quotient distance below measures discrepancy after optimal alignment \citep{srivastava2016functional}. For $\Bf,\Bh\in C(\M;\mathbb R^D)$, we use the quotient metric 
\begin{align}\label{eq: quotient metric d iso}
    d_{\mathrm{iso}}([\Bf],[\Bh]):=\inf_{\phi\in\operatorname{Isom}(\M,g)}\Vert\Bf-\Bh\circ\phi\Vert_{C(\M)}.
\end{align}
Since $\M$ is compact, $\operatorname{Isom}(\M,g)$ is compact \citep[Section~5]{myers1939group}, and the infimum is attained. The uniqueness and convergence statements below are formulated for isometry orbits.

The loss function $\mathcal{L}_{\lambda}(\Bf)$ in \eqref{eq: the core min problem} comprises two terms. The first term, $\mathbb{E} \{\Vert \boldsymbol{X}-\Bf\left(\Bpi_{\Bf}(\boldsymbol{X})\right)\Vert^2 \}$, represents the average of the squared fitting error, illustrated in Figure \ref{fig:proj-curve}. This term quantifies the fidelity of the image $\Bf(\M)$ to data $\boldsymbol{X}$. Minimizing the fitting error alone is at risk of overfitting \citep{duchamp1996extremal}, which is addressed by the second term, the penalty. Appendix~\ref{section: Geometric Interpretation of the Penalty} provides a detailed interpretation of the penalty term $\Vert \nabla{}^2 \boldsymbol{f} \Vert^2_{L^2(\M)}$ from the viewpoints of \textit{Poincaré–Wirtinger’s inequality} \citep[Lemma \ref{lemma: Poincaré-Wirtinger's inequality};][]{brezis2011functional} and Riemannian geometry \citep{lee2018riemannian}. Briefly, when $\Bf: \M\rightarrow\mathbb{R}^D$ is an embedding, the penalty decomposes into (i) its \textit{bending energy} \citep[][Appendix~\ref{section: Curvature Penalty}]{derfner2026branch} and (ii) the discrepancy between the Levi--Civita connections induced by the metric \(g\) and the pull-back metric defined in \eqref{eq: def of metric g_f}; when $\nabla^{2}\boldsymbol{f}=\boldsymbol{0}$, the map
$\boldsymbol{f}\colon \M\to\mathbb{R}^{D}$ sends every geodesic of $(\M,g)$ to a straight line in $\mathbb{R}^{D}$ \citep[][Corollary 1.6]{vilms1970totally}; furthermore, when $d=1$ and $\M=\mathbb{S}^1$, the penalty term also penalizes the length of the fitted curve, serving the role of the penalty introduced by \cite{kegl2000learning}. Appendix~\ref{section: Geometric Interpretation of the Penalty} also justifies defining the penalty using the Hessian $\nabla^2$ rather than the \textit{Laplace-Beltrami operator}. In particular, the null space of the penalty \(\Vert \nabla{}^2 \boldsymbol{f} \Vert^2_{L^2(\M)}\) is a subset of a finite-dimensional space (see Remark~\ref{remark: Affine and harmonic function spaces}). Hence, as \(\lambda\) increases, the optimization problem in \eqref{eq: the core min problem} shrinks toward a parametric problem. Lastly, the upper bound $\Upsilon$ excludes candidate functions that stray arbitrarily far from the region $\mathrm{supp}(\mathbb{P})$ where the data lie; it is not a user-selected tuning parameter and plays no role in the practical implementation of the framework (see Section~\ref{section: The Empirical PA Algorithm}).

The minimization construction in \eqref{eq: the core min problem} is a modified analogue of the framework introduced by \cite{meng_principal_2021}, e.g., see \eqref{eq: def of the ME principal manifolds}. The work by \cite{meng_principal_2021} does not provide a theoretical guarantee of the existence of the minimizers they propose. In contrast, our following theorem ensures that the argmin set in \eqref{eq: the core min problem} is nonempty and that a principal manifold is therefore well defined.
\begin{theorem}\label{thm: existence of f star}
    Given $\lambda>0$, there exists an $\Bf^*_{\lambda}\in\mathscr{F}$ such that $\mathcal{L}_\lambda(\Bf^*_\lambda) = \min_{\Bf\in\mathscr{F}} \mathcal{L}_\lambda(\Bf)$.
\end{theorem}
\noindent Note that the existence of a minimizer $\Bf^*_{\lambda}$, as established in Theorem \ref{thm: existence of f star}, requires only a minimal condition on the data-generating distribution $\mathbb{P}$, i.e., the support $\mathrm{supp}(\mathbb{P})$ is compact.

Here, we examine two extremes, $\lambda=0$ and $\lambda\rightarrow\infty$. When $\lambda=0$, the minimization in \eqref{eq: the core min problem} is essentially equivalent to the framework proposed by \cite{hastie1989principal}, as indicated by \eqref{eq: def of the fitting-error functional}. \cite{duchamp1996extremal} demonstrate that the loss function $\mathcal{L}_{\lambda}(\Bf)$ in \eqref{eq: the core min problem} generally does not have a minimizer when $\lambda=0$, which explains the overfitting results previously observed by \cite{hastie1989principal}. The overfitting phenomenon is also evident in Figure \ref{fig:interval_S1_full}(b,e), where the wiggly green curves correspond to an excessively small $\lambda$. When $\lambda\rightarrow\infty$, the following results establish connections between the proposed PME minimizers in \eqref{eq: the core min problem}, linear PCA, and the expected value $\mathbb{E}\boldsymbol{X}$.
\begin{theorem}\label{thm: PME with lambda=infty}
    Let $\Bf^*_{\lambda}$ be a minimizer defined in \eqref{eq: the core min problem}. Under Assumption \ref{assumption: assumption on the template manifold}, we have
    \begin{enumerate}
        \item When $\partial \M=\emptyset$, $\Bf^*_\lambda$ converges to the constant function $\boldsymbol{m} \mapsto \mathbb{E}\boldsymbol{X}$ as $\lambda\rightarrow\infty$ in the supremum-norm topology, i.e., $\max_{\boldsymbol{m}\in\M} \Vert \Bf^*_\lambda(\boldsymbol{m})-\mathbb{E}\boldsymbol{X}\Vert\rightarrow0$ as $\lambda\to\infty$.
        \item Suppose $\mathfrak M \subseteq \mathbb{R}^d$ is a compact, simply connected, $d$-dimensional subset with smooth boundary, equipped with the Euclidean metric. Assume that the $d$-th largest eigenvalue of the covariance matrix $\mathrm{Cov}(\boldsymbol{X})$ is strictly greater than the $(d+1)$-th largest eigenvalue. Let $\boldsymbol v_1,\ldots,\boldsymbol v_d$ be orthonormal eigenvectors of $\mathrm{Cov}(\boldsymbol X)$ corresponding to its largest $d$ eigenvalues, and define the PCA affine subspace $\mathcal{P}_d:=\mathbb E\boldsymbol X + \mathrm{span}\{\boldsymbol v_1,\ldots,\boldsymbol v_d\}$. If $\Upsilon$ in \eqref{eq: the core min problem} is sufficiently large, then there exists a nonempty subset $\mathscr A\subseteq \mathscr F$ such that $\boldsymbol{f}_\lambda^*$ converges to $\mathscr A$ as $\lambda\to\infty$ in the following sense: for every open neighborhood $\mathscr V$ of $\mathscr A$ in $\mathscr F$ under the supremum-norm topology, there exists $\Lambda>0$ such that $\boldsymbol{f}_\lambda^*\in \mathscr V$ for all $\lambda>\Lambda$. Moreover, every $\boldsymbol{f}_\infty^*\in \mathscr A$ satisfies $\boldsymbol{f}_\infty^*(\M)\subseteq\mathcal{P}_d$.
    \end{enumerate}
\end{theorem}
\noindent The second conclusion of Theorem~\ref{thm: PME with lambda=infty} requires only the existence of a sufficiently large $\Upsilon$; users do not need its exact value in practice (see Section~\ref{section: The Empirical PA Algorithm}). The results of Theorem~\ref{thm: PME with lambda=infty} are illustrated in Figures~\ref{fig:interval_S1_full} and \ref{fig:2d3D_star_moon}. In panels (b,e) of Figure~\ref{fig:interval_S1_full}, when the tuning parameter $\lambda$ is excessively large, the fitted manifold reduces either to a straight line (corresponding to the PCA affine subspace) or to a point (corresponding to $\mathbb{E}\boldsymbol{X}$). Figure~\ref{fig:2d3D_star_moon} shows that the fitted closed surface gradually shrinks to a point (the center of the point cloud) as $\lambda$ increases. Appendix~\ref{appendix: additional simulations} presents additional numerical experiments that validate Theorem~\ref{thm: PME with lambda=infty}. 

While the extreme case of $\lambda=0$ results in overfitting, Theorem \ref{thm: PME with lambda=infty} shows that the opposite extreme, $\lambda\rightarrow\infty$, can lead to underfitting unless the mean or PCA of data is satisfactory for the intended application. Consequently, selecting an appropriate value of $\lambda$ between these extremes is necessary, which will be addressed in Section \ref{section: Tuning Parameter Selection}. Furthermore, as stated below, the minimizing orbit $[\Bf^*_{\lambda}]$ depends continuously on $\lambda$, ensuring stability of the regularized fit as $\lambda$ varies.
\begin{theorem}\label{thm: continuity of the minimizer wrt the tuning parameter}
Suppose $[\underline{\lambda},\overline{\lambda}]\subset(0,\infty)$ and that, for every $\lambda\in[\underline{\lambda},\overline{\lambda}]$, the argmin set in \eqref{eq: the core min problem} is the single orbit $[\boldsymbol{f}_\lambda^*]$. Then, $\lambda\mapsto[\boldsymbol{f}_\lambda^*]$ is continuous under the quotient metric $d_{\mathrm{iso}}$.
\end{theorem}

\begin{figure}[!htbp]
    \centering
    \includegraphics[width=1.0\linewidth]{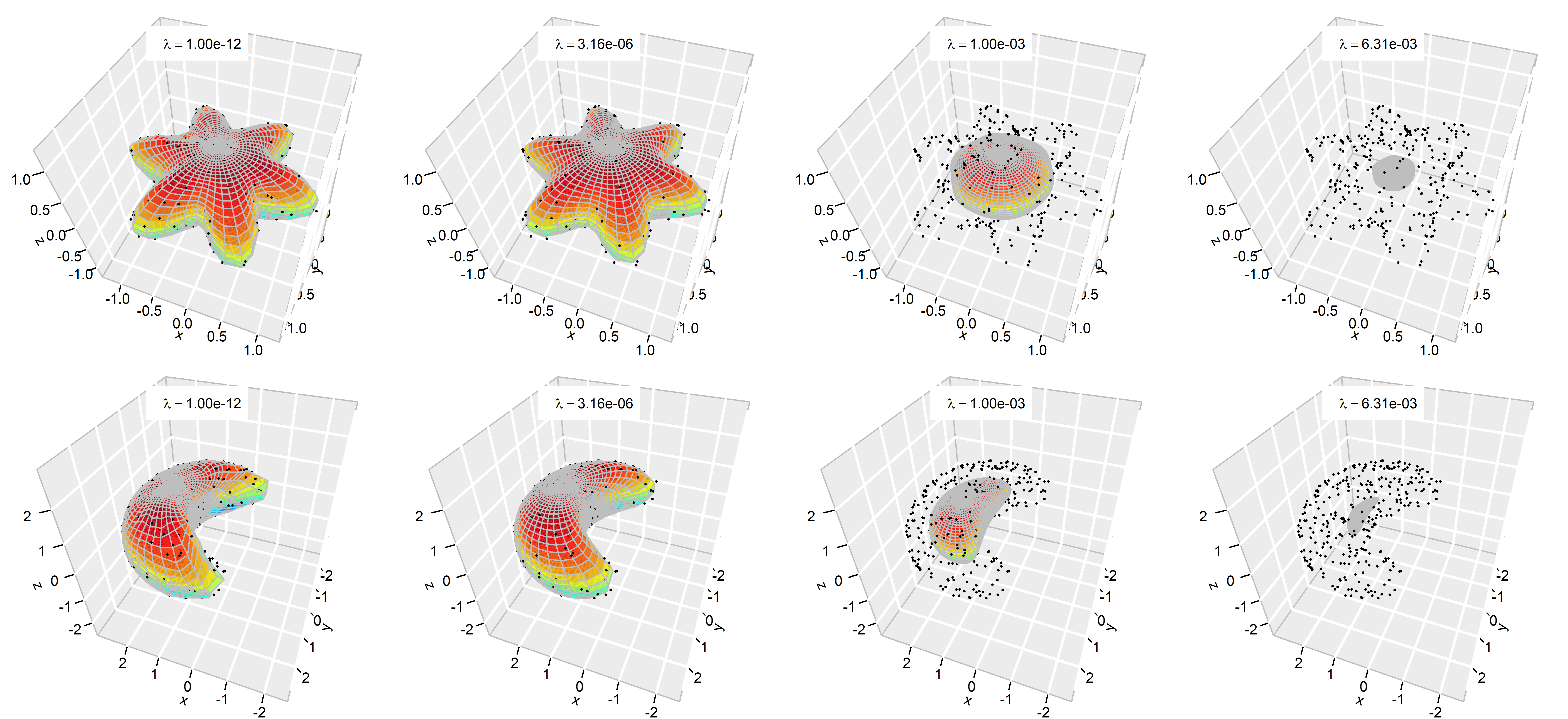}
    \vspace{-2.5em}
    \caption{\footnotesize The star- and cashew-shaped point clouds in the first and second rows are generated via the mechanisms described in Appendices \ref{appendix: Surface of a Flower/Star (2d, 3D)} and \ref{section: Surface of a Moon/Cashew (2d, 3D)}, respectively. We apply the PA algorithm to the point clouds. The top row is initialized using only spherical normalization, and the bottom row using ISOMAP followed by spherical normalization (see Appendix \ref{section: Initialization of the PA algorithm} for details). Each column displays the fitting results corresponding to a prespecified value of $\lambda$. This figure shows that a fitted closed surface gradually shrinks to the center of a point cloud as $\lambda$ approaches $\infty$, as stated in Theorem \ref{thm: PME with lambda=infty}.}
    \label{fig:2d3D_star_moon}
\end{figure}

\subsection{Projection-Adaptation Algorithm}\label{section: Projection-Adaptation Algorithm}

Computing a representative of the minimizing orbit $[\Bf_\lambda^*]$ defined in \eqref{eq: the core min problem} is generally nontrivial, except in the extreme case where $\lambda\rightarrow\infty$ (Theorem \ref{thm: PME with lambda=infty}). In this section, we develop the theoretical foundation for an iterative algorithm used to numerically approximate such a representative. This algorithm was first applied to the PME framework by \cite{Hastie1984PrincipalCurves} and \cite{hastie1989principal}, referred to as the \textit{principal curve algorithm}. Despite its widespread use \citep[e.g.,][]{banfield1992ice, yue2016parameterization}, theoretical guarantees for this algorithm remain limited. As noted by \cite{banfield1992ice}, ``there is no formal proof that the algorithm converges, but \cite{hastie1989principal} report that they have had no convergence problems with more than 40 real and simulated examples.'' \cite{smola2001regularized} outline a generalization, termed the \textit{projection–adaptation (PA) algorithm}, which alternates between a projection step (P-step) and an adaptation step (A-step). \cite{meng_principal_2021} and \cite{zielinski2024longitudinal} adopt the PA algorithm within their PME framework and report numerical convergence across extensive numerical experiments. Nevertheless, a precise convergence statement and a rigorous proof are still lacking. Here, we first describe the PA algorithm within the framework introduced in Section \ref{section: Manifolds via Minimization}. We then state and prove the convergence of the PA algorithm.

To estimate a representative of the minimizing orbit $[\Bf^*_\lambda]$ described in \eqref{eq: the core min problem}, we need the ``bivariate'' surrogate functional $\mathcal{Q}(\cdot\vert\cdot)$ and the updating operator $\mathcal{T}_\lambda(\cdot)$ defined as follows:
\begin{align}\label{eq: def of the Q functional}
    \begin{aligned}
        & \mathcal{Q}_\lambda(\Bf\,\vert\,\Bg) := \mathbb{E} \left\{\Vert \boldsymbol{X}-\Bf\left(\Bpi_{\Bg}(\boldsymbol{X})\right)\Vert^2\right\} + \lambda\cdot \Vert \nabla{}^2 \boldsymbol{f} \Vert^2_{L^2(\M)}, \\
        & \mathcal{T}_\lambda(\Bg) := \argmin_{\Bf\in\mathscr{F}} \mathcal{Q}_\lambda(\Bf\,\vert\,\Bg)\subseteq\mathscr{F}, \quad \text{ for all }\Bf,\,\Bg\in\mathscr{F}.
    \end{aligned}
\end{align}
In general, $\mathcal{Q}_\lambda(\Bf\,\vert\,\Bg) \ne \mathcal{Q}_\lambda(\Bg\,\vert\,\Bf)$. By definition, $\mathcal{L}_\lambda(\Bf)=\mathcal{Q}_\lambda(\Bf\,\vert\,\Bf)$ for all $\Bf\in\mathscr{F}$. The argmin set \(\mathcal{T}_\lambda(\Bg)\) defined in \eqref{eq: def of the Q functional} is nonempty (see Lemma~\ref{lemma: properties of Q and T functionals}). With these preparations, we are ready to present the PA algorithm. A representative of the minimizing orbit $[\Bf^*_\lambda]$ defined in \eqref{eq: the core min problem} can be approximated through the following iterative algorithm.
\begin{align}\label{eq: the core iterative algorithm}
    \Bf_\lambda^{(n+1)} \in \mathcal{T}_\lambda(\Bf^{(n)}_\lambda) = \argmin_{\Bf\in\mathscr{F}} \left[\mathbb{E} \left\{\left\Vert \boldsymbol{X}-\Bf\left(\Bpi_{\Bf^{(n)}_\lambda}(\boldsymbol{X})\right)\right\Vert^2\right\} + \lambda\cdot \Vert \nabla{}^2 \boldsymbol{f} \Vert^2_{L^2(\M)} \right],
\end{align}
for $n=0,1,2,\ldots$. We refer to the computation of the projection $\Bpi_{\Bf^{(n)}_\lambda}(\boldsymbol{X})$ as the P-step; practical methods for this computation are given in Appendix \ref{section: Computing Projection Indices}. We refer to the minimization in \eqref{eq: the core iterative algorithm} as the A-step. The PA algorithm in \eqref{eq: the core iterative algorithm} is equivalent to the principal curve algorithm proposed by \cite{hastie1989principal} when $\lambda=0$ \citep[e.g.,][Theorem 6]{meng_principal_2021}. 

The minimization in \eqref{eq: the core min problem} is an unsupervised learning task that generalizes PCA (see Theorem \ref{thm: PME with lambda=infty}). In contrast, each iteration of the PA algorithm in \eqref{eq: the core iterative algorithm} is a supervised learning problem, namely, nonlinear regression. Specifically, at a given iteration, the P-step provides the ``predictor,'' i.e., $\Bpi_{\Bf^{(n)}_\lambda}(\boldsymbol{X})$, and the data point $\boldsymbol{X}$ serves as the ``response.'' This regression problem is well-studied in the literature \citep[e.g.,][]{wahba1990spline} and can be solved efficiently by exploiting the fact that $H^2(\M)$ is an RKHS when $\dim(\M)<4$ (see Theorem \ref{thm: Rellich–Kondrachov embedding}). Details of the regression fit are provided in Section \ref{section: Empirical Version} and Appendix~\ref{appendix: Reproducing Kernel Hilbert Spaces}. Theorem \ref{thm: the convergence theorem of the core iterative algorithm} (presented later) shows that the unsupervised minimization in \eqref{eq: the core min problem} can be carried out by iteratively solving these nonlinear regression subproblems in \eqref{eq: the core iterative algorithm}.

\begin{figure}[!htbp]
    \centering
\begin{tikzpicture}[scale=1.0, transform shape, x=1.05cm, y=1.05cm,
                    line cap=round, line join=round]

\definecolor{myblue}{RGB}{0,110,255}
\definecolor{myorange}{RGB}{245,140,30}

\draw[-{Stealth[length=3.2mm,width=1.0mm]}, very thick] (0,0) -- (8.7,0);
\draw[-{Stealth[length=3.2mm,width=1.0mm]}, very thick] (0,0) -- (0,4.2);

\node[anchor=west] at (8, 0.25) {\footnotesize $\mathscr{F}$};
\node[anchor=south] at (1.7,3.7) {\footnotesize $\mathcal{L}_{\lambda}(\boldsymbol{f})
$ for $\boldsymbol{f}\in\mathscr{F}$};

\draw[black, very thick, dashed]
  (0.1, 2.8) .. controls (0.2, 3.0) and (0.24,3.6) .. (1.05, 3.32);

\draw[black, very thick]
  (1.1,3.3) .. controls (1.95,3.0) and (2.35,2.35) .. (2.90,1.55)
              .. controls (3.35,0.95) and (4.05,0.72) .. (4.55,0.80)
              .. controls (5.15,0.92) and (5.85,1.55) .. (6.35,2.45)
              .. controls (6.75,3.15) and (6.9,3.85) .. (7.10, 4.1);

\draw[black, very thick, dashed]
  (7.15,4.18) .. controls (7.6, 4.8) and (7.8, 4) .. (8.0, 3.7)
              .. controls (8.1,3.5) and (8.2,3.3) .. (8.5, 3.2);

\fill[red] (4.38,0.78) circle (2.6pt);
\node[anchor=north] at (4.15,0.8) {\footnotesize $\boldsymbol{f}_\lambda^\ast$};

\coordinate (p0) at (6.86,3.62);
\coordinate (p1) at (6.5, 2.75);
\coordinate (p2) at (6, 1.88);
\coordinate (p4) at (5, 0.97);

\fill[myorange] (p0) circle (2.4pt);
\fill[myorange] (p1) circle (2.4pt);
\fill[myorange] (p2) circle (2.4pt);
\fill[myorange] (p4) circle (2.4pt);

\draw[myblue, thick, -{Stealth[length=3.0mm,width=2.2mm]}]
  (p0) to[bend right=25] (p1);
\draw[myblue, thick, -{Stealth[length=3.0mm,width=2.2mm]}]
  (p1) to[bend right=25] (p2);
\draw[myblue, thick, -{Stealth[length=3.0mm,width=2.2mm]}]
  (p2) to[bend right=25] (p4);

\node[anchor=south west] at ($(p0)+(-0.6, 0)$) {\footnotesize $\boldsymbol{f}_\lambda^{(0)}$};
\node[anchor=west]       at ($(p1)+(-2.45, 0.1)$) {\footnotesize $\boldsymbol{f}_\lambda^{(1)}\in\mathcal{T}_\lambda(\boldsymbol{f}_\lambda^{(0)})$};
\node[anchor=east]       at ($(p2)+(0.8, -0.1)$) {\footnotesize $\boldsymbol{f}_\lambda^{(2)}$};
\node[anchor=west]       at ($(p4)+(-0.27,-0.25)$) {\footnotesize$\boldsymbol{f}_\lambda^{(3)}\in\mathcal{T}_\lambda(\boldsymbol{f}_\lambda^{(2)})$};

\coordinate (xL) at (1.10,0);   
\coordinate (xR) at (7.10,0);   

\draw[mygreen, thick, dashed] (1.10,0) -- (1.10,3.27);
\draw[mygreen, thick, dashed] (7.10,0) -- (7.10, 4.1);

\draw[mygreen, thick, decorate,
      decoration={brace, mirror, amplitude=7pt}]
  (1.10, -0.1) -- (7.10, -0.1);

\node[mygreen, anchor=north] at (4.05,-0.25) {\footnotesize $\mathscr{U}$};

\end{tikzpicture}
\vspace{-15pt}
\caption{\footnotesize Illustration of the PA algorithm and the role of Assumption~\ref{assumption: Assumption for the convergence of the iterative algorithm}. The vertical axis denotes the objective $\mathcal{L}_\lambda(\Bf)$ and the horizontal axis represents the function space $\mathscr F$ (illustration only, not to scale). The region $\mathscr U$ (green bracket; boundaries indicated by dashed vertical lines) is a ``basin of attraction'' in which the updating operator $\mathcal T_\lambda$ is intended to operate. Starting from an initialization $\Bf_\lambda^{(0)}\in\mathscr U$, successive iterations $\Bf_\lambda^{(n+1)}\in\mathcal{T}_\lambda(\Bf_\lambda^{(n)})$ (orange points connected by blue arrows) decrease the objective and move toward the minimizing orbit $[\Bf_\lambda^\ast]\subseteq\mathscr U$; the single red point schematically represents this orbit, with isometric reparameterizations suppressed. The dashed portions of the objective curve indicate behavior outside $\mathscr U$, where additional stationary points or irregularities may occur and where the convergence guarantee is not
asserted.}\label{fig:placeholder}
\end{figure}

The iterative scheme in \eqref{eq: the core iterative algorithm} casts the PA algorithm as a \textit{majorization--minimization} (MM) procedure \citep{lange2022legacy} over the function space $\mathscr F$. Existing convergence theory for MM algorithms provides useful guidance for analyzing the convergence of the PA algorithm. However, these results are not directly applicable in the present setting because $\mathscr F$ is infinite-dimensional and the projection index exhibits additional subtleties (e.g., possible nonuniqueness). Therefore, we develop a convergence theory tailored to the PA algorithm, beginning with a descent property. Lemma \ref{lemma: the connection between L and Q} in Appendix \ref{Appendix: Proofs} shows that the scalar-valued sequence $\{\mathcal{L}_\lambda(\Bf_\lambda^{(n)})\}_{n\in\mathbb{N}}$ is nonincreasing, illustrated in Figures \ref{fig:interval_S1_full}(c,f) and \ref{fig:placeholder}. Hence, the limit $\lim_{n\rightarrow\infty} \mathcal{L}_\lambda(\Bf_\lambda^{(n)})$ exists. However, the existence of this limit does not imply convergence of the orbit sequence $\{[\Bf^{(n)}_\lambda]\}_{n\in\mathbb{N}}$. The following assumptions are needed to ensure convergence of $\{[\Bf^{(n)}_\lambda]\}_{n\in\mathbb{N}}$.
\begin{assumption}\label{assumption: Assumption for the convergence of the iterative algorithm}
There exists a region $\mU\subseteq\mathscr{F}$ with the following properties:
\begin{enumerate}
\item (Closedness and symmetry.) $\mU$ is closed in $C(\mathfrak{M};\,\mathbb{R}^D)$ under the supremum norm and isometry-saturated: if $\Bg\in\mU$ and $\phi\in\operatorname{Isom}(\M,g)$, then $\Bg\circ\phi\in\mU$.
    \item (Invariance.) $\mathcal{T}_\lambda(\Bg)\subseteq\mU$ for all $\Bg\in\mU$, i.e., $\mathcal{T}_\lambda(\mU)\subseteq\mU$.
    \item (Unique global minimizer.) There exists a global minimizer $\Bf^*_\lambda\in\mU$; every global minimizer $\Bh\in\mU$ satisfies $[\Bh]=[\Bf^*_\lambda]$, i.e., $\mU$ contains exactly one global minimizing orbit.
    \item (Unique fixed-point orbit.) If $\Bg\in\mU$ satisfies $\Bg\in\mathcal{T}_\lambda(\Bg)$, then $[\Bg]=[\Bf^*_\lambda]$.
\end{enumerate}
\end{assumption}
\noindent Assumption~\ref{assumption: Assumption for the convergence of the iterative algorithm} posits an isometry-saturated basin of attraction $\mU$ around the minimizing orbit $[\Bf_\lambda^\ast]$ (see Figure~\ref{fig:placeholder}). Theorem~\ref{thm: the convergence theorem of the core iterative algorithm} shows that an iteration initialized in $\mathscr U$ converges to this orbit. Appendix \ref{section: Initialization of the PA algorithm} discusses practical initializations. The invariance condition $\mathcal T_\lambda(\mathscr U)\subseteq\mathscr U$ keeps all iterates in $\mathscr U$; the third item of Assumption~\ref{assumption: Assumption for the convergence of the iterative algorithm} permits precisely the isometric reparameterizations of one minimizing orbit, and the fourth rules out fixed points outside that orbit. These requirements go beyond the compactness of $\operatorname{supp}(\mathbb{P})$ used for existence.

In addition to Assumption~\ref{assumption: Assumption for the convergence of the iterative algorithm}, we need the following assumption regarding the relationship between the support $\mathrm{supp}(\mathbb{P})$ and arbitrary uniform accumulation points of the iterates.
\begin{assumption}\label{assumption: singleton assumption for the sequence}
Let $\{\Bf_\lambda^{(n)}\}_{n\in\mathbb{N}}$ be a sequence generated by the iterative algorithm in \eqref{eq: the core iterative algorithm}, and let $\Bf_\lambda^{(\infty)}$ be an arbitrary accumulation point under the supremum norm, i.e., there exists a subsequence $\{\Bf_\lambda^{(n,k)}\}_{k\in\mathbb{N}}$ such that $\lim_{k\rightarrow\infty}\Vert \Bf_\lambda^{(n,k)} - \Bf_\lambda^{(\infty)} \Vert_{C(\M)} = 0$. We assume that the set $\argmin_{\boldsymbol{m}\in\mathfrak{M}}
\Vert \boldsymbol{x} - \Bf_\lambda^{(\infty)}(\boldsymbol{m})
\Vert$ is a singleton for every \(\boldsymbol{x}\in\operatorname{supp}(\mathbb{P})\).
\end{assumption}
\noindent Assumption~\ref{assumption: singleton assumption for the sequence} is imposed on every uniform accumulation map relevant to Theorem~\ref{thm: the convergence theorem of the core iterative algorithm}. For the rest of this paragraph, assume additionally that each such map is injective. For the image $\Bh(\M)$ of any such map $\Bh$, the assumption requires that each $\boldsymbol{x}\in\operatorname{supp}(\mathbb{P})$ admits a unique nearest point in $\Bh(\M)$. Recall that the \emph{medial axis} of a compact set $\mathcal{A}$ is the collection of all points having more than one closest point in $\mathcal{A}$ \citep[see Figure~\ref{fig:ellipse-medial-axis};][]{blum1967transformation}. Then, Assumption~\ref{assumption: singleton assumption for the sequence} can be read as requiring that $\operatorname{supp}(\mathbb{P})$ avoid the medial axis of $\Bh(\M)$. Under the manifold hypothesis (namely, that the data are generated by a low-dimensional manifold contaminated by noise), this amounts to assuming that the noise level is not so large as to push any point of $\mathrm{supp}(\mathbb{P})$ into the medial axis of $\Bh(\M)$ (e.g., see Figure~\ref{fig:ellipse-medial-axis}). Assumption~\ref{assumption: singleton assumption for the sequence} is also naturally interpreted through the ``reach'' and ``tubular-neighborhood'' viewpoint. The \emph{reach}, $\mathrm{reach}(\mathcal{A})$, of a compact set $\mathcal{A}$ is the largest $r\ge 0$ such that every point $\boldsymbol{x}$ with $\mathrm{dist}(\boldsymbol{x},\mathcal{A})<r$ has a unique nearest point on $\mathcal{A}$, i.e., the nearest-point projection is well-defined throughout the \emph{tubular neighborhood} $\{\boldsymbol{x}:\mathrm{dist}(\boldsymbol{x},\mathcal{A})<r\}$ \citep{Federer1969Geometric, fefferman2016testing}. Then, we have the following sufficient condition: when, for every such map $\Bh$, the entire support $\operatorname{supp}(\mathbb P)$ lies in a tubular neighborhood of $\Bh(\M)$ with a positive radius less than $\mathrm{reach}(\Bh(\M))$, Assumption~\ref{assumption: singleton assumption for the sequence} holds. In the manifold learning literature, the reach of an underlying manifold is commonly assumed to be reasonably large, e.g., \cite{genovese2012minimax}, \cite{fefferman2016testing}, and \cite{fefferman2025fitting} focus on submanifolds endowed with reaches bounded uniformly from below. Building on Assumptions \ref{assumption: Assumption for the convergence of the iterative algorithm} and \ref{assumption: singleton assumption for the sequence}, we establish convergence of the PA iterations to the minimizing orbit under the quotient metric $d_{\mathrm{iso}}$ defined in \eqref{eq: quotient metric d iso}.
\begin{theorem}\label{thm: the convergence theorem of the core iterative algorithm}
Assume that Assumption \ref{assumption: Assumption for the convergence of the iterative algorithm} holds, let $\Bf^*_\lambda$ be a representative of the unique minimizing orbit in $\mU$, and let $\{\Bf_\lambda^{(n)}\}_{n\in\mathbb{N}}$ satisfy $\Bf_\lambda^{(n+1)}\in\mathcal{T}_\lambda(\Bf_\lambda^{(n)})$. If $\Bf_\lambda^{(0)}\in\mU$ and Assumption \ref{assumption: singleton assumption for the sequence} holds, then $d_{\mathrm{iso}}([\Bf_\lambda^{(n)}],[\Bf^*_\lambda])\to0$ and $\mathcal{L}_\lambda(\Bf_\lambda^{(n)})\to\mathcal{L}_\lambda(\Bf^*_\lambda)$ as $n\rightarrow\infty$.
\end{theorem}

\section{Empirical Version}\label{section: Empirical Version}

In practice, the underlying distribution $\mathbb{P}$ is unknown, and only finitely many observations $\{\boldsymbol{X}_{i}\}_{i\in[N]}$ in the ambient space $\mathbb{R}^D$ are available, where $N$ is the sample size. We assume these observations are iid draws from $\mathbb{P}$. The framework developed in Sections~\ref{section: Manifolds via Minimization} and~\ref{section: Projection-Adaptation Algorithm} is then specialized by replacing $\mathbb{P}$ with the empirical distribution $\mathbb{P}_N:=\frac{1}{N}\sum_{i=1}^N \delta_{\boldsymbol{X}_i}$. Accordingly, the expectation \(\mathbb{E}(\cdot)=\int(\cdot)d\mathbb{P}\) is replaced with the empirical average \(\mathbb{E}_N(\cdot):=\int(\cdot)d\mathbb{P}_N\). Given $\{\boldsymbol{X}_{i}\}_{i\in[N]}$, because \(\mathbb{P}_N\) is also a probability measure with compact support, the framework outlined in Sections~\ref{section: Manifolds via Minimization} and~\ref{section: Projection-Adaptation Algorithm} still applies to \(\mathbb{P}_N\).

This section presents two main results. First, it is demonstrated that the empirical minimizing orbit $[\Bf^*_{N,\lambda}]$ converges to its population-level counterpart $[\Bf^*_{\lambda}]$ as the sample size $N\rightarrow\infty$. Second, an explicit formula for the PA algorithm for the empirical distribution $\mathbb{P}_N$ is presented, utilizing the RKHS structure of $H^2(\M)$.

\subsection{Consistency of the Empirical Minimizer}\label{section: Consistency of the Empirical Minimizer} Let $\Bf^*_{N,\lambda} \in \argmin_{\Bf\in\mathscr{F}} \mathcal{L}_{N,\lambda}(\Bf)$, where
\begin{align}\label{eq: core iterative algorithm, empirical}
    \mathcal{L}_{N,\lambda}(\Bf):= \frac{1}{N} \sum_{i=1}^N \norm{\boldsymbol{X}_i - \Bf\left(\Bpi_{\Bf} (\boldsymbol{X}_i)\right)}^2 + \lambda\cdot \Vert \nabla{}^2 \boldsymbol{f} \Vert^2_{L^2(\M)},
\end{align}
and $\mathscr{F}$ is defined by \eqref{eq: the core min problem}. Notably, $\Bf^*_{N,\lambda}$ is consistent for $\Bf^*_{\lambda}$ modulo template isometries, in outer probability \citep[][Section 1.2]{van1996weak}, as stated precisely below and illustrated in Figure \ref{fig:PME_monte_carlo}.
\begin{theorem}\label{thm: argmin}
Let $\lambda>0$ and $\{\boldsymbol{X}_i\}_{i\in[N]}\overset{\mathrm{iid}}{\sim}\mathbb{P}$. Suppose $\argmin_{\Bf\in\mathscr F}\mcL_\lambda(\Bf)=[\Bf^*_\lambda]$, and this minimizing orbit is well separated: $\mcL_\lambda(\Bf^*_\lambda)<\inf_{\Bf\in\mathscr F\setminus\mathscr G}\mcL_\lambda(\Bf)$ for every open, isometry-saturated set $\mathscr G\subseteq C(\M;\mathbb R^D)$ containing $[\Bf^*_\lambda]$. Then, $d_{\mathrm{iso}}([\Bf^*_{N,\lambda}],[\Bf^*_\lambda])$ converges to zero in outer probability as $N\rightarrow\infty$, where $d_{\mathrm{iso}}$ is the quotient metric defined in \eqref{eq: quotient metric d iso}.
\end{theorem}

\begin{figure}[!htbp]
    \centering
    \includegraphics[width=1\linewidth]{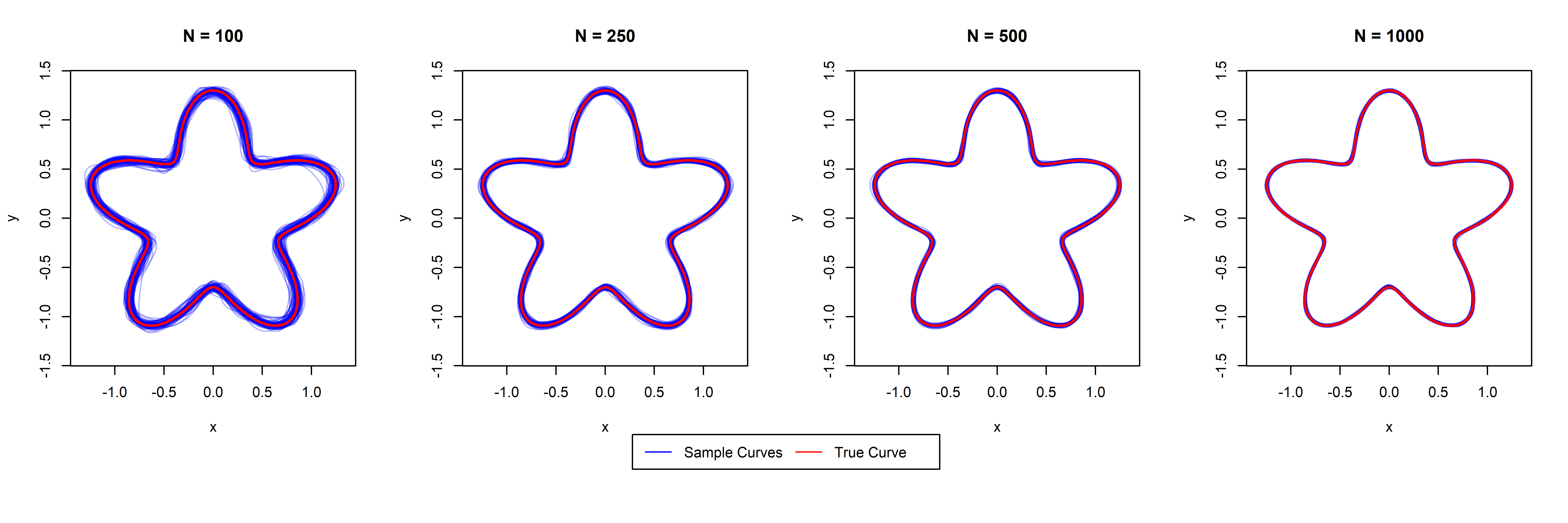}
    \vspace{-35pt}
    \caption{\footnotesize Finite-sample illustration accompanying Theorem \ref{thm: argmin}. The red curve represents the data-generating manifold and is included only as a simulation reference. Theorem \ref{thm: argmin} concerns convergence of the empirical PME fit to its population PME counterpart; it does not assert convergence to the red curve. For each sample size $N$, we repeat the following procedure 100 times: (i) generate a sample of size $N$ using the data-generating mechanism described in Appendix \ref{appendix: Boundary of a Flower/Star}, with noise standard deviation $0.04$; and (ii) apply the PA algorithm with $\lambda = 10^{-8}$ and plot the resulting estimate as a blue curve. The tuning parameter is held fixed across sample sizes and repetitions.}
    \label{fig:PME_monte_carlo}
\end{figure}

\subsection{The Empirical PA Algorithm}\label{section: The Empirical PA Algorithm} 

A representative of the consistent empirical minimizer $[\Bf^*_{N,\lambda}]$ can be approximated numerically using the empirical version of the PA algorithm in \eqref{eq: the core iterative algorithm}, obtained by replacing the distribution $\mathbb{P}$ in \eqref{eq: def of the Q functional} with its empirical counterpart $\mathbb{P}_N$. At each iteration, the resulting constrained update minimizes the empirical version of $\mathcal{Q}_\lambda$ over the function class $\mathscr{F}$.

Recall that the function class $\mathscr{F}$ in \eqref{eq: the core min problem} imposes two constraints: (i) $\Bf\in H^2(\M)$ and (ii) $\Vert \Bf\Vert_{C(\M)}\le \Upsilon$. Enforcing the second constraint at each iteration (A-step) would prevent an application of the \emph{representer theorem} \citep[][Theorem 1.3.1]{wahba1990spline} to the RKHS $H^2(\M)$. We therefore drop the second constraint and compute the following relaxed update over $H^2(\M)$:
\begin{align}\label{eq: argmin via RKHS}
    \Bf_{N,\lambda}^{(n+1)} \in \argmin_{\Bf\in H^2(\M)} \left[ \frac{1}{N} \sum_{i=1}^N \left\{\left\Vert \boldsymbol{X}_i-\Bf\left(\Bpi_{\Bf_{N,\lambda}^{(n)}}(\boldsymbol{X}_i)\right)\right\Vert^2\right\} + \lambda\cdot \Vert \nabla{}^2 \boldsymbol{f} \Vert^2_{L^2(\M)} \right].
\end{align}
The relaxed and constrained empirical PA updates coincide whenever the solution $\Bf_{N,\lambda}^{(n+1)}$ to the relaxed problem in \eqref{eq: argmin via RKHS} satisfies the post hoc condition
\begin{align}\label{eq: post hoc condition for the finite sample PA algorithm}
    \Vert \Bf_{N,\lambda}^{(n+1)} \Vert_{C(\M)}\le \Upsilon.
\end{align}
Indeed, if \eqref{eq: post hoc condition for the finite sample PA algorithm} holds, then the relaxed minimizer $\Bf_{N,\lambda}^{(n+1)}$ belongs to $\mathscr{F}$. Since it minimizes the objective over the larger space $H^2(\M)$, it is also a minimizer over $\mathscr{F}$. The following condition allows us to establish the post hoc bound in \eqref{eq: post hoc condition for the finite sample PA algorithm}.
\begin{align}\label{eq: def of tau star}
    \tau_*:=\inf_{n\ge0}\lambda_{\min}\left(\frac1N\boldsymbol T_n^\T\boldsymbol T_n\right)>0, \quad \text{where } \boldsymbol T_n:=\left(\phi_r\left(\Bpi_{\Bf_{N,\lambda}^{(n)}}(\boldsymbol X_i)\right)\right)_{1\le i\le N,\ 1\le r\le d_0},
\end{align}
the collection $\{\phi_r\}_{r=1}^{d_0}$ is a basis of the following finite-dimensional space (see Lemma~\ref{lemma: finite dimensionality of the kernel space})
\begin{align}\label{eq: H_0 space for scalar-valued functions}
    \mathcal H_0:=\left\{u\in H^2(\M;\mathbb{R}^1):\nabla^2 u=0 \right\},
\end{align}
and $\lambda_{\mathrm{min}}(\cdot)$ denotes the smallest eigenvalue of a matrix. Proposition~\ref{prop: empirical PA sup norm bound} shows that, under \eqref{eq: def of tau star}, there exists a finite constant $\Upsilon_\star$ such that every relaxed update $\Bf_{N,\lambda}^{(n+1)}$ satisfies \eqref{eq: post hoc condition for the finite sample PA algorithm} whenever $\Upsilon\ge\Upsilon_\star$. Accordingly, in practice, we assume that \eqref{eq: def of tau star} holds and that $\Upsilon\ge\Upsilon_\star$; the value of $\Upsilon$ need not be specified. Condition~\eqref{eq: def of tau star} is easily satisfied in common settings. For example, when the template manifold $\M$ is closed, $\mathcal{H}_0=\mathrm{span}\{1\}$ (see Appendix~\ref{appendix: An Orthogonal Decomposition of H2}) and $\tau_*=1>0$.

Given \(\{\boldsymbol{X}_i\}_{i\in[N]}\) and the function \(\Bf_{N,\lambda}^{(n)}\) from the \(n\)th step, the updated estimator \(\Bf_{N,\lambda}^{(n+1)}\) in \eqref{eq: argmin via RKHS} admits a closed-form expression via the representer theorem \citep[][Theorem 1.3.1]{wahba1990spline}. Details of the application of the representer theorem to the PA algorithm are presented in Appendix \ref{appendix: Reproducing Kernel Hilbert Spaces}. Comprehensive implementation details of the PA algorithm are provided in Appendix \ref{appendix: Projection-Adaptation Algorithm}. In particular, Algorithm \ref{alg: empirical pme} in Appendix \ref{appendix: Projection-Adaptation Algorithm} presents a summary of the PA algorithm, and Appendix~\ref{section: An Approximate Spline Solution for Large N} discusses efficient methods for computing the relaxed A-step in \eqref{eq: argmin via RKHS}.

\subsection{A Bayesian Interpretation}\label{section: Bayesian Viewpoint and Tuning Parameter Selection}

The relaxed empirical A-step in \eqref{eq: argmin via RKHS} admits a Gaussian process regression (GPR) formulation that closely parallels the iterative GPR strategy underlying the MrGap algorithm \citep{dunson2021inferring}. In addition to providing a Bayesian interpretation of the tuning parameter $\lambda$, this formulation suggests a natural generalization of the A-step, i.e., the RKHS $H^2(\M)$ and the Hessian penalty $\Vert \nabla^2 \boldsymbol{f} \Vert^2_{L^2(\M)}$ can be replaced by a general RKHS and its associated regularization penalty.

Fix \(n\) and suppress the iteration superscript by writing \(\boldsymbol m_i:=\Bpi_{\Bf_{N,\lambda}^{(n)}}(\boldsymbol X_i)\) for \(i\in[N]\). Conditional on \(\{\boldsymbol m_i\}_{i\in[N]}\) that serve as predictors, \(\{X_{ij}\}_{i\in[N]}\) are the responses for each coordinate \(j\in[D]\). To specify the GPR model, recall from Appendix~\ref{appendix: Reproducing Kernel Hilbert Spaces} the orthogonal decomposition \(H^2(\M;\mathbb R)=\mathcal H_0\oplus\mathcal H_1\) with respect to the inner product \(\langle\cdot,\cdot\rangle_R\) in \eqref{eq: def of the inner product for the representer theorem}, where
\(\mathcal H_0\) is defined in \eqref{eq: H_0 space for scalar-valued functions}. Let \(d_0:=\dim\mathcal H_0 <\infty\) (see Lemma~\ref{lemma: finite dimensionality of the kernel space}), choose a basis \(\{\phi_r\}_{r=1}^{d_0}\) of \(\mathcal H_0\), and let \(R_1\) denote the reproducing kernel of \(\mathcal H_1\). For \(\boldsymbol\beta_j=(\beta_{1j},\ldots,\beta_{d_0j})^\T\), define \(p_{\boldsymbol\beta_j}(\boldsymbol m):=\sum_{r=1}^{d_0}\beta_{rj}\phi_r(\boldsymbol m)\). Given positive variance parameters \(\rho^2,\omega^2,\sigma^2\), consider the following semiparametric model for each coordinate \(j\):
\begin{align}\label{eq: GPR interpretation of a PA iteration}
\begin{aligned}
& X_{ij}=p_{\boldsymbol\beta_j}(\boldsymbol m_i)
    +h_j(\boldsymbol m_i)+\varepsilon_{ij}, \quad \text{ where }\\
 & \boldsymbol\beta_j\sim\mathcal N(\boldsymbol 0,\rho^2\boldsymbol I_{d_0}),\quad
    h_j\sim\mathrm{GP}(0,\omega^2R_1),\quad
    \quad
    \varepsilon_{ij}\overset{\mathrm{iid}}{\sim}\mathcal N(0,\sigma^2).
\end{aligned}
\end{align}
Thus, \eqref{eq: GPR interpretation of a PA iteration} places a Gaussian process prior on the coordinate regression function \(p_{\boldsymbol\beta_j}+h_j\) and assumes independent Gaussian observation noise. Here, \(\mathrm{GP}(0,\omega^2R_1)\) denotes a zero-mean Gaussian process with covariance kernel \(\omega^2R_1\). The prior precision of \(\boldsymbol\beta_j\) is \(\rho^{-2}\boldsymbol I_{d_0}\), which vanishes as \(\rho\to\infty\). In this limit, the prior becomes diffuse and may be viewed formally as the improper flat prior \(\pi(\boldsymbol\beta_j)\propto1\) on \(\mathbb R^{d_0}\). Conditional on \(\{\boldsymbol m_i\}_{i\in[N]}\), the components \(\boldsymbol\beta_j\), \(h_j\), and \(\{\varepsilon_{ij}\}_{i\in[N]}\) are mutually independent for each \(j\); the coordinate-specific collections are also independent across \(j\) and share the same hyperparameters. For \(\rho>0\) and \(j\in[D]\), let \(\widehat f_{j,\rho}^{(n+1)}\) denote the posterior mean of \(p_{\boldsymbol\beta_j}+h_j\) under \eqref{eq: GPR interpretation of a PA iteration} and write
\(\boldsymbol{\widehat{f}}_{\rho}^{(n+1)}
:=(\widehat f_{1,\rho}^{(n+1)},\ldots,\widehat f_{D,\rho}^{(n+1)})\).

\begin{theorem}\label{thm: diffuse prior limit for an empirical A-step}
Fix \(n\) and condition on
\(\{\boldsymbol m_i\}_{i\in[N]}\). Suppose that \(\boldsymbol T_n\), defined in \eqref{eq: def of tau star}, has full column rank. Fix \(\sigma^2,\omega^2>0\) and set
\begin{equation}\label{eq: GPR parameter relation}
    \lambda=\frac{\sigma^2}{N\omega^2}.
\end{equation}
Then, \(\lim_{\rho\to\infty} \Vert \,
    \boldsymbol{\widehat{f}}_{\rho}^{(n+1)}
    -\Bf_{N,\lambda}^{(n+1)}
    \,\Vert_{C(\M)} = 0\), where \(\Bf_{N,\lambda}^{(n+1)}\) denotes the minimizer in \eqref{eq: argmin via RKHS}.
\end{theorem}
\noindent Theorem~\ref{thm: diffuse prior limit for an empirical A-step} provides a Bayesian interpretation of each iteration of the PA algorithm. Under \eqref{eq: GPR interpretation of a PA iteration}, the GPR posterior mean converges uniformly to the relaxed A-step in \eqref{eq: argmin via RKHS} as $\rho \to \infty$, provided that $\lambda$ is given by \eqref{eq: GPR parameter relation}. The relation in \eqref{eq: GPR parameter relation} identifies $\lambda$ as a variance ratio. For fixed $N$, increasing the observation-noise variance $\sigma^2$ or decreasing the rough-process variance $\omega^2$ increases $\lambda$, thereby pulling the fitted function toward a finite-dimensional null space. 

From the GPR perspective embodied in \eqref{eq: GPR interpretation of a PA iteration}, each A-step of the PA algorithm can be generalized by replacing the kernel $R_1$ in \eqref{eq: GPR interpretation of a PA iteration} with another reproducing kernel defined on the template manifold $\M$. Heat and Matérn kernels on closed Riemannian manifolds have been extensively studied in the literature \citep{li2023inference, he2026scalable, wang2025riemannian, li2026statistical}. The theoretical analysis and numerical investigation of PA algorithms based on alternative kernels are left for future work.

\section{Tuning Parameter Selection}\label{section: Tuning Parameter Selection}

The tuning parameter $\lambda$ balances fidelity to the data against geometric smoothness. As discussed in Section~\ref{section: Manifolds via Minimization}, $\lambda=0$ typically leads to overfitting, whereas $\lambda\to\infty$ typically leads to underfitting. Choosing an intermediate value to recover the ``true'' underlying manifold is meaningful only if the data-generating distribution stably determines the relevant geometry. However, Section~\ref{section: Numerical Nonidentifiability} uses a concrete example to show that, without additional restrictions, arbitrarily similar data-generating distributions can correspond to manifolds with arbitrarily different geometry (e.g., see Figure~\ref{fig:noisy-circle-density-curve-grid}). Therefore, Section~\ref{section: Selection Based on Downstream Performance} instead advocates choosing $\lambda$ according to the performance of the fitted manifold in its intended downstream task, e.g., a regression task wherein the responses are assumed to lie on an unknown manifold. 

\subsection{Numerical Nonidentifiability of Underlying-Manifold Geometry}\label{section: Numerical Nonidentifiability}

The purpose of this subsection is to use the reference example and corrugated alternatives in Figure~\ref{fig:noisy-circle-density-curve-grid} as a specific example of a general possibility: without suitable restrictions, important geometric features of the ``true'' underlying manifold need not be stably determined by the data-generating distribution. Appendix~\ref{appendix: Quantile-Corrugated Nonidentifiability Example} gives the complete construction, derivations, and proofs of the example; here we summarize its intuition and main conclusion. 

Informally, a geometric feature is \textit{numerically identifiable} at a reference distribution, over a family of distributions, if nearby data-generating distributions necessarily correspond to nearby values of that feature. Definition~\ref{def: numerical identifiability of geometry} provides a mathematically precise and rigorous formulation of this stability requirement by measuring the maximal variation of the feature within bounded-Lipschitz neighborhoods of the reference distribution. In the example in this subsection, the features are manifold length and the number of radial corrugations.

Begin with panels~(o) and~(t) of Figure~\ref{fig:noisy-circle-density-curve-grid}. Panel~(t) shows the reference manifold $\mathcal{M}_{\star}$, the unit circle, whose length is $2\pi$. To generate $\boldsymbol{X}_{\star}$, a point is sampled uniformly from $\mathcal{M}_{\star}$ and perturbed by isotropic Gaussian noise. The resulting ring-shaped density is in panel~(o). For each integer $q\geq1$, Appendix~\ref{appendix: Quantile-Corrugated Nonidentifiability Example} constructs a smooth corrugated manifold $\mathcal{M}_q$ and a noisy random variable $\boldsymbol{X}_q$. Figure~\ref{fig:noisy-circle-density-curve-grid} displays $\mathcal{M}_q$ and the probability density function (PDF) of $\boldsymbol{X}_q$ for $q\in [9]$. The latent component has support exactly $\mathcal{M}_q$; after adding nondegenerate Gaussian noise, $\boldsymbol{X}_q$ has a positive density and support $\mathbb{R}^2$, containing $\mathcal{M}_q$.

The construction (see Appendix~\ref{appendix: Quantile-Corrugated Nonidentifiability Example}) makes $\mathcal{M}_q$ oscillate radially through the annular region containing most of the probability under $\boldsymbol{X}_{\star}$. Its latent sampling distribution and Gaussian noise level are calibrated so that $\boldsymbol{X}_q$ converges to $\boldsymbol{X}_{\star}$ in distribution as the oscillations become finer (i.e., $q\rightarrow\infty$). For $q=1$, the ten coarse corrugations of $\mathcal{M}_1$ in panel~(f) remain visible in the density in panel~(a). For $q=9$, $\mathcal{M}_9$ has ninety fine corrugations in panel~(s), yet the density of $\boldsymbol{X}_9$ in panel~(n) is visually close to the reference density in panel~(o).

Proposition~\ref{prop: quantile-calibrated numerical nonidentifiability} and Corollary~\ref{cor: quantitative numerical nonidentifiability} in Appendix~\ref{appendix: Quantile-Corrugated Nonidentifiability Example} show that $\boldsymbol{X}_q$ converges to $\boldsymbol{X}_{\star}$ in distribution, while $\mathcal{M}_q$ has exactly $10q$ radial corrugations and the length of the curve $\mathcal{M}_q$ tends to infinity (note that the length of $\mathcal{M}_\star$ is equal to $2\pi$). More precisely, given any bounded-Lipschitz tolerance and any finite thresholds for length and corrugation count, one can choose a sufficiently large $q$ so that the two distributions are closer than that tolerance and $\mathcal{M}_q$ exceeds both geometric thresholds. Thus, every bounded-Lipschitz neighborhood of $\boldsymbol{X}_\star$ contains data-generating distributions $\boldsymbol{X}_q$ whose manifolds $\mathcal{M}_q$ have unbounded length and corrugation count. Distributional proximity between $\boldsymbol{X}_q$ and $\boldsymbol{X}_\star$ therefore does not place a finite local bound on these features, let alone force $\mathcal{M}_q$ to be similar to $\mathcal{M}_\star$.

The example shows that arbitrarily close data-generating distributions can correspond to manifolds with very different geometries. Numerical identifiability may be restored by imposing additional restrictions, e.g., \cite{duchamp1996extremal} suggest restricting variations of planar principal curves to sufficiently low frequencies. Without such restrictions, recovery of the ``true'' underlying manifold is not a numerically well-defined criterion for selecting $\lambda$. Therefore, Section~\ref{section: Selection Based on Downstream Performance} bases the selection of $\lambda$ on the fitted manifold's performance in a downstream task.

\begin{figure}[!htbp]
    \centering
    \includegraphics[width=1\linewidth]{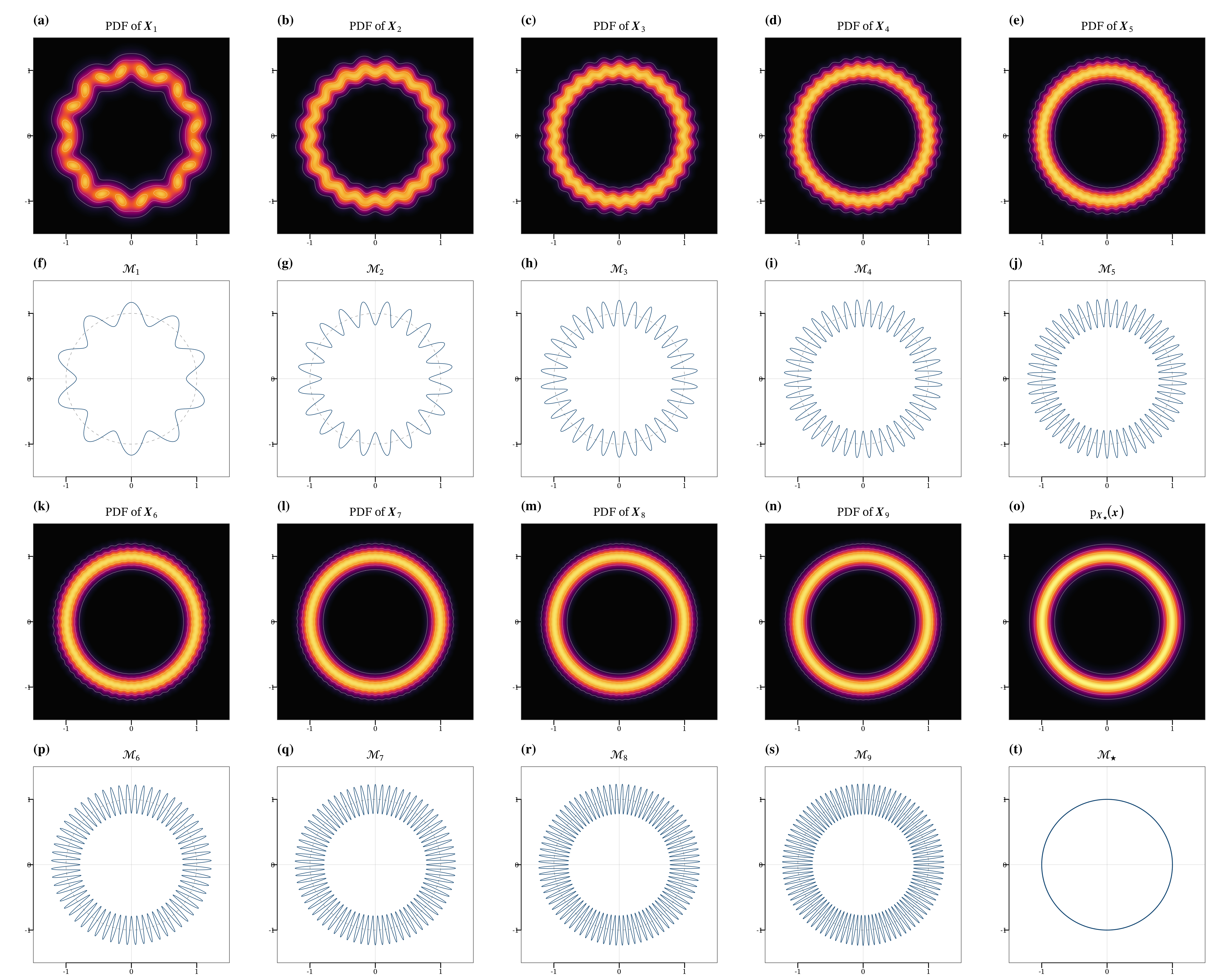}
    \vspace{-2.5em}
    \caption{\footnotesize Illustration of numerical nonidentifiability. Each PDF panel is paired with its underlying manifold below it. Panels~(a)--(e) and~(k)--(n) show the PDFs of $\boldsymbol{X}_1,\ldots,\boldsymbol{X}_9$, and panels~(f)--(j) and~(p)--(s) show the corresponding manifolds $\mathcal{M}_1,\ldots,\mathcal{M}_9$. The PDF of $\boldsymbol{X}_{\star}$ and the reference unit circle $\mathcal{M}_{\star}$ appear in panels~(o) and~(t), respectively; dashed curves mark $\mathcal{M}_{\star}$ for reference. As $q$ increases, $\mathcal{M}_q$ becomes increasingly corrugated, while the PDF of $\boldsymbol{X}_q$ becomes visually closer to that of $\boldsymbol{X}_{\star}$. All panels use common coordinate limits and equal aspect ratios; all PDF panels use the same scale. Appendix~\ref{appendix: Quantile-Corrugated Nonidentifiability Example} gives the computational details for the density panels.}
    \label{fig:noisy-circle-density-curve-grid}
\end{figure}

\subsection{Selection Based on Downstream Performance}
\label{section: Selection Based on Downstream Performance}

Section~\ref{section: Numerical Nonidentifiability} motivates selecting $\lambda$ based on performance in a downstream task rather than on geometric recovery. As a proof of concept, we apply \textit{extrinsic local-linear regression} \citep{lin2017extrinsic} to the task of predicting manifold-valued responses from Euclidean predictors. The extrinsic regression model requires the response manifold $\mathcal{M}$ (i.e., the manifold on which the responses lie) and its embedding $\mathcal{M}\rightarrow\mathbb R^D$ to be known. It estimates the conditional mean of the responses by local-linear smoothing in $\mathbb R^D$ and projects the estimate onto the embedded manifold $\mathcal{M}$. In our setting, however, the response manifold $\mathcal M$ is unknown.
We observe $\{(\boldsymbol t_i,\boldsymbol X_i)\}_{i\in[N]}$, where $\boldsymbol t_i\in\mathbb R^p$ is a predictor and $\boldsymbol X_i\in\mathbb R^D$ is the corresponding observed response.
The observed responses lie in the vicinity of $\mathcal M$ and are corrupted by ambient noise. PME supplies a fitted manifold representation and embedding for use in regression. The remainder of this paragraph gives an overview; the following paragraphs give the details.
We select $\lambda$ by two-level cross-validation (CV).
For each candidate $\lambda$ and validation fold, we fit PME only to the responses $\boldsymbol{X}_i$ in the training set and project them onto the fitted manifold.
Holding this PME fit fixed, leave-one-out CV using the training predictors and their corresponding projected responses selects the local-linear bandwidth $h$ \citep[][Section 2]{lin2017extrinsic}; then, the extrinsic local-linear predictor is fitted at the selected bandwidth using all training predictors and projected responses.
We evaluate this predictor against the observed, unprojected responses $\boldsymbol{X}_i$ in the outer validation fold and select $\lambda$ by averaging the fold-specific mean squared errors.

Let $\mathcal I_1,\ldots,\mathcal I_K$ denote the validation folds that partition $[N]$; let $\mathcal I_{-k}:=[N]\setminus\mathcal I_k$ denote the training set associated with $\mathcal I_k$.
For each $\lambda>0$ and fold $k$, fit PME only to $\{\boldsymbol X_i:i\in\mathcal I_{-k}\}$, obtaining the map $\Bf_{N,\lambda}^{*(-k)}$ and the fitted manifold $\mathcal M_{N,\lambda}^{(-k)}:=\Bf_{N,\lambda}^{*(-k)}(\M)$. The fitted $\mathcal M_{N,\lambda}^{(-k)}$ is an embedded submanifold of $\mathbb R^D$, and its inclusion map supplies the Euclidean embedding used in extrinsic regression. For each $i\in\mathcal I_{-k}$, define the projected training response $\boldsymbol Y_{i,\lambda}^{(-k)}:=\Bf_{N,\lambda}^{*(-k)}\!\bigl(\Bpi_{\Bf_{N,\lambda}^{*(-k)}}(\boldsymbol X_i)\bigr)\in\mathcal M_{N,\lambda}^{(-k)}$.
The set $\{(\boldsymbol t_i,\boldsymbol Y_{i,\lambda}^{(-k)}):i\in\mathcal I_{-k}\}$ is the regression training sample for fold $k$; the observations indexed by $\mathcal I_k$ remain reserved for validation.

Within the training set $\mathcal I_{-k}$, select the local-linear bandwidth $h$ by leave-one-out CV while holding the PME fit $\mathcal M_{N,\lambda}^{(-k)}$ fixed. Denote the selected bandwidth by $\widehat h_{\lambda,k}$. Using $\widehat h_{\lambda,k}$, fit the extrinsic local-linear regression to the training sample and denote the fitted regression function by $\widehat{\boldsymbol\mu}_{\lambda,\widehat h_{\lambda,k}}^{(-k)}$. Define its mean squared error on the unprojected responses in $\mathcal I_k$ by
\[
\operatorname{MSE}_k(\lambda,\widehat h_{\lambda,k})
=
\frac{1}{|\mathcal I_k|}
\sum_{i\in\mathcal I_k}
\left\|
\boldsymbol X_i-
\widehat{\boldsymbol\mu}_{\lambda,\widehat h_{\lambda,k}}^{(-k)}
(\boldsymbol t_i)
\right\|^2.
\]
Averaging the errors gives $\operatorname{CV}(\lambda):=K^{-1}\sum_{k=1}^K\operatorname{MSE}_k(\lambda,\widehat h_{\lambda,k})$. We select $\widehat\lambda\in\arg\min_{\lambda>0}\operatorname{CV}(\lambda)$.

\begin{figure}[ht]
    \centering
    \includegraphics[width=1\linewidth]{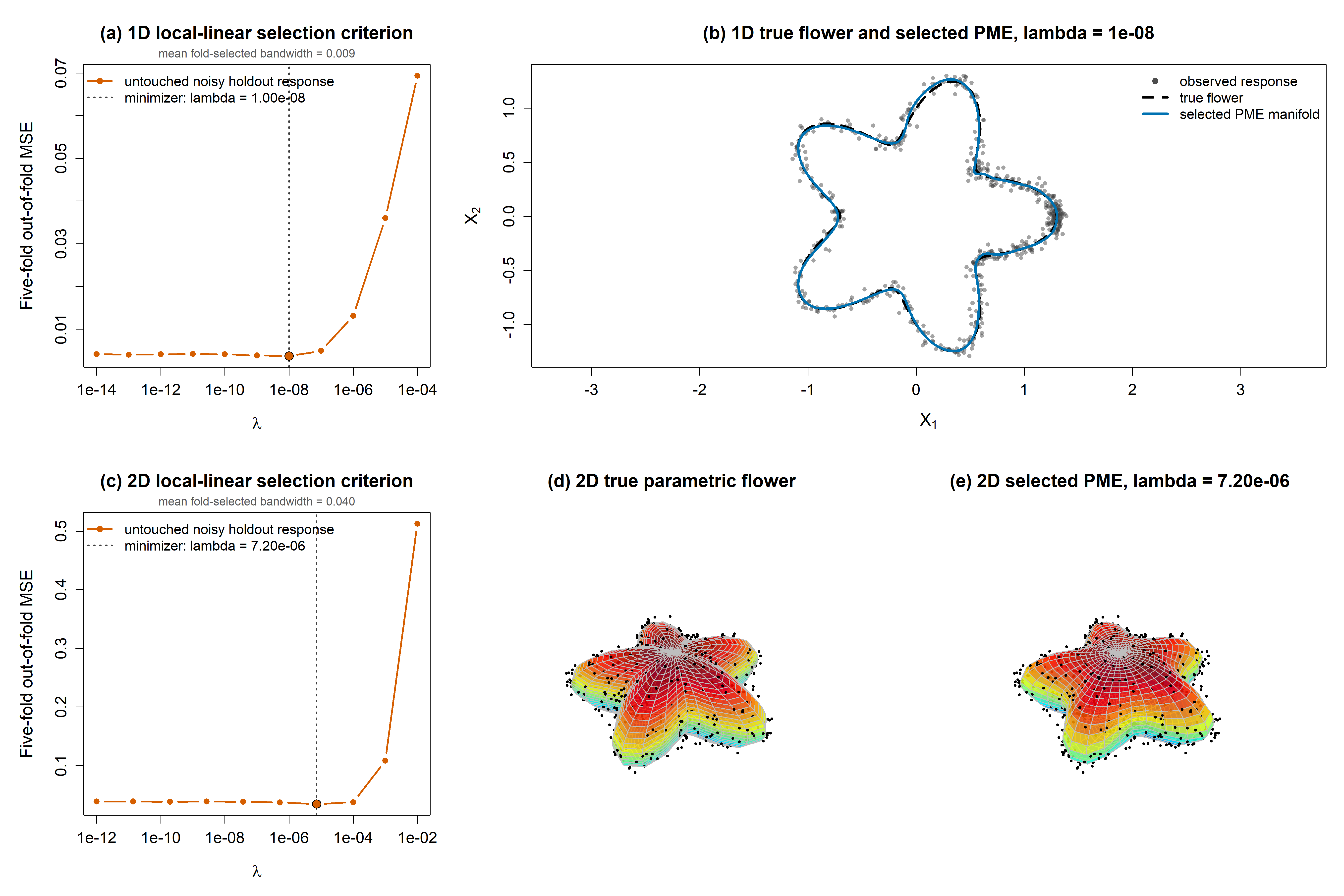}
    \vspace{-2.5em}
    \caption{\footnotesize Selection of $\lambda$ by downstream extrinsic regression \citep{lin2017extrinsic}. Panels~(a) and~(c) show five-fold ($K=5$) cross-validation MSE for the curve and surface experiments, with dashed lines marking the minimizers; panels~(b), (d), and~(e) show the corresponding observed responses, data-generating manifolds, and selected PME fits.}
    \label{fig:extrinsic-regression-lambda-selection}
\end{figure}

We illustrate the procedure with two numerical experiments. The curve experiment uses $N=600$ scalar predictors $t_i\overset{\mathrm{iid}}{\sim}\mathrm{Unif}(-1,1)$ and responses in $\mathbb R^2$ ($D=2$), whereas the surface experiment uses $N=800$ two-dimensional predictors $\boldsymbol t_i\overset{\mathrm{iid}}{\sim}\mathrm{Unif}([-1,1]^2)$ and responses in $\mathbb R^3$ ($D=3$).
In each experiment, the regression function $\boldsymbol\mu$ maps the predictor domain deterministically and noninjectively onto a five-petal flower manifold.
The observed responses satisfy $\boldsymbol X_i=\boldsymbol\mu(\boldsymbol t_i)+\boldsymbol\varepsilon_i$, where $\boldsymbol\varepsilon_i\overset{\mathrm{iid}}{\sim}\mathcal N_D(\boldsymbol 0,0.04^2\boldsymbol I_D)$.
Explicit formulas for $\boldsymbol\mu$ are given in Appendices~\ref{section: Downstream Regression with Planar-Flower Responses, D=2} and~\ref{section: Downstream Regression with Planar-Flower Responses, D=3}, respectively. Figure~\ref{fig:extrinsic-regression-lambda-selection} summarizes the results. Panels~(a) and~(c) report $\operatorname{CV}(\lambda)$ for the curve and surface experiments, respectively, while panels~(b), (d), and~(e) compare the observed responses, data-generating manifolds, and selected PME fits. The criterion is minimized at $\widehat\lambda=10^{-8}$ in the curve experiment and $\widehat\lambda=7.2\times10^{-6}$ in the surface experiment. The fitted manifolds are selected to optimize out-of-fold regression performance; they are not intended to recover the data-generating manifolds.

The selection principle is not specific to extrinsic local-linear regression. It applies to any prespecified downstream task with an appropriate validation loss. If no downstream task is prespecified, Appendix~\ref{section: Latent-Variable Model and Tuning Parameter Selection via Oracle Criterion} provides an alternative, assumption-dependent model-based criterion that typically yields a reasonable value of $\lambda$ in practice.

\section{Conclusion and Future Research}\label{section: discussions and future research}

In this article, we develop theoretical foundations for PME with compact Riemannian templates, extending existing formulations beyond Euclidean templates. This extension accommodates non-Euclidean topology and allows global topological information from TDA to guide template selection. Using Sobolev space theory, we prove the existence of penalized minimizers (i.e., our defined principal manifolds). Under the stated conditions, we also establish convergence of the PA algorithm and consistency of the empirical minimizer for its population counterpart. Both results hold up to Riemannian isometries of the template. Numerical nonidentifiability of the underlying manifold's geometry motivates choosing the tuning parameter according to the fitted manifold's performance in a prespecified downstream task. When no such task is specified, Appendix~\ref{section: Latent-Variable Model and Tuning Parameter Selection via Oracle Criterion} provides an alternative criterion based on a latent-variable model, subject to additional assumptions.

Several directions merit further research. First, improving computational efficiency would make the PA algorithm more practical for large samples. The RKHS formulation in \eqref{eq: notations for the representer theorem} requires solving linear systems whose dimension increases with the sample size. A reduced-basis expansion can reduce this computational cost \citep[e.g.,][]{ma2015efficient, meng2020more}. Another option is to use alternative spline bases, such as those provided by the \texttt{R} packages \texttt{splines} and \texttt{mgcv} \citep{wood2017generalized}. Appendix~\ref{appendix: Projection-Adaptation Algorithm} discusses them in more detail.

Second, we plan to develop practical implementations of the PA algorithm on a wider range of compact Riemannian templates. Our simulations show that the algorithm performs well on $[0,1]$, $\mbS^1$, and $\mbS^2$ when suitably initialized. Extending these implementations to other templates requires computationally efficient representations of the reproducing kernel of $H^2(\M)$. Recent advances in the study of kernels on compact Riemannian manifolds \citep[e.g.,][]{li2023inference, mostowsky2025geometrickernels, wang2025riemannian, li2026statistical, he2026scalable} provide a basis for this work.

Third, we plan to study the asymptotic behavior of a fitted principal manifold near the boundary of the latent manifold. Studies of locally linear embedding \citep{roweis2000nonlinear} have established boundary asymptotics \citep{wu2023locally, kuo2025boundary}. We aim to adapt their techniques to the PME framework.

Finally, our current PME formulation, which uses a single point cloud, could be extended to multiview data. Recent advances in manifold learning and sensor fusion offer a basis for this extension \citep{talmon2019latent, ding2024kernel, ding2026generalized}.

\section*{Code Availability}

Code for the algorithm is publicly available in our \texttt{R} package \texttt{CompactPME} on \url{https://github.com/cmperez024/Compact-Principal-Manifold-Estimation}. Simulation code for replication purposes is also available in the repository.

\section*{Acknowledgments}

We are grateful to Dr.~Jonathan Stewart (Florida State University) for helpful discussions and valuable suggestions regarding asymptotic analysis.

\clearpage
\appendix
\section*{Appendix}
\startcontents[appendix]
\begin{center}
    \textbf{Contents}
\end{center}
\printcontents[appendix]{}{1}{}

\newpage
\section{Notation}\label{appendix: Mathematical Preparations}

\noindent \textbf{Constants and Parameters}

$D$ --- The ambient dimension.

$d$ --- The intrinsic dimension of the underlying manifold.

$\lambda$ --- The regularization parameter/tuning parameter in \eqref{eq: the core min problem}.

$N$ --- The sample size.

\noindent\textbf{Optimization}

$\mcL_\lambda$ --- The population-level loss functional; see \eqref{eq: the core min problem}.

$\mcL_{N,\lambda}$ --- The empirical-level loss functional; see \eqref{eq: core iterative algorithm, empirical}.

$\Bf_\lambda^{(n)}$ --- The $n$-th iteration of the population-level PA algorithm based on regularization parameter $\lambda$; see \eqref{eq: the core iterative algorithm}; not to be confused with the $n$-th derivative.

$\Bf_{N,\lambda}^{(n)}$ --- The $n$-th iteration of the empirical-level PA algorithm based on regularization parameter $\lambda$ and sample size $N$; see \eqref{eq: core iterative algorithm, empirical}.

$\Bf^*_{\lambda}$ --- A minimizer of $\mcL_\lambda(\Bf)$; see \eqref{eq: the core min problem}.

$\Bf^*_{N,\lambda}$ --- A minimizer of the empirical $\mcL_{N, \lambda}(\Bf)$; see \eqref{eq: core iterative algorithm, empirical}.

$\mathcal{Q}_\lambda(\Bf\,\vert\,\Bg)$ --- See \eqref{eq: def of the Q functional}.

$\mathcal{T}_\lambda$ --- The updating operator defined by $\mathcal{T}_\lambda(\Bg):=\argmin_{\Bf\in\mathscr{F}} \mathcal{Q}_\lambda(\Bf\,\vert\,\Bg)$.

\noindent\textbf{Probability and Measure Theory}

$\boldsymbol{X}$ --- An $\mbR^D$-valued random variable drawn from $\mbP$.

$\dist(\boldsymbol{X}, \Bf)$ --- See \eqref{eq: def of the fitting error term}.

$\mbE$ --- Expectation.

$\mbP$ --- Probability measure defined on the Borel $\sigma$-algebra $\mcB(\mbR^D)$.

$\operatorname{supp}(\mbP)$ --- The support of a probability measure $\mbP$.

$\mathbb{P}_N:=\frac{1}{N}\sum_{i=1}^N \delta_{\boldsymbol{X}_i}$ --- Empirical distribution.

\noindent\textbf{Spaces and Manifolds}

$\M$ --- The template manifold with $\dim\M=d$.

$\mathcal{M}$ --- The latent manifold we wish to estimate.

$H^2(\M)$ --- Sobolev space; see Section~\ref{section: Sobolev Spaces on Riemannian Manifolds}.

$C(\M)$ --- Collection of continuous functions defined on $\M$.

$\mathscr{F}$ --- See \eqref{eq: the core min problem}.

\noindent\textbf{Operators and Functions}

$\nabla$ --- The Levi-Civita connection associated with $g$ on $\M$.

$\nabla^{\Bf}$ --- The Levi-Civita connection associated with $g_{\Bf}$ on $\M$.

$\Delta$ --- Laplace-Beltrami operator on $\M$; see \eqref{eq: def of Laplace-Beltrami operator}.

$\nabla^2$ --- The Hessian operator; see \eqref{eq: def of Hess}.

$\Bpi_{\Bf}$ --- Projection index; see Definition~\ref{def: projection index}.

$\norm{\cdot}$ and $\norm{\cdot}_{\mathbb{R}^D}$ --- Euclidean norm.

$\delta$ --- The Euclidean metric.

$g_{\Bf}$ --- Pullback Riemannian metric associated with $\Bf$ and $\M$; see \eqref{eq: def of metric g_f}.

$\II$ --- The second fundamental form.

$\lvert \cdot \rvert_g$ --- The pointwise norm with respect to $g$; see \eqref{eq: vert g norm using bases}.

$\norm{\cdot}_{L^2(\M)}$ --- The $L^2$ norm of a function or a $(0,2)$ tensor; see \eqref{eq: L2 norm on M} and \eqref{eq: def of L2 norm of a (0,2) tensor}.

$A\Subset B$ --- compact containment, meaning that the closure
$\overline A$ is a compact subset of $B$.

\section{Poisson Surface Reconstruction versus Principal Manifold Estimation}\label{appendix: Poisson Surface Reconstruction versus Principal Manifold Estimation}

\textit{Poisson surface reconstruction} \citep[PSR;][]{kazhdan2006poisson} is widely used for reconstructing smooth, watertight surfaces. This section compares PSR with the principal manifold estimation (PME) framework considered in this article.

\subsection{Poisson Surface Reconstruction}

The standard input to PSR is an oriented point cloud, that is, a set of pairs
\[
    \{(\boldsymbol{X}_i,\boldsymbol{n}_i)\}_{i\in[N]}
    \subset \mathbb{R}^3 \times \mathbb{S}^2,
\]
where each $\boldsymbol{X}_i \in \mathbb{R}^3$ is a point location and each $\boldsymbol{n}_i \in \mathbb{S}^2$ is a consistently oriented normal vector. The normal vectors are an essential part of the PSR formulation. Specifically, PSR constructs an implicit function whose gradient is encouraged to agree with the vector field induced by the oriented normals, and the reconstructed surface is then extracted as an isosurface.

A raw point cloud, however, should not be assumed to contain such normal vector information. In many acquisition and storage pipelines, a point cloud consists only of three-dimensional coordinates, while normals are optional attributes supplied either by the acquisition device, by the user, or by a preprocessing step. When normals are not provided, a common baseline is local PCA \citep{singer2012vector, wu2018think}, i.e., one computes the covariance matrix of a local neighborhood and takes the eigenvector associated with the smallest eigenvalue as an estimate of the local normal direction. This step estimates an unoriented normal line. A separate orientation procedure is then needed to choose signs consistently across the point cloud.

Thus, a two-dimensional closed surface can be reconstructed from a three-dimensional point cloud using PSR only after reliable normals have been supplied or estimated and globally oriented. Given such an oriented point cloud, PSR returns an implicit surface, typically represented in practice by an extracted triangulated mesh. If the resulting mesh is a closed manifold and of genus zero, then an additional parameterization step may be applied to obtain a map $\boldsymbol{f}:\mathbb{S}^2 \to \mathbb{R}^3$. Such a spherical parameterization can be constructed, for example, by spherical conformal or related mesh-parameterization methods \citep{jermyn2017elastic}.

\subsection{Comparison}

The PME framework differs from PSR in both its input requirements and its statistical objective. PME is posed directly in terms of the observed point locations $\{\boldsymbol{X}_i\}_{i=1}^N \subset \mathbb{R}^3$ and does not require normal vectors as part of the input. Moreover, PME estimates a parameterized manifold associated with a chosen template space $\mathfrak{M}$, whereas PSR first reconstructs an implicit surface from an oriented point cloud and only subsequently, if appropriate, permits a separate parameterization step.

PSR is highly effective in many dense-sampling regimes and is widely used for large scan data sets. For example, the raw range scans of the ``Stanford Bunny'' in the Stanford 3D Scanning Repository contain 362272 points.\footnote{\url{https://graphics.stanford.edu/data/3Dscanrep/}} With sufficiently dense sampling, accurate and consistently oriented normals, and appropriate reconstruction parameters, PSR can produce a high-quality approximation to the underlying surface. Nevertheless, in finite-sample settings with moderate sample size, noise, uncertain normal estimation, or suboptimal reconstruction parameters, the topology of the reconstructed mesh need not agree with the topology of the data-generating surface. (Accordingly, the comparison below should not be read as a claim that PME uniformly dominates PSR for mesh reconstruction. PSR remains the natural tool when the deliverable is a large-scale watertight mesh from dense oriented scans.)

To illustrate this issue, consider data points $\{\boldsymbol{X}_i\}_{i\in[N]} \subset \mathbb{R}^3$ generated from a latent two-dimensional cashew-shaped surface with Gaussian noise, where $N=1000$. See Section~\ref{section: Surface of a Moon/Cashew (2d, 3D)} for details of the data-generating mechanism. Figure~\ref{fig:cashew_hole} displays the corresponding persistence diagram and a Poisson surface reconstruction. The persistence diagram, computed using the \texttt{TDA} package in \texttt{R} \citep{fasy2021tda}, shows one prominent connected component, no prominent one-dimensional homology feature, and one prominent two-dimensional homology feature. Under the orientable closed-surface classification discussed in Section~\ref{section: Topology}, this pattern is consistent with genus-zero, and hence spherical topology. Thus, $\mathfrak{M}=\mathbb{S}^2$ is a natural template for PME in this example. 

The Poisson surface reconstruction, computed using the \texttt{SurfaceReconstruction} package\footnote{\url{https://github.com/stla/SurfaceReconstruction}} as a wrapper to \texttt{CGAL},\footnote{\url{https://www.cgal.org/}} contains a visible missing patch on the right-hand side. Consequently, the reconstructed mesh does not match either the topology suggested by the persistence diagram or the known topology of the data-generating mechanism. This example does not imply that PSR is intrinsically unreliable; rather, it illustrates that a generic reconstruction pipeline can return an incorrect topology when the sampling density, noise level, normal estimation, or reconstruction parameters are not sufficiently favorable. If one is aware of the correct or estimated topology (such as by using TDA methods), one can apply mesh repairing algorithms to induce the correct topology \citep{attene2013polygon}.

In contrast, once the template $\mathfrak{M}$ has been selected, PME preserves the topology of that template by construction. Therefore, if the topological diagnostic correctly identifies $\mathfrak{M}=\mathbb{S}^2$, the resulting PME estimate avoids this particular failure mode of producing a surface with a different topology (see Figures~\ref{fig:2d3D_star_moon} and~\ref{fig:sph_results_panel_8}).

\begin{figure}[!htbp]
    \centering
    \includegraphics[width=1\linewidth]{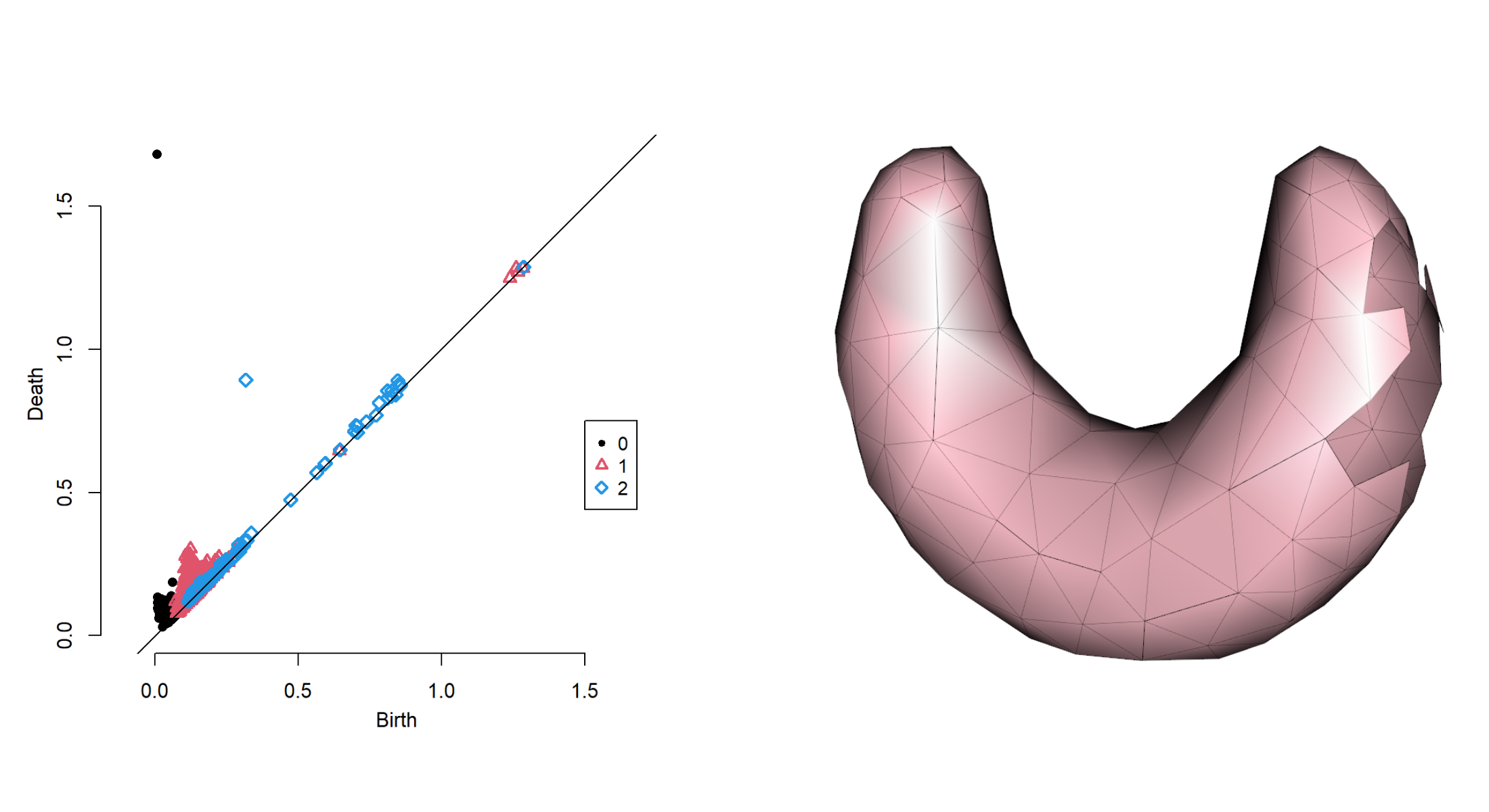}
    \vspace{-0.1cm}
    \caption{\footnotesize Left: persistence diagram for $N=1000$ noisy samples from a two-dimensional cashew-shaped surface. Right: Poisson surface reconstruction of the same point cloud, showing a missing patch on the right side.}
    \label{fig:cashew_hole}
\end{figure}

\section{Geometric Interpretation of the Penalty}\label{section: Geometric Interpretation of the Penalty}

In this section, we provide a detailed interpretation of the penalty \(\Vert \nabla^2 \boldsymbol{f} \Vert^2_{L^2(\M)}\) used in \eqref{eq: the core min problem}. First, we explain why we use the Hessian \(\nabla^2\), rather than the Laplace-Beltrami operator, in defining the penalty term (see Section \ref{section: Hessian Penalty versus Laplace-Beltrami Penalty}). Second, we describe the relationship between our penalty \(\Vert \nabla^2 \boldsymbol{f} \Vert^2_{L^2(\M)}\) and the curve-length penalty proposed by \cite{kegl2000learning} (see Section \ref{section: Curve Length Penalty}). Finally, through the orthogonal decomposition in Section \ref{section: An Orthogonal Decomposition of the Penalty}, we show that the penalty is the sum of (i) the bending energy \citep{derfner2026branch} of the fitted embedding (see Section \ref{section: Curvature Penalty}) and (ii) a discrepancy between the Levi--Civita connections induced by the metric \(g\) and the pull-back metric \(g_{\Bf}:=\Bf^*\delta\) (see Section \ref{section: Metric Mismatch Penalty}).

\subsection{Hessian Penalty versus Laplace-Beltrami Penalty}\label{section: Hessian Penalty versus Laplace-Beltrami Penalty}

Theorem~\ref{thm: PME with lambda=infty} provides an initial interpretation of the Hessian-based penalty \(\Vert \nabla^2 \boldsymbol{f} \Vert^2_{L^2(\M)}\), showing that a heavily penalized estimator tends to have low complexity (e.g., a hyperplane determined by finitely many parameters or a single point). Generally, the Hessian penalty \(\Vert \nabla^2 \boldsymbol{f} \Vert^2_{L^2(\M)}\) favors estimators that lie close to the finite-dimensional space 
$$\mathcal{N}(\nabla^2):=\left\{\Bf\in H^2(\M):\,\Vert \nabla^2 \boldsymbol{f} \Vert^2_{L^2(\M)}=0 \right\}$$ 
(see Remark~\ref{remark: Affine and harmonic function spaces} in Appendix~\ref{appendix: Reproducing Kernel Hilbert Spaces}). This feature helps reduce overfitting effectively.

By contrast, a penalty defined through the Laplace--Beltrami operator need not favor a finite-dimensional space. For any smooth scalar-valued function \(f\) on \(\M\), the Laplace--Beltrami operator \(\Delta\) is defined by
\begin{align}\label{eq: def of Laplace-Beltrami operator}
    \Delta f(\boldsymbol{m}) := \operatorname{tr}_g(\nabla^2 f)(\boldsymbol{m}) = \sum_{i=1}^d \nabla^2 f(e_i,e_i),
\end{align}
where \(\{e_i\}_{i\in[d]}\) is any \(g\)-orthonormal basis of the tangent space \(T_{\boldsymbol{m}}\M\). For a vector-valued map \(\Bf=(f_1,\ldots,f_D)\), we define $\Delta \Bf := (\Delta f_1,\ldots,\Delta f_D)$. Accordingly, one may consider the penalty $\Vert \Delta \boldsymbol{f} \Vert^2_{L^2(\M)}=\sum_{j=1}^D\Vert \Delta f_j \Vert^2_{L^2(\M)}$ \citep[e.g.,][Section 2.2]{wahba1990spline}. When the boundary \(\partial \M\) is empty, the Hessian penalty \(\Vert \nabla^2 \boldsymbol{f} \Vert^2_{L^2(\M)}\) and the Laplace--Beltrami penalty \(\Vert \Delta \boldsymbol{f} \Vert^2_{L^2(\M)}\) are related through \textit{Bochner's formula} \citep[][Chapter 3]{li2012geometric}:
\begin{align}\label{Bochner's formula}
    \Vert \Delta \boldsymbol{f} \Vert^2_{L^2(\M)}
    =
    \Vert \nabla^2 \boldsymbol{f} \Vert^2_{L^2(\M)}
    +
    \sum_{j=1}^D\int_{\M}\mathrm{Ric}(\nabla f_j,\nabla f_j)\,d\mathrm{vol}_g,
\end{align}
where \(\mathrm{Ric}(\cdot,\cdot)\) denotes the Ricci curvature tensor. An estimator that is heavily penalized by \(\Vert \Delta \boldsymbol{f} \Vert^2_{L^2(\M)}\) tends to lie close to the null space $\mathcal{N}(\Delta):=\{\Bf\in H^2(\M): \Vert \Delta \boldsymbol{f} \Vert^2_{L^2(\M)}=0\}$, i.e., the space of harmonic maps. However, unlike the Hessian penalty, the null space \(\mathcal{N}(\Delta)\) may be infinite-dimensional when the boundary $\partial\M\ne\emptyset$ (see Remark~\ref{remark: Affine and harmonic function spaces}). Consequently, the Laplace-Beltrami penalty may be less effective in controlling overfitting, since the infinitely many dimensions still permit substantial flexibility in manifold estimation. Moreover, on Euclidean domains with boundary in dimensions $d\geq2$, the Laplace--Beltrami penalty does not in general yield the linear-PCA conclusion of Theorem~\ref{thm: PME with lambda=infty}(ii), because its null space can contain nonaffine harmonic maps. The flat one-dimensional interval is an exception: when $\M=[0,1]$ carries the Euclidean metric $g=dt\otimes dt$, one has $\Delta\Bf=\Bf''$ and $\nabla^2\Bf=\Bf''\,dt\otimes dt$, so the two penalties coincide.
For these reasons, we adopt and focus on the Hessian penalty \(\Vert \nabla^2 \boldsymbol{f} \Vert^2_{L^2(\M)}\).

\subsection{Curve Length Penalty}\label{section: Curve Length Penalty}

\cite{kegl2000learning} propose a principal curve framework that minimizes the loss function in \eqref{eq: the loss utilized by Kegl et al.}, where its penalty term is equal to the length of the fitted curve \citep[][Chapter 1]{doCarmo1976DifferentialGeometry}. When the template manifold $(\M,g)$ is the unit circle $\mathbb{S}^1$ endowed with the metric $g$ induced from $\mathbb R^2$, the curve length penalty utilized by \cite{kegl2000learning} is dominated by our proposed penalty $\Vert \nabla{}^2 \boldsymbol{f} \Vert^2_{L^2(\mathbb{S}^1)}$. Specifically, we have that
\begin{align}\label{eq: the dominating inequality related to Kegl et al.}
    \text{the length of the curve }\Bf(\mathbb{S}^1) \le (2\pi)^{3/2} \cdot \Vert \nabla^2 \Bf \Vert_{L^2(\mathbb{S}^1)},
\end{align}
and the proof of \eqref{eq: the dominating inequality related to Kegl et al.} is given in Appendix \ref{Appendix: Proofs} using Poincaré–Wirtinger's inequality (Lemma~\ref{lemma: Poincaré-Wirtinger's inequality}). Therefore, in this case, our proposed $\Vert \nabla^2 \Bf \Vert^2_{L^2(\mathbb{S}^1)}$ penalizes the curve length, as in \cite{kegl2000learning}. In addition, the inequality in \eqref{eq: the dominating inequality related to Kegl et al.} provides an alternative validation of Theorem~\ref{thm: PME with lambda=infty}(i).

The curve-length inequality in \eqref{eq: the dominating inequality related to Kegl et al.} also explains the large-\(\lambda\) fit in Figure~\ref{fig:interval_S1_full}(e). When \(\lambda\) is excessively large, the objective strongly penalizes \(\Vert\nabla^2\Bf\Vert_{L^2(\mathbb S^1)}^2\). By \eqref{eq: the dominating inequality related to Kegl et al.}, a small Hessian norm forces the fitted closed curve to have very small length. Thus, the corresponding fit is nearly point-like.

\subsection{An Orthogonal Decomposition of the Penalty}\label{section: An Orthogonal Decomposition of the Penalty}

We decompose the proposed penalty $\Vert \nabla{}^2 \boldsymbol{f} \Vert^2_{L^2(\M)}$ into two additive terms, each with a geometric interpretation. 

Let $(\M,g)$ be a template manifold satisfying Assumption \ref{assumption: assumption on the template manifold}. Assume henceforth that $\boldsymbol{f}:\M\rightarrow\mathbb{R}^D$ is an embedding \citep[][Chapter 2]{lee2018riemannian}. Precisely, $\Bf$ is a homeomorphism onto its image $\Bf(\M)$ with the subspace topology, and it is an immersion, i.e., $d\boldsymbol{f}:T_{\boldsymbol{m}}\M\rightarrow T_{\boldsymbol{f}(\boldsymbol{m})}\mathbb{R}^D\cong\mathbb{R}^D$ is injective for any $\boldsymbol{m}\in\M$. Then, we may define a new metric $g_{\Bf}=\Bf^*\delta$ on $\M$ by
\begin{align}\label{eq: def of metric g_f}
    g_{\boldsymbol{f}}(u,v):=\delta\big(\, d\boldsymbol{f}(u),\, d\boldsymbol{f}(v) \,\big) \text{ for any }u,v\in\Gamma(T\M),
\end{align}
where $\delta$ is the Euclidean metric on $\mathbb{R}^D$, and $(\M, g_{\Bf})$ is a $d$-dimensional embedded submanifold of $\mathbb{R}^D$ \citep[][Chapter 8]{lee2018riemannian}. Note that the newly defined metric $g_{\Bf}$ in \eqref{eq: def of metric g_f} may differ from the original metric $g$. Since $g_{\Bf}$ is induced by the Euclidean metric $\delta$ on $\mathbb{R}^D$, and $\Bf:\M\rightarrow\mathbb{R}^D$ is an embedding, we identify $(\M, g_{\Bf})$ with its image $\Bf(\M) \subseteq \mathbb{R}^D$.

Let $\nabla^{\boldsymbol{f}}$ be the Levi-Civita connection associated with $g_{\Bf}$ on $\M$. Recall that $\nabla$ denotes the Levi-Civita connection associated with $g$ on $\M$. We will demonstrate that the difference between connections $\nabla$ and $\nablaf$ constitutes an orthogonal component of our penalty term. 

For each $\boldsymbol{m}\in\M$, let $T_{\boldsymbol{f}(\boldsymbol{m})}\boldsymbol{f}(\M)$ be the tangent space and $N_{\boldsymbol{f}(\boldsymbol{m})}\boldsymbol{f}(\M)$ its normal space in $\mathbb{R}^D$. Then, we have the following decomposition of the Hessian $\nabla^2\boldsymbol{f}$, 
\begin{align}\label{eq: the key decomposition of the Hessian}
    \nabla^2\boldsymbol{f}(u,v) = \II(u,v) + d\boldsymbol{f}\big(\nablaf_u v - \nabla_u v\big)  \text{ for any }u,v\in\Gamma(T\M),
\end{align}
whose proof can be found in Appendix \ref{Appendix: Proofs}, where $\II(u,v) \in N_{\boldsymbol{f}(\boldsymbol{m})}\boldsymbol{f}(\M)$ is the second fundamental form of the embedded submanifold $\boldsymbol{f}(\M)\subset\mathbb{R}^D$. Section \ref{section: A Detailed Interpretation} interprets its $g$-weighted norm as the bending energy of the embedding \citep{derfner2026branch}.

The three terms in \eqref{eq: the key decomposition of the Hessian} can be viewed as vectors in $\mathbb{R}^D$. For each $\boldsymbol{m}\in\M$ and $u,v\in \Gamma(T\M)$, we have $(\nablaf_u v - \nabla_u v)|_{\boldsymbol m}\in T_{\boldsymbol{m}}\M$, which implies $d\boldsymbol{f}\big( \,(\nablaf_u v - \nabla_u v )|_{\boldsymbol m}\big)\in T_{\boldsymbol{f}(\boldsymbol{m})}\boldsymbol{f}(\M)$. Note $\II(u,v) |_{\boldsymbol m}\in N_{\boldsymbol{f}(\boldsymbol{m})}\boldsymbol{f}(\M)$ and that $T_{\boldsymbol{f}(\boldsymbol{m})}\boldsymbol{f}(\M)$ and $N_{\boldsymbol{f}(\boldsymbol{m})}\boldsymbol{f}(\M)$ are perpendicular to each other. Then, we have 
\begin{align}\label{eq: pointwise orthogonal decomposition of the penalty term}
\| \nabla^2\boldsymbol{f}(u,v) \|^2_{\mathbb{R}^D} = \| \II(u,v) \|^2_{\mathbb{R}^D} + \| d\boldsymbol{f}\big(\nablaf_u v - \nabla_u v\big) \|^2_{\mathbb{R}^D}, 
\end{align}
where $\| \cdot \|_{\mathbb{R}^D}$ denotes the Euclidean norm in $\mathbb{R}^D$. The orthogonal decomposition in \eqref{eq: pointwise orthogonal decomposition of the penalty term} holds at every point $\boldsymbol{f}(\boldsymbol{m})\in\mathbb{R}^D$ for $\boldsymbol{m}\in\M$. Using the pointwise orthogonal decomposition in \eqref{eq: pointwise orthogonal decomposition of the penalty term}, we have the following $L^2$-decomposition
\begin{align}\label{eq: the key L2 norm decomposition of the penalty}
\begin{aligned}
    \Vert \nabla{}^2 \boldsymbol{f} \Vert^2_{L^2(\M)}
=
\underbrace{ \|\II\|_{L^2(\M)}^2 }_{\text{bending energy}} \;+\;
\underbrace{ \sum_{i,k=1}^d\int_{\M}\left\Vert\, d\boldsymbol{f}\big(\nablaf_{e_i} e_k - \nabla_{e_i} e_k\big)\big\vert_{\boldsymbol{f}(\boldsymbol{m})} \,\right\Vert^2_{\mathbb{R}^D} \;d \vol_g(\boldsymbol{m}) }_{\text{connection mismatch between }g\text{ and }g_{\boldsymbol f}},
\end{aligned}
\end{align}
where $\{e_k\}_{k\in[d]}$ is any $g$-orthonormal basis of the tangent space $T_{\boldsymbol{m}}\M$ at $\boldsymbol{m} \in \M$. Appendix \ref{Appendix: Proofs} provides the detailed derivation from equation \eqref{eq: pointwise orthogonal decomposition of the penalty term} to equation \eqref{eq: the key L2 norm decomposition of the penalty}.

\subsection{A Riemannian Interpretation}\label{section: A Detailed Interpretation}

\subsubsection{Bending Energy}\label{section: Curvature Penalty}

The first term in \eqref{eq: the key L2 norm decomposition of the penalty} is
\[
\begin{aligned}
\|\II\|_{L^2(\M)}^2=\int_{\M}|\II(\boldsymbol m)|_g^2\,d\vol_g(\boldsymbol m),\qquad
|\II(\boldsymbol m)|_g^2
:=\sum_{i,k=1}^d
\|\II_{\boldsymbol m}(e_i,e_k)\|_{\mathbb R^D}^2,
\end{aligned}
\]
where, at each $\boldsymbol m\in\M$, $\{e_i\}_{i=1}^d$ is any $g$-orthonormal basis of $T_{\boldsymbol m}\M$. Since $\II_{\boldsymbol m}(e_i,e_k)$ is the normal component of $\nabla^2\boldsymbol f(e_i,e_k)$, this term penalizes the normal bending of the embedding $\boldsymbol f$ in the ambient Euclidean space \citep{derfner2026branch}.

\subsubsection{Connection Mismatch Penalty}\label{section: Metric Mismatch Penalty} 
The second term on the right-hand side of \eqref{eq: the key L2 norm decomposition of the penalty} quantifies the discrepancy between the Levi--Civita connections $\nabla$ and $\nablaf$, associated with the metrics $g$ and $g_{\boldsymbol f}$, respectively. 

By the fundamental theorem of Riemannian geometry \citep[][Theorem~5.10]{lee2018riemannian}, $\nabla \neq \nablaf$ implies $g \neq g_{\boldsymbol f}$. However, $g \neq g_{\boldsymbol f}$ does not imply $\nabla \neq \nablaf$. For example, a constant rescaling $g_{\boldsymbol f}=c\,g$ with $c>0$ has the same Levi--Civita connection as $g$ while generally $g_{\boldsymbol f}\ne g$. Thus, this term measures connection mismatch rather than arbitrary metric mismatch; in particular, it does not detect constant metric rescaling or collapse along a sequence of such rescalings.

\section{Reproducing Kernel Hilbert Spaces}\label{appendix: Reproducing Kernel Hilbert Spaces}

In order to carry out the PA algorithm detailed in \eqref{eq: argmin via RKHS}, it is instructive to consider the case where projection indices  $\{\Bpi_{\Bf_{N,\lambda}^{(n)}}(\boldsymbol{X}_i)\}_{i\in [N]}$  are fixed. In this case, the solution to minimizing the functional in \eqref{eq: argmin via RKHS} is identical to the framework proposed in the smoothing spline literature. One particular result is the \textit{representer theorem} \citep[][Theorem 1.3.1]{wahba1990spline}, which uses the RKHS structure of $H^2(\M)=H^2(\M ; \mbR^1)$ to provide a closed-form solution for a spline fit $f_j \in H^2(\M)$ for each $j \in [D]$. Stacking the results gives us a function in $H^2(\M ; \mbR^D),$ i.e., $\Bf = (f_1,\ldots,f_D)$. We henceforth focus on scalar-valued functions. To apply the representer theorem to $H^2(\M)=H^2(\M ; \mbR^1)$, we first establish some mathematical preparations in Sections \ref{appendix: An Orthogonal Decomposition of H2} and \ref{appendix: Inner Products Associated with Reproducing Kernels}. 

\subsection{Direct Sum Decompositions}\label{appendix: An Orthogonal Decomposition of H2}

Define a linear subspace of $H^2(\M ; \mbR^1)$ by
\begin{align}\label{eq: def of H_0}
\mathcal{H}_0 \;:=\; \left\{\, f\in H^2(\M ; \mbR^1)\;:\;\nabla^{2}f = 0 \text{ almost everywhere on } \M \, \right\}.
\end{align}
The following lemma shows that $\mathcal{H}_0$ is finite-dimensional. Its proof can be found in Appendix~\ref{Appendix: Proofs}.
\begin{lemma}\label{lemma: finite dimensionality of the kernel space}
Let $(\M,g)$ be a Riemannian manifold satisfying Assumption \ref{assumption: assumption on the template manifold}. Then, $\mathcal{H}_0$ is a finite-dimensional linear subspace of $H^2(\M ; \mbR^1)$. More precisely, $\dim \mathcal{H}_0 \;\le\; 1+\dim \M$.
\end{lemma}
\begin{remark}\label{remark: Affine and harmonic function spaces}
\begin{enumerate}
    \item Lemma~\ref{lemma: finite dimensionality of the kernel space} implies that the null space $\mathcal{N}(\nabla^2)=\{\Bf\in H^2(\M;\mathbb{R}^D):\,\Vert \nabla^2 \boldsymbol{f} \Vert^2_{L^2(\M)}=0 \}$, as defined in Section~\ref{section: Manifolds via Minimization}, is finite-dimensional. 

    \item The space of harmonic maps  
    \begin{align*}
        \mathcal{N}(\Delta):=\left\{\Bf\in H^2(\M): \Vert \Delta \boldsymbol{f} \Vert^2_{L^2(\M)}=0 \right\},
    \end{align*}
    introduced in Section~\ref{section: Hessian Penalty versus Laplace-Beltrami Penalty}, may be infinite-dimensional when $\partial\M\ne\emptyset$. For example, let $\M=\{(x,y)\in\mathbb{R}^2: x^2+y^2\le 1\}$, the closed unit disk in \(\mathbb{R}^2\) equipped with the Euclidean metric. Then the Laplace-Beltrami operator reduces to the usual Laplacian 
    \begin{align*}
        \Delta=\frac{\partial^2}{\partial x^2}+\frac{\partial^2}{\partial y^2}.
    \end{align*}
    Define a sequence of functions \(\{u_n\}_{n\in\mathbb{N}}\) on \(\M\) by $u_n(x,y):=\mathrm{Re}(x+iy)^n$ for $(x,y)\in\M$. Since \(z=x+iy\mapsto z^n\) is holomorphic on \(\mathbb{C}\), its real part \(u_n\) is harmonic on \(\M\). Hence, \(\{u_n\}_{n\in\mathbb{N}}\subseteq \mathcal{N}(\Delta)\). It is easy to show that the infinitely many functions $\{u_n\}_{n\in\mathbb{N}}$ are linearly independent.
\end{enumerate}
\end{remark}

Since $\mathcal{H}_0$ is finite-dimensional, it is closed in $H^2(\M)=H^2(\M;\mathbb{R}^1)$, and it is also closed in $L^2(\M)$. We choose a closed linear complement $\mathcal H_1$ of $\mathcal H_0$ in $H^2(\M)$. Then, we have the direct sum decomposition
\begin{align}\label{eq: a general def of H_1}
    H^2(\M)=\mathcal H_0\oplus\mathcal H_1.
\end{align}
That is, for every $f\in H^2(\M)$, there exist unique $f_0\in\mathcal{H}_0$ and $f_1\in\mathcal{H}_1$ such that $f=f_0+f_1$. One possible way of constructing such a linear complement is to use the $L^2$-orthogonal complement $\mathcal{H}_0^{\perp_{L^2}}$, that is,
\begin{align}\label{eq: def of H_1}
    \begin{aligned}
        \mathcal{H}_1 :&=\mathcal{H}_0^{\perp_{L^2}}\cap H^2(\M) \\
    &=\left\{f\in H^2(\M): \int_{\M} f(\boldsymbol{m}) \tilde{f}(\boldsymbol{m}) \, d\vol_g(\boldsymbol{m}) = 0 \text{ for all }\tilde{f}\in\mathcal{H}_0 \right\}.
    \end{aligned}
\end{align}
Note that the $\mathcal{H}_1$ defined in \eqref{eq: def of H_1} is also orthogonal to $\mathcal{H}_0$ in $H^2(\M)$.

The closed linear complement $\mathcal{H}_1$ of $\mathcal{H}_0$ is not unique. Therefore, \eqref{eq: def of H_1} is not the only method for constructing such a complement. The following discussion permits $\mathcal{H}_1$ to be any closed linear complement of $\mathcal{H}_0$ that satisfies \eqref{eq: a general def of H_1}. We present three constructions of $\mathcal{H}_1$ corresponding to three representative template manifolds, as detailed below.
\begin{itemize}
    \item When $\M=[0,1]$ is equipped with the standard Euclidean metric, we have $\mathcal{H}_0=\mathrm{span}\{1,t\}$, and we may choose 
    \begin{align}\label{eq: H_1 for [0,1]}
        \mathcal{H}_1=\left\{u\in H^2([0,1]):u(0)=u'(0)=0 \right\}.
    \end{align}
    Note that \eqref{eq: H_1 for [0,1]} is not constructed by \eqref{eq: def of H_1}. We may also construct $\mathcal{H}_1$ as follows using \eqref{eq: def of H_1}, where we define
    \begin{align}\label{eq: alternative H_1 for [0,1]}
        \mathcal{H}_1=\left\{u\in H^2([0,1]): \int_0^1u(t)\,dt = \int_0^1 t \cdot u(t)\,dt =0 \right\}.
    \end{align}

    \item When $\M=\mathbb{S}^1=\mathbb R/\mathbb Z$ is equipped with the metric $dt\otimes dt$ and parameterized by the interval $[0,1)$, we have $\mathcal{H}_0=\mathrm{span}\{1\}$, and we may construct $\mathcal{H}_1$ using \eqref{eq: def of H_1}, i.e., 
    \begin{align}\label{eq: H_1 for S1}
        \mathcal H_1=\left\{ u\in H^2(\mathbb S^1):\int_0^1u(t)\,dt=0 \right\}.
    \end{align}

    \item When $\M=\mathbb{S}^2$ equipped with the standard unit-sphere metric, we have $\mathcal{H}_0=\mathrm{span}\{1\}$, and we may construct $\mathcal{H}_1$ using \eqref{eq: def of H_1}, i.e., 
    \begin{align}
        \mathcal H_1=\left\{ u\in H^2(\mathbb S^2):\int_{\mathbb S^2}u\,d\vol_g=0 \right\}.
    \end{align}
\end{itemize}

\subsection{Inner Products Associated with Reproducing Kernels}\label{appendix: Inner Products Associated with Reproducing Kernels}

To apply the representer theorem to $H^2(\M)$, we first define an inner product on $H^2(\M)$. Specifically, for any $f=f_0+f_1$ and $\tilde{f}=\tilde{f}_0+\tilde{f}_1\in H^2(\M)$ with $\{f_0,\tilde{f}_0\}\subseteq\mathcal{H}_0$ and $\{f_1, \tilde{f}_1\}\subseteq\mathcal{H}_1$, we define an inner product by
\begin{align}\label{eq: def of the inner product for the representer theorem}
    \begin{aligned}
        & \langle f,\tilde{f}\rangle_R := \int_{\M} f_0(\boldsymbol{m})\cdot \tilde{f}_0(\boldsymbol{m}) \, d\vol_g(\boldsymbol{m}) + \left\langle \nabla^2 f_1,\, \nabla^2  \tilde{f}_1 \right\rangle_{L^2(\M)}, \\
    & \text{where }\left\langle \nabla^2 f_1,\, \nabla^2  \tilde{f}_1 \right\rangle_{L^2(\M)}:= \int_{\M} \left\langle \nabla^2 f_1(\boldsymbol{m}), \nabla^2  \tilde{f}_1(\boldsymbol{m}) \right\rangle_g \, d\vol_g(\boldsymbol{m}),
    \end{aligned}
\end{align}
and $\langle \cdot,\cdot\rangle_g$ denotes the pointwise inner product induced by the Riemannian metric $g$, i.e., 
$$\langle \nabla^2 f,\nabla^2 \tilde f \rangle_g = \sum_{ijk\ell} g^{ik}g^{j\ell}(\nabla^2 f)_{ij}(\nabla^2 \tilde f)_{k\ell}$$ 
in each coordinate chart, compatible with \eqref{eq: vert cdot vert g}. The first term is simply an inner product on the finite-dimensional null space $\mathcal H_0$, chosen to be the $L^2$ inner product.\footnote{Any two norms defined on a finite-dimensional linear space are equivalent \citep[][Section 11.1]{brezis2011functional}.}

The inner product $\langle\cdot, \cdot\rangle_R$ induces a norm $\Vert f\Vert_R:=\sqrt{\langle f, f \rangle_R}$. Importantly, it is equivalent to the following Sobolev norm that we use
\begin{align*}
    \Vert f\Vert_{H^2(\M)}:=\left(\Vert f\Vert_{L^2(\M)}^2+\Vert\nabla^2 f\Vert_{L^2(\M)}^2\right)^{1/2}.
\end{align*}
Specifically, there exist constants $A$ and $B$ such that 
\begin{align}\label{eq: Sobolev norm equivalence}
    A\cdot\Vert f\Vert_R \le\Vert f\Vert_{H^2(\M)} \le B\cdot \Vert f\Vert_R \quad \text{ for all }f\in H^2(\M).
\end{align}
The proof of \eqref{eq: Sobolev norm equivalence} can be found in Appendix \ref{Appendix: Proofs}. When $\dim\M<4$, Lemma \ref{lemma: general Sobolev inequalities on manifolds}, together with \eqref{eq: Sobolev norm equivalence}, implies
\begin{align}\label{eq: reasoning of H2 equipped with R}
    \Vert f\Vert_{C(\M)} \le C_H\cdot \Vert f\Vert_{H^2(\M)} \le C_H\cdot B \cdot \Vert f\Vert_R \quad \text{ for all }f\in H^2(\M),
\end{align}
which further implies that $H^2(\M)$, equipped with the inner product $\langle\cdot, \cdot\rangle_R$, is an RKHS. 

In addition, we define the inner product $\langle\cdot, \cdot\rangle_{R_1}$ on $\mathcal{H}_1$ by 
\begin{align}\label{eq: def of the inner product associated with R1}
    \begin{aligned}
        \langle f, \tilde{f}\rangle_{R_1}&:= \left\langle \nabla^2 f,\, \nabla^2  \tilde{f} \right\rangle_{L^2(\M)} \\
    &= \int_{\M} \left\langle \nabla^2 f(\boldsymbol{m}), \nabla^2  \tilde{f}(\boldsymbol{m}) \right\rangle_g \, d\vol_g(\boldsymbol{m}) \quad \text{ for all }f, \tilde{f}\in \mathcal{H}_1.
    \end{aligned}
\end{align}
It is easy to verify that $\Vert f\Vert_{R_1}:=\sqrt{\langle f, f\rangle_{R_1}}$ defines a norm on $\mathcal{H}_1$. Then, \eqref{eq: reasoning of H2 equipped with R} implies that
\begin{align*}
    \Vert f\Vert_{C(\M)} \le C_H\cdot B \cdot \Vert f\Vert_{R} = C_H\cdot B \cdot \Vert f\Vert_{R_1} \quad \text{ for all }f\in \mathcal{H}_1,
\end{align*}
where the equality above follows from the definition in \eqref{eq: def of the inner product for the representer theorem}. Hence, $(\mathcal{H}_1, \langle\cdot,\cdot\rangle_{R_1})$ is an RKHS as well, and we denote its reproducing kernel by $R_1(\cdot,\cdot)$. By the definition of a reproducing kernel, together with \eqref{eq: def of the inner product associated with R1}, we have
\begin{align}\label{eq: L2 representation of R1}
    R_1(\boldsymbol{m},\boldsymbol{m}') = \langle R_1(\cdot,\boldsymbol{m}),\, R_1(\cdot,\boldsymbol{m}')\rangle_{R_1}=\left\langle \nabla^2 R_1(\cdot, \boldsymbol{m}),\, \nabla^2  R_1(\cdot, \boldsymbol{m}') \right\rangle_{L^2(\M)}.
\end{align}

\subsection{Representer Theorem and an Equivalence of Norms}

Note that $\mathcal{H}_1$ is also orthogonal to $\mathcal{H}_0$ with respect to the inner product $\langle\cdot,\cdot\rangle_{R}$ defined in \eqref{eq: def of the inner product for the representer theorem}. Let $P_1: H^2(\M) \rightarrow \mathcal{H}_1$ denote the corresponding $\langle\cdot,\cdot\rangle_{R}$-orthogonal projection. Then, for every $f=f_0+f_1$ with $f_j\in\mathcal{H}_j$ for $j=0,1$, we have
\begin{align*}
    \Vert \nabla^2 f\Vert_{L^2(\M)}^2 = \Vert \nabla^2 f_1\Vert_{L^2(\M)}^2 = \Vert f_1\Vert_{R_1}^2 = \Vert P_1 f\Vert_{R_1}^2 = \Vert P_1 f\Vert_{R}^2.
\end{align*}
In addition, let $\{\phi_j\}_{j\in[d_0]}$ be a linear basis of $\mathcal{H}_0$, where $d_0:=\dim\mathcal{H}_0$.

Given the preceding preparations, we now present the representer theorem for the minimization problem stated in \eqref{eq: argmin via RKHS}.
\begin{theorem}\label{thm: wahba representation theorem}
Suppose we observe a point cloud $\{\boldsymbol{X}_i=(X_{i1},\ldots,X_{iD})^\T\}_{i\in[N]}\subseteq\mathbb{R}^D$, and $\Bf_{N,\lambda}^{(n)}=(f_{N,\lambda,1}^{(n)},\ldots, f_{N,\lambda,D}^{(n)})$ has been fitted. Denote
\begin{align}\label{eq: notations for the representer theorem}
\begin{aligned}
& \boldsymbol{m}_i := \boldsymbol{\pi}_{\Bf_{N,\lambda}^{(n)}}\left(\boldsymbol{X}_i\right) \quad \text{ for all }i\in[N], \\
& \boldsymbol{K}:=(K_{ij})_{1\le i,j \le N} \in\mbR^{N \times N},\\ 
&\text{where } K_{ij} = R_1(\boldsymbol{m}_i,\boldsymbol{m}_j), \text{ and $R_1$ is the reproducing kernel of }\mathcal{H}_1, \\
& \boldsymbol{M} := \boldsymbol{K} + N \lambda \boldsymbol{I}_N  \in \mbR^{N\times N}, \text{ where }\boldsymbol{I}_N  \text{ denotes the $N$-by-$N$ identity matrix},\\
& \boldsymbol{T}:=(T_{ij})_{1\le i \le N,\, 1\le j\le d_0} \in \mbR^{N \times d_0}, \\
&\text{where } T_{ij}=\phi_j(\boldsymbol{m}_i), \text{ and }\{\phi_j\}_{j=1}^{d_0} \text{ is a linear basis of } \mathcal{H}_0, \\
& \boldsymbol{X}:= (\boldsymbol{X}_1,\ldots, \boldsymbol{X}_N)^\T=(X_{ij})_{1\le i\le N,\, 1\le j\le D} \in\mathbb{R}^{N\times D},\\
& (\theta_{ij})_{1\le i\le d_0,\, 1\le j\le D} :=(\boldsymbol T^\T \boldsymbol M^{-1} \boldsymbol T)^{-1}\boldsymbol T^\T \boldsymbol M^{-1} \boldsymbol X  \in \mbR^{d_0 \times D},  \\
& (\alpha_{ij})_{1\le i\le N,\, 1\le j\le D} := \boldsymbol M^{-1} \left( \boldsymbol I_N - \boldsymbol T(\boldsymbol T^\T \boldsymbol M^{-1} \boldsymbol T)^{-1} \boldsymbol T^\T \boldsymbol M^{-1} \right)\boldsymbol X \in \mbR^{N \times D}.
\end{aligned}
\end{align}
Then, the minimizing function $\Bf_{N,\lambda}^{(n+1)}=(f_{N,\lambda,1}^{(n+1)},\ldots,f_{N,\lambda,D}^{(n+1)}) : \M \to \mbR^D$  is given by
\begin{align}\label{eq: minimizer formula in the representer theorem}
    f_{N,\lambda,j}^{(n+1)}(\boldsymbol{m}) = \sum_{i=1}^{d_0} \theta_{ij} \cdot \phi_i(\boldsymbol{m}) + \sum_{\ell=1}^N \alpha_{\ell j}\cdot R_1(\boldsymbol{m},\boldsymbol m_\ell).
\end{align}
\end{theorem}
\noindent Theorem \ref{thm: wahba representation theorem} follows directly from an application of the representer theorem \citep[][Theorem 1.3.1]{wahba1990spline}. In \eqref{eq: notations for the representer theorem}, we assume that $\boldsymbol T$ has full column rank, equivalently that $\boldsymbol T^\T\boldsymbol M^{-1}\boldsymbol T$ is nonsingular. If this rank condition fails, the displayed inverse should be replaced by a Moore--Penrose pseudoinverse and the minimizer need not be unique along the $\mathcal H_0$ directions.

To tailor the application of Theorem \ref{thm: wahba representation theorem} to a particular template manifold $\M$, we must identify the reproducing kernel $R_1$ of the RKHS $\mathcal{H}_1$ associated with the chosen closed complement in \eqref{eq: a general def of H_1}. For three particular template manifolds: 
\begin{itemize}
    \item $\mathbb S^1 =$ unit circle, 
    \item $\mathbb S^2 = $ unit sphere, and 
    \item the flat unit interval $[0,1]$,
\end{itemize}
the corresponding reproducing kernels are known \citep[e.g.,][]{wahbacircle, wahbasphere, wahba1990spline, hsing2015theoretical}. The results are restated in the following corollary.

\begin{corollary}\label{cor: minimizerkernels} 
The chosen roughness kernel $R_1$ on $\mathcal{H}_1$ and the matrix $\boldsymbol T$ defined in \eqref{eq: notations for the representer theorem} vary with the template manifold $\M$ as follows:
\begin{enumerate}
    \item When $\M = [0,1]$ equipped with the standard Euclidean metric and $\mathcal{H}_1$ is defined by \eqref{eq: H_1 for [0,1]}, we have that
    \begin{align*}
        R_1(s,t) =  \frac13\min(s,t)^3 - \frac12\min(s,t)^2(s+t) + \min(s,t)(st), \quad \text{ and }\quad \boldsymbol{T} = \begin{pmatrix}
            1 & m_1 \\
            1 & m_2 \\
            \vdots & \vdots \\
            1 & m_N
        \end{pmatrix}.
    \end{align*}

    \item Let $\M = \mathbb S^1=\mathbb R/\mathbb Z$ with metric $g=dt\otimes dt$, and parameterize it by the interval $[0,1)$. Suppose $\mathcal{H}_1$ is defined by \eqref{eq: H_1 for S1}. For $s, t \in [0,1)$ we have
    $$R_1(s,t) = -\frac{1}{24}B_4([s-t]),  \quad \text{ and }\quad \boldsymbol{T} =(1,1,\ldots,1)^\T,$$ where $B_4(t)=t^4 - 2t^3 + t^2 - \frac{1}{30}$ is the fourth Bernoulli polynomial and $[t] = t - \lfloor t \rfloor$ denotes the fractional part operator.\footnote{If instead $\mathbb S^1$ is viewed as the embedded unit circle with its standard arc-length metric, this kernel is multiplied by $(2\pi)^3$.}

    \item Let $\M = \mathbb S^2$ be equipped with the standard unit-sphere metric, and we parameterize it by unit vectors $\boldsymbol s,\boldsymbol t \in \mbR^3$. There is a topologically equivalent norm on $H^2(\M)$, as presented in Lemma~\ref{lemma: hessian laplace equivalence} later in this section, under which we have the closed form reproducing kernel
    $$R_1(\boldsymbol s,\boldsymbol t) = \frac{1}{4\pi}  \left(q_2(\boldsymbol s^\T \boldsymbol t) - \frac{1}{3}\right),  \quad \text{ and }\quad \boldsymbol{T} =(1,1,\ldots,1)^\T \in \mbR^N,$$ where 
    $$q_2(t)=\begin{cases}
        0.5 & t=1 \\ 
        0.5 \left( \log\left(1 + \sqrt{\frac{2}{1-t}} \right)(1-4t+3t^2) - 3\sqrt{2(1-t)^3}+4-3t \right) & t <1.
    \end{cases}$$
    
\end{enumerate}
\end{corollary}
\noindent  Item (i) of Corollary~\ref{cor: minimizerkernels} is discussed in \cite{wahba1990spline} but its precise form requires a further derivation (see Appendix \ref{Appendix: Proofs}). Item (ii) of Corollary~\ref{cor: minimizerkernels} is directly  adapted from \cite{wahbacircle}. Item (iii) is from \cite{wahbasphere} but necessitates a discussion on what is meant by ``topologically equivalent norm.'' In \cite{wahbasphere}, the penalty of interest is technically the Laplace-Beltrami operator $\Delta$, defined in \eqref{eq: def of Laplace-Beltrami operator}, on the surface, i.e.,
\begin{align}\label{eq: Laplace-Beltrami penalty}
    \norm{\Delta \Bf}_{L^2(\mbS^2)}^2 =\sum_{i=1}^D \int_{\mbS^2} \left\vert \Delta f_i \right\vert^2 d\vol_g
\end{align}
where $\Bf = (f_1,f_2,\ldots,f_D)$ and $f_i\in H^2(\mathbb{S}^2)$ for $i\in[D]$. This is not the same as the penalty operator given by the $L^2$ norm of the Hessian, $\norm{\nabla^2 \Bf}_{L^2(\mbS^2)}$, as defined in \eqref{eq: L2 norm of Hf}. We demonstrate a result that bridges these two penalties. 
\begin{lemma}\label{lemma: hessian laplace equivalence}
    Suppose $\Bf \in H^2(\mbS^2)$. Then,
    $$\norm{\nabla^2 \Bf}_{L^2(\mbS^2)} \leq \norm{\Delta \Bf}_{L^2(\mbS^2)} \leq \sqrt 2 \norm{\nabla^2 \Bf}_{L^2(\mbS^2)}.$$
\end{lemma}
\noindent The proof is provided in Appendix \ref{Appendix: Proofs}. The result shows that the Laplace-Beltrami-based penalty in \cite{wahbasphere}, i.e., the one defined in \eqref{eq: Laplace-Beltrami penalty}, is equivalent to the Hessian-based penalty on $\mbS^2$. This allows us to use the reproducing kernel associated with the Laplace-Beltrami operator-based penalty $\norm{\Delta \Bf}_{L^2(\mbS^2)}^2$. The notion of ``topologically equivalent norm'' refers to both this equivalence and another equivalence imposed in \citep[][Section 3]{wahbasphere}, allowing for a convenient closed form for the reproducing kernel. This form is the one given in Corollary \ref{cor: minimizerkernels}(iii).

The discussion above shows the equivalence between the Hessian and Laplace-Beltrami penalties on $\mbS^2$. Using the Laplace-Beltrami penalty can be advantageous due to its prevalence in the literature. Moreover, in $\mbS^1$, we can prove that these two formulations are actually identical. As $\mbS^1$ is a one-dimensional manifold, we can choose local coordinates $t \in [0,1)$ in which the Riemannian metric can be expressed  as $g = dt\otimes dt$. This implies the Christoffel symbols vanish, so we can express the covariant Hessian and Laplace-Beltrami operator as follows: 
\begin{align*}
    & (\nabla ^2 f)_{ij} = \partial_{ij} f - \Gamma_{ij}^k \partial_k f \implies \nabla^2f = f''(t) dt \otimes dt,\\
    & \Delta_g  f  = \mathrm{tr}_g (\nabla^2 f) \implies \Delta_g f =f''(t),
\end{align*}
which follows from \cite[][Example 4.22]{lee2018riemannian} and \eqref{eq: def of Laplace-Beltrami operator}. We furthermore remark that $d\vol_g = dt$ \citep[][Proposition 2.4.1(c)]{lee2018riemannian}. Computing the integrals, we see that
\begin{align*}
    \norm{\nabla^2 f}^2_{L^2(\mbS^1)} = \int_{\mbS^1} | \nabla^2 f|_g^2 \,d\vol_g = \int_0^1 |f''(t)|^2 | dt \otimes dt|_g^2 \,d\vol_g =  \int_0^1 |f''(t)|^2 \,dt,
\end{align*}
and $\norm{\Delta_g f}^2_{L^2(\mbS^1)} = \int_{\mbS^1} | \Delta_g f |^2 d\vol_g =  \int_0^1 |f''(t)|^2dt$. Hence, there is no need to distinguish between the Hessian and Laplace-Beltrami penalties in this case.

\subsection{Closed Form of the Penalty Term}

Computing $\norm{\nabla^2 \Bf_{N,\lambda}^{(n+1)}}_{L^2(\M)}$ is helpful for visualizing the value of $\mcL_{N,\lambda}(\Bf_{N,\lambda}^{(n+1)})$ as iterations progress (e.g., Figure \ref{fig:interval_S1_full}). The following lemma provides a closed form for the penalty term.
\begin{lemma}\label{lemma:  closed form penalty circle}
    Suppose we have data $\{\boldsymbol{X}_i\}_{i\in[N]} \subseteq \mbR^D$ and $\{\boldsymbol m_i :=\boldsymbol{\pi}_{\Bf_{N,\lambda}^{(n)}}(\boldsymbol{X}_i) \}_{i\in[N]}\subseteq \M$. Using the notations defined in Theorem \ref{thm: wahba representation theorem}, the penalty of the optimal solution is given by
    \begin{align}\label{eq: closed form penalty circle}
    \norm{\nabla^2 \Bf_{N,\lambda}^{(n+1)}}^2_{L^2(\M)} =\sum_{j=1}^D \boldsymbol{\alpha}_{j}^\T\boldsymbol{K}\boldsymbol{\alpha}_j,
    \end{align}
    where $\boldsymbol{\alpha}_j = (\alpha_{1,j},\ldots, \alpha_{N,j})^\T$.
\end{lemma}
\noindent The proof of Lemma \ref{lemma:  closed form penalty circle} is given in Appendix \ref{Appendix: Proofs}.

\subsection{Equivalent Sobolev Norms}\label{section: Equivalent Sobolev Norms}

On a compact smooth Riemannian manifold $(\M,g)$ with empty or smooth boundary, a standard norm is
\begin{align}
\label{eq: Sobolev norm 1}
    \|\Bf\|_{H^2_{\mathrm{std}}(\M)}
    :=
    \left(
        \|\Bf\|_{L^2(\M)}^2
        +
        \|\nabla \Bf\|_{L^2(\M)}^2
        +
        \|\nabla^2 \Bf\|_{L^2(\M)}^2
    \right)^{1/2}.
\end{align}
Equivalently, one may use the following sum norm \citep[][Section 2.2]{hebey2000nonlinear}
\begin{align}
\label{eq: Sobolev norm 2}
    \|\Bf\|_{L^2(\M)}
    +
    \|\nabla \Bf\|_{L^2(\M)}
    +
    \|\nabla^2 \Bf\|_{L^2(\M)}.
\end{align}
Indeed, for all nonnegative $a,b,c$, we have the following elementary inequalities
\begin{align}\label{elementary inequalities for Sobolev norm equivalence}
    (a^2+b^2+c^2)^{1/2}
    \le
    a+b+c
    \le
    \sqrt{3}\,(a^2+b^2+c^2)^{1/2},
\end{align}
and they imply the equivalence between the norms defined in \eqref{eq: Sobolev norm 1} and \eqref{eq: Sobolev norm 2}.

In Section~\ref{section: Sobolev Spaces on Riemannian Manifolds}, however, we
use the norm
\begin{align}
\label{eq: Sobolev norm 3}
    \|\Bf\|_{H^2(\M)}
    :=
    \left(
        \|\Bf\|_{L^2(\M)}^2
        +
        \|\nabla^2 \Bf\|_{L^2(\M)}^2
    \right)^{1/2},
\end{align}
which omits the first-order term $\|\nabla\Bf\|_{L^2(\M)}$. We now show that
\eqref{eq: Sobolev norm 3} is nevertheless equivalent to the standard norms in
\eqref{eq: Sobolev norm 1} and \eqref{eq: Sobolev norm 2}.

\begingroup
\color{black}
All estimates below are first established on the common smooth core
$C^\infty(\M;\mathbb R^D)$; the completed spaces are identified only after
the norm comparison has been proved. It is enough to work first with a
scalar function $u\in C^\infty(\M)$.

Choose finitely many nested regular charts
$V_\alpha\Subset U_\alpha\Subset\widetilde U_\alpha$
(where $A\Subset B$ denotes compact containment, meaning that the closure
$\overline A$ is a compact subset of $B$) such that
$\{V_\alpha\}_\alpha$ covers $\M$. After rescaling the coordinates, an
interior chart has full-ball images, whereas a chart meeting $\partial\M$
has images relatively open in the closed half-space
$\overline{\mathbb H}^{\,d}:=\{x_d\geq0\}$. More precisely, put
\[
\Omega_\alpha:=
\varphi_\alpha(U_\alpha\cap\operatorname{int}\M)
=
\begin{cases}
B_2(0),&U_\alpha\cap\partial\M=\varnothing,\\
B_2^+(0):=B_2(0)\cap\{x_d>0\},&U_\alpha\cap\partial\M\ne\varnothing,
\end{cases}
\]
and arrange that $\varphi_\alpha(V_\alpha\cap\operatorname{int}\M)$ is
the corresponding ball or open half-ball of radius $1$. The chart extends
to the larger patch $\widetilde U_\alpha$, so the metric coefficients,
volume density, and Christoffel symbols extend smoothly to
$\overline{\Omega}_\alpha$ and are uniformly bounded there. Finiteness of
the atlas makes all resulting constants uniform after taking maxima.

For a boundary chart, fix a bounded linear $H^2$ extension operator across
the flat face of $B_2^+(0)$ to a full bounded smooth Euclidean domain; choose
it also bounded on $L^2$. For an interior chart use the ordinary bounded
linear full-ball extension. Applying Theorem~7.28 of
\cite{gilbarg1998elliptic} to the
extension of $u_\alpha:=u\circ\varphi_\alpha^{-1}$, restricting back to
$\Omega_\alpha$, expanding the Euclidean $H^2$ norm, and absorbing its
small first-derivative term gives, for sufficiently small
$\varepsilon>0$,
\begin{align}
\label{eq: first absorption}
\|Du_\alpha\|_{L^2(\Omega_\alpha)}
\le
K_{\alpha,0}(\varepsilon)\|u_\alpha\|_{L^2(\Omega_\alpha)}
+
\theta_\alpha(\varepsilon)
\|D^2u_\alpha\|_{L^2(\Omega_\alpha)},
\end{align}
where $K_{\alpha,0}(\varepsilon)<\infty$ and
$\theta_\alpha(\varepsilon)\to0$ as $\varepsilon\downarrow0$. This
extension-and-restriction step is what permits the Euclidean estimate to
be used at the boundary; no boundary chart image is treated as an open
full-space domain.

In local coordinates,
$\partial_{ij}u_\alpha=(\nabla^2u)_{ij}
+\Gamma_{ij}^k\partial_ku_\alpha$. Uniform metric and Jacobian
comparisons therefore give constants $A_\alpha,B_\alpha<\infty$ such that
\begin{align}
\label{eq: compare hessian concise}
\|D^2u_\alpha\|_{L^2(\Omega_\alpha)}
\le
A_\alpha\|\nabla^2u\|_{L^2(U_\alpha)}
+
B_\alpha\|Du_\alpha\|_{L^2(\Omega_\alpha)}.
\end{align}
Because the atlas is finite, $\varepsilon$ may be chosen so that
$\theta_\alpha(\varepsilon)B_\alpha\leq1/2$ for every $\alpha$.
Substitution of \eqref{eq: compare hessian concise} into
\eqref{eq: first absorption}, absorption of the last term, and the same
metric and Jacobian comparisons yield
\begin{align}
\label{eq: local geometric estimate}
\|\nabla u\|_{L^2(U_\alpha)}
\le
C_{\alpha,0}\|u\|_{L^2(U_\alpha)}
+
C_{\alpha,2}\|\nabla^2u\|_{L^2(U_\alpha)}.
\end{align}

The sets $U_\alpha$ cover $\M$ because the smaller $V_\alpha$ do, and the
finite family $\{U_\alpha\}_\alpha$ has finite overlap multiplicity.
Squaring \eqref{eq: local geometric estimate}, summing in $\alpha$, and
enlarging the constant gives
\begin{align}
\label{eq: global first derivative bound}
\|\nabla u\|_{L^2(\M)}
\le
C_\M
\bigl(
\|u\|_{L^2(\M)}
+
\|\nabla^2u\|_{L^2(\M)}
\bigr),
\end{align}
where $C_\M$ depends only on $(\M,g)$ and the fixed finite atlas.

Applying \eqref{eq: global first derivative bound} componentwise shows,
for every smooth $\Bf=(f_1,\ldots,f_D)$, that
\[
\|\nabla\Bf\|_{L^2(\M)}
\le
C_\M
\bigl(
\|\Bf\|_{L^2(\M)}
+
\|\nabla^2\Bf\|_{L^2(\M)}
\bigr).
\]
Consequently,
$\|\Bf\|_{H^2_{\mathrm{std}}(\M)}
\leq C_\M'\|\Bf\|_{H^2(\M)}$ on
$C^\infty(\M;\mathbb R^D)$, while the reverse inequality follows
immediately from the definitions. Thus, the two norms are equivalent on
their common smooth core. The identity on that core extends uniquely to
bounded inverse maps between the two completions. In particular, the
completion under \eqref{eq: Sobolev norm 3} is canonically the standard
Sobolev space $H^2(\M;\mathbb R^D)$, and only after this identification do
the two norm inequalities and \eqref{eq: global first derivative bound}
extend by continuity to every element of $H^2(\M;\mathbb R^D)$.
\endgroup

No boundary condition on $\Bf$ is needed for this equivalence, because the
$L^2$ term is retained in \eqref{eq: Sobolev norm 3}. By contrast,
$\|\nabla^2\Bf\|_{L^2(\M)}$ alone is only a seminorm: every constant map has
$\nabla\Bf=0$ and $\nabla^2\Bf=0$. Dropping the $L^2$ term would require an
extra normalization.

We adopt the definition in \eqref{eq: Sobolev norm 3} for two reasons. First, the first-order derivative term $\|\nabla\Bf\|_{L^2(\M)}$ is not directly relevant to the PME framework. Second, this formulation simplifies the proofs of the proposed theorems.

\subsection{Post Hoc Condition in \eqref{eq: post hoc condition for the finite sample PA algorithm}}\label{Appendix: Post Hoc Condition}

The relaxed update in \eqref{eq: argmin via RKHS} does not by itself imply the post hoc bound in \eqref{eq: post hoc condition for the finite sample PA algorithm}. The following condition and proposition give a sufficient criterion under which the post hoc bound in \eqref{eq: post hoc condition for the finite sample PA algorithm} holds. Let \(\mathcal H_0=\{u\in H^2(\M):\nabla^2u=0\}\). For the empirical PA algorithm in \eqref{eq: argmin via RKHS}, define
\[
\boldsymbol m_i^{(n)}:=\Bpi_{\Bf_{N,\lambda}^{(n)}}(\boldsymbol X_i),
\qquad
\boldsymbol T_n:=\bigl(\phi_r(\boldsymbol m_i^{(n)})\bigr)_{1\le i\le N,\ 1\le r\le d_0}.
\]
Recall the assumption in \eqref{eq: def of tau star}:
\[
\tau_*:=\inf_{n\ge0}\lambda_{\min}\left(\frac1N\boldsymbol T_n^\T\boldsymbol T_n\right)>0.
\]
Furthermore, we define
\begin{align*}
& R_N:=\left(\frac1N\sum_{i=1}^N\Vert\boldsymbol X_i\Vert^2\right)^{1/2}, \\
& \kappa:=\sup_{\boldsymbol m\in\M}\sqrt{R_1(\boldsymbol m,\boldsymbol m)}, \\
& b_0:=\sup_{\boldsymbol m\in\M}\left(\sum_{r=1}^{d_0}\phi_r(\boldsymbol m)^2\right)^{1/2},    
\end{align*}
where \(R_1\) is the reproducing kernel of \(\mathcal H_1\). The constant \(\kappa\) is finite; alternatively, by the embedding bound in \eqref{eq: reasoning of H2 equipped with R}, one may use \(\kappa\le C_HB\).
\begin{proposition}\label{prop: empirical PA sup norm bound}
Under the assumption in \eqref{eq: def of tau star} and for \(\lambda>0\), every relaxed update in \eqref{eq: argmin via RKHS} satisfies
\[
\Vert \Bf_{N,\lambda}^{(n+1)}\Vert_{C(\M)}
\le
R_N\left[
\frac{\kappa}{\sqrt{\lambda}}
+b_0\tau_*^{-1/2}\left(2+\frac{\kappa}{\sqrt{\lambda}}\right)
\right]
=:\Upsilon_\star,
\qquad n\ge0.
\]
Consequently, if \(\Upsilon\ge\Upsilon_\star\), then the relaxed update lies in \(\mathscr F\) at every iteration and is also feasible for the constrained empirical PA step. 
\end{proposition}

\section{Projection–Adaptation Algorithm}\label{appendix: Projection-Adaptation Algorithm}

Following Theorem \ref{thm: wahba representation theorem} and its ensuing Corollary~\ref{cor: minimizerkernels}, we now provide the empirical Projection–Adaptation (PA) algorithm in detail in Algorithm \ref{alg: empirical pme}.

\begin{algorithm}[ht]
    \caption{Empirical Projection–Adaptation (PA) Algorithm}\label{alg: empirical pme}
    \begin{algorithmic}[1]
        \Require (i) a point cloud $\{\boldsymbol{X}_i\}_{i\in[N]}$; (ii) a choice of $\lambda >0$; (iii) a stopping criterion $\epsilon_{\text{stop}} > 0$; (iv) a positive integer $\operatorname{itr}$ limiting the number of spline fits, including the initial fit; (v) a choice of template manifold $\M$ satisfying Assumptions \ref{assumption: assumption on the template manifold} and \ref{assumption: topological assumption}.

        \State Initialize projection indices $\boldsymbol m_1,\ldots \boldsymbol m_N\in\M$ (see Section \ref{section: Initialization of the PA algorithm} for details).

        \State Fit the initial spline $\Bf_{N,\lambda}^{(0)} = \left(f_{N,\lambda,1}^{(0)}, \ldots, f_{N,\lambda,D}^{(0)}\right)$ by using $\boldsymbol X:=(\boldsymbol{X}_1, \ldots, \boldsymbol{X}_N)^\T$, $\lambda$, projection indices $\boldsymbol m_1,\ldots, \boldsymbol m_N$, and the formula in \eqref{eq: minimizer formula in the representer theorem}.


        \State Initialize $\epsilon \gets \epsilon_{\text{stop}} + 1$, $n  \gets 1$, and $D\left(\Bf_{N,\lambda}^{(0)}\right) \gets \sum_{i=1}^N\norm{\boldsymbol{X}_i - \Bf_{N,\lambda}^{(0)}(\Bpi_{\Bf_{N,\lambda}^{(0)}}(\boldsymbol{X}_i))}^2$.

        \While{$\epsilon > \epsilon_{\text{stop}}$ and $n < \operatorname{itr}$}
            \State Compute $\Bpi_{\Bf_{N,\lambda}^{(n-1)}}(\boldsymbol{X}_i) =\argmin_{\boldsymbol{m}\in\mathfrak{M}} \norm{\boldsymbol{X}_i - \Bf_{N,\lambda}^{(n-1)}(\boldsymbol m)}$ for all $i \in [N]$ (see Section \ref{section: Computing Projection Indices} for details).
            \State  Compute $\Bf_{N,\lambda}^{(n)}$ based on $\boldsymbol X=(\boldsymbol{X}_1, \ldots, \boldsymbol{X}_N)^\T$, $\lambda$, and projection indices $\{\Bpi_{\Bf_{N,\lambda}^{(n-1)}}(\boldsymbol{X}_i)\}_{i\in[N]}$ (see Theorem \ref{thm: wahba representation theorem} and Section \ref{section: An Approximate Spline Solution for Large N}).
            
            \State Set $D(\Bf_{N,\lambda}^{(n)}) \gets \sum_{i=1}^N\norm{\boldsymbol{X}_i - \Bf_{N,\lambda}^{(n)}(\Bpi_{\Bf_{N,\lambda}^{(n-1)}}(\boldsymbol{X}_i))}^2$
            \State Update $\epsilon \gets \left| D\left(\Bf_{N,\lambda}^{(n)}\right)  - D\left(\Bf_{N,\lambda}^{(n-1)}\right) \right| / D\left(\Bf_{N,\lambda}^{(n-1)}\right) $
            \State Update $n \gets n+1$
        \EndWhile

        \Return $\Bf_{N,\lambda}^{(n-1)}$ as the numerical PA output.
    \end{algorithmic}

\end{algorithm}

\subsection{Initialization}\label{section: Initialization of the PA algorithm}
One aspect of Algorithm \ref{alg: empirical pme} requiring further discussion is the initialization of the projection indices---computing the initial $\boldsymbol m_1,\ldots \boldsymbol m_N \in \M$ so that we can compute the initial function $\Bf_{N,\lambda}^{(0)}$ using the formula in \eqref{eq: minimizer formula in the representer theorem}. One should strategize based on the choice of template manifold $\M$ as well as the structure of the observed point cloud $\{\boldsymbol{X}_i\}_{i \in [N]}$.

\subsubsection{Template $\M=[0,1]$} For the case of $\M = [0,1]$, we find satisfactory results from initialization via ISOMAP \citep{tenenbaum2000global} and then rescaling the ISOMAP coordinates to the interval $[0,1]$, preserving the original order. Given one-dimensional ISOMAP coordinates $y_1, \ldots, y_N \in \mbR$ corresponding to a dataset of $N$ points, the projection indices are given by
$$m_i := \frac{y_i - \min_{i \in [N]} y_i}{\max_{i \in [N]} y_i - \min_{i \in [N]} y_i}.$$

\subsubsection{Template $\M=\mathbb{S}^1$ and Ambient Space $\mathbb{R}^2$} For the template manifold $\M=\mbS^1$ in ambient space $\mbR^2$, Assumption \ref{assumption: topological assumption} motivates initialization by projecting the data onto a circle and determining their ordering, provided that the data exhibit circular structure. This calculation involves centering the data at their empirical mean and computing the angle from the origin to the data point relative to the positive $x$-axis. This value can then be rescaled to $[0,1)$. This works well when the underlying manifold satisfies a star-shape-like property. In the literature, subsets $S$ of a vector space $V$ are said to be star-shaped at a point $x \in S$ if for every $y \in S$, the line segment from $x$ to $y$ is inside $S$.

This notion of star-shape will have to be modified for our purposes, since, for example, manifolds diffeomorphic to $\mbS^1$ typically have no interior in $\mbR^2$. Therefore, we define the following star-shape-like property. Suppose $\mcM$ is a latent manifold in $\mbR^D$. Let $\boldsymbol x\in\mbR^D$ be a specified reference point; for implementation, take $\boldsymbol x=\overline{\boldsymbol X}:=N^{-1}\sum_{i=1}^N\boldsymbol X_i$, the empirical mean of the observed point cloud. For the geometric example of a circle centered at the origin, one may take $\boldsymbol x=\boldsymbol 0$. Choose any point $\boldsymbol y\in\mcM$. Draw an infinitely long ray from $\boldsymbol x$ to $\boldsymbol y$. The ray should only intersect the manifold $\mcM$ once. When the latent manifold satisfies this property, circular projection suffices. But if it does not, the circular projection does not behave well: to remedy this, we recommend applying ISOMAP to get a two-dimensional representation before applying the circular projection.

An illustration of this property is given by Figure  \ref{fig:example_star_convex}. Suppose the latent manifold $\mathcal{M}$ is diffeomorphic to $\M=\mathbb S^1$. The approximate center of $\mcM$ is indicated by the black point. Observe that drawing a vector from the black point to point $p_2$ crosses the manifold $\mathcal{M}$ twice. This leads to the following issue. Suppose $\theta$ denotes the angular distance from the positive $x$-axis (grey dotted line). Such values of $\theta$ determine the ordering induced by the circular projection. Let $\theta_{13}$ denote the angular distance between $p_1$ and $p_3$, and $\theta_{23}$ for $p_2$ and $p_3$. Circular projection gives $\theta_{23}<\theta_{13}$, so the projected angular distance between $p_2$ and $p_3$ is smaller than that between $p_1$ and $p_3$. In the illustrated configuration, the geodesic distances have the opposite ordering: $d_{\mathcal{M}}(p_2,p_3)>d_{\mathcal{M}}(p_1,p_3)$, where $d_{\mathcal{M}}$ denotes intrinsic geodesic distance on $\mathcal{M}$. A consequence of this is that the spline fit may take shortcuts: instead of approximately following the geodesic path between $p_2$ and $p_3$ (passing through $p_1$ in this illustration), it may cross through the ``interior,'' bypassing $p_1$.

\begin{figure}[ht]
    \centering
    \includegraphics[width=0.45\linewidth]{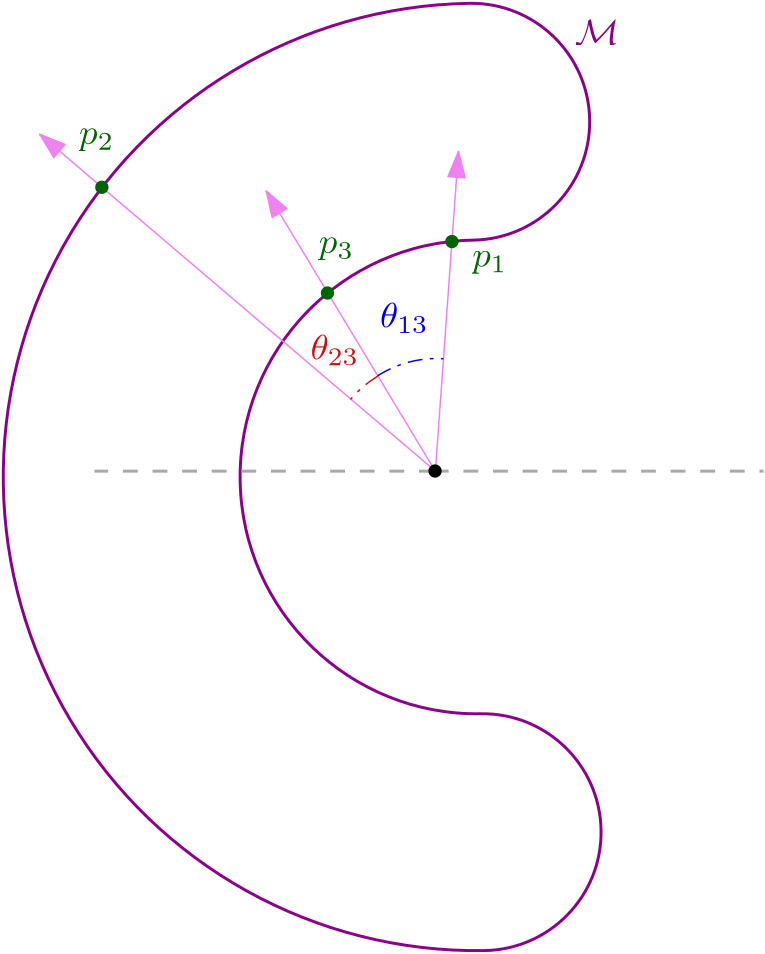}
    \caption{\footnotesize A latent manifold $\mathcal{M}$ with approximate center at the black point. Taking the ordering counterclockwise along the manifold, we have the ordering $(p_1, p_2, p_3)$. Let $\theta_{ij}$ denote the angular distance on $\mathbb{S}^1$ between the circular projections of $p_i$ and $p_j$ from the black point. Here $\theta_{23}<\theta_{13}$, whereas the illustrated geodesic distances satisfy $d_{\mathcal{M}}(p_2,p_3)>d_{\mathcal{M}}(p_1,p_3)$. Thus, circular projection reverses this distance ordering.}
    \label{fig:example_star_convex}
\end{figure}

ISOMAP can help resolve this issue by ``unwinding'' a point cloud into a configuration that satisfies a star-shape-like property. In Figure \ref{fig:isomap_unwind}, we display a point cloud along a cashew-shaped closed curve and color the points based on the true ordering. If one were to use the circular projection on the left panel directly, we would encounter the issues described in the previous paragraph. Under ISOMAP, the true ordering is preserved under the circular projection method (as shown in the right panel), resulting in better initializations.
\begin{figure}[!htbp]
    \centering
    \includegraphics[width=0.75\linewidth]{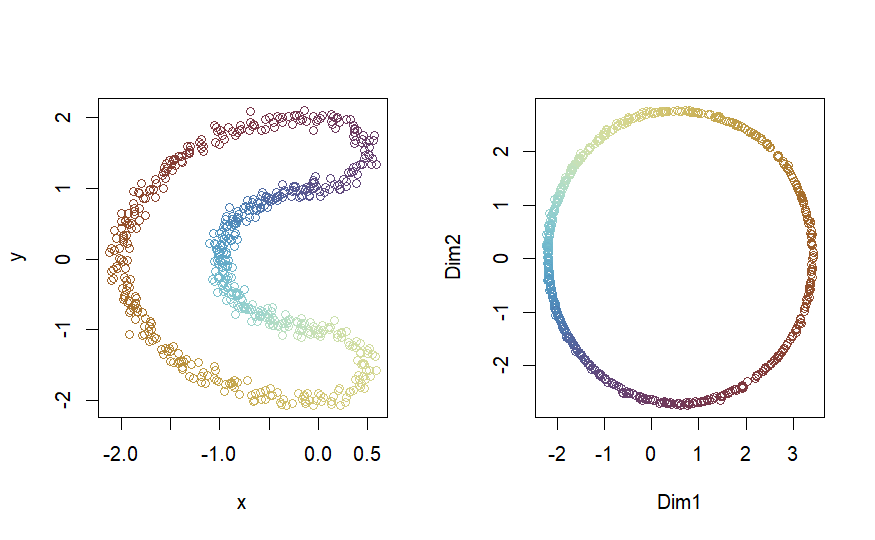}
    \caption{\footnotesize The left figure depicts the original dataset. The right panel shows the corresponding two-dimensional ISOMAP representation. The colors denote the original ordering of the data. Observe how the ordering along the curve is preserved.}
    \label{fig:isomap_unwind}
\end{figure}

When it comes to using the PA algorithm, Figure \ref{fig:isomap_unwind_first_vs_converged} demonstrates the poor initialization when a naive circular projection is used on point clouds failing to satisfy the star property. We use the same dataset as in Figure \ref{fig:isomap_unwind}. The left plot shows the initial manifold estimate, and the right side displays the converged estimate. Observe how the ISOMAP method results in a much better initialization than the naive method. Interestingly, once we let the algorithm run to convergence in both cases, we see similar curves that describe the data well. However, in the converged non-ISOMAP case, small regions of sharp local bending are present on the curve. This suggests two things: (1) in some cases, the circular projection method can be robust to data not satisfying the star property, but not without consequences (e.g., sharp local bending), and (2) using ISOMAP to initialize can lead to a better fit, despite this robustness.

\begin{figure}[!htbp]
    \centering
    \includegraphics[width=1\linewidth]{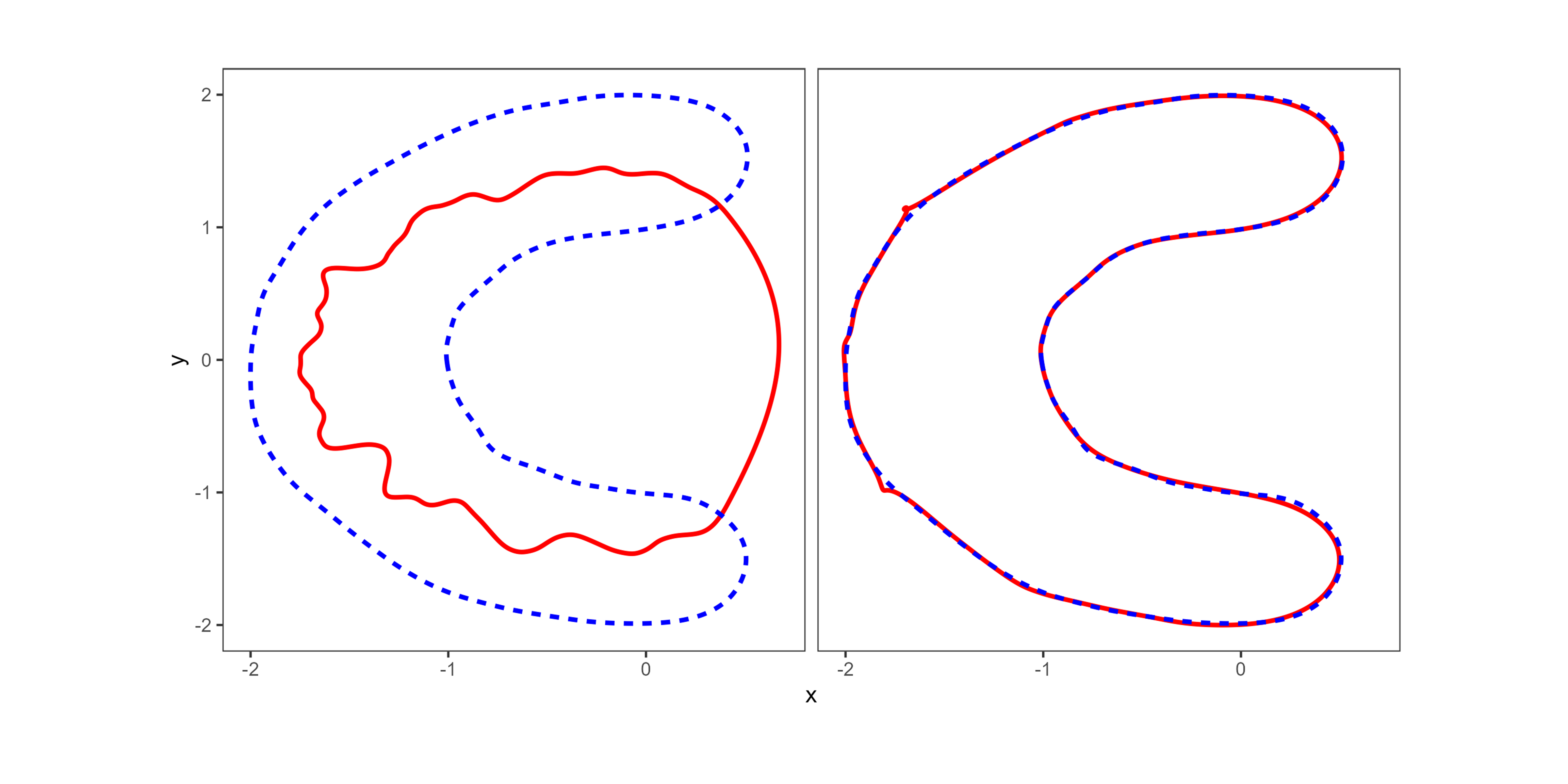}
     \vspace{-3em}
   \caption{\footnotesize Comparison of initial and converged PME estimates under different initialization schemes. The blue dotted curve corresponds to ISOMAP followed by circular projection, whereas the red solid curve corresponds to raw circular projection. The left panel shows the initialization, and the right panel shows the converged fit.}
    \label{fig:isomap_unwind_first_vs_converged}
\end{figure}

One may ask, since ISOMAP can be used to find a $d$-dimensional coordinate system to represent the data, why use a two-dimensional representation of the data and then project onto a circle, as opposed to finding a one-dimensional representation through ISOMAP directly? Figure \ref{fig:isomap_1d_vs_2d} demonstrates why. Using a one-dimensional representation directly for periodic data produces extremely poor results at the initialization step. Using ISOMAP to acquire a 2-dimensional system and then projecting onto a circle leads to a significantly better initialization. This is always the approach we use when dealing with irregular point clouds. Similarly for $\mbS^2$ in $\mbR^3$, despite it being a two-dimensional manifold, we apply ISOMAP to get three-dimensional coordinates and then project onto $\mbS^2$ by centering the data and dividing each vector by its Euclidean norm.
\begin{figure}[!htbp]
    \centering
   \includegraphics[width=1\linewidth]{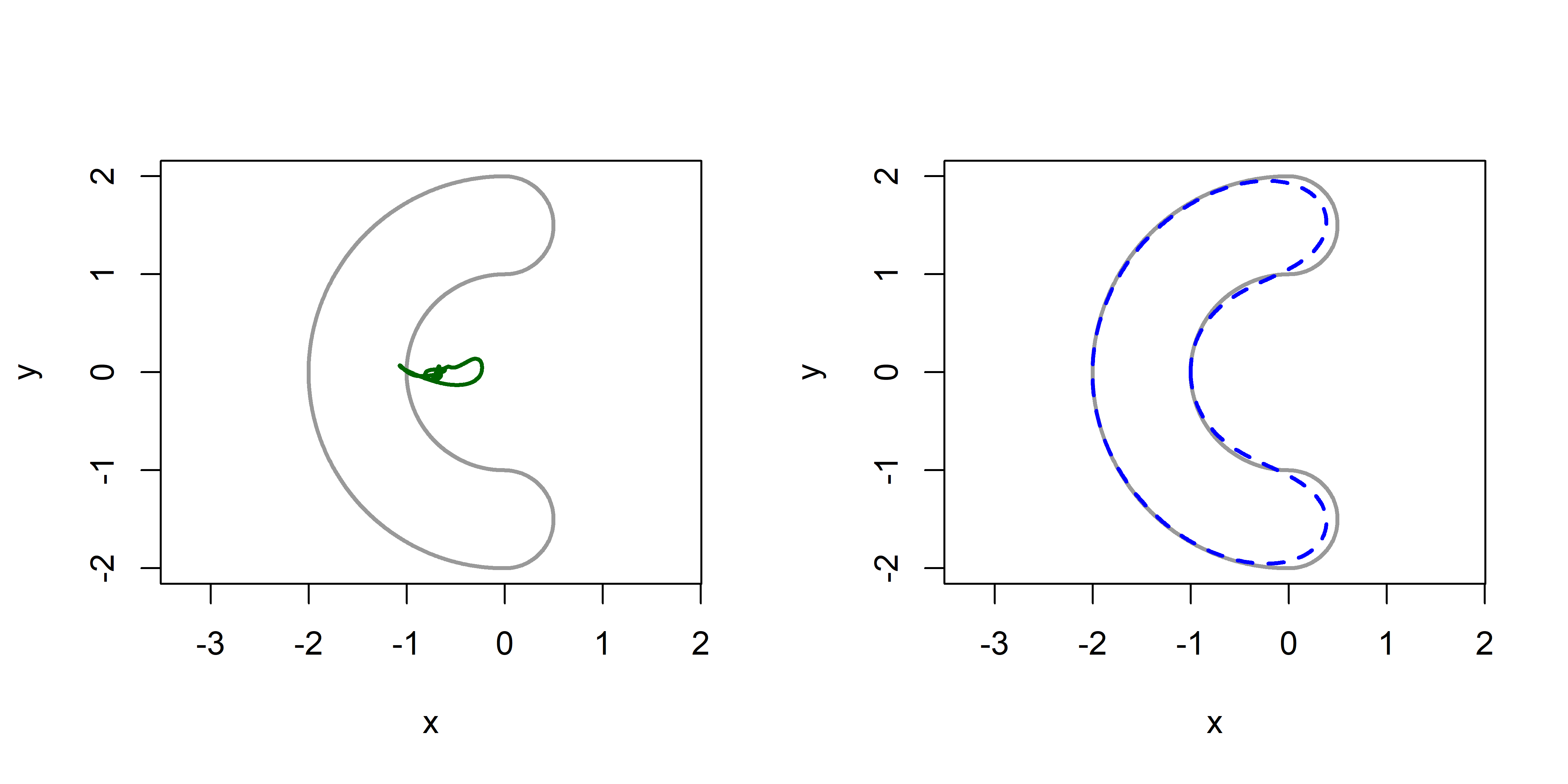}
    \vspace{-3em}
    \caption{\footnotesize Underlying manifold (grey solid curve) alongside different initialization curves based on associated noisy data. Left: 1D ISOMAP initialization (green solid curve). Right: 2D ISOMAP plus circular projection initialization (blue dotted curve).}
    \label{fig:isomap_1d_vs_2d}
\end{figure}

\subsubsection{Template $\M=\mathbb{S}^1$ and Ambient Space $\mathbb{R}^3$} In this situation, we always recommend using a two-dimensional ISOMAP step followed by a circular projection. In ambient $\mbR^3$ space, a non--self--intersecting closed curve can behave erratically, and ISOMAP tends to do a satisfactory job in finding the circular structure and retaining it in a two-dimensional coordinate system.

\subsubsection{Template $\M=\mathbb{S}^2$ and Ambient Space $\mathbb{R}^3$} The approach is similar to that of $\mbS^1$ in $\mbR^2$. Since, in our algorithm, we parameterize by directions in $\mbR^3$ instead of spherical coordinates, our initialization approach is to simply center the data and then normalize the vectors by their Euclidean norm to acquire unit vectors that describe their location on the sphere. A three-dimensional ISOMAP step should be used when the point cloud does not satisfy a star-shape-like property. One then normalizes the ISOMAP coordinate vectors in the aforementioned fashion.

\subsubsection{ISOMAP Parameters}
In all applications of ISOMAP, we must specify the nearest-neighbor parameter $k$. There are theoretically justified methods for choosing $k$, such as those studied by \cite{samko2006selection}. But for our purposes, we find that taking the smallest possible $k$ such that the ISOMAP graph is connected usually suffices. It is also possible to perform ISOMAP based on a distance parameter $\epsilon$, though this was not explored in our work.

\subsection{Computing Projection Indices}\label{section: Computing Projection Indices}

We discuss the computation of projection indices $\Bpi_{\Bf_{N,\lambda}^{(n)}}(\boldsymbol{X}_i) =\argmin_{\boldsymbol{m} \in \M}\norm{\boldsymbol{X}_i - \Bf_{N,\lambda}^{(n)}(\boldsymbol m)}$ for any $n = 0, 1, 2, \ldots$ and all $\lambda>0$. For the template manifolds $[0,1]$ and $\mathbb S^1$, since they are intrinsically of dimension $1$ and are bounded, the procedure is simple. In our implementation, a partition-based search provides a favorable balance between accuracy and speed.

Accordingly, let $M \in \mbN$ be a large positive integer. Let $\mcP = \{x_j\}_{j=1}^M$ be a partition of $[0,1]$, i.e., $0=x_1 < x_2 < \ldots < x_M=1.$ The projection index problem can then be reformulated as
$$\hat{\Bpi}_{\Bf_{N,\lambda}^{(n)}}(\boldsymbol{X}_i) =\argmin_{t \in \mcP}\norm{\boldsymbol{X}_i - \Bf_{N,\lambda}^{(n)}(t)}.$$
Due to the finiteness of $\mcP$, for a given data point $\boldsymbol{X}_i$, it suffices to check the value of $\norm{\boldsymbol{X}_i - \Bf_{N,\lambda}^{(n)}(t)}$ for each $t \in \mcP$ and pick $t$ such that the quantity is minimized.  A sufficiently fine partition typically gives a good approximation to the global minimizer. This also applies to $\mathbb S^1$ since we can parameterize it with the interval $[0,1)$ and construct a finite partition as above, with the modification that $t_M=1-\epsilon$ for $\epsilon >0$ small. This is to account for the fact that the endpoints $\{0,1\}$ are identified with each other in $\mbS^1$.

In the case of the template manifold $\mathbb S^2$, the addition of an extra dimension complicates the above approach. It is possible to generate a grid of points on a Fibonacci sphere as a partition of $\mbS^2$ and perform the search as described above. While this can find a global solution, it requires a very fine partition, at which point the method has comparable speed and accuracy to other methods, which we describe next. For $\mathbb S^2$, we recommend solving the projection index problem using techniques such as nonlinear minimization (NLM). In our case, we use the \texttt{nlm} function in \texttt{R} \citep{nlmschnabel1985modular, nlmdennis1996numerical}. Although NLM provides a good balance between accuracy and speed, it may return only a local minimizer. Because the parameter space is spherical, one may obtain slightly more accurate results by using manifold optimization routines such as those implemented in the \texttt{ManifoldOptim} package \citep{martin2020manifoldoptim}. However, in our testing, the \texttt{nlm} approach yielded faster computation than that of manifold optimization with a similar degree of accuracy.

\subsection{An Approximate Spline Solution for Large Sample Sizes}\label{section: An Approximate Spline Solution for Large N}

Suppose we have \(N\) iid observations \(\{\boldsymbol{X}_i\}_{i\in[N]}\) of \(\mbR^D\)-valued random variables associated with some latent manifold \(\mathcal{M}\). For a fixed iteration $n$, write
$\boldsymbol m_i:=\Bpi_{\Bf_{N,\lambda}^{(n)}}(\boldsymbol X_i)$.
Let $R_1$ be the reproducing kernel of $\mathcal H_1$ with its Hessian
inner product, and let $\{\phi_r\}_{r=1}^{d_0}$ be a basis of $\mathcal H_0$.
By Theorem~\ref{thm: wahba representation theorem} and
\eqref{eq: minimizer formula in the representer theorem}, every coordinate
of a relaxed minimizer in \eqref{eq: argmin via RKHS} satisfies
\[
f_{N,\lambda,j}^{(n+1)}
\in \mathcal H_0+
\operatorname{span}\{R_1(\cdot,\boldsymbol m_i):i\in[N]\},
\qquad j\in[D].
\]
The Hessian penalty leaves $\mathcal H_0$ unpenalized, so the nullspace
coefficients must be retained independently of the $R_1$ coefficients.
The coefficient calculation uses both the kernel matrix
$\boldsymbol K=(R_1(\boldsymbol m_i,\boldsymbol m_\ell))_{i,\ell=1}^N$
and the nullspace design matrix
$\boldsymbol T=(\phi_r(\boldsymbol m_i))_{i\in[N],\,r\in[d_0]}$,
together with linear solves involving the positive-definite matrix
$\boldsymbol M=\boldsymbol K+N\lambda\boldsymbol I_N$ as in
\eqref{eq: notations for the representer theorem}.
A dense factorization of $\boldsymbol M$ costs $O(N^3)$ operations,
making the full calculation expensive for large $N$.

At each fixed set of projection indices, the relaxed A-step
is a penalized least-squares problem with a Hessian penalty.
We discuss two approaches to reducing its computational cost:
(1) reduced expansions in the same roughness kernel $R_1$ and
(2) alternative spline bases.

First, choose $K<N$ knots
$\widetilde{\boldsymbol m}_k=\boldsymbol m_{i_k}$, $k=1,\ldots,K$,
from the projection indices, and restrict each coordinate to
\[
f_j(\boldsymbol m)
=\sum_{r=1}^{d_0}\theta_{rj}\phi_r(\boldsymbol m)
+\sum_{k=1}^{K}a_{kj}R_1(\boldsymbol m,\widetilde{\boldsymbol m}_k),
\qquad j\in[D].
\]
All $d_0$ nullspace basis functions remain available and unpenalized;
only the roughness-kernel expansion is reduced. The coefficients
$\theta_{rj}$ and $a_{kj}$ are chosen jointly by minimizing the original
objective in \eqref{eq: argmin via RKHS} over this restricted space,
using all $N$ observations. This gives an approximation to the full
relaxed A-step, and is exact only when a full minimizer belongs to the
restricted space. No approximation-error bound or convergence guarantee
for the resulting reduced-basis iteration is asserted here. This type of strategy is studied by \cite{ma2015efficient} and \cite{meng2020more}. A different data-reduction strategy, with reduced sample size $N'<N$, is used by \cite{meng_principal_2021}, which introduces a density estimation algorithm to obtain a reduced dataset $\{\boldsymbol{X}'_{i}\}_{i\in[N']} \subset \mbR^D$---not necessarily a subset of the original data $\{\boldsymbol{X}_{i}\}_{i\in[N]}$---and proceeds with the PA algorithm as usual. Unlike reducing only the kernel expansion, this changes the empirical fitting term and requires separate approximation justification.

There exist other methods that forgo the reproducing kernel basis. Consider the case of the template manifold $[0,1]$. For a fixed iteration, the minimizer coincides exactly with a natural cubic spline whose knots are the projection indices $\{\boldsymbol{m}_i\}_{i\in[N]}$ \citep{berlinet2011reproducing}. In the literature, it is standard to further impose a set of basis functions, a common choice being the cubic B-spline basis functions \citep{green1994nonparametric}. This provides a great numerical advantage, as evidenced in functions such as \texttt{smooth.spline} in \texttt{R}, which is used in our code  as a numerically efficient alternative for fits on the template manifold $[0,1]$. For the template $\mathbb S^1$, we consider splines on the interval $[0,1]$ with a periodic boundary condition. These are called cyclic cubic splines, which can be expanded in terms of a cardinal spline basis \citep{wood2017generalized}. This is the approach implemented in \texttt{mgcv} \citep{wood2017generalized}. A change of basis spanning the full spline space leaves the A-step unchanged. A smaller spline space instead gives a restricted approximation and must still retain all of $\mathcal H_0$.

These are not the only methods available, as there exists a substantial literature on the efficient computation and approximation of smoothing splines. There remains considerable scope for improving the efficiency of solving the minimization problem in \eqref{eq: argmin via RKHS}.

\section{Details and Proofs for the Quantile-Corrugated Nonidentifiability Example}
\label{appendix: Quantile-Corrugated Nonidentifiability Example}

This appendix provides the complete construction and proofs for the example summarized in Section~\ref{section: Numerical Nonidentifiability}. The argument proceeds in four steps. Subsection~\ref{appendix: Reference Distribution and Polar Representation} derives the reference noisy-circle distribution and its polar representation. Subsection~\ref{appendix: Construction and Coupling of Corrugated Alternatives} constructs the corrugated candidate manifolds, defines latent distributions whose supports are exactly those manifolds, and couples the alternatives to the reference. Subsection~\ref{appendix: Properties and Failure of Numerical Identifiability} establishes the geometric, support, and convergence properties, formalizes numerical identifiability as a local stability requirement, and proves that this requirement fails. Finally, Subsection~\ref{appendix: Density Formula and Numerical Evaluation} derives the exact alternative densities and describes the numerical evaluation used for the density panels in Figure~\ref{fig:noisy-circle-density-curve-grid}.

\subsection{Reference Distribution and Polar Representation}
\label{appendix: Reference Distribution and Polar Representation}

Work in $\mathbb{R}^2$ and define
\[
    \boldsymbol{u}(\theta):=(\cos\theta,\sin\theta)^\T,
    \qquad
    \mathcal{M}_{\star}
    :=\{\boldsymbol{u}(\theta):0\leq\theta<2\pi\}.
\]
Thus, $\mathcal{M}_{\star}$ is the unit circle and has length $2\pi$. Let $T_{\star}\sim\mathrm{Unif}[0,2\pi)$ and $\boldsymbol{Z}_{\star}\sim\mathcal{N}_2(\boldsymbol{0},\boldsymbol{I}_2)$ be independent. Set $\sigma_{\star}:=0.1$ and define the reference random variable $\boldsymbol{X}_{\star}:=\boldsymbol{u}(T_{\star}) + \sigma_{\star}\boldsymbol{Z}_{\star}$. The noise-free signal $\boldsymbol{u}(T_{\star})$ has support exactly $\mathcal{M}_{\star}$. By contrast, the nondegenerate Gaussian perturbation gives $\boldsymbol{X}_{\star}$ an everywhere-positive density, so its data-generating distribution has support $\mathbb{R}^2$. The alternatives below are designed to approximate the complete noisy random variable $\boldsymbol{X}_{\star}$, not merely its latent circular signal.

We first derive the density of $\boldsymbol{X}_{\star}$. Conditional on $T_{\star}=t$, the density at $\boldsymbol{x}\in\mathbb{R}^2$ is
\[
    \frac{1}{2\pi\sigma_{\star}^2}
    \exp\!\left\{-\frac{\Vert\boldsymbol{x}-\boldsymbol{u}(t)\Vert^2}
    {2\sigma_{\star}^2}\right\}.
\]
Write $\boldsymbol{x}=r\boldsymbol{u}(\theta)$, where $r:=\Vert\boldsymbol{x}\Vert$. Since $\Vert r\boldsymbol{u}(\theta)-\boldsymbol{u}(t)\Vert^2 = r^2+1-2r\cos(t-\theta)$, averaging the conditional density over the uniform variable $T_{\star}$ gives
\begin{align*}
    p_{\star}(\boldsymbol{x})
    =\frac{1}{4\pi^2\sigma_{\star}^2}
      \int_0^{2\pi}
      \exp\!\left[-\frac{r^2+1-2r\cos(t-\theta)}
      {2\sigma_{\star}^2}\right]\dd t
    =\frac{1}{2\pi\sigma_{\star}^2}
      \exp\!\left\{-\frac{r^2+1}{2\sigma_{\star}^2}\right\}
      I_0\!\left(\frac{r}{\sigma_{\star}^2}\right),
\end{align*}
where $I_0(a):=\frac{1}{2\pi}\int_0^{2\pi}e^{a\cos s}\,\dd s$ is the modified Bessel function of the first kind of order zero. Hence,
\begin{equation}\label{eq: target noisy-circle density}
    p_{\star}(\boldsymbol{x})
    =\frac{1}{2\pi\sigma_{\star}^2}
      \exp\!\left\{-\frac{\Vert\boldsymbol{x}\Vert^2+1}
      {2\sigma_{\star}^2}\right\}
      I_0\!\left(\frac{\Vert\boldsymbol{x}\Vert}
      {\sigma_{\star}^2}\right),
    \qquad \boldsymbol{x}\in\mathbb{R}^2.
\end{equation}

Because $\boldsymbol{X}_{\star}$ is absolutely continuous with respect to Lebesgue measure, $\mathbb{P}(\boldsymbol{X}_{\star}=\boldsymbol{0})=0$. We may therefore write $\boldsymbol{X}_{\star}=R_{\star}\boldsymbol{u}(\Theta_{\star})$ almost surely, where $R_{\star}:=\Vert\boldsymbol{X}_{\star}\Vert>0$ and $\Theta_{\star}\in[0,2\pi)$ is its polar angle. Here, $T_{\star}$ is the latent angle used to generate the point on $\mathcal{M}_{\star}$, whereas $\Theta_{\star}$ is the polar angle of the noisy observation $\boldsymbol{X}_{\star}$; these variables need not be equal. The polar-coordinate Jacobian and \eqref{eq: target noisy-circle density} yield
\[
    f_{R_{\star},\Theta_{\star}}(r,\theta)
    =r p_{\star}\{r\boldsymbol{u}(\theta)\}
    =\frac{f_R(r)}{2\pi},
    \qquad r>0,\quad 0\leq\theta<2\pi,
\]
where
\begin{equation}\label{eq: Rice radial density for nonidentifiability example}
    f_R(r)
    :=\frac{r}{\sigma_{\star}^2}
      \exp\!\left\{-\frac{r^2+1}{2\sigma_{\star}^2}\right\}
      I_0\!\left(\frac{r}{\sigma_{\star}^2}\right),
    \qquad r>0.
\end{equation}
The factorization shows directly that $\Theta_{\star}\sim\mathrm{Unif}[0,2\pi)$ and that $R_{\star}$ and $\Theta_{\star}$ are independent. In particular, $R_{\star}$ has the Rice density in \eqref{eq: Rice radial density for nonidentifiability example}.

Let $F_R(r):=\mathbb{P}(R_{\star}\leq r)$ and define the generalized inverse by
\[
    F_R^{-1}(v):=\inf\{r\geq0:F_R(r)\geq v\},
    \qquad 0<v<1.
\]
Since $f_R(r)>0$ for every $r>0$, $F_R$ is continuous and strictly increasing on $[0,\infty)$, with $F_R(0)=0$ and $F_R(r)\uparrow1$ as $r\to\infty$. Consequently, $F_R^{-1}$ is continuous and strictly increasing on $(0,1)$, with $F_R^{-1}(v)\downarrow0$ as $v\downarrow0$ and $F_R^{-1}(v)\uparrow\infty$ as $v\uparrow1$.

\subsection{Construction and Coupling of the Corrugated Alternatives}
\label{appendix: Construction and Coupling of Corrugated Alternatives}

The construction approximates the reference in two coordinated steps. First, it compresses the reference radial probability coordinate to a central quantile interval, producing a radius that converges to $R_{\star}$. Second, it places that radius on a corrugated manifold within the same shrinking angular cell as $\Theta_{\star}$. This arrangement makes the support of the latent distribution equal to the entire candidate manifold while keeping the latent signal close to $\boldsymbol{X}_{\star}$ under the coupling. For each integer $q\geq1$, define

\[
\begin{aligned}
    \eta_q&:=\frac{1}{10(q+1)},
    &\qquad \alpha_q&:=F_R^{-1}(\eta_q),
    &\qquad \beta_q&:=F_R^{-1}(1-\eta_q),\\
    w_q&:=\beta_q-\alpha_q,
    &\qquad \kappa_q&:=10q.
\end{aligned}
\]
Here, $0<\eta_q<1/2$, so $0<\alpha_q<\beta_q<\infty$ and $w_q>0$. The interval $(\alpha_q,\beta_q)$ contains the central $1-2\eta_q$ probability mass of $R_{\star}$, while the integer $\kappa_q$ determines the number of corrugations. Define
\[
    \rho_q(\theta)
    :=\alpha_q+\frac{w_q}{2}\{1-\cos(\kappa_q\theta)\},
    \qquad
    \boldsymbol{c}_q(\theta)
    :=\rho_q(\theta)\boldsymbol{u}(\theta),
\]
and let $\mathcal{M}_q :=\{\boldsymbol{c}_q(\theta):0\leq\theta<2\pi\}$. The function $\rho_q$ is smooth, positive, and $2\pi$-periodic because $\kappa_q$ is an integer. Its minima and maxima occur at
\[
    \rho_q\!\left(\frac{2\pi j}{\kappa_q}\right)=\alpha_q,
    \qquad
    \rho_q\!\left(\frac{(2j+1)\pi}{\kappa_q}\right)=\beta_q,
    \qquad j=0,\ldots,\kappa_q-1.
\]
Between consecutive extrema, $\rho_q$ is strictly monotone. We count one radial corrugation as one complete minimum-to-maximum-to-minimum cycle of $\rho_q$. Over $[0,2\pi)$, the function completes exactly $\kappa_q=10q$ such cycles. Consequently, $\mathcal{M}_q$ has exactly $10q$ radial corrugations and is contained in the annulus $\mathcal{M}_q \subseteq\{\boldsymbol{x}\in\mathbb{R}^2: \alpha_q\leq\Vert\boldsymbol{x}\Vert\leq\beta_q\}$. Because $\eta_q\downarrow0$, the endpoint behavior of $F_R^{-1}$ established above gives
\begin{equation}\label{eq: quantile endpoint limits}
    \alpha_q\downarrow0,
    \qquad
    \beta_q\uparrow\infty.
\end{equation}

We now construct a latent signal $\boldsymbol{S}_q$ with two key properties: its distribution has support exactly $\mathcal{M}_q$, and, under an appropriate coupling, $\boldsymbol{S}_q$ converges almost surely to $\boldsymbol{X}_{\star}$. The construction first produces a radius close to $R_{\star}$ by compressing its probability coordinate. It then uses an inverse branch of $\rho_q$ to place that radius on $\mathcal{M}_q$ within the angular cell containing $\Theta_{\star}$.

By the probability integral transform, $U_{\star}:=F_R(R_{\star})\sim\mathrm{Unif}(0,1)$, and $U_{\star}$ is independent of $\Theta_{\star}$. Set
\[
    U_q:=\eta_q+(1-2\eta_q)U_{\star}
    \sim\mathrm{Unif}(\eta_q,1-\eta_q),
\]
and let $R_q:=F_R^{-1}(U_q)$. For $\alpha_q<r<\beta_q$,
\[
\begin{split}
    \mathbb{P}(R_q\leq r)
    =\mathbb{P}\{U_q\leq F_R(r)\} =\frac{F_R(r)-\eta_q}{1-2\eta_q}.
\end{split}
\]
Differentiation gives the truncated Rice density
\begin{equation}\label{eq: truncated Rice density}
    g_q(r)
    :=\frac{f_R(r)}{1-2\eta_q}
      \mathbbm{1}_{(\alpha_q,\beta_q)}(r).
\end{equation}
Thus, every open subinterval of $(\alpha_q,\beta_q)$ receives positive probability.

Partition $[0,2\pi)$ into the $\kappa_q$ cells
\[
    \mathcal{I}_{j,q}
    :=\left[\frac{2\pi j}{\kappa_q},
    \frac{2\pi(j+1)}{\kappa_q}\right),
    \qquad j=0,\ldots,\kappa_q-1,
\]
and define
\[
    J_q:=\left\lfloor\frac{\kappa_q\Theta_{\star}}{2\pi}\right\rfloor.
\]
Then $\Theta_{\star}\in\mathcal{I}_{J_q,q}$ almost surely. Since $J_q$ is a function of $\Theta_{\star}$ whereas $R_q$ is a function of $R_{\star}$, the independence of $R_{\star}$ and $\Theta_{\star}$ implies that $J_q$ and $R_q$ are independent. Moreover,
\[
    \mathbb{P}(J_q=j)=\frac{1}{\kappa_q},
    \qquad j=0,\ldots,\kappa_q-1.
\]

For a prescribed radius $r\in(\alpha_q,\beta_q)$, define the phase
\[
    \xi_q(r):=\arccos\!\left\{1-\frac{2(r-\alpha_q)}{w_q}\right\}
\]
so that $\xi_q$ maps $(\alpha_q,\beta_q)$ continuously and strictly increasingly onto $(0,\pi)$. In the $j$th angular cell, the two solutions of $\rho_q(\theta)=r$ are
\[
    \frac{2\pi j+\xi_q(r)}{\kappa_q}
    \quad\text{and}\quad
    \frac{2\pi(j+1)-\xi_q(r)}{\kappa_q}.
\]
Set $\Xi_q:=\xi_q(R_q)$. Let $\{B_q\}_{q\geq1}$ be independent Bernoulli$(1/2)$ random variables, independent of all variables introduced above; $B_q$ selects one of the two inverse branches. Define, modulo $2\pi$,
\begin{equation}\label{eq: exact inverse-branch angle}
    \Theta_q
    :=\begin{cases}
        (2\pi J_q+\Xi_q)/\kappa_q, & B_q=0,\\[2mm]
        \{2\pi(J_q+1)-\Xi_q\}/\kappa_q, & B_q=1.
    \end{cases}
\end{equation}
For either branch, $\cos(\kappa_q\Theta_q)=\cos(\Xi_q)$, and hence
\[
    \rho_q(\Theta_q)
    =\alpha_q+\frac{w_q}{2}\{1-\cos(\Xi_q)\}
    =R_q.
\]
Therefore, $\boldsymbol{S}_q:=R_q\boldsymbol{u}(\Theta_q)
=\boldsymbol{c}_q(\Theta_q)\in\mathcal{M}_q$ almost surely. Finally, let a single random vector $\boldsymbol{Z}\sim\mathcal{N}_2(\boldsymbol{0},\boldsymbol{I}_2)$ be independent of all variables above, set $\sigma_q:=\sigma_{\star}/\sqrt q$, and define $\boldsymbol{X}_q
:=\boldsymbol{S}_q+\sigma_q\boldsymbol{Z}$.

\begin{figure}[ht]
\centering
\makebox[\textwidth][c]{%
    \hspace*{-0.4cm}%
\begin{tikzpicture}[
    font=\footnotesize,
    line cap=round,
    line join=round,
    >=Latex,
    manifold/.style={blue!75!black,line width=1.1pt},
    cellray/.style={black!22,line width=0.35pt},
    selectedplusray/.style={orange!85!black,line width=0.9pt},
    selectedminusray/.style={green!60!black,line width=0.9pt},
    branchdivider/.style={black!45,densely dashed,line width=0.6pt},
    guide/.style={black!48,densely dashed,line width=0.55pt},
    radiusguide/.style={purple!75!black,densely dotted,line width=0.8pt},
    plusguide/.style={
        draw=orange!95!black,densely dotted,line width=1.0pt
    },
    minusguide/.style={
        draw=green!65!black,densely dotted,line width=1.0pt
    },
    pluspoint/.style={
        circle,fill=orange!95!black,draw=white,
        line width=0.4pt,inner sep=1.8pt
    },
    minuspoint/.style={
        circle,fill=white,draw=green!65!black,
        line width=0.9pt,inner sep=1.6pt
    },
    labelbox/.style={
        fill=white,fill opacity=0.94,text opacity=1,inner sep=1.5pt
    }
]
\begin{scope}[shift={(3.35,3.20)}]
    %
    \def\radscale{2.25}
    \def\alphaval{0.840971}
    \def\betaval{1.169111}
    \def\midval{1.005041}
    \def\halfwidth{0.164070}
    \def\rqval{1.087076} 
    \def\rstardisplayrad{1.68}
    \pgfmathsetmacro{\alpharad}{\radscale*\alphaval}
    \pgfmathsetmacro{\betarad}{\radscale*\betaval}
    \pgfmathsetmacro{\rqrad}{\radscale*\rqval}
    \fill[orange!10]
        (0,0)--(36:3.02)
        arc[start angle=36,end angle=54,radius=3.02]--cycle;
    \fill[green!8]
        (0,0)--(54:3.02)
        arc[start angle=54,end angle=72,radius=3.02]--cycle;
    \foreach \a in {0,36,...,324}{
        \draw[cellray] (0,0)--(\a:2.90);
    }
    \draw[selectedplusray] (0,0)--(36:3.05);
    \draw[selectedminusray] (0,0)--(72:3.05);
    \draw[branchdivider] (0,0)--(54:3.05);
    \draw[guide] (0,0) circle[radius=\alpharad];
    \draw[guide] (0,0) circle[radius=\betarad];
    \draw[radiusguide] (0,0) circle[radius=\rqrad];
    \draw[manifold]
        plot[
            domain=0:360,
            samples=721,
            smooth,
            variable=\t
        ]
        ({\t}:{\radscale*(\midval-\halfwidth*cos(10*\t))});
    \foreach \j in {0,...,9}{
        \ifnum\j=1\relax
        \else
            \node[font=\scriptsize]
                at ({36*\j+18}:3.24)
                {$\mathcal I_{\j,1}$};
        \fi
    }
    \draw[black,line width=0.7pt]
        (0,0)--(52:2.90);
    \fill[black]
        (52:\rstardisplayrad) circle[radius=1.8pt];
    \draw[orange!90!black,line width=0.7pt]
        (0,0)--(48:\rqrad);
    \node[pluspoint] (sone)
        at (48:\rqrad) {};
    \node[minuspoint] (sonealt)
        at (60:\rqrad) {};
    \node[
        orange!85!black,
        labelbox,
        anchor=west
    ] (sonelabel)
        at (3.08,1.60) {$\boldsymbol S_1\;(B_1=0)$};
    \draw[orange!85!black,->]
        (sonelabel.west)
        to[out=180,in=-20]
        (sone);
    \node[
        green!55!black,
        labelbox,
        anchor=east
    ] (sonealtlabel)
        at (0.58,2.48) {$B_1=1$};
    \draw[green!55!black,->]
        (sonealtlabel.east)
        to[out=0,in=150]
        (sonealt);
    \fill[black]
        (0,0) circle[radius=1.25pt];
    \node[anchor=north east]
        at (-0.06,-0.05)
        {$\boldsymbol 0$};
    \node[
        anchor=south,
        orange!70!black,
        labelbox
    ]
        at (48:3.72)
        {$\mathcal I_{J_1,1},\quad J_1=1$};
    \node[
        anchor=west,
        align=left,
        labelbox
    ] (xstarlabel)
        at (3.08,2.30)
        {noisy benchmark realization\\
        $\boldsymbol X_\star$};
    \draw[black!75,->]
        (xstarlabel.west)
        to[out=180,in=-15]
        (52:\rstardisplayrad);
    \node[anchor=east,labelbox] (betalabel)
        at (-3.02,1.57)
        {$\beta_1$};
    \draw[black!55,->]
        (betalabel.east)--(154:\betarad);
    \node[
        anchor=east,
        purple!75!black,
        labelbox
    ] (rqlabel)
        at (-3.02,0.32)
        {$R_1$};
    \draw[purple!75!black,->]
        (rqlabel.east)--(171:\rqrad);
    \node[anchor=east,labelbox] (alphalabel)
        at (-3.02,-0.53)
        {$\alpha_1$};
    \draw[black!55,->]
        (alphalabel.east)--(196:\alpharad);
    \node[font=\small]
        at (0,3.82)
        {Angular cells for $q=1$ ($\kappa_1=10$)};
\end{scope}
\begin{scope}[shift={(8.05,1.10)}]
    \fill[orange!8] (0,0) rectangle (3.60,2.88);
    \fill[green!7] (3.60,0) rectangle (7.20,2.88);
    \draw[->]
        (-0.16,0)--(7.72,0);
    \draw[->]
        (0,-0.16)--(0,3.12)
        node[above] {$\rho_1$};
    \draw[branchdivider] (3.60,0)--(3.60,2.50);
    \draw[manifold]
        plot[
            domain=0:360,
            samples=241,
            smooth,
            variable=\t
        ]
        ({\t/50},{1.5-cos(\t)});
    \draw[radiusguide]
        (0.12,2.0)--(7.20,2.0);
    \draw[plusguide]
        (2.40,0)--(2.40,2.0);
    \draw[minusguide]
        (4.80,0)--(4.80,2.0);
    \node[pluspoint]
        at (2.40,2.0) {};
    \node[minuspoint]
        at (4.80,2.0) {};
    \node[
        anchor=south east,
        align=right,
        labelbox,
        text=orange!95!black
    ] (pluslabel)
        at (2.12,3.00)
        {$\theta_{J_1,1}^{+}(R_1)$\\
         $B_1=0$};
    \draw[orange!95!black,->]
        (pluslabel.south east)
        to[out=-35,in=120]
        (2.40,2.05);
    \node[
        anchor=south west,
        align=left,
        labelbox,
        text=green!65!black
    ] (minuslabel)
        at (5.08,3.00)
        {$\theta_{J_1,1}^{-}(R_1)$\\
         $B_1=1$};
    \draw[green!65!black,->]
        (minuslabel.south west)
        to[out=-145,in=60]
        (4.80,2.05);
    \draw
        (0.08,0.5)--(-0.08,0.5);
    \node[anchor=east,labelbox]
        at (-0.13,0.5)
        {$\alpha_1$};
    \draw
        (0.08,2.0)--(-0.08,2.0);
    \node[anchor=east,labelbox,text=purple!75!black]
        at (-0.13,2.0)
        {$R_1$};
    \draw
        (0.08,2.5)--(-0.08,2.5);
    \node[anchor=east,labelbox]
        at (-0.13,2.5)
        {$\beta_1$};
    \draw
        (0,-0.07)--(0,0.07);
    \node[anchor=north]
        at (0,-0.15)
        {$0$};
    \draw[draw=orange!95!black,line width=0.9pt]
        (2.40,-0.07)--(2.40,0.07);
    \node[
        anchor=north east,
        align=right,
        text=orange!95!black
    ]
        at (2.9,-0.17)
        {$\xi_1(R_1)$};
    \draw
        (3.60,-0.07)--(3.60,0.07);
    \node[anchor=north]
        at (3.60,-0.17)
        {$\pi$};
    \draw[draw=green!65!black,line width=0.9pt]
        (4.80,-0.07)--(4.80,0.07);
    \node[
        anchor=north west,
        align=left,
        text=green!65!black
    ]
        at (4.2,-0.17)
        {$2\pi-\xi_1(R_1)$};
    \draw
        (7.20,-0.07)--(7.20,0.07);
    \node[anchor=north]
        at (7.20,-0.17)
        {$2\pi$};
    \node[anchor=north]
        at (3.60,-0.92)
        {$\phi=\kappa_1\theta-2\pi J_1$};
\end{scope}
\end{tikzpicture}
    \hspace*{0.4cm}%
}
\caption{\footnotesize
Schematic illustration, not drawn to scale, of one realization of the
coupling between the noisy benchmark and the $q=1$ corrugated
alternative. Here $\kappa_1=10$, so
each angular cell has width $2\pi/\kappa_1=36^\circ$.
\textit{Left:} The black point is a noisy benchmark realization
$\boldsymbol X_\star=R_\star\boldsymbol u(\Theta_\star)$. Its polar angle
$\Theta_\star$ lies in the shaded cell
$\mathcal I_{1,1}=[36^\circ,72^\circ)$ and therefore determines
$J_1=1$. The orange and green halves of the cell denote its increasing
and decreasing branches, separated at $54^\circ$ (local phase $\pi$).
Compressing the radial probability coordinate of $R_\star$
gives the displayed radius
$R_1=\alpha_1+3(\beta_1-\alpha_1)/4$. The purple dotted circle of radius
$R_1$ intersects the
corrugated manifold at two eligible locations within the selected cell.
For the displayed draw $B_1=0$, the increasing-branch location
$\theta_{J_1,1}^{+}(R_1)$ is selected, giving the filled orange
latent point $\boldsymbol S_1$; the open green point is the location that
would instead be selected if $B_1=1$. The corresponding noisy
corrugated observation is subsequently obtained as
$\boldsymbol X_1=\boldsymbol S_1+\sigma_1\boldsymbol Z$ and is not shown.
\textit{Right:} The selected cell is displayed in the local phase
coordinate $\phi=\kappa_1\theta-2\pi J_1$, which maps
$\mathcal I_{J_1,1}$ onto $[0,2\pi)$. In this coordinate,
$\rho_1=\alpha_1+\{(\beta_1-\alpha_1)/2\}(1-\cos\phi)$. The orange and
green backgrounds distinguish the increasing and decreasing branches,
whose boundary is $\phi=\pi$. The colored vertical guides mark the two
phase solutions $\xi_1(R_1)$ and $2\pi-\xi_1(R_1)$, while the purple
horizontal guide marks the selected radius $R_1$.}
\label{fig:corrugation-angular-cells}
\end{figure}

Placing these variables on a common probability space is only a device for proving convergence; it does not change their marginal distributions. Three approximation errors vanish as $q\to\infty$: the radial probability-coordinate compression approaches the identity, the angular-cell width $2\pi/\kappa_q$ approaches zero, and the noise scale $\sigma_q$ approaches zero. Meanwhile, the number of corrugations grows without bound. The next subsection proves the manifold, support, convergence, and geometric divergence claims.

\subsection{Properties and Failure of Numerical Identifiability}
\label{appendix: Properties and Failure of Numerical Identifiability}

For a random vector $\boldsymbol{V}$, write $\mathbb{P}_{\boldsymbol{V}}$ for its distribution. The following proposition collects the properties of the construction.

\begin{proposition}[Quantile-calibrated corrugated alternatives]
\label{prop: quantile-calibrated numerical nonidentifiability}
Under the construction above, the following statements hold.
\begin{enumerate}
    \item For every $q\geq1$, $\mathcal{M}_q$ is a compact, connected, smooth embedded one-dimensional submanifold of $\mathbb{R}^2$ without boundary.
    \item The support of the latent distribution is exactly the candidate manifold: $\supp(\mathbb{P}_{\boldsymbol{S}_q})=\mathcal{M}_q$ for every $q\geq1$.
    \item Under the coupling in Subsection~\ref{appendix: Construction and Coupling of Corrugated Alternatives}, $\boldsymbol{X}_q\xrightarrow{\mathrm{a.s.}}\boldsymbol{X}_{\star}$ as $q\to\infty$. Consequently, $\boldsymbol{X}_q$ converges to $\boldsymbol{X}_{\star}$ in distribution.
    \item For every $q\geq1$, $\mathcal{M}_q$ has exactly $10q$ radial corrugations and satisfies
    \[
        \operatorname{length}(\mathcal{M}_q)
        \geq20q(\beta_q-\alpha_q).
    \]
    In particular, $\operatorname{length}(\mathcal{M}_q)\to\infty$ as $q\to\infty$.
\end{enumerate}
\end{proposition}

\begin{proof}
\noindent\emph{Validity of the candidate manifolds.}
The functions $\rho_q$ and $\boldsymbol{c}_q$ are smooth and $2\pi$-periodic, with $\rho_q(\theta)\geq\alpha_q>0$. Differentiating $\boldsymbol{c}_q(\theta)=\rho_q(\theta)\boldsymbol{u}(\theta)$ and using the orthogonality of $\boldsymbol{u}(\theta)$ and $\boldsymbol{u}'(\theta)$ gives
\[
    \Vert\boldsymbol{c}_q'(\theta)\Vert^2
    =\rho_q'(\theta)^2+\rho_q(\theta)^2
    \geq\alpha_q^2>0.
\]
Thus, the map induced by $\boldsymbol{c}_q$ on the parameter circle is an immersion. To verify injectivity, suppose $\boldsymbol{c}_q(\theta_1)=\boldsymbol{c}_q(\theta_2)$. Equality of norms gives $\rho_q(\theta_1)=\rho_q(\theta_2)>0$, and division by the common positive radius gives $\boldsymbol{u}(\theta_1)=\boldsymbol{u}(\theta_2)$. Hence, $\theta_1=\theta_2$ modulo $2\pi$. A smooth injective immersion of a compact manifold into $\mathbb{R}^2$ is an embedding. Its image $\mathcal{M}_q$ is compact and connected because the parameter circle is compact and connected, proving the first statement.

\noindent\emph{Support of the latent distribution.}
Because the density $g_q$ in \eqref{eq: truncated Rice density} is positive throughout $(\alpha_q,\beta_q)$, the support of $R_q$ is $[\alpha_q,\beta_q]$. The function $\xi_q$ maps $(\alpha_q,\beta_q)$ continuously and strictly increasingly onto $(0,\pi)$, so the support of $\Xi_q=\xi_q(R_q)$ is $[0,\pi]$. As shown in Subsection~\ref{appendix: Construction and Coupling of Corrugated Alternatives}, $J_q$ is independent of $R_q$ and
\[
    \mathbb{P}(J_q=j)=\frac{1}{\kappa_q},
    \qquad j=0,\ldots,\kappa_q-1,
\]
while the two values of $B_q$ each have probability $1/2$ and select the two inverse branches in every cell. It follows that every nonempty open arc of the parameter circle has positive probability under $\Theta_q$. Because $\boldsymbol{c}_q$ is a homeomorphism from the parameter circle onto $\mathcal{M}_q$, every relatively open nonempty subset of $\mathcal{M}_q$ has positive probability under $\boldsymbol{S}_q=\boldsymbol{c}_q(\Theta_q)$. Since $\boldsymbol{S}_q\in\mathcal{M}_q$ almost surely, $\supp(\mathbb{P}_{\boldsymbol{S}_q})=\mathcal{M}_q$.

\noindent\emph{Almost sure and distributional convergence.}
The radial probability-coordinate compression satisfies
\[
    U_q-U_{\star}=\eta_q(1-2U_{\star}),
    \qquad
    |U_q-U_{\star}|\leq\eta_q\rightarrow0.
\]
Since $F_R^{-1}$ is continuous on $(0,1)$ and $U_{\star}\in(0,1)$ almost surely,
\begin{equation}\label{eq: alternative radius converges}
    R_q\rightarrow R_{\star}
    \qquad\text{almost surely}.
\end{equation}
By construction, $\Theta_q$ and $\Theta_{\star}$ belong to the closure of the same angular cell, including when $\Theta_{\star}$ lies on a cell boundary. Therefore, writing $d_{\mathbb{S}^1}$ for circular angular distance, we have
\[
    d_{\mathbb{S}^1}(\Theta_q,\Theta_{\star})
    \leq\frac{2\pi}{\kappa_q}=\frac{\pi}{5q}.
\]
The chord length on the unit circle is bounded by circular angular distance. Combining this fact with \eqref{eq: alternative radius converges} yields
\[
\begin{split}
    \Vert\boldsymbol{S}_q-\boldsymbol{X}_{\star}\Vert
    &\leq |R_q-R_{\star}|
       +R_{\star}\Vert\boldsymbol{u}(\Theta_q)
       -\boldsymbol{u}(\Theta_{\star})\Vert\\
    &\leq |R_q-R_{\star}|
       +R_{\star}d_{\mathbb{S}^1}(\Theta_q,\Theta_{\star})
    \rightarrow0
    \qquad\text{almost surely}.
\end{split}
\]
The same random vector $\boldsymbol{Z}$ is used for every $q$, so $\sigma_q\boldsymbol{Z}=q^{-1/2}\sigma_{\star}\boldsymbol{Z}\to\boldsymbol{0}$ almost surely. Therefore,
\[
    \Vert\boldsymbol{X}_q-\boldsymbol{X}_{\star}\Vert
    \leq\Vert\boldsymbol{S}_q-\boldsymbol{X}_{\star}\Vert
      +\sigma_q\Vert\boldsymbol{Z}\Vert
    \rightarrow0
    \qquad\text{almost surely}.
\]
Almost sure convergence implies that $\boldsymbol{X}_q$ converges to $\boldsymbol{X}_{\star}$ in distribution.

\noindent\emph{Geometric divergence.}
The extrema displayed in Subsection~\ref{appendix: Construction and Coupling of Corrugated Alternatives}, together with strict monotonicity between consecutive extrema, show that $\rho_q$ completes exactly $\kappa_q=10q$ radial corrugations. Moreover,
\[
    \rho_q'(\theta)
    =\frac{\kappa_qw_q}{2}\sin(\kappa_q\theta).
\]
Therefore,
\[
\begin{split}
    \operatorname{length}(\mathcal{M}_q)
    &=\int_0^{2\pi}
      \sqrt{\rho_q(\theta)^2+\rho_q'(\theta)^2}\,\dd\theta\\
    &\geq\int_0^{2\pi}|\rho_q'(\theta)|\,\dd\theta
     =\frac{\kappa_qw_q}{2}
       \int_0^{2\pi}|\sin(\kappa_q\theta)|\,\dd\theta=2\kappa_qw_q
     =20q(\beta_q-\alpha_q).
\end{split}
\]
Here, $\int_0^{2\pi}|\sin(\kappa_q\theta)|\,\dd\theta=4$ because $\kappa_q$ is a positive integer. Since $\eta_q$ decreases with $q$, $\alpha_q$ decreases and $\beta_q$ increases. Hence,
\[
    \operatorname{length}(\mathcal{M}_q)
    \geq20q(\beta_1-\alpha_1)
    \rightarrow\infty.
\]
This proves all four statements.
\end{proof}

The proposition exhibits the central contrast: the data-generating distributions converge, while the associated manifold geometry diverges. To express this contrast as a failure of numerical identifiability, we must specify both the distance between distributions and the required stability property. For probability distributions $P$ and $Q$ on $\mathbb{R}^2$, define their bounded-Lipschitz distance by
\[
    d_{\mathrm{BL}}(P,Q)
    :=\sup_{\substack{\Vert h\Vert_{\infty}\leq1\\
        \mathrm{Lip}(h)\leq1}}
      \left|\int h\,\dd P-\int h\,\dd Q\right|,
\]
where $\mathrm{Lip}(h):=\sup_{\boldsymbol{x}\neq\boldsymbol{y}}
|h(\boldsymbol{x})-h(\boldsymbol{y})|/
\Vert\boldsymbol{x}-\boldsymbol{y}\Vert$. This metric metrizes convergence in distribution on $\mathbb{R}^2$. Let $P_q:=\mathbb{P}_{\boldsymbol{X}_q}$ and $P_{\star}:=\mathbb{P}_{\boldsymbol{X}_{\star}}$ denote the corresponding data-generating distributions.

\begin{definition}[Numerical identifiability of a geometric feature]
\label{def: numerical identifiability of geometry}
Let $\mathfrak{A}$ be a specified family of generative models. Each model $a\in\mathfrak{A}$ determines a data-generating distribution $P_a$ and an underlying manifold $\mathcal{M}_a$. Fix a reference model $a_{\star}\in\mathfrak{A}$, write $P_{\star}:=P_{a_{\star}}$ and $\mathcal{M}_{\star}:=\mathcal{M}_{a_{\star}}$, and let $\Gamma$ be a real-valued geometric feature of the manifold. For $\varepsilon>0$, define the geometric uncertainty radius $\mathfrak{R}_{\Gamma,\star}(\varepsilon)\in[0,\infty]$ by
\[
    \mathfrak{R}_{\Gamma,\star}(\varepsilon)
    :=\sup\left\{
       \left|\Gamma(\mathcal{M}_a)
       -\Gamma(\mathcal{M}_{\star})\right|:
       a\in\mathfrak{A},\ 
       d_{\mathrm{BL}}(P_a,P_{\star})<\varepsilon
    \right\}.
\]
The feature $\Gamma$ is \emph{numerically identifiable at the reference} over $\mathfrak{A}$ if
\[
    \mathfrak{R}_{\Gamma,\star}(\varepsilon)
    \rightarrow0
    \qquad\text{as }\varepsilon\downarrow0.
\]
\end{definition}

Definition~\ref{def: numerical identifiability of geometry} requires the inverse correspondence from the data-generating distribution to the geometric feature to be uniformly stable at the reference. Equivalently, for every sequence $a_n\in\mathfrak{A}$ satisfying $d_{\mathrm{BL}}(P_{a_n},P_{\star})\to0$, one must have $\Gamma(\mathcal{M}_{a_n})\to\Gamma(\mathcal{M}_{\star})$. This requirement is stronger than exact identifiability, which asks only whether $P_a=P_{\star}$ implies $\Gamma(\mathcal{M}_a)=\Gamma(\mathcal{M}_{\star})$. Numerical identifiability also requires nearby distributions to correspond to nearby feature values. A necessary, but weaker, condition is local boundedness: $\mathfrak{R}_{\Gamma,\star}(\varepsilon)<\infty$ for some $\varepsilon>0$.

We apply this definition to the family consisting of the reference model and the corrugated alternatives constructed above. On this family, define
\[
    \Gamma_{\mathrm{len}}(\mathcal{M})
    :=\operatorname{length}(\mathcal{M}),
    \qquad
    \Gamma_{\mathrm{cor}}(\mathcal{M}_{\star}):=0,
    \qquad
    \Gamma_{\mathrm{cor}}(\mathcal{M}_q):=10q.
\]
Thus, $\Gamma_{\mathrm{cor}}$ counts the radial corrugations in this family, with the uncorrugated reference assigned the value zero. The following corollary shows that neither feature is bounded within any bounded-Lipschitz neighborhood of the reference data-generating distribution.

\begin{corollary}[Failure of local boundedness]
\label{cor: quantitative numerical nonidentifiability}
For every $\varepsilon>0$ and every finite $L_0,C_0>0$, there exists an integer $q\geq1$ such that
\[
    d_{\mathrm{BL}}(P_q,P_{\star})<\varepsilon,
    \qquad
    \operatorname{length}(\mathcal{M}_q)>L_0,
    \qquad
    \Gamma_{\mathrm{cor}}(\mathcal{M}_q)>C_0.
\]
\end{corollary}

\begin{proof}
Proposition~\ref{prop: quantile-calibrated numerical nonidentifiability} implies $\boldsymbol{X}_q\to\boldsymbol{X}_{\star}$ in distribution, and therefore $d_{\mathrm{BL}}(P_q,P_{\star})\to0$. The same proposition gives $\operatorname{length}(\mathcal{M}_q)\to\infty$ and $\Gamma_{\mathrm{cor}}(\mathcal{M}_q)=10q\to\infty$. Choose $q$ large enough that all three inequalities hold simultaneously.
\end{proof}

Because $L_0$ and $C_0$ in the corollary are arbitrary, they may be chosen larger than the corresponding reference values by any prescribed amount. Hence, every $\varepsilon>0$ satisfies
\[
    \mathfrak{R}_{\Gamma_{\mathrm{len}},\star}(\varepsilon)=\infty,
    \qquad
    \mathfrak{R}_{\Gamma_{\mathrm{cor}},\star}(\varepsilon)=\infty.
\]
Thus, the construction fails numerical identifiability in the sense of Definition~\ref{def: numerical identifiability of geometry} and even fails the weaker local-boundedness condition. This is the precise sense in which manifold length and corrugation count are numerically nonidentifiable in the example.

The almost sure convergence in Proposition~\ref{prop: quantile-calibrated numerical nonidentifiability} is a property of the particular coupling used in the proof. Its implication for observable data is that $\boldsymbol{X}_q$ converges to $\boldsymbol{X}_{\star}$ in distribution. For any fixed $q$, the distributions of $\boldsymbol{X}_q$ and $\boldsymbol{X}_{\star}$ are not asserted to be equal. The construction therefore establishes a failure of numerical identifiability for these geometric features, not a failure of exact identifiability.

The conclusion depends on the model family under consideration. Both the latent sampling distribution and the noise level vary with $q$, so the result applies to any family containing the reference and the constructed alternatives and permitting these components to vary accordingly. It need not apply to narrower families that fix or otherwise constrain the latent distribution, noise level, or manifold geometry.

Finally, the Gaussian perturbations give every data-generating distribution in the example support on all of $\mathbb{R}^2$. Hence, the example does not satisfy the compact-support assumption used in Section~\ref{section: Manifolds via Minimization} to prove existence of a PME minimizer. Its purpose here is solely to demonstrate the failure of numerically stable recovery of manifold length and corrugation count.

\subsection{Density Formula and Numerical Evaluation}
\label{appendix: Density Formula and Numerical Evaluation}

This subsection first derives the exact density of $\boldsymbol{X}_q$ and then describes the numerical approximation used for the density panels in Figure~\ref{fig:noisy-circle-density-curve-grid}. For $r\in(\alpha_q,\beta_q)$ and $j=0,\ldots,\kappa_q-1$, define
\[
\begin{aligned}
    \theta_{j,q}^{+}(r)
    &:=\frac{2\pi j+\xi_q(r)}{\kappa_q},
    &\qquad
    \theta_{j,q}^{-}(r)
    &:=\frac{2\pi(j+1)-\xi_q(r)}{\kappa_q},\\
    \boldsymbol{m}_{j,q}^{\pm}(r)
    &:=r\boldsymbol{u}\{\theta_{j,q}^{\pm}(r)\},
    &
    \phi_{\tau}^{\mathrm{G}}(\boldsymbol{z})
    &:=\frac{1}{2\pi\tau^2}
      \exp\!\left\{-\frac{\Vert\boldsymbol{z}\Vert^2}{2\tau^2}\right\}.
\end{aligned}
\]
Conditional on $R_q=r$, the cell index $J_q$ is uniform on $\{0,\ldots,\kappa_q-1\}$, the branch indicator $B_q$ is Bernoulli$(1/2)$, and both variables are independent of $R_q$. Therefore, $\boldsymbol{S}_q$ is conditionally uniform over the $2\kappa_q$ inverse-branch locations $\boldsymbol{m}_{j,q}^{\pm}(r)$:
\[
    \mathbb{P}(\boldsymbol{S}_q\in A\mid R_q=r)
    =\frac{1}{2\kappa_q}
      \sum_{j=0}^{\kappa_q-1}
      \sum_{s\in\{+,-\}}
      \mathbbm{1}_{A}\{\boldsymbol{m}_{j,q}^{s}(r)\}
\]
for every Borel set $A\subseteq\mathbb{R}^2$. Adding the independent Gaussian perturbation gives
\[
    p_q(\boldsymbol{x})
    =\mathbb{E}\left[
      \phi_{\sigma_q}^{\mathrm{G}}
      (\boldsymbol{x}-\boldsymbol{S}_q)
      \right].
\]
Conditioning first on $R_q$ and then using the density $g_q$ in \eqref{eq: truncated Rice density} gives
\begin{equation}\label{eq: alternative corrugated density}
    p_q(\boldsymbol{x})
    =\frac{1}{2\kappa_q(1-2\eta_q)}
      \sum_{j=0}^{\kappa_q-1}
      \sum_{s\in\{+,-\}}
      \int_{\alpha_q}^{\beta_q}
      \phi_{\sigma_q}^{\mathrm{G}}
      \{\boldsymbol{x}-\boldsymbol{m}_{j,q}^{s}(r)\}
      f_R(r)\,\dd r.
\end{equation}
No discrete endpoint contributions are needed in \eqref{eq: alternative corrugated density}, because $R_q$ has the density $g_q$ in \eqref{eq: truncated Rice density} and therefore has no atoms at $\alpha_q$ or $\beta_q$.

Equation~\eqref{eq: alternative corrugated density} also makes the distinction between latent and observed support explicit. Every Gaussian density in the integrand is strictly positive, so $p_q(\boldsymbol{x})>0$ for every $\boldsymbol{x}\in\mathbb{R}^2$. Consequently,
\[
    \supp(\mathbb{P}_{\boldsymbol{X}_q})=\mathbb{R}^2,
    \qquad
    \mathcal{M}_q\subseteq\supp(\mathbb{P}_{\boldsymbol{X}_q})
\]
for every $q\geq1$. In contrast, $\supp(\mathbb{P}_{\boldsymbol{S}_q})=\mathcal{M}_q$. Thus, $\mathcal{M}_q$ is contained in the support of the observed random variable, but it is the latent signal---not $\boldsymbol{X}_q$ itself---whose support is exactly $\mathcal{M}_q$. For fixed $q$, the bounded integration interval permits differentiation under the integral sign for derivatives of every order, so $p_q$ is infinitely differentiable. Finally, Tonelli's theorem and $\int_{\mathbb{R}^2}\phi_{\sigma_q}^{\mathrm{G}}(\boldsymbol{x}-\boldsymbol{m})\,\dd\boldsymbol{x}=1$ give
\[
\begin{split}
    \int_{\mathbb{R}^2}p_q(\boldsymbol{x})\,\dd\boldsymbol{x}
    &=\frac{1}{1-2\eta_q}
      \int_{\alpha_q}^{\beta_q}f_R(r)\,\dd r\\
    &=1,
\end{split}
\]
confirming the normalization directly.

The radial quantiles used in the displayed examples can be computed without numerical integration of \eqref{eq: Rice radial density for nonidentifiability example}. Indeed,
\[
    \frac{\boldsymbol{X}_{\star}}{\sigma_{\star}}
    =\frac{\boldsymbol{u}(T_{\star})}{\sigma_{\star}}
      +\boldsymbol{Z}_{\star},
\]
and, conditional on $T_{\star}$, the vector $\boldsymbol{u}(T_{\star})/\sigma_{\star}$ is fixed with squared norm $1/\sigma_{\star}^2$. Thus, conditionally on $T_{\star}$, $R_{\star}^2/\sigma_{\star}^2$ has a noncentral chi-squared distribution with two degrees of freedom and noncentrality $1/\sigma_{\star}^2=100$. This conditional distribution does not depend on $T_{\star}$, so the same statement holds unconditionally. Therefore,
\[
    F_R^{-1}(v)
    =\sigma_{\star}
      \sqrt{F_{\chi'^2_2(100)}^{-1}(v)},
    \qquad 0<v<1.
\]

Table~\ref{tab: quantile-corrugated parameters} records the resulting parameters for the alternatives shown in Figure~\ref{fig:noisy-circle-density-curve-grid}. The entries are obtained from the noncentral chi-squared quantile representation above.

\begin{table}[ht]
\centering
\caption{\footnotesize Parameters of the nine corrugated alternatives displayed in Figure~\ref{fig:noisy-circle-density-curve-grid}.}
\label{tab: quantile-corrugated parameters}
\footnotesize
\begin{tabular}{c|ccccc}
\hline
$q$ & $\eta_q$ & $\alpha_q$ & $\beta_q$ & $\kappa_q$ & $\sigma_q$\\
\hline
1 & 0.050000 & 0.840971 & 1.169111 & 10 & 0.100000\\
2 & 0.033333 & 0.822126 & 1.187980 & 20 & 0.070711\\
3 & 0.025000 & 0.809562 & 1.200560 & 30 & 0.057735\\
4 & 0.020000 & 0.800215 & 1.209920 & 40 & 0.050000\\
5 & 0.016667 & 0.792811 & 1.217335 & 50 & 0.044721\\
6 & 0.014286 & 0.786701 & 1.223454 & 60 & 0.040825\\
7 & 0.012500 & 0.781514 & 1.228649 & 70 & 0.037796\\
8 & 0.011111 & 0.777015 & 1.233155 & 80 & 0.035355\\
9 & 0.010000 & 0.773049 & 1.237128 & 90 & 0.033333\\
\hline
\end{tabular}
\end{table}

In particular, $\alpha_9=0.773049\ldots$ and $\beta_9=1.237128\ldots$. Because $\alpha_q$ decreases and $\beta_q$ increases with $q$, every displayed alternative satisfies
\[
    0.773\leq\Vert\boldsymbol{x}\Vert\leq1.238,
    \qquad
    \boldsymbol{x}\in\mathcal{M}_q,
    \quad 1\leq q\leq9,
\]
with the bounds rounded outward.

For the numerical evaluation of the density panels, let $L:=4{,}096$ and define midpoint nodes in the truncated Rice-probability coordinate by
\[
    v_{\ell,q}
    :=\eta_q+(1-2\eta_q)
      \frac{\ell-\frac{1}{2}}{L},
    \qquad
    r_{\ell,q}:=F_R^{-1}(v_{\ell,q}),
    \qquad \ell=1,\ldots,L.
\]
Applying the midpoint rule to \eqref{eq: alternative corrugated density} gives the direct approximation
\[
    \widehat p_q(\boldsymbol{x})
    :=\frac{1}{2\kappa_qL}
      \sum_{\ell=1}^{L}
      \sum_{j=0}^{\kappa_q-1}
      \sum_{s\in\{+,-\}}
      \phi_{\sigma_q}^{\mathrm{G}}
      \{\boldsymbol{x}-\boldsymbol{m}_{j,q}^{s}(r_{\ell,q})\}.
\]
The cancellation of $1-2\eta_q$ follows from integrating in the truncated probability coordinate rather than directly in $r$.

For computational efficiency, this approximation is evaluated in two stages. First, mass $1/(2\kappa_qL)$ is placed at every inverse-branch location $\boldsymbol{m}_{j,q}^{s}(r_{\ell,q})$, and each mass is deposited bilinearly onto the four nearest points of a common $501\times501$ grid over $[-1.5,1.5]^2$. The grid spacing is $3/500=0.006$. Bilinear deposition preserves the total source mass, and the annular bounds above ensure that every source point lies strictly within the computational domain. Second, the gridded probability masses are convolved with the sampled $\mathcal{N}_2(\boldsymbol{0},\sigma_q^2\boldsymbol{I}_2)$ density. The arrays are zero-padded to $1{,}024\times1{,}024$ before applying the fast Fourier transform, which prevents circular wraparound. No grid-area factor is introduced at this stage because the first array contains probability masses, rather than sampled density values.

The reference panel is evaluated directly from \eqref{eq: target noisy-circle density}. To avoid overflow when the Bessel argument is large, the exponentially scaled Bessel function is used through the equivalent expression
\[
    p_{\star}(\boldsymbol{x})
    =\frac{1}{2\pi\sigma_{\star}^2}
      \exp\!\left\{-\frac{(\Vert\boldsymbol{x}\Vert-1)^2}
      {2\sigma_{\star}^2}\right\}
      \left[e^{-a}I_0(a)\right],
    \qquad
    a:=\frac{\Vert\boldsymbol{x}\Vert}{\sigma_{\star}^2}.
\]
Here, the bracketed quantity is the exponentially scaled modified Bessel function. As a numerical check, trapezoidal integration over the plotting grid gives total masses of at least $0.999995$ for all ten displayed densities. The largest evaluated density is below $0.65$ in every panel, so the common color scale $[0,0.65]$ does not truncate the displayed values.

Panels~(a)--(e) and~(k)--(n) of Figure~\ref{fig:noisy-circle-density-curve-grid} display the nine alternative densities, and panel~(o) displays the reference density. The corresponding manifolds appear directly below them in panels~(f)--(j) and~(p)--(t). The common coordinate limits, aspect ratio, and density scale make each vertical comparison direct. These finite examples provide the visual intuition used in Section~\ref{section: Numerical Nonidentifiability}; Proposition~\ref{prop: quantile-calibrated numerical nonidentifiability} and Corollary~\ref{cor: quantitative numerical nonidentifiability} establish the corresponding limiting distributional and geometric conclusions.

\section{Tuning-Parameter Selection Motivated by a Latent-Variable Model}\label{section: Latent-Variable Model and Tuning Parameter Selection via Oracle Criterion}

When no downstream analysis dictates a choice of $\lambda$, we prefer values for which the conditional mean squared residual varies as little as possible over the template, relative to its overall level. The latent-variable model below motivates this principle. Note that the formal rule compares the penalized population fits and does not assume exact recovery of the latent manifold.

\subsection{Latent-variable motivation}

Let $\boldsymbol{f}_0:\M\rightarrow\mathbb{R}^D$ be deterministic, smooth, and an embedding, and let the $\M$-valued latent variable $\boldsymbol{T}$ have a density $q$ with respect to Riemannian volume that is bounded away from zero almost everywhere. Let the $\mathbb{R}^D$-valued random variable $\boldsymbol{\zeta}$ be independent of $\boldsymbol{T}$, with
\[
    \mathbb{E}\boldsymbol{\zeta}=\boldsymbol{0},
    \qquad
    \mathrm{Cov}(\boldsymbol{\zeta})=\sigma^2\boldsymbol{I}_D,
    \qquad
    \sigma^2>0.
\]
For $\boldsymbol{m}\in\M$, let $\boldsymbol{\Pi}^{\perp}_{\boldsymbol{m}}\in\mathbb{R}^{D\times D}$ denote the orthogonal projection matrix onto
\[
    N_{\boldsymbol{f}_0(\boldsymbol{m})}\boldsymbol{f}_0(\M)
    =\left(d\boldsymbol{f}_0(T_{\boldsymbol{m}}\M)\right)^\bot,
\]
and define $\boldsymbol{\varepsilon}:=\boldsymbol{\Pi}^{\perp}_{\boldsymbol{T}}\boldsymbol{\zeta}$. Suppose that the $\mathbb{R}^D$-valued observation $\boldsymbol{X}$ is generated by
\begin{align}\label{eq: the latent-variable model}
    \boldsymbol{X}=\boldsymbol{f}_0(\boldsymbol{T})+\boldsymbol{\varepsilon}.
\end{align}
Assume $\Vert\boldsymbol{\varepsilon}\Vert\leq B$ almost surely for some finite $B$. This bound ensures that the distribution of $\boldsymbol{X}$ has compact support, while the normal-space construction gives
\[
    \mathbb{E}(\boldsymbol{\varepsilon}\mid\boldsymbol{T})=\boldsymbol{0},
    \qquad
    \mathrm{Cov}(\boldsymbol{\varepsilon}\mid\boldsymbol{T})
    =\sigma^2\boldsymbol{\Pi}^{\perp}_{\boldsymbol{T}},
    \qquad
    \mathbb{E}(\Vert\boldsymbol{\varepsilon}\Vert^2\mid\boldsymbol{T})
    =(D-d)\sigma^2.
\]
The constant conditional noise level in the last display motivates comparing population fits by the relative variation of their conditional squared residuals over the template.

\subsection{Population criterion}

Fix a nonempty compact candidate set $\Lambda\subset(0,\infty)$, which may be a closed interval or a prespecified finite grid. For each $\lambda\in\Lambda$, choose a representative $\Bf_\lambda^*$ of a population minimizing orbit of~\eqref{eq: the core min problem}. Fix a Borel measurable projection index $\Bpi_{\Bf_\lambda^*}$ as in Lemma~\ref{lemma: Borel measurable selection of nearest-point}, and define
\[
    \boldsymbol{M}_\lambda
    :=\Bpi_{\Bf_\lambda^*}(\boldsymbol{X}),
    \qquad
    \boldsymbol{R}_\lambda
    :=\boldsymbol{X}-\Bf_\lambda^*(\boldsymbol{M}_\lambda).
\]
Here, $\boldsymbol{M}_\lambda$ is the projection index, $\boldsymbol{R}_\lambda$ is the corresponding residual, and $\mu_\lambda$ denotes the law of $\boldsymbol{M}_\lambda$. Assumption~\ref{assumption: the assumption for lambda selection} (presented later) makes the minimizing orbit unique and requires the nearest-point argmin to be a singleton almost surely. The index and its law depend on the chosen representative only through the corresponding isometric reparameterization.

Because $\M$ is a compact manifold and hence a standard Borel space, choose a Borel measurable nonnegative function $\Phi(\lambda,\cdot):\M\to[0,\infty)$ satisfying
\begin{align}\label{eq: def of residual function}
    \Phi(\lambda,\boldsymbol{M}_\lambda)
    =\mathbb{E}\left(\Vert\boldsymbol{R}_\lambda\Vert^2
    \,\middle|\,\boldsymbol{M}_\lambda\right)
    \quad\text{almost surely}.
\end{align}
Write $\Phi_\lambda:=\Phi(\lambda,\cdot)$. Thus, for fixed $\lambda$, $\Phi_\lambda$ is a Borel version of the conditional mean squared residual as a function of the projection index.

To define the uniform averaging used by the criterion, let $\nu_g$ denote normalized Riemannian volume measure on $\M$, i.e.,
\[
    d\nu_g(\boldsymbol{m})
    :=\frac{d\vol_g(\boldsymbol{m})}{\vol_g(\M)}.
\]
The coverage condition in Assumption~\ref{assumption: the assumption for lambda selection} below ensures that all such Borel versions agree $\nu_g$-almost everywhere, so we identify $\Phi_\lambda$ with this equivalence class; the regularity condition places it in $L^2(\nu_g)$. Define
\[
    \overline{\Phi}_\lambda
    :=\int_\M\Phi_\lambda(\boldsymbol{m})\,d\nu_g(\boldsymbol{m}),
    \qquad
    \mathcal{C}(\lambda)
    :=\frac{\Vert\Phi_\lambda-\overline{\Phi}_\lambda\Vert_{L^2(\nu_g)}}
    {\overline{\Phi}_\lambda}.
\]
The quantity $\overline{\Phi}_\lambda$ averages the conditional squared residual over $\M$ under normalized volume; we call it the \emph{manifold-wide residual average}. The numerator of $\mathcal{C}(\lambda)$ measures how much the conditional residual level varies over $\M$, and the denominator removes its overall scale.

Indeed, if $\widetilde{\Bf}_\lambda^*=\Bf_\lambda^*\circ\phi$ for $\phi\in\operatorname{Isom}(\M,g)$, then, under the singleton condition and choosing the transported conditional-mean version, we have
\[
    \widetilde{\boldsymbol M}_\lambda=\phi^{-1}(\boldsymbol M_\lambda),
    \qquad
    \widetilde{\boldsymbol R}_\lambda=\boldsymbol R_\lambda,
    \qquad
    \widetilde\Phi_\lambda=\Phi_\lambda\circ\phi
    \quad \nu_g\text{-a.e.}
\]
Because $\phi$ preserves $\nu_g$, both $\overline\Phi_\lambda$ and $\mathcal C(\lambda)$ are independent of the representative. For scalar $u,v\in L^2(\nu_g)$, write $[u]:=\{u\circ\phi:\phi\in\operatorname{Isom}(\M,g)\}$ and define
\[
    d_{2,\mathrm{iso}}([u],[v])
    :=\inf_{\phi\in\operatorname{Isom}(\M,g)}
    \Vert u-v\circ\phi\Vert_{L^2(\nu_g)}.
\]

We impose the following assumption.
\begin{assumption}\label{assumption: the assumption for lambda selection}
For every $\lambda\in\Lambda$, the population minimizers in~\eqref{eq: the core min problem} form the single orbit $[\Bf_\lambda^*]$, and $\argmin_{\boldsymbol{m}\in\M}\Vert\boldsymbol{X}-\Bf_\lambda^*(\boldsymbol{m})\Vert$ is a singleton almost surely for one, hence every, representative. In addition:
\begin{enumerate}[label=(\roman*)]
    \item The measure $\mu_\lambda$ has a density $p_\lambda$ with respect to $\nu_g$, and there is a constant $c_M>0$ such that $p_\lambda\geq c_M$ $\nu_g$-almost everywhere for every $\lambda\in\Lambda$ (equivalently, for every representative of the orbit).
    \item The map $\lambda\mapsto[\Phi_\lambda]$ is continuous under $d_{2,\mathrm{iso}}$, and there is a constant $b>0$ such that $\inf_{\lambda\in\Lambda}\overline{\Phi}_\lambda\geq b$.
    \item The criterion has a strictly separated minimizer: there exists $\lambda_0\in\Lambda$ such that, for every $\epsilon>0$ for which the set below is nonempty,
    \[
        \kappa_\epsilon
        :=\inf_{\substack{\lambda\in\Lambda\\
        |\lambda-\lambda_0|\geq\epsilon}}
        \big\{\mathcal{C}(\lambda)-\mathcal{C}(\lambda_0)\big\}>0.
    \]
\end{enumerate}
\end{assumption}

The image $\Bf_{\lambda_0}^*(\M)$ is called the \emph{penalized pseudo-true fit}. If $\boldsymbol{U}\sim\nu_g$, then $\mathcal{C}(\lambda)$ is the coefficient of variation of $\Phi_\lambda(\boldsymbol{U})$. We denote the minimizer produced by the rule by $\lambda^*$:
\begin{align}\label{eq: choice of the optimal lambda}
    \lambda^*
    :=\argmin_{\lambda\in\Lambda}\mathcal{C}(\lambda)
    =\argmin_{\lambda\in\Lambda}
    \left\{\frac{\sqrt{\mathrm{Var}\big(\Phi(\lambda,\boldsymbol{U})\big)}}
    {\mathbb{E}\big(\Phi(\lambda,\boldsymbol{U})\big)}\right\},
    \qquad \boldsymbol{U}\sim\nu_g.
\end{align}

\begin{theorem}\label{thm: the theorem for lambda selection}
Under Assumption~\ref{assumption: the assumption for lambda selection}, the population criterion in~\eqref{eq: choice of the optimal lambda} is finite, independent of both the version of $\Phi$ and the representative of the minimizing orbit, and continuous on $\Lambda$. Its unique minimizer is $\lambda^*=\lambda_0$.

For the conditional sample result, let $\widehat{\Phi}_{N,\lambda}\in L^2(\nu_g)$ be random regression-function estimates such that \footnotemark
\[
    \sup_{\lambda\in\Lambda}
    d_{2,\mathrm{iso}}([\widehat\Phi_{N,\lambda}],[\Phi_\lambda])
    \overset{\mathbb{P}^*}{\longrightarrow}0.
\]
\footnotetext{For the definitions of outer probability $\mathbb{P}^*$ and convergence in outer probability $\overset{\mathbb{P}^*}{\longrightarrow}$, see \citep[][Sections~1.2 and~1.9]{van1996weak}.}
For a deterministic sequence $\tau_N>0$ with $\tau_N\downarrow0$, set
\[
    \widehat{\overline{\Phi}}_{N,\lambda}
    :=\int_\M\widehat{\Phi}_{N,\lambda}(\boldsymbol{m})\,d\nu_g(\boldsymbol{m}),
    \qquad
    \widehat{\mathcal{C}}_N(\lambda)
    :=\frac{\Vert\widehat{\Phi}_{N,\lambda}-\widehat{\overline{\Phi}}_{N,\lambda}\Vert_{L^2(\nu_g)}}
    {\max\{\widehat{\overline{\Phi}}_{N,\lambda},\tau_N\}}.
\]
The truncation by $\tau_N$ prevents division by a nonpositive or near-zero estimated average. If a data-dependent choice $\widehat{\lambda}_N\in\Lambda$ has optimization gap
\[
    r_N
    :=\widehat{\mathcal{C}}_N(\widehat{\lambda}_N)
    -\inf_{\lambda\in\Lambda}\widehat{\mathcal{C}}_N(\lambda)
    =o_{\mathbb{P}^*}(1),
\]
then $\widehat{\lambda}_N\overset{\mathbb{P}^*}{\longrightarrow}\lambda_0$.
\end{theorem}

\begin{figure}[ht]
    \centering
    \begin{tikzpicture}
        \node[anchor=south west,inner sep=0] (figGone) at (0,0) {\includegraphics[width=1.0\linewidth]{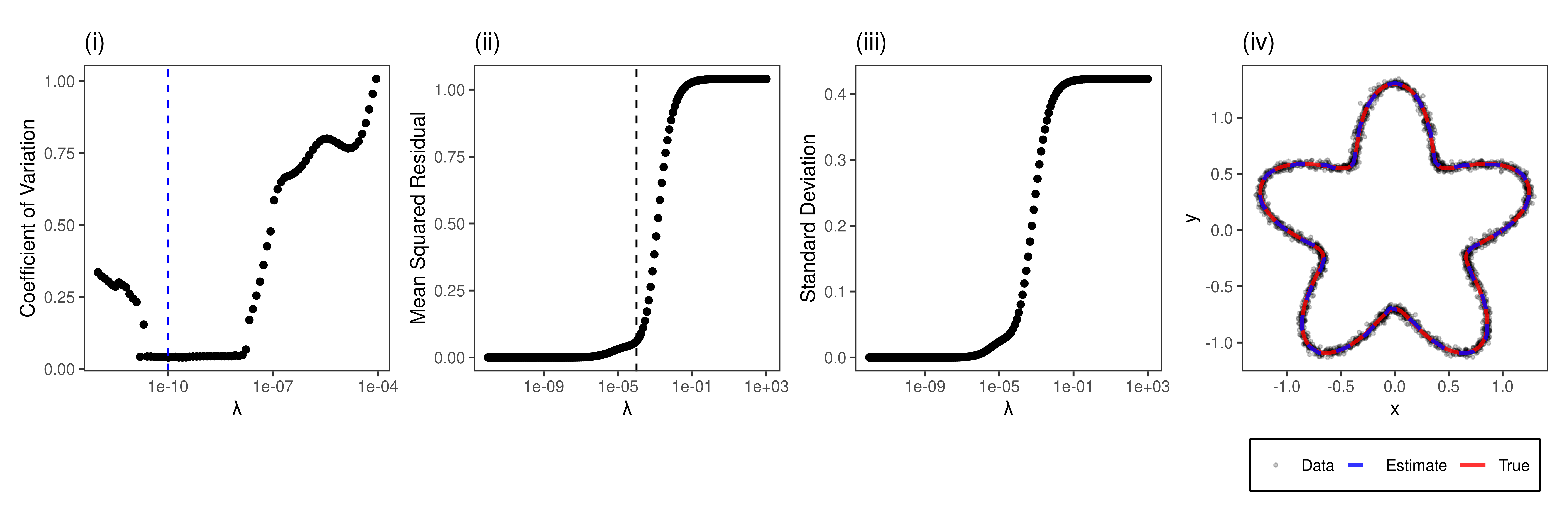}};
    \end{tikzpicture}
    \vspace{-25pt}
    \caption{\footnotesize (i) The estimated criterion $\widehat{\mathcal{C}}_N(\lambda)$; the blue dotted line marks $\widehat{\lambda}_N$, selected over the elbow-restricted set $\Lambda_{\mathrm{sel}}$. (ii) The manifold-wide residual average $\widehat{\overline{\Phi}}_{N,\lambda}$; the black dotted line marks $\lambda_{\mathrm{elbow}}$. (iii) The estimated standard deviation of $\widehat{\Phi}_{N,\lambda}(\boldsymbol{U})$, where $\boldsymbol{U}\sim\nu_g$. (iv) The data (gray dots), the estimate associated with $\widehat{\lambda}_N$ (blue dotted curve), and the true data-generating curve (red solid curve).}
    \label{fig:pme_demo_cv_lambda_selection}
\end{figure}

\subsection{Sample implementation}

Figures~\ref{fig:pme_demo_cv_lambda_selection} and~\ref{fig:sph_results_panel_8} illustrate the following implementation for curve and surface templates, respectively. Given observations $\{\boldsymbol{X}_i\}_{i\in[N]}$, first fit the PA map $\Bf^*_{N,\lambda}:\M\to\mathbb{R}^D$ for each $\lambda\in\Lambda$, and define the empirical projection index and residual by
\[
    \widehat{\boldsymbol{M}}_{i,\lambda}
    :=\Bpi_{\Bf^*_{N,\lambda}}(\boldsymbol{X}_i),
    \qquad
    \widehat{\boldsymbol{R}}_{i,\lambda}
    :=\boldsymbol{X}_i
    -\Bf^*_{N,\lambda}(\widehat{\boldsymbol{M}}_{i,\lambda}),
    \qquad i\in[N].
\]
\noindent These quantities form the regression sample
\[
    \mathcal{D}_{N,\lambda}
    :=\left\{\left(\widehat{\boldsymbol{M}}_{i,\lambda},
    \Vert\widehat{\boldsymbol{R}}_{i,\lambda}\Vert^2\right):
    i\in[N]\right\}.
\]
\noindent Next, apply the nonparametric regression method in \citep{pelletier2006non} to $\mathcal{D}_{N,\lambda}$, and denote the fitted function by $\widehat{\Phi}_{N,\lambda}$. Related local-regression methods for manifold or circular predictors are described in \citep{cheng2013local, di2009local}. Finally, approximate the integrals defining $\widehat{\overline{\Phi}}_{N,\lambda}$ and $\widehat{\mathcal{C}}_N(\lambda)$ using independent draws from $\nu_g$. On the finite candidate grids used below, select
\[
    \widehat{\lambda}_N
    \in\argmin_{\lambda\in\Lambda}\widehat{\mathcal{C}}_N(\lambda).
\]

\begin{figure}[H]
    \centering
    \begin{tikzpicture}
        \node[anchor=south west,inner sep=0] (figGflower) at (0,0) {\includegraphics[width=1\linewidth]{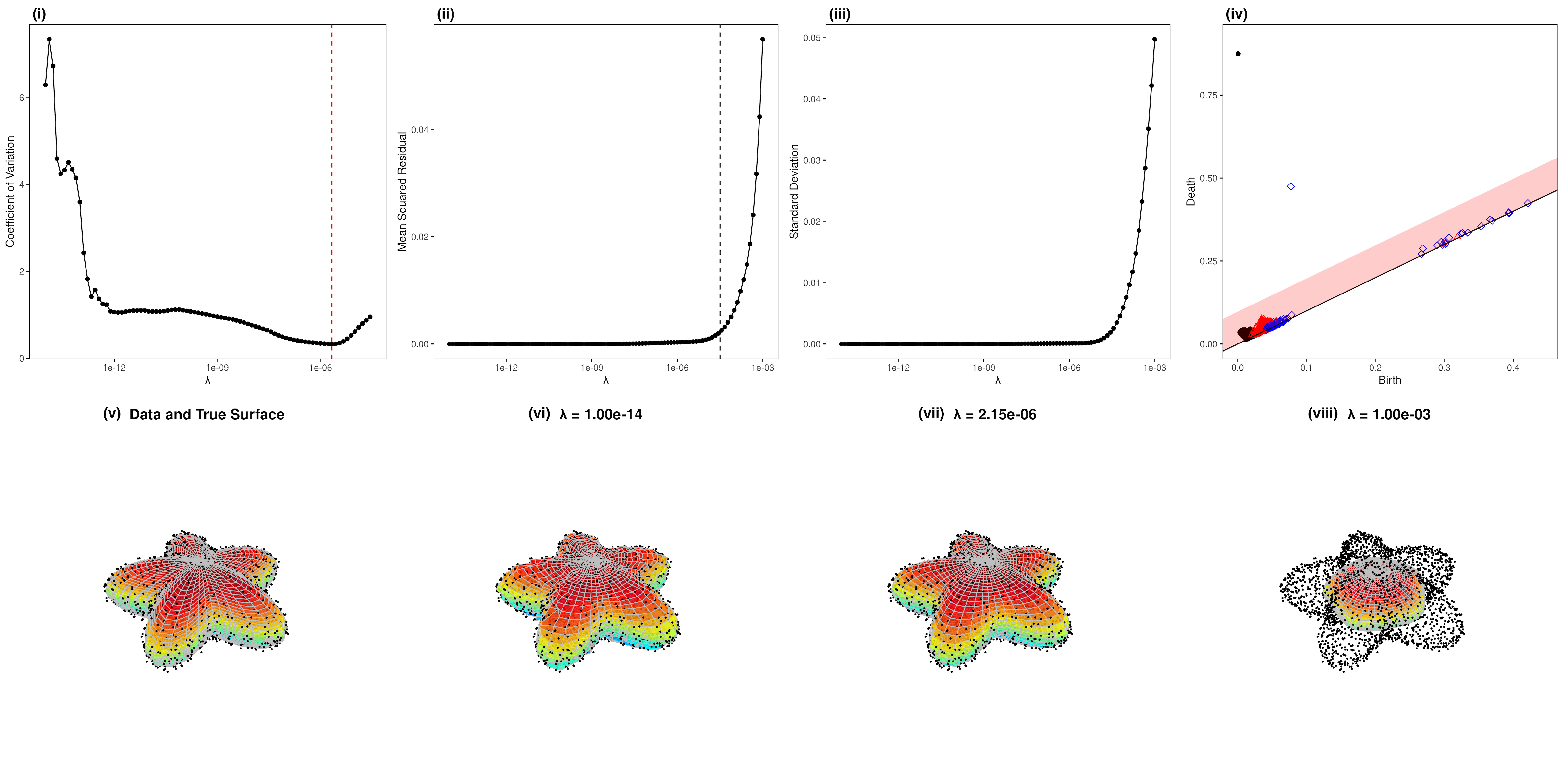}};
    \end{tikzpicture}
    \begin{tikzpicture}
        \node[anchor=south west,inner sep=0] (figGcashew) at (0,0) {\includegraphics[width=1\linewidth]{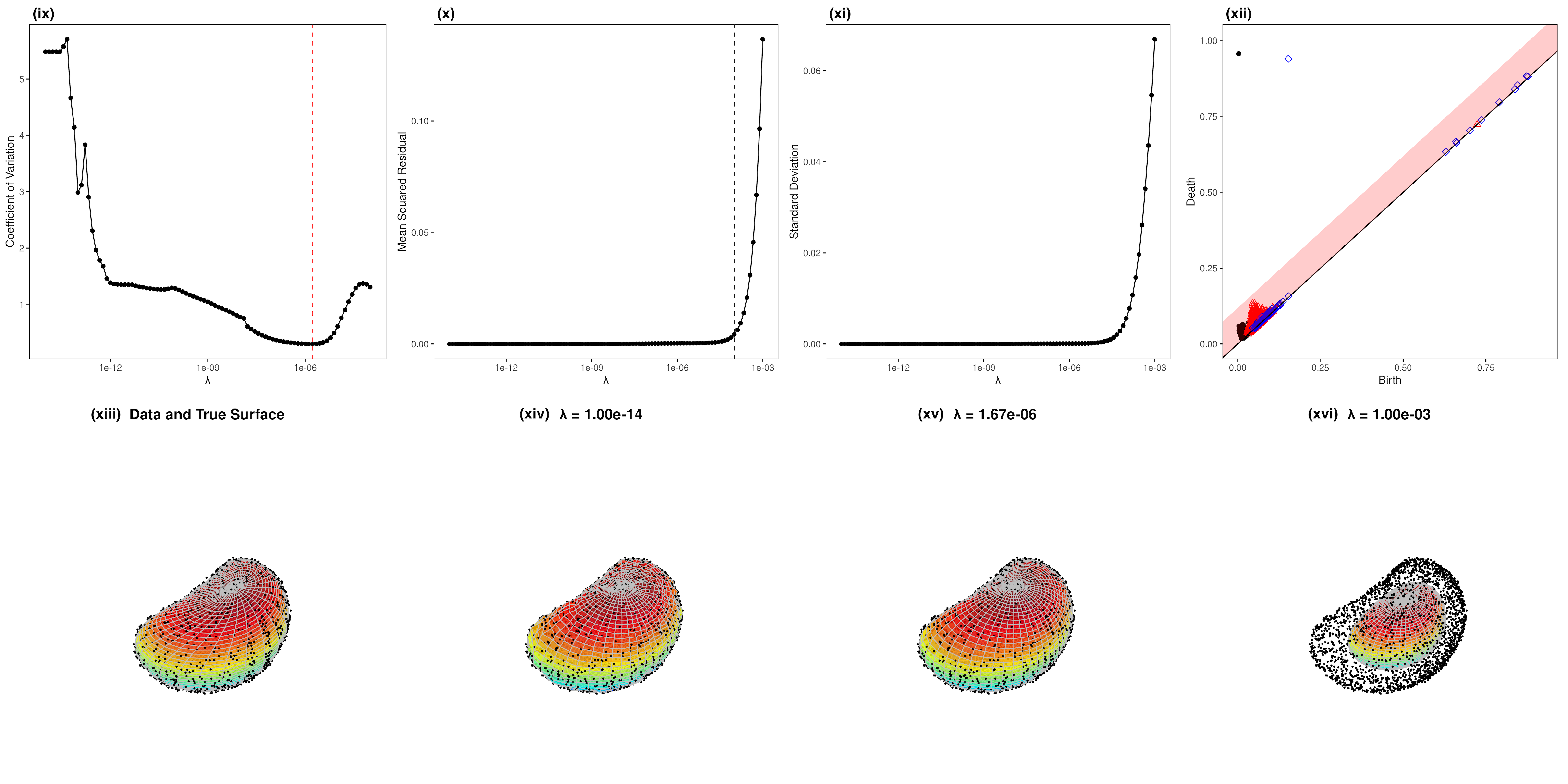}};
    \end{tikzpicture}
    \vspace{-35pt}
    \caption{\footnotesize Panels (i)--(viii) show the flower/star example; panels (ix)--(xvi) show the moon/cashew example. Panels (i)--(iii) and (ix)--(xi) show, respectively, $\widehat{\mathcal{C}}_N(\lambda)$, $\widehat{\overline{\Phi}}_{N,\lambda}$, and the estimated standard deviation of $\widehat{\Phi}_{N,\lambda}(\boldsymbol{U})$ for $\boldsymbol{U}\sim\nu_g$. The red dotted lines in panels (i) and (ix) mark $\widehat{\lambda}_N$, selected over the experiment-specific elbow-restricted set $\Lambda_{\mathrm{sel}}$; the black dotted lines in panels (ii) and (x) mark the corresponding $\lambda_{\mathrm{elbow}}$. The persistence diagrams in panels (iv) and (xii) indicate one significant zero-dimensional feature and one significant two-dimensional feature, with no significant one-dimensional feature; the pink band is a 99\% bootstrap confidence band. Panels (v) and (xiii) show the point clouds and their true data-generating surfaces. Panels (vi)--(viii) and (xiv)--(xvi) show, respectively, overfitting, selected, and underfitting fits.}
    \label{fig:sph_results_panel_8}
\end{figure}

\subsection{Numerical illustrations}

For each example below, let $\Lambda_{\mathrm{grid}}$ be the prespecified finite grid. We first compute the diagnostic curves for every $\lambda\in\Lambda_{\mathrm{grid}}$. Let $\lambda_{\mathrm{elbow}}$ denote the elbow point of $\lambda\mapsto\widehat{\overline{\Phi}}_{N,\lambda}$ marked by the black vertical line, and define the restricted selection set $\Lambda_{\mathrm{sel}}:=\{\lambda\in\Lambda_{\mathrm{grid}}:\lambda<\lambda_{\mathrm{elbow}}\}$. Only values in $\Lambda_{\mathrm{sel}}$ are considered for selection, so the preceding selection rule is applied with $\Lambda_{\mathrm{sel}}$ in place of $\Lambda$. Values at and beyond the elbow point remain in the diagnostic plots for comparison only. 

\begin{itemize}
    \item Figure~\ref{fig:pme_demo_cv_lambda_selection} uses the flower-shaped point cloud generated by the mechanism in Appendix~\ref{appendix: Boundary of a Flower/Star}. In this experiment, only tuning parameters smaller than the elbow point shown in panel~(ii) are considered, and the reported $\widehat{\lambda}_N$ minimizes $\widehat{\mathcal{C}}_N$ over $\Lambda_{\mathrm{sel}}$. At the selected value, the fitted curve closely follows the five-petal data-generating curve while smoothing the surrounding noise.

    \item Figure~\ref{fig:sph_results_panel_8} uses the flower/star and moon/cashew point clouds generated by the mechanisms in Appendices~\ref{appendix: Surface of a Flower/Star (2d, 3D)} and~\ref{section: Surface of a Moon/Cashew (2d, 3D)}, respectively. In each surface experiment, only tuning parameters smaller than its corresponding elbow point (panel~(ii) for the flower/star and panel~(x) for the moon/cashew) are considered, and $\widehat{\lambda}_N$ minimizes $\widehat{\mathcal{C}}_N$ over the corresponding $\Lambda_{\mathrm{sel}}$. In both examples, the selected fit closely reproduces the corresponding data-generating surface while avoiding the local roughness of the overfitting fit and the shrinkage of the underfitting fit toward the sample mean. 
\end{itemize}

\section{Additional Simulations}\label{appendix: additional simulations}

This section presents additional numerical experiments that validate the two conclusions of Theorem~\ref{thm: PME with lambda=infty}.

\subsection{One-Dimensional Templates in Ambient Three-Dimensional Space}

\subsubsection{Interval Template}

Figure \ref{fig:1d3D_sine_spiral} shows the algorithm results for a sinusoidal curve and a helix, computed using the exact method from Theorem \ref{thm: wahba representation theorem}, which requires inversion of a full $N\times N$ matrix. The data are generated by sampling $t_i$ uniformly from $[0,1]$ and evaluating the parameterized curves
\begin{align*}
     t_i \mapsto \mcO \Bf(t_i) + \boldsymbol \epsilon_i, \qquad
 t_i \mapsto \mcO \Bg(t_i) + \boldsymbol \epsilon_i,
\end{align*}
where $\Bf(t) = (t, \sin(2 \pi t),\cos(2\pi t))^{\T}$, $\Bg(t) = (\cos(4\pi t),\sin(4\pi t),3t)^{\T}$, and $\mcO$ is the rotation through angle $\pi/6$ about the axis in the direction $(1,1,0.2)^{\T}$. Each dataset contains $N=500$ observations, and the independent noise vectors satisfy $\boldsymbol\epsilon_i\sim\mathcal N_3(\boldsymbol 0,4\times10^{-4}\boldsymbol I_3)$, corresponding to a coordinatewise standard deviation of $0.02$.

Observe that $\lambda=10^{-12}$ results in overfitting, $\lambda=10^{-6}$ provides smoothing, and $\lambda=1$ results in a PCA line, as predicted by Theorem \ref{thm: PME with lambda=infty}. We also see that $\lambda = 10^{-6}$ visually provides a good fit to the data.

\begin{figure}[!htbp]
    \centering
    \includegraphics[width=1\linewidth]{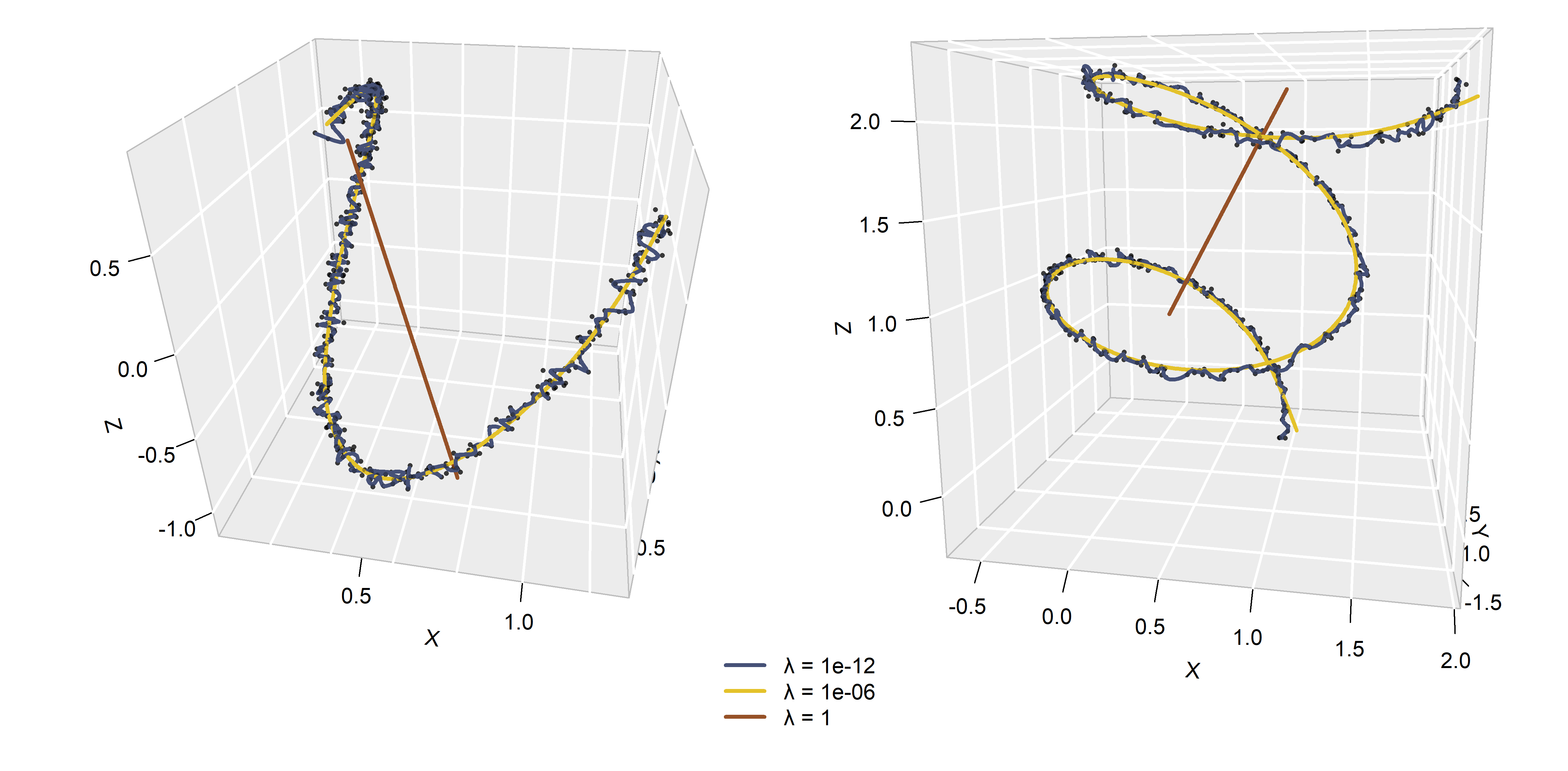}
    \caption{\footnotesize Fits on rotated sinusoidal and helix datasets in ambient $\mbR^3$ space for various values of $\lambda$.}
    \label{fig:1d3D_sine_spiral}
\end{figure}

\subsubsection{Circle Template}
For the case \(\M=\mathbb{S}^1\), we use closed parametrized curves. When \(\lambda\) is large, \(\fstar_{N,\lambda}\) collapses to the sample mean. For the star-shaped dataset, we sample $N=500$ independent parameters $t_i\sim\operatorname{Unif}(0,1)$ and set $\boldsymbol X_i=\Bf(t_i)+\boldsymbol\epsilon_i$, where
\[
\Bf(t)=(r(t)\cos(2\pi t),r(t)\sin(2\pi t),\sin(3\pi t))^{\T},
\qquad r(t)=1+0.3\sin(10\pi t).
\]
The moon-shaped dataset is generated by \texttt{moon1d3D(500, noise=0.02)}. The helper constructs two adjacent half circles whose endpoints are joined by smaller half circles, using deterministic grids along these arcs; the third coordinate is $\cos(4\pi s)$, where $s$ runs uniformly from $0$ to $1$ over the concatenated points. The nominal helper argument $500$ produces $1000$ three-dimensional points. No additional rotation is applied to this dataset. In both examples, independent noise vectors with distribution $\mathcal N_3(\boldsymbol 0,4\times10^{-4}\boldsymbol I_3)$ are added, so the coordinatewise noise standard deviation is $0.02$.

\begin{figure}[!htbp]
    \centering
    \includegraphics[width=1.0\linewidth]{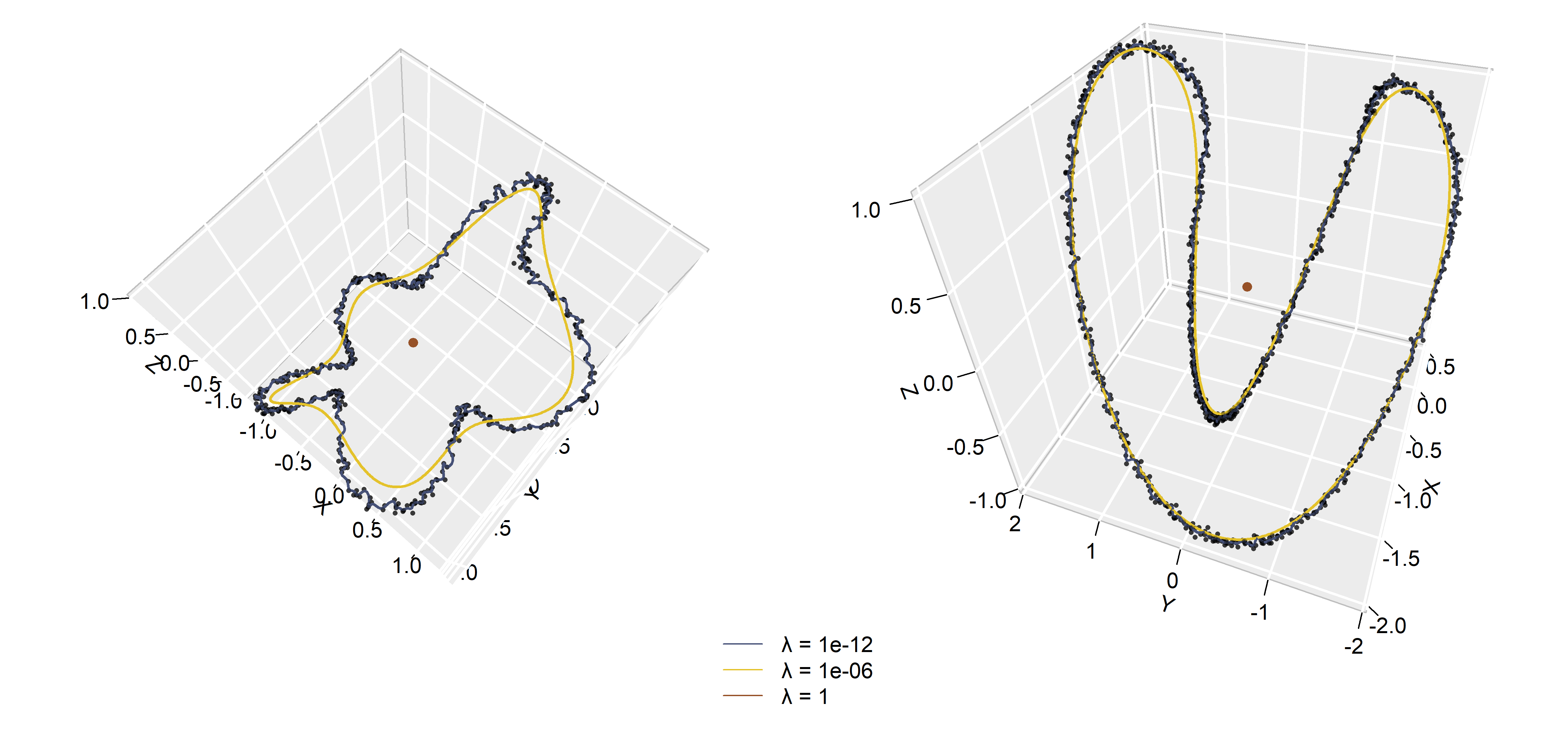}
    \caption{\footnotesize Fits on a star-shaped dataset and a moon-shaped dataset, respectively, for varying levels of $\lambda$.}
    \label{fig:1d3D_star_moon}
\end{figure}

The results in Figure \ref{fig:1d3D_star_moon} demonstrate overfitting/interpolation at $\lambda = 10^{-12}$, moderate
smoothing at $\lambda =  10^{-6}$, and collapse to the sample mean at $\lambda =1$. In particular, the case $\lambda=1$ validates Theorem~\ref{thm: PME with lambda=infty}.

\begin{figure}[H]
    \centering
    \includegraphics[width=1\linewidth]{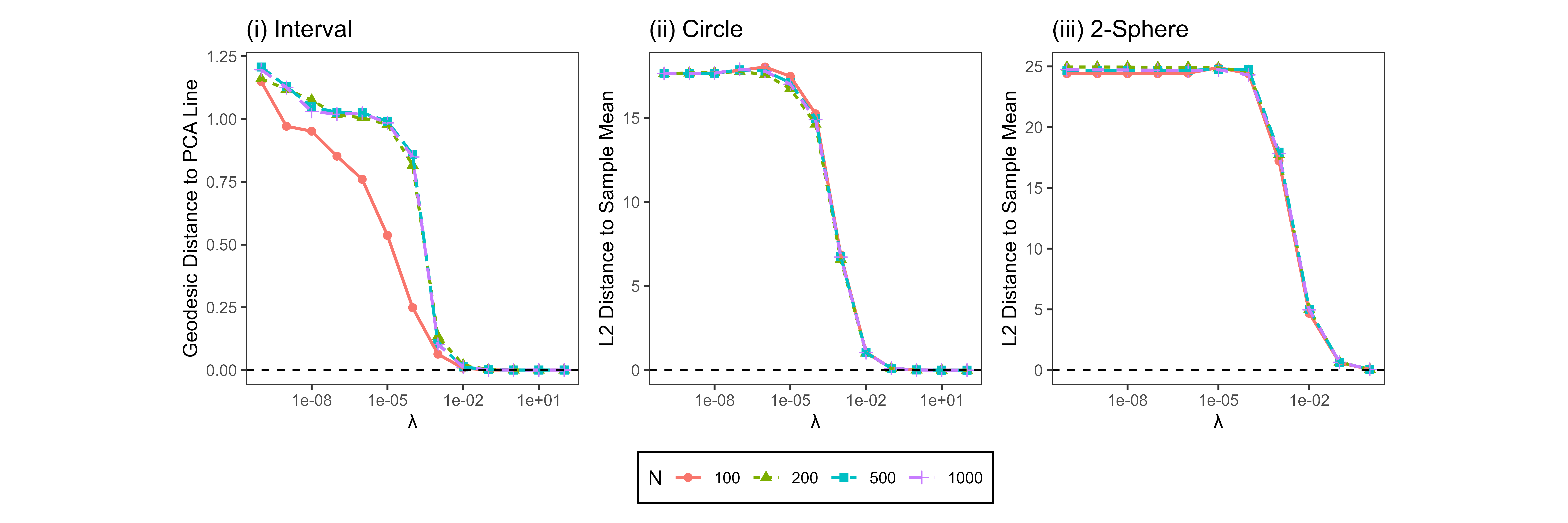}
    \caption{\footnotesize (i): Plot of geodesic distance in the $\M=[0,1]$ case between $\fstar_{N,\lambda}$ and $\fstar_{N,\infty}$ (PCA-line) as a function of $\lambda$, where each line corresponds to a different $N$. (ii) and (iii): $L^2$ distance between  $\fstar_{N,\lambda}$ and $\fstar_{N,\infty}$ (sample mean) as a function of $\lambda$ for the $\M=\mbS^1$ and $\M=\mbS^2$ cases, respectively. }
    \label{fig:dist_lambda_infty}
\end{figure}

\subsection{Asymptotics of the Tuning Parameter}

Theorem \ref{thm: PME with lambda=infty} characterizes the limiting behavior of $\fstar_{N,\lambda}$ as $\lambda\to\infty$. For the case of $[0,1]$, it will take the form of a PCA line, and for closed manifolds such as $\mbS^1$ and $\mbS^2$, it will collapse to the sample mean. We present numerical evidence demonstrating this property as $\lambda$ increases.

For the case of the interval template, there may be infinitely many reparameterizations of a line. Therefore, we use the geodesic distance between two planar curves to assess convergence \citep{srivastava2016functional}. For $\mbS^1$ and $\mbS^2$, because $\fstar_{N,\lambda}$ collapses to a point, we use the standard $L^2$ distance between functions. The results are shown in Figure~\ref{fig:dist_lambda_infty}. Observe how in all cases, as $\lambda$ increases, the distance between $\fstar_{N,\lambda}$ and the PCA line (for the interval) or the sample mean (for $\mbS^1$ and $\mbS^2$) approaches $0$.

\section{Mathematical Proofs}\label{Appendix: Proofs}

In this section, we present the proofs of our proposed theorems and lemmas.

\subsection{Proofs of the Results in Section \ref{section: Mathematical Preliminaries}}

We first review the notion of Hölder spaces on Riemannian manifolds \citep[][Section 2.6]{hebey2000nonlinear}. Given a smooth, compact Riemannian manifold $(\M, g)$ and $0<\alpha\le 1$, let $C^{0,\alpha}(\M)$ be the set of continuous functions $u:\M\rightarrow\mathbb{R}$ for which the norm
\begin{align*}
    \Vert u\Vert_{C^{0,\alpha}(\M)}=\max_{\boldsymbol{m}\in\M} \vert u(\boldsymbol{m})\vert + \sup_{\boldsymbol{m}\ne \boldsymbol{m}'} \frac{\vert u(\boldsymbol{m})-u(\boldsymbol{m}')\vert}{d_g(\boldsymbol{m},\,\boldsymbol{m}')^\alpha}
\end{align*}
is finite, where $\alpha$ is called the Hölder exponent, and $d_g(\cdot,\cdot)$ denotes the distance associated with the Riemannian metric $g$. Denote
\begin{align*}
    [u]_{C^{0,\alpha}(\M)}:=\sup_{\boldsymbol{m}\ne \boldsymbol{m}'} \frac{\vert u(\boldsymbol{m})-u(\boldsymbol{m}')\vert}{d_g(\boldsymbol{m},\,\boldsymbol{m}')^\alpha}.
\end{align*}
For any two Hölder exponents $0<\alpha<\beta\le 1$, it is easy to verify that $C^{0,\beta}(\M) \subseteq C^{0,\alpha}(\M)$ and $[u]_{C^{0,\alpha}(\M)} \le C_{\alpha,\beta}\cdot [u]_{C^{0,\beta}(\M)}$, where $C_{\alpha,\beta}$ is a constant depending on $\alpha$, $\beta$, and the diameter of $\M$. Hereafter, $u$ denotes a real-valued function, in contrast to the vector-valued function $\Bf=(f_1,\ldots,f_D)$.

We begin with a lemma, which serves as preparation for the proof of Theorem \ref{thm: Rellich–Kondrachov embedding}.
\begin{lemma}\label{lemma: general Sobolev inequalities on manifolds}
Let \((\M,g)\) be a smooth, compact \(d\)-dimensional Riemannian manifold, and assume that \(\partial \M\) is either empty or smooth. If \(d<4\), then \(H^2(\M)\subseteq C^{0,\frac13}(\M)\). Moreover, there exists a constant \(C_H>0\), depending only on \((\M,g)\), such that
\[
\|u\|_{C^{0,\frac13}(\M)} \le C_H \|u\|_{H^2(\M)}
\qquad\text{for all }u\in H^2(\M).
\]
\end{lemma}
\noindent The Hölder exponent \(1/3\) in Lemma~\ref{lemma: general Sobolev inequalities on manifolds} is not optimal, but it is sufficient for the proof of Theorem~\ref{thm: Rellich–Kondrachov embedding}. Lemma~\ref{lemma: general Sobolev inequalities on manifolds} is an analogue of Theorem~2.8 in the monograph by \cite{hebey2000nonlinear}. For the reader's convenience, we include a proof of Lemma~\ref{lemma: general Sobolev inequalities on manifolds}, following the same strategy as employed in the proof of Theorem~2.8 by \cite{hebey2000nonlinear}.

\begin{proof}[\textbf{Proof of Lemma \ref{lemma: general Sobolev inequalities on manifolds}}]
Put \(\alpha:=1/3\). We first separate two Sobolev norms that occur in the argument. The norm used in the statement of Lemma~\ref{lemma: general Sobolev inequalities on manifolds} is the norm defined in \eqref{eq: Sobolev norm 3}, while the standard second-order Sobolev norm is defined by
\[
\|u\|_{H^2_{\mathrm{std}}(\M)}
:=
\bigl(
\|u\|_{L^2(\M)}^2
+
\|\nabla u\|_{L^2(\M)}^2
+
\|\nabla^2 u\|_{L^2(\M)}^2
\bigr)^{1/2},
\]
as defined in \eqref{eq: Sobolev norm 1}. Here, \(\nabla u\) is the covariant differential of the scalar function \(u\), and \(\nabla^2u\) is its covariant Hessian. By Appendix~\ref{section: Equivalent Sobolev Norms}, the norms \(\|\cdot\|_{H^2(\M)}\) and \(\|\cdot\|_{H^2_{\mathrm{std}}(\M)}\) are equivalent. In particular, there exists \(C_{\mathrm{eq}}\ge 1\) such that
\[
\|u\|_{H^2_{\mathrm{std}}(\M)} \le C_{\mathrm{eq}} \|u\|_{H^2(\M)}
\qquad\text{for all }u\in H^2(\M).
\]

We first prove the estimate for $u\in C^\infty(\M)$; passage to the
completed space and construction of its continuous representative are made at the end. Choose a finite regular atlas of nested coordinate patches\footnote{Recall that $A\Subset B$ denotes compact containment, meaning that the closure
$\overline A$ is a compact subset of $B$.}
\[
\Omega_s\Subset\widetilde\Omega_s,\qquad s=1,\ldots,N,
\]
such that $\{\Omega_s\}_{s=1}^N$ covers $\M$. After rescaling, an interior
chart maps $\Omega_s$ to $B_2(0)$, while a boundary chart maps it to
$B_2(0)\cap\overline{\mathbb H}^{\,d}$, relatively open in
$\overline{\mathbb H}^{\,d}:=\{x_d\geq0\}$; in either case the same
coordinate map is defined on the larger patch
$\widetilde\Omega_s$. Put
\[
U_s:=\varphi_s(\Omega_s\cap\operatorname{int}\M)
=
\begin{cases}
B_2(0),&\Omega_s\cap\partial\M=\varnothing,\\
B_2^+(0):=B_2(0)\cap\{x_d>0\},&
\Omega_s\cap\partial\M\ne\varnothing.
\end{cases}
\]
Values at the flat boundary of a boundary chart are understood on
$\overline U_s$. Choose a smooth partition of unity
$\{\eta_s\}_{s=1}^N$ with
$K_s:=\operatorname{supp}\eta_s\Subset\Omega_s$ and
$\sum_{s=1}^N\eta_s=1$, and define
\[
u_s:=(\eta_su)\circ\varphi_s^{-1}\qquad\text{on }U_s.
\]
For each boundary chart, fix a bounded linear $H^2$ extension operator across
the flat face to a full bounded smooth Euclidean domain, also bounded on
$L^2$; the compact support of $\eta_s$ inside $\Omega_s$ keeps $u_s$ away
from the curved chart boundary. Interior charts use the corresponding
full-ball estimate directly. The larger patches ensure that all metric,
Jacobian, Christoffel-symbol, and cutoff bounds used below hold on compact
closures, and finiteness of the atlas makes the constants uniform.

We now prove carefully that the local Euclidean \(H^2_{\mathrm{std}}\) norm of \(u_s\) is controlled by the global covariant \(H^2_{\mathrm{std}}\) norm of \(u\). Fix \(s\). Write \(x=(x^1,\dots,x^d)\) for the coordinates on \(U_s\). Let
\[
g^{(s)}_{ij},\qquad
J_s:=\sqrt{\det(g^{(s)}_{ij})},\qquad
\Gamma^k_{ij,s}
\]
denote the metric coefficients, the Riemannian volume density, and the Christoffel symbols in this chart. Since the coefficients are smooth on the relevant compact chart closures and the atlas is finite, there are constants \(0<c_s\le C_s<\infty\) such that, on \(U_s\),
\[
c_s |\xi|^2
\le
g^{(s)}_{ij}(x)\xi^i\xi^j
\le
C_s|\xi|^2,
\qquad
c_s\le J_s(x)\le C_s,
\qquad
|\Gamma^k_{ij,s}(x)|\le C_s.
\]
Thus, Euclidean coordinate norms and Riemannian tensor norms are equivalent on each chart, and the constants can later be made uniform by taking maxima over the finite atlas.

Let \(w\in C^\infty(\M)\) be supported in \(\Omega_s\), and set \(\widetilde w=w\circ\varphi_s^{-1}\). Repeated indices are summed. Since \(w\) is scalar-valued,
\[
\partial_i \widetilde w = (\nabla_i w)\circ\varphi_s^{-1},
\qquad
\partial_i\partial_j \widetilde w
=
\bigl((\nabla^2 w)_{ij}+\Gamma^k_{ij,s}\nabla_k w\bigr)\circ\varphi_s^{-1}.
\]
Using the preceding identities, the change of variables formula \(d\operatorname{vol}_g=J_s(x)\,dx\), and the boundedness of \(J_s^{-1}\), we obtain
\begin{align*}
    & \|\widetilde w\|_{L^2(U_s)}
\le a_s\|w\|_{L^2(\Omega_s)}, \\
& \|\nabla_{\!x}\widetilde w\|_{L^2(U_s)}
\le a_s\|\nabla w\|_{L^2(\Omega_s)}, \\
& \|\nabla_{\!x}^2\widetilde w\|_{L^2(U_s)}
\le
a_s\bigl(
\|\nabla^2w\|_{L^2(\Omega_s)}
+
\|\nabla w\|_{L^2(\Omega_s)}
\bigr).
\end{align*}
Here, \(\nabla_{\!x}\) and \(\nabla_{\!x}^2\) denote the Euclidean gradient and Euclidean Hessian in the coordinate variable \(x\). Adding these three estimates gives
\[
\|\widetilde w\|_{H^2_{\mathrm{std}}(U_s)}
\le
A_s
\bigl(
\|w\|_{L^2(\Omega_s)}
+
\|\nabla w\|_{L^2(\Omega_s)}
+
\|\nabla^2 w\|_{L^2(\Omega_s)}
\bigr)
\le
A\|w\|_{H^2_{\mathrm{std}}(\M)}.
\]
Here, \(A:=\sqrt{3}\max_{1\le s\le N}A_s<\infty\); the factor \(\sqrt{3}\) follows from applying the Cauchy--Schwarz inequality to the sum of the three \(L^2\) norms.

We now take \(w=\eta_s u\). Multiplication by the fixed cutoff \(\eta_s\) is bounded on the covariant \(H^2_{\mathrm{std}}\) norm. Indeed,
\[
\|\eta_su\|_{L^2(\M)}
\le
\|\eta_s\|_{L^\infty(\M)}\|u\|_{L^2(\M)},
\]
while the first and second covariant product rules are
\[
\nabla(\eta_s u)=\eta_s\nabla u+u\nabla\eta_s \quad \text{ and } \quad \nabla^2(\eta_s u)
=
\eta_s\nabla^2u
+
u\nabla^2\eta_s
+
\nabla\eta_s\otimes\nabla u
+
\nabla u\otimes\nabla\eta_s,
\]
respectively. Hence, since \(\eta_s\), \(\nabla\eta_s\), and \(\nabla^2\eta_s\) are smooth fixed tensors on the compact manifold \(\M\), there is a constant \(B_s<\infty\) such that
\[
\|\eta_s u\|_{H^2_{\mathrm{std}}(\M)}
\le
B_s\|u\|_{H^2_{\mathrm{std}}(\M)}
\le
B\|u\|_{H^2_{\mathrm{std}}(\M)}
\]
where \(B:=\max_{1\le s\le N}B_s<\infty\). Combining the local coordinate estimate with the cutoff estimate gives \(C_1:=AB>0\), independent of \(u\) and \(s\), such that
\[
\|u_s\|_{H^2_{\mathrm{std}}(U_s)}
\le C_1 \|u\|_{H^2_{\mathrm{std}}(\M)}
\le C_1 C_{\mathrm{eq}} \|u\|_{H^2(\M)}
\qquad\text{for } s=1,\dots,N.
\]

For an interior chart, apply the Euclidean Sobolev--Morrey
estimate directly on the ball. For a boundary chart, apply it to the fixed
bounded $H^2$ extension across the flat face and then restrict to
$\overline U_s$. Thus, by the ``general Sobolev inequalities''
\citep[][Section 5.6.3]{evans1998pde}, for each \(s\) there is
\(C_{2,s}<\infty\) such that
\[
\|u_s\|_{C^{0,\alpha}(\overline U_s)}
\le
C_{2,s}\|u_s\|_{H^2_{\mathrm{std}}(U_s)}.
\]
Since \(N<\infty\), after replacing \(C_{2,s}\) by
\(C_2:=\max_{1\le s\le N}C_{2,s}\), we have
\[
\|u_s\|_{C^{0,\frac13}(\overline U_s)}
\le C_2 \|u_s\|_{H^2_{\mathrm{std}}(U_s)}
\qquad\text{for } s=1,\dots,N.
\]

It remains to patch the local H\"older estimates into a global one. Let
\(\psi_s:=\eta_su\), so that \(u=\sum_{s=1}^N\psi_s\). We claim that, for each \(s\), there is a
constant \(D_s<\infty\), depending only on the chart and cutoff, such that
\[
\|\psi_s\|_{C^{0,\alpha}(\M)}
\le
D_s\|u_s\|_{C^{0,\alpha}(\overline U_s)}.
\]
The supremum-norm part is immediate. For the seminorm, recall
$K_s=\operatorname{supp}\eta_s\Subset\Omega_s$. Choose $\rho_s>0$ so that
the closed $2\rho_s$-neighborhood of $K_s$ is contained in $\Omega_s$.
If both $p,q$ lie outside $K_s$, then
$\psi_s(p)-\psi_s(q)=0$. Otherwise relabel the points so that $p\in K_s$.
When $d_g(p,q)<\rho_s$, one has $q\in\Omega_s$, and every length-minimizing
curve from $p$ to $q$ remains inside the larger chart. The differential
of $\varphi_s$ is bounded on the compact neighborhood swept out by these
curves, so integration along such a curve gives $|\varphi_s(p)-\varphi_s(q)|\le L_s d_g(p,q)$. It follows that
\[
|\psi_s(p)-\psi_s(q)|
\le
L_s^\alpha
[u_s]_{C^{0,\alpha}(\overline U_s)}
d_g(p,q)^\alpha.
\]
When $d_g(p,q)\ge\rho_s$, the far-pair estimate gives
\[
|\psi_s(p)-\psi_s(q)|
\le
2\|\psi_s\|_{C(\M)}
\le
2\rho_s^{-\alpha}
\|u_s\|_{C^{0,\alpha}(\overline U_s)}
d_g(p,q)^\alpha.
\]
These cases prove the claim without asserting that the chart map is
globally Lipschitz with respect to the ambient geodesic distance. Taking \(C_3:=\sum_{s=1}^N D_s\), we get
\[
\|u\|_{C^{0,\alpha}(\M)}
\le
\sum_{s=1}^N\|\psi_s\|_{C^{0,\alpha}(\M)}
\le
C_3\sum_{s=1}^N\|u_s\|_{C^{0,\alpha}(\overline U_s)}.
\]
Combining the preceding estimates yields
\[
\|u\|_{C^{0,\frac13}(\M)}
\le
C_3 \sum_{s=1}^N \|u_s\|_{C^{0,\frac13}(\overline U_s)}
\le
N C_1 C_2 C_3 C_{\mathrm{eq}}\|u\|_{H^2(\M)}.
\]
Thus, with $C_H:=NC_1C_2C_3C_{\mathrm{eq}}$, the claimed estimate holds
for every $u\in C^\infty(\M)$. Now let $u$ be an arbitrary element of the
completion $H^2(\M)$ and choose $u_n\in C^\infty(\M)$ such that
$u_n\to u$ in $H^2(\M)$. Applying the smooth estimate to $u_n-u_m$ shows
that $\{u_n\}$ is Cauchy in the Banach space
$C^{0,\frac13}(\M)$; let $v$ be its limit there. Convergence in
$H^2(\M)$ implies convergence to the class $u$ in $L^2(\M)$, while
uniform convergence to $v$ implies convergence to $v$ in $L^2(\M)$
because $\M$ has finite volume. Hence $v$ represents the $L^2$ class
$u$.

This representative is independent of the chosen smooth approximating
sequence and is unique: any two resulting continuous representatives
agree almost everywhere and therefore everywhere, since Riemannian volume
has full support. We henceforth identify $u$ with this representative.
Passing to the limit in the smooth estimate gives $\|u\|_{C^{0,\frac13}(\M)}
\le C_H\|u\|_{H^2(\M)}$.
\end{proof}

\begin{proof}[\textbf{Proof of Theorem \ref{thm: Rellich–Kondrachov embedding}}]
We first treat the scalar-valued case \(H^2(\M)=H^2(\M;\mathbb{R})\). The norm on \(C^{0,\frac13}(\M)\) is $\|u\|_{C^{0,\frac13}(\M)}
=
\|u\|_{C(\M)}+[u]_{C^{0,\frac13}(\M)}$, so \(\|u\|_{C(\M)}\le \|u\|_{C^{0,\frac13}(\M)}\). Therefore, Lemma~\ref{lemma: general Sobolev inequalities on manifolds} gives
\[
\|u\|_{C(\M)}
\le
\|u\|_{C^{0,\frac13}(\M)}
\le
C_H \|u\|_{H^2(\M)}
\qquad\text{for all }u\in H^2(\M).
\]
Hence, the inclusion map $\iota:H^2(\M)\to C(\M)$ is well defined and continuous.

We prove compactness by the sequential characterization of compactness in metric spaces. Let \(\{u_n\}_{n\in\mathbb{N}}\subset H^2(\M)\) be a bounded sequence. Write
\[
M:=\sup_{n\in\mathbb{N}} \|u_n\|_{H^2(\M)} < \infty.
\]
Applying Lemma~\ref{lemma: general Sobolev inequalities on manifolds} again gives
\[
\sup_{n\in\mathbb{N}} \|u_n\|_{C^{0,\frac13}(\M)}
\le
C_H M.
\]
In particular, \(\{u_n\}_{n\in\mathbb{N}}\) is uniformly bounded in \(C(\M)\). Moreover, for every \(\boldsymbol m,\boldsymbol m'\in \M\),
\[
|u_n(\boldsymbol m)-u_n(\boldsymbol m')|
\le
[u_n]_{C^{0,\frac13}(\M)}\, d_g(\boldsymbol m,\boldsymbol m')^{1/3}
\le
C_H M\, d_g(\boldsymbol m,\boldsymbol m')^{1/3}.
\]
This estimate is uniform in \(n\). Hence, given \(\varepsilon>0\), if \(M>0\), we may take \(\delta=(\varepsilon/(C_HM))^3\), and if \(M=0\) the sequence is identically zero in \(H^2(\M)\). In either case, whenever
\(d_g(\boldsymbol m,\boldsymbol m')<\delta\), the difference \(|u_n(\boldsymbol m)-u_n(\boldsymbol m')|\) is less than \(\varepsilon\), uniformly in \(n\). Thus, \(\{u_n\}_{n\in\mathbb{N}}\) is uniformly equicontinuous on the compact metric space \((\M,d_g)\).

By the Arzel\`a--Ascoli theorem \citep[][Theorem 4.25]{brezis2011functional}, a uniformly bounded and uniformly equicontinuous sequence in \(C(\M)\) has a uniformly convergent subsequence. Therefore there exist a subsequence \(\{u_{n_k}\}_{k\in\mathbb{N}}\) and a function \(u^\ast\in C(\M)\) such that
\[
\|u_{n_k}-u^\ast\|_{C(\M)} \rightarrow 0
\qquad\text{as }k\to\infty.
\]
This proves that every bounded sequence in \(H^2(\M)\) has a subsequence whose image under \(\iota\) converges in \(C(\M)\). Hence \(\iota:H^2(\M)\to C(\M)\) is compact in the scalar-valued case.

We now treat the vector-valued case. Equip \(H^2(\M;\mathbb{R}^D)\) with the componentwise
Hilbert norm
\[
\|\boldsymbol f\|_{H^2(\M;\mathbb{R}^D)}
=
\left(\sum_{j=1}^D \|f_j\|_{H^2(\M)}^2\right)^{1/2},
\]
where $\boldsymbol f=(f_1,\dots,f_D)$, and equip \(C(\M;\mathbb{R}^D)\) with the supremum norm induced by the Euclidean norm on \(\mathbb{R}^D\). For continuity, the scalar estimate gives
\begin{align*}
\|\boldsymbol f\|_{C(\M;\mathbb{R}^D)}
&=
\sup_{\boldsymbol m\in\M}
\left(\sum_{j=1}^D |f_j(\boldsymbol m)|^2\right)^{1/2}  \\
&\le
\sum_{j=1}^D \|f_j\|_{C(\M)}
\le
C_H\sum_{j=1}^D \|f_j\|_{H^2(\M)} \le
C_H\sqrt D\,\|\boldsymbol f\|_{H^2(\M;\mathbb{R}^D)}.
\end{align*}
Thus, the vector-valued inclusion is continuous.

For compactness, let $\{\boldsymbol f_n=(f_{n,1},\dots,f_{n,D})\}_{n\in\mathbb{N}}\subseteq H^2(\M;\mathbb{R}^D)$ be bounded, say
\[
\sup_{n\in\mathbb{N}}\|\boldsymbol f_n\|_{H^2(\M;\mathbb{R}^D)}<\infty.
\]
Then, for each fixed \(j\), the component sequence
\(\{f_{n,j}\}_{n\in\mathbb{N}}\) is bounded in \(H^2(\M)\). Applying the scalar compactness
result to the first component gives a subsequence along which \(f_{n,1}\) converges in
\(C(\M)\). Applying the scalar compactness result to the second component along that
subsequence gives a further subsequence along which both the first and second components
converge in \(C(\M)\). Continuing in this way through the finitely many components
\(j=1,\dots,D\), we obtain a single subsequence
\(\{\boldsymbol f_{n_k}\}_{k\in\mathbb{N}}\) and functions
\(f^\ast_1,\dots,f^\ast_D\in C(\M)\) such that
\[
\|f_{n_k,j}-f^\ast_j\|_{C(\M)} \rightarrow 0
\qquad\text{for every }j=1,\dots,D.
\]
Set \(\boldsymbol f^\ast := (f^\ast_1,\dots,f^\ast_D)\). Then, using the Euclidean norm on
\(\mathbb{R}^D\),
\[
\|\boldsymbol f_{n_k}-\boldsymbol f^\ast\|_{C(\M;\mathbb{R}^D)}
\le
\sum_{j=1}^D \|f_{n_k,j}-f^\ast_j\|_{C(\M)}
\rightarrow 0.
\]
Hence the inclusion $H^2(\M;\mathbb{R}^D) \subseteq C(\M;\mathbb{R}^D)$ is compact as well.
\end{proof}

\begin{proof}[\textbf{Proof of Lemma \ref{lemma: Borel measurable selection of nearest-point}}]
We verify the hypotheses of the Kuratowski--Ryll-Nardzewski measurable
selection theorem \citep[][Theorem 5.2.1]{kuratowski1965general, srivastava1998course} for
$\boldsymbol{\Psi}:\mathbb{R}^D \rightarrow 2^\mathfrak{M}$.

\noindent \underline{Part 1 (notation and basic facts)}.
Let $\Bf(\mathfrak{M}):=\{\boldsymbol{f}(\boldsymbol{m}):\boldsymbol{m}\in\M\}\subseteq\mathbb{R}^D$. Since
$\Bf$ is continuous and $\mathfrak{M}$ is compact, $\Bf(\mathfrak{M})$ is compact.
For any nonempty compact $A\subseteq\mathbb{R}^D$ define the distance function
\begin{align*}
    d(\boldsymbol{x}, A) := \min_{\boldsymbol{y}\in A}\|\boldsymbol{x}-\boldsymbol{y}\|\ \ \text{ for all }\boldsymbol{x}\in\mathbb{R}^D.
\end{align*}
It is standard that
$\boldsymbol{x}\mapsto d(\boldsymbol{x},A)$ is continuous and that if $A\subseteq B$ then
$d(\boldsymbol{x},B)\le d(\boldsymbol{x},A)$ for all $\boldsymbol{x}\in\mathbb{R}^D$.

\noindent \underline{Part 2 (closed, nonempty values)}.
For fixed $\boldsymbol{x}$, the map
\begin{align*}
\varphi_{\boldsymbol{x}}:\ & \mathfrak{M}\to\mathbb{R}, \\
& \boldsymbol{m} \mapsto \varphi_{\boldsymbol{x}}(\boldsymbol{m})=\|\boldsymbol{x}-\Bf(\boldsymbol{m})\|
\end{align*}
is continuous on compact $\mathfrak{M}$; hence, it attains its minimum and
$\boldsymbol{\Psi}(\boldsymbol{x})\neq\emptyset$, where $\boldsymbol{\Psi}$ is defined in Lemma~\ref{lemma: Borel measurable selection of nearest-point}. Being an argmin set of a
continuous function, $\boldsymbol{\Psi}(\boldsymbol{x})=\varphi_{\boldsymbol{x}}^{-1}(\min_{\boldsymbol{m}\in\mathfrak{M}}\varphi_{\boldsymbol{x}}(\boldsymbol{m}))$ is closed in
$\mathfrak{M}$; thus, it is compact. Therefore, $\boldsymbol{\Psi}$ takes
closed nonempty values.

\noindent \underline{Part 3 (a countable base with closed containment)}.
As a compact Riemannian manifold, $\mathfrak{M}$ is metrizable
and hence regular; it also satisfies the second axiom of countability
\citep[][Appendix A]{lee2018riemannian}. Let
$\{B_j\}_{j\in\mathbb{N}}$ be a countable topological base. For every open
$U\subseteq\mathfrak{M}$ and $\boldsymbol{m}\in U$, regularity gives an open
set $V$ such that
$\boldsymbol{m}\in V\subseteq\overline V\subseteq U$. Choosing $j$ with
$\boldsymbol{m}\in B_j\subseteq V$ then yields
$\boldsymbol{m}\in B_j\subseteq\overline{B_j}\subseteq\overline V\subseteq U$.

\noindent \underline{Part 4 (characterization of the hit set)}. Fix an arbitrary open $U\subseteq\mathfrak{M}$, define the ``hit set''
\[
H_U := \{\boldsymbol{x}\in\mathbb{R}^D:\ \boldsymbol{\Psi}(\boldsymbol{x})\cap U\neq\emptyset\}.
\]
For each $j$ with $B_j\subseteq \overline{B_j}\subseteq U$, set $K_j:=\Bf(\overline{B_j})$, which is compact. Consider
\[
E_j := \left\{\boldsymbol{x}\in\mathbb{R}^D:\ d(\boldsymbol{x}, K_j)=d\left(\boldsymbol{x},\Bf(\mathfrak{M})\right)\right\}.
\]
Since $\boldsymbol{x}\mapsto d(\boldsymbol{x}, K_j)$ and $\boldsymbol{x}\mapsto d\left(\boldsymbol{x},\Bf(\mathfrak{M})\right)$ are continuous, each $E_j$ is closed. We claim
\begin{equation}\label{eq:hit-decomp}
H_U = \bigcup_{j:\,B_j \,\subseteq \, \overline{B_j}\,\subseteq U} E_j .
\end{equation}
\noindent \underline{Part 4.1 (proof of ``$\subseteq$'' in \eqref{eq:hit-decomp})}. Let $\boldsymbol{x}\in H_U$ be arbitrary. Since $\boldsymbol{\Psi}(\boldsymbol{x})\cap U\neq\emptyset$, there exists $\boldsymbol{m}_\ast\in\boldsymbol{\Psi}(\boldsymbol{x})\cap U$.
By Part~3, choose $j$ with $\boldsymbol{m}_\ast\in B_j\subseteq\overline{B_j}\subseteq U$, which implies that $\Bf(\boldsymbol{m}_\ast) \in K_j$.
Then
\[
d(\boldsymbol{x}, K_j)=\min_{\boldsymbol{y}\in K_j} \Vert \boldsymbol{x}-\boldsymbol{y} \Vert
\le \|\boldsymbol{x}-\Bf(\boldsymbol{m}_\ast)\|
= \min_{\boldsymbol{m}'\in\mathfrak{M}}\|\boldsymbol{x}-\Bf(\boldsymbol{m}')\|
= d(\boldsymbol{x}, \Bf(\mathfrak{M})),
\]
while $K_j\subseteq \Bf(\mathfrak{M})$ implies $d(\boldsymbol{x}, \Bf(\mathfrak{M}))\le d(\boldsymbol{x}, K_j)$. Hence, $d(\boldsymbol{x}, \Bf(\mathfrak{M})) = d(\boldsymbol{x}, K_j)$, which implies that
$\boldsymbol{x}\in E_j$. That is, $H_U \subseteq \bigcup_{j:\,B_j \,\subseteq \, \overline{B_j}\,\subseteq U} E_j$.

\noindent \underline{Part 4.2 (proof of ``$\supseteq$'' in \eqref{eq:hit-decomp})}. Let $\boldsymbol{x} \in \bigcup_{j:\,B_j \,\subseteq \, \overline{B_j}\,\subseteq U} E_j$ be arbitrary. There exists a $j$ such that $B_j \,\subseteq \, \overline{B_j}\,\subseteq U$ and $\boldsymbol{x}\in E_j$. By compactness of
$\overline{B_j}$, choose $\boldsymbol{m}_j\in\overline{B_j}\subseteq U$ such that $\|\boldsymbol{x}-\Bf(\boldsymbol{m}_j)\| =\min_{\boldsymbol{m}\in \overline{B_j}} \Vert \boldsymbol{x}-\Bf(\boldsymbol{m})\Vert=d(\boldsymbol{x}, K_j)$. Since $\boldsymbol{x}\in E_j$, we have
\begin{align*}
    \|\boldsymbol{x}-\Bf(\boldsymbol{m}_j)\|=d(\boldsymbol{x}, K_j)=d(\boldsymbol{x}, \Bf(\mathfrak{M}))=\min_{\boldsymbol{m}'\in\mathfrak{M}}\|\boldsymbol{x}-\Bf(\boldsymbol{m}')\|,
\end{align*}
hence, $\boldsymbol{m}_j\in\boldsymbol{\Psi}(\boldsymbol{x})$. Because
$\boldsymbol{m}_j\in\overline{B_j}\subseteq U$, we have $\boldsymbol{m}_j \in \boldsymbol{\Psi}(\boldsymbol{x}) \cap U \ne \emptyset$, i.e., $\boldsymbol{x}\in H_U$, which implies that $H_U \supseteq \bigcup_{j:\,B_j \,\subseteq \, \overline{B_j}\,\subseteq U} E_j$. This proves \eqref{eq:hit-decomp}.

\noindent \underline{Part 5 (Application of the Kuratowski--Ryll-Nardzewski measurable selection theorem)}. Note that the compact Riemannian manifold $\mathfrak{M}$ is a Polish space. By \eqref{eq:hit-decomp}, $H_U$ is a countable union of the closed sets $E_j$, hence, $H_U$ is Borel measurable in $\mathbb{R}^D$. Since the open set $U$ was arbitrary, the Kuratowski--Ryll-Nardzewski measurable selection theorem implies that there exists a Borel measurable selection $\Bpi_{\Bf}:\mathbb{R}^D\to\mathfrak{M}$ with
$\Bpi_{\Bf}(\boldsymbol{x})\in\boldsymbol{\Psi}(\boldsymbol{x})$ for all
$\boldsymbol{x}\in\mathbb{R}^D$. 
\end{proof}

\subsection{Proofs of the Results in Section \ref{section: Manifolds via Minimization}}

The following lemma serves as preparation for the proof of Theorem \ref{thm: existence of f star}.
\begin{lemma}\label{lemma: continuity of Delta wrt f}
We denote 
\begin{align}\label{eq: def of the fitting error term}
    \dist(\boldsymbol{x}, \Bf) := \min_{\boldsymbol{m} \in \M} \norm{\boldsymbol{x} - \Bf(\boldsymbol{m})}^2 = \norm{\boldsymbol{x} - \Bf(\Bpi_{\Bf}(\boldsymbol{x}))}^2
\end{align}
for all $\boldsymbol{x}\in\mathbb{R}^D$ and $\Bf\in\mathscr{F}$. Then, we have
\begin{align}\label{eq: continuity of Delta wrt f}
    \left| \dist(\boldsymbol{x}, \Bf) - \dist(\boldsymbol{x}, \Bg) \right| \le 4\Upsilon\cdot \max_{\boldsymbol{m}\in\mathfrak{M}} \norm{\Bf(\boldsymbol{m}) - \Bg(\boldsymbol{m})}.
\end{align}
for all $\boldsymbol{x}\in\operatorname{supp}(\mathbb{P})$ and $\Bf,\Bg\in\mathscr{F}$.
\end{lemma}
\begin{proof}[\textbf{Proof of Lemma \ref{lemma: continuity of Delta wrt f}}]
By Lemma \ref{lemma: Borel measurable selection of nearest-point}, we have the following Borel measurable projection indices of $\Bf$ and $\Bg$, respectively.
\begin{align*}
    \Bpi_{\Bf}(\boldsymbol{x}) \in \argmin_{\boldsymbol{m}\in\mathfrak{M}}\Vert \boldsymbol{x}-\Bf(\boldsymbol{m})\Vert^2\ \ \ \text{ and }\ \ \ \Bpi_{\Bg}(\boldsymbol{x}) \in \argmin_{\boldsymbol{m}\in\mathfrak{M}}\Vert \boldsymbol{x}-\Bg(\boldsymbol{m})\Vert^2.
\end{align*}
When $\dist(\boldsymbol{x}, \Bf) \le \dist(\boldsymbol{x},\Bg)$, we have
\begin{align*}
    \begin{aligned}
        &\ \ \ \left| \dist(\boldsymbol{x}, \Bf) - \dist(\boldsymbol{x}, \Bg) \right| \\
    &= \dist(\boldsymbol{x}, \Bg) - \dist(\boldsymbol{x}, \Bf) \\
    &= \dist(\boldsymbol{x}, \Bg) - \norm{\boldsymbol{x}- \Bf(\Bpi_{\Bf}(\boldsymbol{x}))}^2 \\
    &\leq \norm{\boldsymbol{x} - \Bg(\Bpi_{\Bf}(\boldsymbol{x}))}^2 - \norm{\boldsymbol{x} - \Bf(\Bpi_{\Bf}(\boldsymbol{x}))}^2 \\
    &= \Big(\norm{\boldsymbol{x} - \Bg(\Bpi_{\Bf}(\boldsymbol{x}))} + \norm{\boldsymbol{x} - \Bf(\Bpi_{\Bf}(\boldsymbol{x}))}\Big)\Big(\norm{\boldsymbol{x} - \Bg(\Bpi_{\Bf}(\boldsymbol{x}))} - \norm{\boldsymbol{x} - \Bf(\Bpi_{\Bf}(\boldsymbol{x}))}\Big) \\
    &\leq \Big(\norm{\boldsymbol{x} - \Bg(\Bpi_{\Bf}(\boldsymbol{x}))} + \norm{\boldsymbol{x} - \Bf(\Bpi_{\Bf}(\boldsymbol{x}))}\Big) \cdot\norm{\Bg(\Bpi_{\Bf}(\boldsymbol{x})) - \Bf(\Bpi_{\Bf}(\boldsymbol{x}))} \\
    &\leq \Big(2\norm{\boldsymbol{x}} + \norm{\Bg(\Bpi_{\Bf}(\boldsymbol{x}))} + \norm{ \Bf(\Bpi_{\Bf}(\boldsymbol{x}))}\Big) \cdot\norm{\Bg(\Bpi_{\Bf}(\boldsymbol{x})) - \Bf(\Bpi_{\Bf}(\boldsymbol{x}))}.
    \end{aligned}
\end{align*}
where the second-to-last inequality follows from the reverse triangle inequality. Recall from \eqref{eq: the core min problem} that, for $\Bf,\Bg\in\mathscr{F}$,
\begin{align}\label{eq: C bound from the def of F}
    \max_{\boldsymbol{m}\in\mathfrak{M}}\Vert \Bf(\boldsymbol{m}) \Vert \le \Upsilon \ \ \text{ and }\ \ \max_{\boldsymbol{m}\in\mathfrak{M}}\Vert \Bg(\boldsymbol{m}) \Vert \le \Upsilon.
\end{align}
Together with $\Vert\boldsymbol{x}\Vert\le \Upsilon$ for $\boldsymbol{x}\in\operatorname{supp}(\mathbb{P})$, \eqref{eq: C bound from the def of F} yields
\begin{align}\label{eq: fidelity upper bound}
    \left| \dist(\boldsymbol{x}, \Bf) - \dist(\boldsymbol{x}, \Bg) \right| \le 4\Upsilon\cdot\norm{\Bg(\Bpi_{\Bf}(\boldsymbol{x})) - \Bf(\Bpi_{\Bf}(\boldsymbol{x}))}.
\end{align} 
When $\dist(\boldsymbol{x}, \Bf) > \dist(\boldsymbol{x},\Bg)$, applying the same bound once again, we have
\begin{align}\label{eq: critical bound discussion for the convergence of the ftting error term}
     \begin{aligned}
         &\ \ \ \left| \dist(\boldsymbol{x}, \Bf) - \dist(\boldsymbol{x}, \Bg) \right| \\
    &= \dist(\boldsymbol{x}, \Bf) - \dist(\boldsymbol{x}, \Bg) \\
    &= \dist(\boldsymbol{x}, \Bf) - \norm{\boldsymbol{x}- \Bg(\Bpi_{\Bg}(\boldsymbol{x}))}^2 \\
    &\leq \norm{\boldsymbol{x} - \Bf(\Bpi_{\Bg}(\boldsymbol{x}))}^2 - \norm{\boldsymbol{x}- \Bg(\Bpi_{\Bg}(\boldsymbol{x}))}^2 \\
    &= \Big(\norm{\boldsymbol{x} - \Bf(\Bpi_{\Bg}(\boldsymbol{x}))} + \norm{\boldsymbol{x}- \Bg(\Bpi_{\Bg}(\boldsymbol{x}))}\Big)\Big(\norm{\boldsymbol{x} - \Bf(\Bpi_{\Bg}(\boldsymbol{x}))} - \norm{\boldsymbol{x}- \Bg(\Bpi_{\Bg}(\boldsymbol{x}))}\Big) \\
    &\leq \Big(\norm{\boldsymbol{x} - \Bf(\Bpi_{\Bg}(\boldsymbol{x}))} + \norm{\boldsymbol{x}- \Bg(\Bpi_{\Bg}(\boldsymbol{x}))}\Big) \cdot\norm{\Bf(\Bpi_{\Bg}(\boldsymbol{x})) - \Bg(\Bpi_{\Bg}(\boldsymbol{x}))} \\
    &\le 4\Upsilon\cdot\norm{\Bf(\Bpi_{\Bg}(\boldsymbol{x})) - \Bg(\Bpi_{\Bg}(\boldsymbol{x}))}.
     \end{aligned}
\end{align}
Combining \eqref{eq: fidelity upper bound} and \eqref{eq: critical bound discussion for the convergence of the ftting error term}, we have \eqref{eq: continuity of Delta wrt f}.
\end{proof}

\begin{proof}[\textbf{Proof of Theorem \ref{thm: existence of f star}}] We divide the proof into three parts. In Part 1, we show that there exists a (sub)sequence, $\{\fnk\}_{k\in\mathbb{N}}$, that converges strongly in $C(\M)$. In Part 2, we use the result from Part 1 to establish the convergence of the fitting-error term $\mathbb{E}\dist(\boldsymbol{X}, \fnk)$. In Part 3, we leverage the previous results to demonstrate that the infimum $ \inf_{\Bf \in \mathscr{F}}\mathcal{L}_\lambda(\Bf)$ is attained at the limit of the aforementioned subsequence, i.e., $ \inf_{\Bf \in \mathscr{F}}\mathcal{L}_\lambda(\Bf) = \lim_{k\rightarrow\infty} \mathcal{L}_\lambda(\fnk)$.\\
\underline{Part 1.} By definition of infimum, there exists a sequence $\{\Bf^{(n)}=(f_1^{(n)}, f_2^{(n)},\ldots,f_D^{(n)})\}_{n\in\mathbb{N}}\subseteq\mathscr{F}$ such that
\begin{align*}
    \lim_{n\rightarrow\infty} \mathcal{L}_\lambda(\Bf^{(n)}) = \inf_{\Bf\in\mathscr{F}} \mathcal{L}_\lambda(\Bf) = \inf_{\Bf\in\mathscr{F}}\left[ \mathbb{E}\dist(\boldsymbol{X}, \Bf) + \lambda\cdot \Vert \nabla{}^2 \boldsymbol{f} \Vert^2_{L^2(\M)}\right]<\infty.
\end{align*}
Since the sequence $\{\mathcal{L}_\lambda(\Bf^{(n)})\}_{n\in\mathbb{N}}$ converges to a (finite) infimum, and  $\lambda>0$, we have
\begin{align}\label{eq: unform boundedness in C(M)}
    \sup_{n\in\mathbb{N}} \left\{\Vert \nabla{}^2 \boldsymbol{f}^{(n)} \Vert^2_{L^2(\M)}\right\}=:C_1<\infty.
\end{align}
By the definition of the collection $\mathscr{F}$, we have
\begin{align}\label{eq: upper bound of the C norm of fn}
\sup_{n\in\mathbb{N}}\left(\max_{\boldsymbol{m}\in\mathfrak{M}}\Vert \Bf^{(n)}(\boldsymbol{m})\Vert\right) \le \Upsilon,
\end{align}
which implies that
\begin{align*}
    \sup_{n\in\mathbb{N}}\left(\int_{\mathfrak{M}}\Vert \Bf^{(n)}(\boldsymbol{m})\Vert^2\,d \vol_g(\boldsymbol{m})\right) \le \vol_g(\mathfrak{M})\cdot \Upsilon^2=:C_2
\end{align*}
Then, we have the following boundedness in $H^2(\M)$
\begin{align}\label{eq: uniform bound of the H2 norm of fn}
        \sup_{n\in\mathbb{N}}\Vert \boldsymbol{f}^{(n)}\Vert^2_{H^2(\mathfrak{M})} = \sup_{n\in\mathbb{N}}\left\{\Vert \Bf^{(n)}\Vert_{L^2(\mathfrak{M})}^2 + \Vert \nabla^2 \Bf^{(n)}\Vert_{L^2(\mathfrak{M})}^2\right\} \le C_1 + C_2 < +\infty.
\end{align}
Since $H^2(\mathfrak{M})$ is reflexive \citep[][Proposition 2.3]{hebey2000nonlinear}, Theorem 3.18 of \cite{brezis2011functional} implies that there exists a subsequence $\{\Bf^{(n,k)}\}_{k\in\mathbb{N}}$ of $\{\Bf^{(n)}\}_{n\in\mathbb{N}}$ that converges weakly, i.e., for every $j\in[D]$,
\begin{align}\label{eq: the weak convergence of f_{n,k}}
    f_j^{(n,k)} \overset{w}{\rightharpoonup} f^*_j\in H^2(\mathfrak{M})\ \ \text{ as }k\rightarrow\infty,
\end{align}
where $\overset{w}{\rightharpoonup}$ denotes the weak convergence in the Hilbert space $H^2(\mathfrak{M})$. Define 
\begin{align*}
    \Bf^*:=(f^*_1, f^*_2,\ldots,f^*_D)\in H^2(\mathfrak{M};\,\mathbb{R}^D).
\end{align*}
Theorem \ref{thm: Rellich–Kondrachov embedding}, together with Remark 2 of Chapter 6 of \cite{brezis2011functional}, implies that $\Bf^*\in C(\mathfrak{M};\,\mathbb{R}^D)$ and
\begin{align}\label{eq: convergence to f^*}
    \lim_{k\rightarrow\infty}\Vert \Bf^{(n,k)}-\Bf^*\Vert_{C(\mathfrak{M};\,\mathbb{R}^D)}=\lim_{k\rightarrow\infty}\max_{\boldsymbol{m}\in\mathfrak{M}}\Vert \Bf^{(n,k)}(\boldsymbol{m})-\Bf^*(\boldsymbol{m})\Vert=0.
\end{align}
Furthermore, combining \eqref{eq: upper bound of the C norm of fn} and \eqref{eq: convergence to f^*}, we have
\begin{align}\label{eq: upper bound of the C norm of f*}
 \max_{\boldsymbol{m}\in\mathfrak{M}}\Vert \Bf^*(\boldsymbol{m})\Vert = \lim_{k\rightarrow\infty}\max_{\boldsymbol{m}\in\mathfrak{M}}\Vert \Bf^{(n,k)}(\boldsymbol{m})\Vert \le \Upsilon.
\end{align}
Therefore, $\Bf^*\in\mathscr{F}$.

\noindent\underline{Part 2.} We now establish the convergence of the fitting-error term $\mathbb{E}\dist(\boldsymbol{X}, \fnk)$. Using Lemma \ref{lemma: continuity of Delta wrt f} and \eqref{eq: convergence to f^*}, we have 
\begin{align*}
    \lim_{k\rightarrow\infty} \mathbb{E} \left| \dist(\boldsymbol{X}, \fstar) - \dist(\boldsymbol{X}, \fnk) \right| & = \lim_{k\rightarrow\infty}\int_{\mathbb{R}^D} \left| \dist(\boldsymbol{x}, \fstar) - \dist(\boldsymbol{x}, \fnk) \right| \,\mathbb{P}(\mathrm{d}\boldsymbol{x}) \\
    & \le 4\Upsilon\cdot \lim_{k\rightarrow\infty} \max_{\boldsymbol{m}\in\mathfrak{M}} \norm{\Bf^*(\boldsymbol{m}) - \Bf^{(n,k)}(\boldsymbol{m})} =0,
\end{align*}
which implies that
\begin{align}\label{eq: data convergence}
    \lim_{k \to \infty}\mathbb{E} \dist(\boldsymbol{X}, \fnk) = \mathbb{E}\dist(\boldsymbol{X}, \fstar).
\end{align}

\noindent\underline{Part 3.} The strong convergence in \eqref{eq: convergence to f^*} implies that 
\begin{align*}
    \Vert \Bf^{(n,k)}-\Bf^*\Vert^2_{L^2(\mathfrak{M})} &= \int_{\mathfrak{M}} \Vert \Bf^{(n,k)}(\boldsymbol{m})-\Bf^*(\boldsymbol{m})\Vert^2 \,d \vol_g(\boldsymbol{m}) \\
    & \le \vol_g(\mathfrak{M})\cdot\max_{\boldsymbol{m}\in\mathfrak{M}}\Vert \Bf^{(n,k)}(\boldsymbol{m})-\Bf^*(\boldsymbol{m})\Vert ^2\rightarrow0,\ \ \text{ as }k\rightarrow\infty,
\end{align*}
which further implies that
\begin{align}\label{eq: convergence of L2 norm}
    \lim_{k\rightarrow\infty} \Vert \Bf^{(n,k)} \Vert_{L^2(\mathfrak{M})} = \Vert \Bf^* \Vert_{L^2(\mathfrak{M})}.
\end{align}
The weak convergence in \eqref{eq: the weak convergence of f_{n,k}}, together with the weak lower semicontinuity of the norm $\Vert \cdot \Vert_{H^2(\mathfrak{M})}$ \citep[][Proposition 3.5(iii)]{brezis2011functional}, implies that
\begin{align*}
    \Vert f_j^* \Vert_{L^2(\mathfrak{M})}^2 + \Vert \nabla^2 f_j^* \Vert_{L^2(\mathfrak{M})}^2  &=\Vert f_j^* \Vert_{H^2(\mathfrak{M})}^2 \\
    & \le \liminf_{k\rightarrow\infty} \Vert f_j^{(n,k)} \Vert_{H^2(\mathfrak{M})}^2 \\
    &= \liminf_{k\rightarrow\infty}\left(\Vert f_j^{(n,k)} \Vert_{L^2(\mathfrak{M})}^2 + \Vert \nabla^2 f_j^{(n,k)} \Vert_{L^2(\mathfrak{M})}^2\right) \\
    & = \Vert f_j^* \Vert_{L^2(\mathfrak{M})}^2 + \liminf_{k\rightarrow\infty} \Vert \nabla^2 f_j^{(n,k)} \Vert_{L^2(\mathfrak{M})}^2,
\end{align*}
for every $j\in[D]$, where the last equality follows from the componentwise $L^2$-convergence implied by \eqref{eq: convergence to f^*}. Therefore, 
\begin{align}\label{eq: the lower semicontinuity tool}
    \Vert \nabla^2 f_j^* \Vert_{L^2(\mathfrak{M})}^2 \le \liminf_{k\rightarrow\infty} \Vert \nabla^2 f_j^{(n,k)} \Vert_{L^2(\mathfrak{M})}^2,\ \ \text{ for every }j\in[D].
\end{align} 
Applying \eqref{eq: data convergence} and \eqref{eq: the lower semicontinuity tool} to the objective functional \eqref{eq: the core min problem} evaluated at $\fstar$,
\begin{align}\label{eq: lower semicontinuity argument}
    \begin{aligned}
        \mathcal{L}_\lambda (\fstar) &= \mathbb{E}\dist(\boldsymbol{X}, \fstar) + \lambda\cdot\sum_{j=1}^D \Lnorm{ \nabla ^2 f_j^*}^2\\
    &\leq \lim_{k \to \infty}\mathbb{E} \dist(\boldsymbol{X}, \fnk) + \lambda\cdot{ \sum_{j=1}^D \liminf_{k \to \infty}\Lnorm{ \nabla ^2 \fnkj}^2}\\
    &\le \liminf_{k\to \infty} \Big(  \mathbb{E} \dist(\boldsymbol{X}, \fnk) +{\lambda\cdot \sum_{j=1}^D \Lnorm{ \nabla ^2 \fnkj}^2}\Big) \\
    &= \liminf_{k \to \infty} \mathcal{L}_\lambda(\fnk). 
    \end{aligned}
\end{align} 
Since the sequence $\fn$ is defined such that $\lim_{n\to \infty}\mathcal{L}_\lambda(\fn) =\inf_{\Bf \in \mathscr{F}} \mathcal{L}_\lambda(\Bf)$, any subsequence will converge to the same limit, and thus,
 \begin{align*}
     \mathcal{L}_\lambda (\fstar) \le \liminf_{k \to \infty} \mathcal{L}_\lambda(\fnk) = \lim_{k \to \infty} \mathcal{L}_\lambda(\fnk) = \inf_{\Bf \in \mathscr{F}} \mathcal{L}_\lambda(\Bf) \le \mathcal{L}_\lambda (\fstar).
 \end{align*}
Therefore, the infimum $\inf_{\Bf \in \mathscr{F}} \mathcal{L}_\lambda(\Bf) = \mathcal{L}_\lambda (\fstar)$ is achieved by $\Bf^*_\lambda:=\Bf^*\in\mathscr{F}$.
\end{proof}

The following lemma is a standard result in many textbooks \citep[e.g.,][Theorem 2.3.3]{durrett2019probability}.
\begin{lemma}\label{lemma: elementary mathematical analysis lemma}
Let \(\{\boldsymbol{y}_n\}_{n\in\mathbb{N}}\) be a sequence of elements of a topological space. If every subsequence \(\{\boldsymbol{y}_{n,m}\}_{m\in\mathbb{N}}\) has a further subsequence \(\{\boldsymbol{y}_{n,m,k}\}_{k\in\mathbb{N}}\) 
that converges to \(\boldsymbol{y}\), then \(\boldsymbol{y}_n \to \boldsymbol{y}\).
\end{lemma}

The following lemma is a generalization of Lemma \ref{lemma: elementary mathematical analysis lemma}.
\begin{lemma}\label{lemma: a generalization of the elementary mathematical analysis lemma}
Let $\mathscr{X}$ be a topological space, and let $\mathscr{A}\subseteq \mathscr{X}$. Let $\{\boldsymbol{y}_n\}_{n\in\mathbb{N}}$ be a sequence in $\mathscr{X}$. Assume that every subsequence of $\{\boldsymbol{y}_n\}_{n\in\mathbb{N}}$ has a further subsequence converging to a point in $\mathscr{A}$. Then, $\{\boldsymbol{y}_n\}_{n\in\mathbb{N}}$ converges to $\mathscr{A}$ in the following sense: for every open set $\mathscr{V}\subseteq \mathscr{X}$ with $\mathscr{A}\subseteq \mathscr{V}$, there exists $N\in\mathbb{N}$ such that $\boldsymbol{y}_n\in \mathscr{V}$ for all $n\ge N$.
\end{lemma}
\begin{proof}[\textbf{Proof of Lemma \ref{lemma: a generalization of the elementary mathematical analysis lemma}.}]
We argue by contradiction. Suppose $\{\boldsymbol{y}_n\}_{n\in\mathbb{N}}$ does not converge to $\mathscr{A}$. Then, by definition, there exists an open set $\mathscr{V}\subseteq \mathscr{X}$ such that $\mathscr{A}\subseteq \mathscr{V}$, but $\{\boldsymbol{y}_n\}_{n\in\mathbb{N}}$ is not eventually contained in $\mathscr{V}$. That is, for every $N\in\mathbb{N}$, there exists $n\ge N$ such that $\boldsymbol{y}_n\notin \mathscr{V}$.

We now construct a subsequence $\{\boldsymbol{y}_{n_m}\}_{m\in\mathbb{N}}$ with $\boldsymbol{y}_{n_m}\notin \mathscr{V}$ for all $m\in\mathbb{N}$. Specifically, 
\begin{itemize}
    \item there exists $n_1\ge 1$ such that $\boldsymbol{y}_{n_1}\notin \mathscr{V}$;
    \item having chosen $\boldsymbol{y}_{n_m}$, there exists $n_{m+1}\ge \max\{m, n_{m}+1\}$ such that $\boldsymbol{y}_{n_{m+1}}\notin \mathscr{V}$.
\end{itemize}

By assumption, the subsequence $\{\boldsymbol{y}_{n_m}\}_{m\in\mathbb{N}}$ has a further subsequence $\{\boldsymbol{y}_{n_{m_k}}\}_{k\in\mathbb{N}}$ that converges to some point $\boldsymbol{y}\in \mathscr{A}$. Since $\mathscr{A}\subseteq \mathscr{V}$, we have $\boldsymbol{y}\in \mathscr{V}$. Because $\mathscr{V}$ is an open neighborhood of $\boldsymbol{y}$ and $\boldsymbol{y}_{n_{m_k}}\to \boldsymbol{y}$, there exists $K\in\mathbb{N}$ such that
\[
\boldsymbol{y}_{n_{m_k}}\in \mathscr{V} \qquad \text{for all } k\ge K.
\]
This contradicts the fact that $\boldsymbol{y}_{n_{m_k}}\notin \mathscr{V}$ for all $k\in\mathbb{N}$. The contradiction shows that $\{\boldsymbol{y}_n\}_{n\in\mathbb{N}}$ must converge to $\mathscr{A}$.
\end{proof}

\begin{lemma}\label{lemma: solutions to the Hessian equation}
Let $\mathfrak M$ be a compact, connected, $d$-dimensional Riemannian manifold, and let $\boldsymbol h\in H^2(\mathfrak M;\,\mathbb{R}^D)$. Assume $\|\nabla^2 \boldsymbol h\|_{L^2(\mathfrak M)}^2=0$. Then, we have:
\begin{enumerate}
    \item If $\partial\mathfrak M=\emptyset$, i.e., $\mathfrak M$ is closed, then $\boldsymbol h$ is constant on $\mathfrak M$.
    \item Suppose $\mathfrak M \subseteq \mathbb{R}^d$ is a compact, simply connected,
    $d$-dimensional subset with a smooth boundary, equipped with the Euclidean
    metric. Then, $\boldsymbol h$ has an affine representative, i.e.,
    there exist $\boldsymbol a\in\mathbb R^D$ and
    $\boldsymbol B\in\mathbb R^{D\times d}$ such that
    \begin{align}\label{eq: affine functions}
        \boldsymbol h(\boldsymbol{x})=\boldsymbol a+\boldsymbol B\boldsymbol{x},
        \qquad \boldsymbol{x}\in\mathfrak M.
    \end{align}
\end{enumerate}
\end{lemma}
\begin{proof}[\textbf{Proof of Lemma \ref{lemma: solutions to the Hessian equation}}]
Write $\boldsymbol h=(h_1,\dots,h_D)$, where each component $h_\ell\in H^2(\mathfrak M;\mathbb{R}^1)$. Because $\|\nabla^2 \boldsymbol h\|_{L^2(\mathfrak M)}^2= \sum_{\ell=1}^D \int_{\M} |\nabla^2 h_\ell (\boldsymbol{m}) |^2 \, d\,\mathrm{vol}_g(\boldsymbol{m})=0$, we have $\int_{\M} |\nabla^2 h_\ell (\boldsymbol{m}) |^2 \, d\,\mathrm{vol}_g(\boldsymbol{m})=0$ for all $\ell\in[D]$. Thus, it suffices to prove the scalar statement for one component $h\in H^2(\mathfrak M)=H^2(\mathfrak M;\mathbb{R}^1)$ with $\int_{\M} |\nabla^2 h (\boldsymbol{m}) |^2 \, d\,\mathrm{vol}_g(\boldsymbol{m})=0$. 

Obviously, $\nabla^2 h=0$ almost everywhere (a.e.) on $\M$.

\medskip
\noindent
\underline{Part 1: Closed manifold case (\(\partial\mathfrak M=\emptyset\)).} Assume $\mathfrak M$ is closed.

The Laplace--Beltrami operator $\Delta$ is the trace of the Hessian, i.e., $\Delta h = \operatorname{tr}_g(\nabla^2 h)$, defined in \eqref{eq: def of Laplace-Beltrami operator}. Since $\nabla^2 h=0$ a.e., it follows that $\Delta h=0$ a.e. on $\M$. By standard interior elliptic regularity for second-order elliptic equations \citep[][Section 6.3.1, Theorem 3, ``Infinite differentiability in the interior'']{evans1998pde}, weakly harmonic functions are locally smooth (e.g., in each sufficiently small coordinate chart), which implies that $h\in C^\infty(\M)$. Then, Green's identity holds on the closed manifold $\M$, i.e., 
\begin{align*}
    \int_{\mathfrak M} |\nabla h|^2\,d \vol_g = -\int_{\mathfrak M} h\,\Delta h\,d \vol_g.
\end{align*}
Since $\Delta h=0$, we have that $\int_{\mathfrak M} |\nabla h|^2\,d \vol_g=0$, which implies that $\nabla h=0$ on $\M$. Hence, $h$ is constant on $\mathfrak M$. Applying this to each component $h_\ell$, we conclude that $\boldsymbol h$ is a constant vector-valued function on $\mathfrak M$.

\noindent\underline{Part 2: Euclidean domain with smooth boundary.}
Suppose $\mathfrak M\subseteq\mathbb R^d$ satisfies the assumption in
part (ii), and put $\Omega:=\operatorname{int}(\mathfrak M)$.
The domain assumptions imply that $\Omega$ is connected and
$\overline\Omega=\mathfrak M$. In the ambient Euclidean coordinates,
the Levi-Civita connection is the ordinary derivative. Hence, for every
component $h_\ell$, the zero covariant Hessian condition gives
$\partial_i\partial_j h_\ell=0$ weakly on $\Omega$ for all $i,j\in[d]$.
Since $\Omega$ is connected, each weak first derivative
$\partial_i h_\ell$ is a constant $b_i^\ell\in\mathbb R$.
It follows that the weak gradient of
$h_\ell-\sum_{i=1}^d b_i^\ell x_i$ vanishes, so this difference is
a constant $a^\ell\in\mathbb R$ on $\Omega$.
Defining $\boldsymbol a=(a^1,\ldots,a^D)^\top$ and
$\boldsymbol B=(b_i^\ell)_{\ell,i}\in\mathbb R^{D\times d}$, we obtain
$\boldsymbol h(\boldsymbol x)=\boldsymbol a+\boldsymbol B\boldsymbol x$
almost everywhere on $\Omega$.
Since the smooth boundary $\partial\mathfrak M$ has zero $d$-dimensional
volume, this is equality as an $H^2(\mathfrak M;\mathbb R^D)$ equivalence
class. The affine map supplies its continuous representative on all of
$\mathfrak M$, proving \eqref{eq: affine functions}.
\end{proof}

\begin{proof}[\textbf{Proof of Theorem \ref{thm: PME with lambda=infty}}]  Suppose $\{\lambda_n\}_{n\in\mathbb{N}}$ is an arbitrary sequence of positive real numbers such that $\lim_{n\rightarrow\infty}\lambda_n=\infty$.\\
\underline{Part A: Preparations.} We proceed by applying Lemmas \ref{lemma: elementary mathematical analysis lemma} and \ref{lemma: a generalization of the elementary mathematical analysis lemma}. To facilitate this, we provide several necessary preparations first. Let $\{\boldsymbol{f}_{\lambda_{n,k}}^*\}_{k\in\mathbb{N}}$ be an arbitrary subsequence of $\{\boldsymbol{f}_{\lambda_{n}}^*\}_{n\in\mathbb{N}}$.

\noindent\underline{Step 1: Boundedness of the subsequence $\{\boldsymbol{f}_{\lambda_{n,k}}^*\}_{k\in\mathbb{N}}$ in $H^2(\mathfrak{M})$.}\\
Fix any $\boldsymbol{g}\in \mathscr{F}$. Since $\boldsymbol{f}_{\lambda_{n,k}}^*$ minimizes $\mathcal{L}_{\lambda_{n,k}}(\Bf)$ over $\mathscr{F}$, we have
\begin{align*}
    \lambda_{n,k} \cdot \Vert \nabla^2 \boldsymbol{f}_{\lambda_{n,k}}^* \Vert^2_{L^2(\mathfrak{M})} \le \mathcal{L}_{\lambda_{n,k}}(\boldsymbol{f}_{\lambda_{n,k}}^*) \le \mathcal{L}_{\lambda_{n,k}}(\boldsymbol{g}) =\mathbb{E}\,\dist(\boldsymbol{X},\boldsymbol{g})+\lambda_{n,k} \cdot \Vert \nabla^2 \boldsymbol{g} \Vert^2_{L^2(\mathfrak{M})}.
\end{align*}
Dividing each side by $\lambda_{n,k}$ and taking $k \to \infty$, we obtain
\begin{align*}
    \limsup_{k\rightarrow\infty} \Vert \nabla^2 \boldsymbol{f}_{\lambda_{n,k}}^* \Vert^2_{L^2(\mathfrak{M})} &\le \lim_{k\rightarrow\infty} \frac{1}{\lambda_{n,k}} \left\{\mathbb{E}\,\dist(\boldsymbol{X},\boldsymbol{g})+\lambda_{n,k} \cdot \Vert \nabla^2 \boldsymbol{g} \Vert^2_{L^2(\mathfrak{M})}\right\} \\
    & = \lim_{k \to \infty} \frac{1}{\lambda_{n,k}} \mathbb{E}\,\dist(\boldsymbol{X},\boldsymbol{g}) +  \Vert \nabla^2 \boldsymbol{g} \Vert^2_{L^2(\mathfrak{M})} \\
    &= \Vert \nabla^2 \boldsymbol{g} \Vert^2_{L^2(\mathfrak{M})} < +\infty.
\end{align*}
Therefore, $\sup_{k\in\mathbb{N}} \Vert \nabla^2 \boldsymbol{f}_{\lambda_{n,k}}^* \Vert^2_{L^2(\mathfrak{M})} <+\infty$. In addition, the definition of $\mathscr{F}$ implies that $\sup_{k\in\mathbb{N}}\Vert \boldsymbol{f}_{\lambda_{n,k}}^* \Vert_{L^2(\mathfrak{M})}<+\infty$. Hence, the sequence $\{\boldsymbol{f}_{\lambda_{n,k}}^*\}_{k\in\mathbb{N}}$ is bounded in $H^2(\mathfrak{M})$.

\noindent\underline{Step 2: Weak convergence in $H^2(\mathfrak{M})$ and strong convergence in $C(\mathfrak{M})$.}\\
Theorem 3.18 of \cite{brezis2011functional}, together with the result from Step 1, implies that there exists a further subsequence $\{\boldsymbol{f}_{\lambda_{n,k,l}}^*\}_{l\in\mathbb{N}}$ and $\tilde{\boldsymbol{f}}\in H^2(\mathfrak{M})$ such that
\begin{align}\label{eq: pca weak convergence, proof of the continuity in lambda}
    \boldsymbol{f}_{\lambda_{n,k,l}}^* \overset{w}{\rightharpoonup} \tilde{\boldsymbol{f}} \text{ in } H^2(\mathfrak{M}) \quad \text{ as }l\rightarrow\infty.
\end{align}
Theorem \ref{thm: Rellich–Kondrachov embedding}, together with Remark 2 of Chapter 6 of \cite{brezis2011functional}, implies that $\tilde{\boldsymbol{f}} \in C(\mathfrak{M})$ and
\begin{align}\label{eq: pca uniform limit for the continuity wrt lambda}
\lim_{l\rightarrow\infty}\max_{\boldsymbol{m}\in\mathfrak{M}} \Vert \boldsymbol{f}_{\lambda_{n,k,l}}^*(\boldsymbol{m}) - \tilde{\boldsymbol{f}}(\boldsymbol{m}) \Vert = 0.
\end{align}
Moreover, since each $\boldsymbol{f}_{\lambda_{n,k,l}}^*\in \mathscr{F}$ satisfies the uniform bound $\max_{\boldsymbol{m}\in\mathfrak{M}}\Vert \boldsymbol{f}_{\lambda_{n,k,l}}^*(\boldsymbol{m}) \Vert \le \Upsilon$ defining $\mathscr{F}$, the uniform limit $\tilde{\boldsymbol{f}}$ in \eqref{eq: pca uniform limit for the continuity wrt lambda} satisfies the same bound. Thus, $\tilde{\boldsymbol{f}}\in \mathscr{F}$.

\noindent\underline{Step 3: Convergence of the fitting-error term.}\\
Lemma \ref{lemma: continuity of Delta wrt f}, together with \eqref{eq: pca uniform limit for the continuity wrt lambda}, implies that
\begin{equation}\label{eq: pca datafit_conv}
\lim_{l\rightarrow\infty}\mathbb{E}\,\dist(\boldsymbol{X}, \boldsymbol{f}_{\lambda_{n,k,l}}^*) = \mathbb{E}\,\dist(\boldsymbol{X},\tilde{\boldsymbol{f}}).
\end{equation}

\noindent\underline{Step 4: Lower semicontinuity of the limit function's penalty}\\
\eqref{eq: pca uniform limit for the continuity wrt lambda} implies $\lim_{l\rightarrow\infty}\Vert \boldsymbol{f}_{\lambda_{n,k,l}}^* - \tilde{\boldsymbol{f}} \Vert_{L^2(\mathfrak{M})}=0$. In addition, the weak convergence in \eqref{eq: pca weak convergence, proof of the continuity in lambda}, together with the weak lower semicontinuity of the norm $\Vert \cdot \Vert_{H^2(\mathfrak{M})}$ \citep[][Proposition 3.5(iii)]{brezis2011functional}, implies that
\begin{align*}
    \Vert \tilde{\boldsymbol{f}} \Vert_{L^2(\mathfrak{M})}^2 + \|\nabla^2 \tilde{\boldsymbol{f}}\|_{L^2(\M)}^2 =\Vert \tilde{\boldsymbol{f}} \Vert_{H^2(\mathfrak{M})}^2 &\le \liminf_{l\rightarrow\infty} \Vert \boldsymbol{f}_{\lambda_{n,k,l}}^* \Vert_{H^2(\mathfrak{M})}^2 \\
    &= \Vert \tilde{\boldsymbol{f}} \Vert_{L^2(\mathfrak{M})}^2 + \liminf_{l\rightarrow\infty} \|\nabla^2 \boldsymbol{f}_{\lambda_{n,k,l}}^*\|_{L^2(\M)}^2.
\end{align*}
Therefore,
\begin{equation}\label{eq: pca lsc_penalty}
\|\nabla^2 \tilde{\boldsymbol{f}}\|_{L^2(\M)}^2 \le \liminf_{l\rightarrow\infty} \|\nabla^2 \boldsymbol{f}_{\lambda_{n,k,l}}^*\|_{L^2(\M)}^2.
\end{equation}

\noindent\underline{Step 5: Controlling the upper bound of the penalty of the sequence}\\Fix an arbitrary $\boldsymbol{h}\in \mathscr{F}$. Optimality of $\boldsymbol{f}_{\lambda_{n,k,l}}^*$ for $\mathcal{L}_{\lambda_{n,k,l}}(\Bf)$ gives
\[
\mathbb{E}\,\dist(\boldsymbol{X}, \boldsymbol{f}_{\lambda_{n,k,l}}^*)+\lambda_{n,k,l} \cdot \|\nabla^2 \boldsymbol{f}_{\lambda_{n,k,l}}^*\|^2_{L^2(\M)} \le \mathbb{E}\,\dist(\boldsymbol{X}, \boldsymbol{h})+\lambda_{n,k,l} \cdot \|\nabla^2 \boldsymbol{h}\|^2_{L^2(\M)}.
\]
Divide by $\lambda_{n,k,l}$ and take $\limsup_{l\to\infty}$ on the left-hand and right-hand sides. Then,
\begin{align*}
\quad &\limsup _{l \to \infty}\left\{ \frac{1}{\lambda_{n,k,l} } \mathbb{E}\,\dist(\boldsymbol{X}, \boldsymbol{f}_{\lambda_{n,k,l}}^*) + \|\nabla^2 \boldsymbol{f}_{\lambda_{n,k,l}}^*\|^2_{L^2(\M)} \right\}\\ \le &\limsup_{l \to \infty}\left\{ \frac{1}{\lambda_{n,k,l}}\mathbb{E}\,\dist(\boldsymbol{X}, \boldsymbol{h})+ \|\nabla^2 \boldsymbol{h}\|^2_{L^2(\M)} \right\}.
\end{align*}
Since $\lambda_{n,k,l} \rightarrow\infty$ as $l\rightarrow\infty$, we have that
\[
\limsup _{l \to \infty} \|\nabla^2 \boldsymbol{f}_{\lambda_{n,k,l}}^*\|^2_{L^2(\M)} \le \|\nabla^2 \boldsymbol{h}\|^2_{L^2(\M)}.
\]
Since this must hold for all $\boldsymbol{h} \in \mF,$ this includes $\boldsymbol{h} = 0$. Thus, we have that
\begin{equation}\label{eq: pca usc_penalty}
\limsup _{l \to \infty} \|\nabla^2 \boldsymbol{f}_{\lambda_{n,k,l}}^*\|^2_{L^2(\M)} =  0.
\end{equation}

\noindent \underline{Step 6: Deducing the limit point.}\\
By \eqref{eq: pca lsc_penalty}, \eqref{eq: pca usc_penalty},  $\limsup$ bounding $\liminf$ from above, and non-negativity of the norm, it follows that $\|\nabla^2 \tilde{\boldsymbol{f}}\|^2_{L^2(\M)} = 0$. 

\noindent\underline{Step 7: Subsequential limit as a minimizer of the limiting problem.}\\ Let $\Bh \in \mF$ satisfy $\Vert \nabla^2 \Bh \Vert^2_{L^2(\M)} = 0$. By the optimality of $\fstar_{\lambda_{n,k,l}}$ for the objective $\mcL_{\lambda_{n,k,l}}(\Bf)$, we have
$$\mbE \dist(\boldsymbol{X},\fstar_{\lambda_{n,k,l}}) \leq\mbE \dist(\boldsymbol{X},\fstar_{\lambda_{n,k,l}}) + \lambda_{n,k,l} \cdot \|\nabla^2 \fstar_{\lambda_{n,k,l}}\|^2_{L^2(\M)} \leq \mbE \dist(\boldsymbol{X},\Bh).$$
Taking limits on the left- and right-hand sides and using \eqref{eq: pca datafit_conv}, we have
$$\mbE \dist(\boldsymbol{X},\tilde{\Bf}) = \lim_{l \to \infty}\mbE \dist(\boldsymbol{X},\fstar_{\lambda_{n,k,l}}) \leq\mbE \dist (\boldsymbol{X},\Bh).$$ 
The expression above holds for all $\Bh$ such that $\Vert \nabla^2 \Bh \Vert^2_{L^2(\M)} = 0$. Since $\|\nabla^2 \tilde{\boldsymbol{f}}\|^2_{L^2(\M)} = 0$, $\tilde{\Bf}$ must be a minimizer. That is,
\begin{equation}\label{eq: limiting problem lambda infinity infty}
    \tilde{\Bf} \in \argmin_{\Bh \in \mF :\; \Vert \nabla^2 \Bh \Vert^2_{L^2(\M)} = 0} \mbE \dist(\boldsymbol{X},\Bh).
\end{equation}

\noindent\underline{Part B: special cases.} We now aim to prove the two cases in Theorem \ref{thm: PME with lambda=infty}, pertaining to two special cases of manifolds. From Part A, we know the minimization is reduced to the form given by \eqref{eq: limiting problem lambda infinity infty}. 

\noindent\underline{Case 1: Manifold without boundary.} Assume $\mathfrak M$ is closed, i.e., $\partial\M=\emptyset$. Due to Lemma \ref{lemma: solutions to the Hessian equation}, the minimization in \eqref{eq: limiting problem lambda infinity infty} is equivalent to 
\begin{align*}
    \argmin_{\Bh\in\mathscr{F}:\,\Bh = \boldsymbol{c} \text{ is a constant function}} \; \mathbb{E}\Vert \boldsymbol{X} - \boldsymbol{c} \Vert^2,
\end{align*}
which has a unique minimizer, $\boldsymbol{h}=\mathbb{E}\boldsymbol{X}$. Then, \eqref{eq: limiting problem lambda infinity infty} implies that $\tilde{\Bf}(\boldsymbol{m})=\mathbb{E}\boldsymbol{X}$ for all $\boldsymbol{m}\in\M$. Therefore, the argument in Part A shows that every subsequence of $\{\Bf_{\lambda_n}^*\}_{n\in\mathbb{N}}$ admits a further subsequence converging to the constant function $\boldsymbol{m}\mapsto \mathbb{E}\boldsymbol{X}$ in the supremum-norm topology. The first assertion of Theorem \ref{thm: PME with lambda=infty} then follows from Lemma \ref{lemma: elementary mathematical analysis lemma}.

\noindent\underline{Case 2: Euclidean Template with Smooth Boundary.}\par
Assume the hypotheses in the second assertion of the theorem. Write
\[
\boldsymbol\mu:=\mathbb E\boldsymbol X, \qquad
\boldsymbol\Sigma:=\mathrm{Cov}(\boldsymbol X), \qquad
V_d:=\operatorname{span}\{\boldsymbol v_1,\ldots,\boldsymbol v_d\},
\]
so that $\mathcal P_d=\boldsymbol\mu+V_d$, and let
$\eta_1\ge\cdots\ge\eta_D\ge0$ be the eigenvalues of $\boldsymbol\Sigma$.

Because $\M\subseteq\mathbb R^d$ is equipped with the Euclidean metric, the covariant Hessian is the ordinary weak Hessian. If $\Vert\nabla^2\Bh\Vert_{L^2(\M)}=0$, then, on the connected domain $\operatorname{int}(\M)$, every weak first derivative of every component of $\Bh$ is constant. Hence there exist $\boldsymbol a\in\mathbb R^D$ and $\boldsymbol B\in\mathbb R^{D\times d}$ such that $\Bh(\boldsymbol m)=\boldsymbol a+\boldsymbol B\boldsymbol m$ almost everywhere on $\M$. Both sides are continuous by Theorem~\ref{thm: Rellich–Kondrachov embedding}, so the equality holds everywhere on $\M$. The converse is immediate. Therefore, the set of minimizers in \eqref{eq: limiting problem lambda infinity infty} is
\begin{align}\label{eq: PCA from the viewpoint of minimizing the distance functional}
\mathscr A
&:=\argmin_{\substack{\Bh\in\mathscr F:\;\Vert\nabla^2\Bh\Vert_{L^2(\M)}=0}}
\mathbb E\,\dist(\boldsymbol X,\Bh) \notag\\
&=\argmin_{\substack{\Bh\in\mathscr F:\;\Bh(\boldsymbol m)=\boldsymbol a+\boldsymbol B\boldsymbol m
\text{ for some }\boldsymbol a,\boldsymbol B}}
\mathbb E\,\dist(\boldsymbol X,\Bh).
\end{align}
This defines the set $\mathscr A$ appearing in the second assertion of the theorem.

\smallskip
\noindent\textit{PCA lower bound.} Fix any zero-Hessian map $\Bh\in\mathscr F$ and let $\mathcal S_{\Bh}:=\operatorname{aff}(\Bh(\M))$. Since $\Bh$ is affine, $\dim(\mathcal S_{\Bh})\le d$. For any affine subspace $\mathcal S=\boldsymbol c+W$ with $q:=\dim(W)\le d$ and orthogonal projection $P_W$ onto $W$, we have
\begin{align*}
\mathbb E\!\left[\operatorname{dist}(\boldsymbol X,\mathcal S)^2\right]
&=\operatorname{tr}\!\left\{(\boldsymbol I_D-P_W)\boldsymbol\Sigma\right\}
+\operatorname{dist}(\boldsymbol\mu,\mathcal S)^2 \\
&\ge \operatorname{tr}(\boldsymbol\Sigma)-\sum_{j=1}^q\eta_j
\ge \operatorname{tr}(\boldsymbol\Sigma)-\sum_{j=1}^d\eta_j
=\mathbb E\!\left[\operatorname{dist}(\boldsymbol X,\mathcal P_d)^2\right],
\end{align*}
where the first inequality is the usual variational characterization of the leading eigenspaces \citep[][Chapter 2]{jolliffe2002principal}. Since $\Bh(\M)\subseteq\mathcal S_{\Bh}$, it follows that
\begin{align}\label{eq: corrected PCA lower bound}
\mathbb E\,\dist(\boldsymbol X,\Bh)
&=\mathbb E\!\left[\operatorname{dist}\bigl(\boldsymbol X,\Bh(\M)\bigr)^2\right] \notag\\
&\ge \mathbb E\!\left[\operatorname{dist}\bigl(\boldsymbol X,\mathcal S_{\Bh}\bigr)^2\right]
\ge \mathbb E\!\left[\operatorname{dist}\bigl(\boldsymbol X,\mathcal P_d\bigr)^2\right].
\end{align}

\smallskip
\noindent\textit{Attainment of the lower bound.} Since $\M$ is a full-dimensional subset of $\mathbb R^d$ with smooth boundary, it has nonempty interior. Choose $\boldsymbol z_0\in\operatorname{int}(\M)$ and $r>0$ such that $\overline B(\boldsymbol z_0,r)\subseteq\M$, and choose a linear isometry $U:\mathbb R^d\to V_d$. Compactness of $\operatorname{supp}(\mathbb P)$ implies that
\[
R_P:=\max_{\boldsymbol x\in\operatorname{supp}(\mathbb P)}
\left\Vert\Pi_{\mathcal P_d}(\boldsymbol x)-\boldsymbol\mu\right\Vert
\]
is finite. Define the affine map
\[
\boldsymbol h_{\mathrm{PCA}}(\boldsymbol m)
:=\boldsymbol\mu+\frac{R_P}{r}\,U(\boldsymbol m-\boldsymbol z_0).
\]
Then $\Vert\nabla^2\boldsymbol h_{\mathrm{PCA}}\Vert_{L^2(\M)}=0$ and $\boldsymbol h_{\mathrm{PCA}}(\M)\subseteq\mathcal P_d$. Moreover, because $U$ maps the radius-$r$ ball in $\mathbb R^d$ onto the radius-$r$ ball in $V_d$, the choice of $R_P$ gives
\begin{equation}\label{eq: h PCA covers PCA projections}
\boldsymbol h_{\mathrm{PCA}}(\M)
\supseteq \boldsymbol\mu+\overline B_{V_d}(\boldsymbol 0,R_P)
\supseteq \Pi_{\mathcal P_d}\bigl(\operatorname{supp}(\mathbb P)\bigr),
\end{equation}
where $\overline B_{V_d}(\boldsymbol 0,R_P)$ denotes the closed radius-$R_P$ ball in $V_d$.

Let
\[
\Upsilon_0:=\max_{\boldsymbol m\in\M}
\left\Vert\boldsymbol h_{\mathrm{PCA}}(\boldsymbol m)\right\Vert<\infty.
\]
In the second assertion of the theorem, it is sufficient that
\begin{equation}\label{eq: sufficiently large Upsilon}
\Upsilon\ge
\max\left\{
\sup_{\boldsymbol x\in\operatorname{supp}(\mathbb P)}\Vert\boldsymbol x\Vert,
\Upsilon_0
\right\}.
\end{equation}
Under \eqref{eq: sufficiently large Upsilon}, $\boldsymbol h_{\mathrm{PCA}}\in\mathscr F$. For every $\boldsymbol x\in\operatorname{supp}(\mathbb P)$, the two image inclusions above imply
\[
\operatorname{dist}\bigl(\boldsymbol x,\boldsymbol h_{\mathrm{PCA}}(\M)\bigr)
=\operatorname{dist}(\boldsymbol x,\mathcal P_d).
\]
Thus, $\boldsymbol h_{\mathrm{PCA}}$ attains the lower bound in \eqref{eq: corrected PCA lower bound}. In particular, $\mathscr A$ is nonempty and
\[
\min_{\substack{\Bh\in\mathscr F:\;\Vert\nabla^2\Bh\Vert_{L^2(\M)}=0}}
\mathbb E\,\dist(\boldsymbol X,\Bh)
=\mathbb E\!\left[\operatorname{dist}\bigl(\boldsymbol X,\mathcal P_d\bigr)^2\right].
\]

\smallskip
\noindent\textit{Characterization of the limiting set.} Let $\boldsymbol f_\infty^*\in\mathscr A$ and set $\mathcal S_{\boldsymbol f_\infty^*}:=\operatorname{aff}(\boldsymbol f_\infty^*(\M))$. Equality holds throughout \eqref{eq: corrected PCA lower bound}. To identify the affine hull, write $\mathcal S_{\boldsymbol f_\infty^*}=\boldsymbol c+W$, with $q=\dim(W)\le d$. The eigengap $\eta_d>\eta_{d+1}\ge0$ implies $\eta_d>0$. Consequently, equality with the $d$-dimensional PCA risk in the preceding variational bound forces $q=d$; equality in the leading-eigenspace variational characterization forces $W=V_d$; and equality in the nonnegative translation term forces $\boldsymbol\mu\in\boldsymbol c+W$. Hence
\[
\mathcal S_{\boldsymbol f_\infty^*}=\boldsymbol\mu+V_d=\mathcal P_d,
\qquad
\boldsymbol f_\infty^*(\M)\subseteq\mathcal P_d.
\]

\smallskip
\noindent\textit{Convergence to $\mathscr A$.} Part A shows that, for every sequence $\lambda_n\to\infty$, each subsequence of $\{\boldsymbol f_{\lambda_n}^*\}_{n\in\mathbb N}$ has a further subsequence converging in the supremum norm to an element of $\mathscr A$. Lemma~\ref{lemma: a generalization of the elementary mathematical analysis lemma} therefore implies that every such sequence converges to $\mathscr A$ in the sense of open neighborhoods.

Now let $\mathscr V\subseteq\mathscr F$ be open in the supremum-norm topology and satisfy $\mathscr A\subseteq\mathscr V$. If there were no $\Lambda>0$ such that $\boldsymbol f_\lambda^*\in\mathscr V$ for every $\lambda>\Lambda$, then for each $n$ one could choose $\lambda_n>n$ with $\boldsymbol f_{\lambda_n}^*\notin\mathscr V$. This contradicts the preceding sequential conclusion. Therefore, such a $\Lambda$ exists, which proves the second assertion of Theorem~\ref{thm: PME with lambda=infty}.
\end{proof}

\begin{proof}[\textbf{Proof of Theorem \ref{thm: continuity of the minimizer wrt the tuning parameter}}]
It suffices to show that, for any sequence $\{\lambda_n\}_{n\in\mathbb{N}}\subseteq [\underline{\lambda},\overline{\lambda}]$ with $\lambda_n\to\lambda$, arbitrary minimizer representatives satisfy
\begin{align}\label{eq: continuity of the minimizer wrt the tuning parameter}
    d_{\mathrm{iso}}([\boldsymbol f_{\lambda_n}^*],[\boldsymbol f_\lambda^*])\rightarrow0.
\end{align}
By Lemma \ref{lemma: elementary mathematical analysis lemma}, it is enough to show that every subsequence of these quotient distances has a further subsequence converging to zero.

\noindent\underline{Step 1: Boundedness of the subsequence $\{\boldsymbol{f}_{\lambda_{n,k}}^*\}_{k\in\mathbb{N}}$ in $H^2(\mathfrak{M})$.}\\
Fix any $\boldsymbol{g}\in \mathscr{F}$. Since $\boldsymbol{f}_{\lambda_{n,k}}^*$ minimizes $\mathcal{L}_{\lambda_{n,k}}(\Bf)$ over $\mathscr{F}$, we have
\begin{align*}
    \lambda_{n,k} \cdot \Vert \nabla^2 \boldsymbol{f}_{\lambda_{n,k}}^* \Vert^2_{L^2(\mathfrak{M})} \le \mathcal{L}_{\lambda_{n,k}}(\boldsymbol{f}_{\lambda_{n,k}}^*) &\le \mathcal{L}_{\lambda_{n,k}}(\boldsymbol{g}) =\mathbb{E}\,\dist(\boldsymbol{X},\boldsymbol{g})+\lambda_{n,k} \cdot \left\Vert \nabla^2 \boldsymbol{g}\right\Vert^2_{L^2(\mathfrak{M})},
\end{align*}
which implies that
\begin{align*}
    \limsup_{k\rightarrow\infty}\Vert \nabla^2 \boldsymbol{f}_{\lambda_{n,k}}^* \Vert^2_{L^2(\mathfrak{M})} &\le \lim_{k\rightarrow\infty} \frac{1}{\lambda_{n,k}} \left\{\mathbb{E}\,\dist(\boldsymbol{X},\boldsymbol{g})+\lambda_{n,k} \cdot \left\Vert \nabla^2 \boldsymbol{g}\right\Vert^2_{L^2(\mathfrak{M})}\right\} = \frac{\mathcal{L}_\lambda(\boldsymbol{g})}{\lambda} < +\infty.
\end{align*}
Therefore, $\sup_{k\in\mathbb{N}} \Vert \nabla^2 \boldsymbol{f}_{\lambda_{n,k}}^* \Vert^2_{L^2(\mathfrak{M})} <\infty$. In addition, the definition of $\mathscr{F}$ implies that $\sup_{k\in\mathbb{N}}\Vert \boldsymbol{f}_{\lambda_{n,k}}^* \Vert_{L^2(\mathfrak{M})}<+\infty$. Hence, the sequence $\{\boldsymbol{f}_{\lambda_{n,k}}^*\}_{k\in\mathbb{N}}$ is bounded in $H^2(\mathfrak{M})$.

\noindent\underline{Step 2: Weak convergence in $H^2(\mathfrak{M})$ and strong convergence in $C(\mathfrak{M})$.}\\
Theorem 3.18 of \cite{brezis2011functional}, together with the result from Step 1, implies that there exists a further subsequence $\{\boldsymbol{f}_{\lambda_{n,k,l}}^*\}_{l\in\mathbb{N}}$ and $\tilde{\boldsymbol{f}}\in H^2(\mathfrak{M})$ such that
\begin{align}\label{eq: weak convergence, proof of the continuity in lambda}
    \boldsymbol{f}_{\lambda_{n,k,l}}^* \overset{w}{\rightharpoonup} \tilde{\boldsymbol{f}} \text{ in } H^2(\mathfrak{M}) \quad \text{ as }l\rightarrow\infty.
\end{align}
Theorem \ref{thm: Rellich–Kondrachov embedding}, together with Remark 2 of Chapter 6 of \cite{brezis2011functional}, implies that $\tilde{\boldsymbol{f}} \in C(\mathfrak{M})$ and
\begin{align}\label{eq: uniform limit for the continuity wrt lambda}
\lim_{l\rightarrow\infty}\max_{\boldsymbol{m}\in\mathfrak{M}} \Vert \boldsymbol{f}_{\lambda_{n,k,l}}^*(\boldsymbol{m}) - \tilde{\boldsymbol{f}}(\boldsymbol{m}) \Vert = 0.
\end{align}
Moreover, since each $\boldsymbol{f}_{\lambda_{n,k,l}}^*\in \mathscr{F}$ satisfies the uniform bound $\max_{\boldsymbol{m}\in\mathfrak{M}}\Vert \boldsymbol{f}_{\lambda_{n,k,l}}^*(\boldsymbol{m}) \Vert \le \Upsilon$ defining $\mathscr{F}$, the uniform limit $\tilde{\boldsymbol{f}}$ in \eqref{eq: uniform limit for the continuity wrt lambda} satisfies the same bound. Thus, $\tilde{\boldsymbol{f}}\in \mathscr{F}$.

\noindent\underline{Step 3: Convergence of the fitting-error term.}\\
Lemma \ref{lemma: continuity of Delta wrt f}, together with \eqref{eq: uniform limit for the continuity wrt lambda}, implies that
\begin{equation}\label{eq:datafit_conv}
\lim_{l\rightarrow\infty}\mathbb{E}\,\dist(\boldsymbol{X}, \boldsymbol{f}_{\lambda_{n,k,l}}^*) = \mathbb{E}\,\dist(\boldsymbol{X},\tilde{\boldsymbol{f}}).
\end{equation}

\noindent\underline{Step 4: Lower semicontinuity of the penalty.}\\
\eqref{eq: uniform limit for the continuity wrt lambda} implies $\lim_{l\rightarrow\infty}\Vert \boldsymbol{f}_{\lambda_{n,k,l}}^* - \tilde{\boldsymbol{f}} \Vert_{L^2(\mathfrak{M})}=0$. In addition, the weak convergence in \eqref{eq: weak convergence, proof of the continuity in lambda}, together with the weak lower semicontinuity of the norm $\Vert \cdot \Vert_{H^2(\mathfrak{M})}$ \citep[][Proposition 3.5(iii)]{brezis2011functional}, implies that
\begin{align*}
    \Vert \tilde{\boldsymbol{f}} \Vert_{L^2(\mathfrak{M})}^2 +  \|\nabla^2 \tilde{\boldsymbol{f}}\|_{L^2(\M)}^2 &=\Vert \tilde{\boldsymbol{f}} \Vert_{H^2(\mathfrak{M})}^2\\ 
    &\le \liminf_{l\rightarrow\infty} \Vert \boldsymbol{f}_{\lambda_{n,k,l}}^* \Vert_{H^2(\mathfrak{M})}^2 \\
    &= \Vert \tilde{\boldsymbol{f}} \Vert_{L^2(\mathfrak{M})}^2 + \liminf_{l\rightarrow\infty} \|\nabla^2 \boldsymbol{f}_{\lambda_{n,k,l}}^*\|_{L^2(\M)}^2.
\end{align*}
Therefore, we have that
\begin{equation}\label{eq:lsc_penalty}
\|\nabla^2 \tilde{\boldsymbol{f}}\|_{L^2(\M)}^2 \le \liminf_{l\rightarrow\infty} \|\nabla^2 \boldsymbol{f}_{\lambda_{n,k,l}}^*\|_{L^2(\M)}^2.
\end{equation}

\noindent\underline{Step 5: The subsequential limit is a minimizer at $\lambda$.}
Fix any $\boldsymbol{h}\in \mathscr{F}$. Optimality of $\boldsymbol{f}_{\lambda_{n,k,l}}^*$ for $\mathcal{L}_{\lambda_{n,k,l}}(\Bf)$ gives
\[
\mathbb{E}\,\dist(\boldsymbol{X}, \boldsymbol{f}_{\lambda_{n,k,l}}^*)+\lambda_{n,k,l} \cdot\|\nabla^2 \boldsymbol{f}_{\lambda_{n,k,l}}^*\|^2_{L^2(\M)} \le \mathbb{E}\,\dist(\boldsymbol{X}, \boldsymbol{h})+\lambda_{n,k,l}\cdot \|\nabla^2 \boldsymbol{h}\|^2_{L^2(\M)}.
\]
Take $\liminf_{l\to\infty}$ on the left-hand and right-hand sides, use \eqref{eq:datafit_conv} and \eqref{eq:lsc_penalty}, and note that $\lim_{l\rightarrow\infty}\lambda_{n,k,l}=\lambda$. Then, we have 
\begin{align*}
    \mathcal{L}_{\lambda}(\tilde{\boldsymbol{f}}) &= \mathbb{E}\,\dist(\boldsymbol{X}, \tilde{\boldsymbol{f}})+\lambda \cdot \|\nabla^2 \tilde{\boldsymbol{f}}\|^2_{L^2(\M)} \\
    & = \lim_{l\rightarrow\infty}\mathbb{E}\,\dist(\boldsymbol{X}, \boldsymbol{f}_{\lambda_{n,k,l}}^*)+\left(\lim_{l\rightarrow\infty}\lambda_{n,k,l}\right) \cdot \|\nabla^2 \tilde{\boldsymbol{f}}\|^2_{L^2(\M)} \\
    &\le \liminf_{l\to\infty} \left\{ \mathbb{E}\,\dist(\boldsymbol{X}, \boldsymbol{f}_{\lambda_{n,k,l}}^*)+\lambda_{n,k,l} \cdot\|\nabla^2 \boldsymbol{f}_{\lambda_{n,k,l}}^*\|^2_{L^2(\M)} \right\} \\
    &\le \mathbb{E}\,\dist(\boldsymbol{X}, \boldsymbol{h})+\lambda\cdot \|\nabla^2 \boldsymbol{h}\|^2_{L^2(\M)}.
\end{align*}
That is, $\mathcal{L}_{\lambda}(\tilde{\boldsymbol{f}})
\le \mathcal{L}_{\lambda}(\boldsymbol{h})$. Since $\boldsymbol{h}\in \mathscr{F}$ was arbitrary, $\tilde{\boldsymbol{f}}\in\arg\min_{\boldsymbol{h}\in \mathscr{F}} \mathcal{L}_\lambda(\boldsymbol{h})$.

\noindent\underline{Step 6: Orbit uniqueness implies full quotient convergence.}
By Step 5 and the assumed orbit uniqueness, $\tilde{\boldsymbol f}\in[\boldsymbol f_\lambda^*]$. Hence, along the further subsequence in \eqref{eq: uniform limit for the continuity wrt lambda},
\[
    d_{\mathrm{iso}}([\boldsymbol f_{\lambda_{n,k,l}}^*],[\boldsymbol f_\lambda^*])
    \leq
    \Vert\boldsymbol f_{\lambda_{n,k,l}}^*-\tilde{\boldsymbol f}\Vert_{C(\M)}
    \rightarrow0.
\]
Thus, every subsequence of the quotient distances in \eqref{eq: continuity of the minimizer wrt the tuning parameter} has a further subsequence converging to zero. Lemma \ref{lemma: elementary mathematical analysis lemma} completes the proof.
\end{proof}

\subsection{Proofs of the Results in Section \ref{section: Projection-Adaptation Algorithm}}

The following lemmas serve as preparation for the proof of Theorem \ref{thm: the convergence theorem of the core iterative algorithm}.
\begin{lemma}\label{lemma: the connection between L and Q}
\begin{enumerate}
    \item For any $\Bf$ and $\Bg\in\mathscr{F}$, we have $\mathcal{L}_\lambda(\Bf) \le \mathcal{Q}_\lambda(\Bf\,\vert\,\Bg)$.

    \item Let $\{\Bf^{(n)}\}_{n\in\mathbb{N}}$ be a sequence generated by the iterative algorithm in \eqref{eq: the core iterative algorithm}. The scalar-valued sequence $\{\mathcal{L}_\lambda(\Bf_\lambda^{(n)})\}_{n\in\mathbb{N}}$ is nonincreasing.
\end{enumerate}
\end{lemma}
\begin{proof}[\textbf{Proof of Lemma \ref{lemma: the connection between L and Q}}]
    i) The definition of projection indices implies that 
    \begin{align*}
        \mathcal{L}_\lambda(\Bf) &= \int_{\mathbb{R}^D} \Vert \boldsymbol{x}-\Bf\left(\Bpi_{\Bf}(\boldsymbol{x})\right)\Vert^2 \, \mathbb{P}(\mathrm{d}\boldsymbol{x}) + \lambda\cdot \Vert \nabla{}^2 \boldsymbol{f} \Vert^2_{L^2(\M)} \\ 
        &\le \int_{\mathbb{R}^D} \Vert \boldsymbol{x}-\Bf\left(\Bpi_{\Bg}(\boldsymbol{x})\right)\Vert^2 \, \mathbb{P}(\mathrm{d}\boldsymbol{x}) + \lambda\cdot \Vert \nabla{}^2 \boldsymbol{f} \Vert^2_{L^2(\M)} = \mathcal{Q}_\lambda(\Bf \,\vert\, \Bg),
    \end{align*}
    for all $\Bf$ and $\Bg\in\mathscr{F}$.

    ii) The result from part i) implies that 
    \begin{align}\label{eq: Monotone descent and existence of a limit}
    \begin{aligned}
        \min_{\Bf\in\mathscr{F}} \mathcal{L}_\lambda(\Bf) \le \mathcal{L}_\lambda(\Bf_\lambda^{(n+1)})
        &\le \mathcal{Q}_\lambda(\Bf_\lambda^{(n+1)}\,\vert\,\Bf^{(n)}_\lambda) \\
        &= \min_{\Bf\in\mathscr{F}}\mathcal{Q}_\lambda(\Bf\,\vert\,\Bf^{(n)}_\lambda) \le \mathcal{Q}_\lambda(\Bf_\lambda^{(n)}\,\vert\,\Bf^{(n)}_\lambda) = \mathcal{L}_\lambda(\Bf_\lambda^{(n)}).
    \end{aligned}
\end{align}
\end{proof}

\begin{lemma}\label{lemma: convergence of projection indices}
    Let $\{\boldsymbol{h}^{(k)}\}_{k\in\mathbb{N}}$ be a sequence in $\mathscr{F}$, and $\boldsymbol{h}^{(\infty)}\in \mathscr{F}$. Suppose
\begin{enumerate}
    \item $\lim_{k\rightarrow\infty}\max_{\boldsymbol{m}\in\mathfrak{M}} \Vert \boldsymbol{h}^{(k)}(\boldsymbol{m}) - \boldsymbol{h}^{(\infty)}(\boldsymbol{m}) \Vert = 0$, and 
    \item for every $\boldsymbol{x}\in\mathrm{supp}(\mathbb{P})$, the following set 
    \begin{align*}
        \argmin_{\boldsymbol{m}'\in\mathfrak{M}}\Vert \boldsymbol{x}-\boldsymbol{h}^{(\infty)}(\boldsymbol{m}') \Vert:=\left\{\boldsymbol{m}\in\mathfrak{M}:\,\Vert \boldsymbol{x}-\boldsymbol{h}^{(\infty)}(\boldsymbol{m}) \Vert = \min_{\boldsymbol{m}'\in\mathfrak{M}}\Vert \boldsymbol{x}-\boldsymbol{h}^{(\infty)}(\boldsymbol{m}') \Vert\right\}
    \end{align*}
    is a singleton, denoted as $\{\boldsymbol{m}^{(\infty)}(\boldsymbol{x})\}$.
\end{enumerate}
Then, for every $\boldsymbol{x}\in\mathrm{supp}(\mathbb{P})$, 
\begin{align*}
    \lim_{k\rightarrow\infty} \boldsymbol{m}^{(k)}(\boldsymbol{x}) = \boldsymbol{m}^{(\infty)}(\boldsymbol{x}),
\end{align*}
where $\boldsymbol{m}^{(k)}(\boldsymbol{x}) \in \argmin_{\boldsymbol{m}\in\mathfrak{M}} \Vert \boldsymbol{x} - \boldsymbol{h}^{(k)}(\boldsymbol{m}) \Vert$ is allowed to be any Borel measurable selection. 
\end{lemma}
\begin{proof}[\textbf{Proof of Lemma \ref{lemma: convergence of projection indices}}]
    Fix $\boldsymbol{x}\in\mathrm{supp}(\mathbb{P})$. 
    \begin{align*}
        \phi_k(\boldsymbol{m}):=\Vert \boldsymbol{x} - \boldsymbol{h}^{(k)}(\boldsymbol{m}) \Vert \ \ \text{ and }\ \ \phi_\infty(\boldsymbol{m}):=\Vert \boldsymbol{x} - \boldsymbol{h}^{(\infty)}(\boldsymbol{m}) \Vert.
    \end{align*}
    Obviously, $\phi_k$ and $\phi_\infty$ are continuous functions. Then, $\boldsymbol{m}^{(k)}\in \argmin_{\boldsymbol{m}\in\mathfrak{M}} \phi_k(\boldsymbol{m})$ for all $k$, and $\boldsymbol{m}^{(\infty)}$ is the element in the singleton $\argmin_{\boldsymbol{m}\in\mathfrak{M}} \phi_\infty(\boldsymbol{m})$. (Throughout this proof, we suppress ``$(\boldsymbol{x})$'', i.e., $\boldsymbol{m}^{(k)}(\boldsymbol{x})$ and $\boldsymbol{m}^{(\infty)}(\boldsymbol{x})$ are denoted as $\boldsymbol{m}^{(k)}$ and $\boldsymbol{m}^{(\infty)}$, respectively.) By assumption,
\begin{align}\label{eq: uniform convergence of phi_k}
\lim_{k\rightarrow\infty}\max_{\boldsymbol{m}\in\mathfrak{M}} \vert \phi_k(\boldsymbol{m}) - \phi_\infty(\boldsymbol{m}) \vert \le \lim_{k\rightarrow\infty}\max_{\boldsymbol{m}\in\mathfrak{M}} \Vert \boldsymbol{h}^{(k)}(\boldsymbol{m}) - \boldsymbol{h}^{(\infty)}(\boldsymbol{m}) \Vert = 0
\end{align}
    
For any subsequence $\{\boldsymbol{m}^{(k, j)}\}_{j\in\mathbb{N}}$ of $\{\boldsymbol{m}^{(k)}\}_{k\in\mathbb{N}}$, the compactness of $\mathfrak{M}$ implies that there exists a further subsequence $\{\boldsymbol{m}^{(k, j, l)}\}_{l\in\mathbb{N}}$ that is convergent, i.e., there exists an $\widehat{\boldsymbol{m}} \in\mathfrak{M}$ such that $\boldsymbol{m}^{(k,j,l)} \rightarrow \widehat{\boldsymbol{m}}$ as $l\rightarrow\infty$.

    By \eqref{eq: uniform convergence of phi_k}, for any $\epsilon>0$, there exists a positive integer $K$, independent of $\boldsymbol{m}$, such that for all $k\ge K$,
    \begin{align}\label{eq: lemma proof eq 1, convergence of projection indices}
        -\epsilon \le \phi_k(\boldsymbol{m}) - \phi_\infty(\boldsymbol{m}) \le \epsilon \ \ \text{ for all }\boldsymbol{m}\in\mathfrak{M}.
    \end{align}
    Because $\boldsymbol{m}^{(k)}$ minimizes $\phi_k$, we have
    \begin{align}\label{eq: lemma proof eq 2, convergence of projection indices}
        \phi_k(\boldsymbol{m}^{(k)}) \le \phi_k(\boldsymbol{m})\ \ \text{ for all }\boldsymbol{m}\in\mathfrak{M}.
    \end{align}
    Use \eqref{eq: lemma proof eq 1, convergence of projection indices} twice in \eqref{eq: lemma proof eq 2, convergence of projection indices},
    \begin{align*}
        \phi_\infty(\boldsymbol{m}^{(k)}) - \epsilon \le \phi_k(\boldsymbol{m}^{(k)}) \le \phi_k(\boldsymbol{m}) \le \phi_\infty(\boldsymbol{m}) +\epsilon \ \ \text{ for all }k \ge K \text{ and all }\boldsymbol{m}\in\mathfrak{M}.
    \end{align*}
    Taking minimum over $\boldsymbol{m}$ on the right gives
    \begin{align*}
        \phi_\infty(\boldsymbol{m}^{(k)}) \le \min_{\boldsymbol{m}\in\mathfrak{M}} \phi_\infty(\boldsymbol{m}) + 2\epsilon \ \ \text{ for all }k \ge K,
    \end{align*}
    which implies that
    \begin{align*}
\phi_\infty(\widehat{\boldsymbol{m}}) = \lim_{l\rightarrow\infty}\phi_\infty(\boldsymbol{m}^{(k,j,l)}) \le \min_{\boldsymbol{m}\in\mathfrak{M}} \phi_\infty(\boldsymbol{m}) + 2\epsilon,
    \end{align*}
where the first equality follows from the continuity of $\phi_\infty$. Since $\epsilon>0$ is arbitrary, we have
    \begin{align*}
        \phi_\infty(\widehat{\boldsymbol{m}}) \le \min_{\boldsymbol{m}\in\mathfrak{M}} \phi_\infty(\boldsymbol{m}),
    \end{align*}
    which implies that $\widehat{\boldsymbol{m}}\in\argmin\phi_\infty=\{\boldsymbol{m}^{(\infty)}\}$. Hence, $\widehat{\boldsymbol{m}}=\boldsymbol{m}^{(\infty)}$. That is, every subsequence of $\{\boldsymbol{m}^{(k)}\}_{k\in\mathbb{N}}$ has a further subsequence converging to $\boldsymbol{m}^{(\infty)}$. By Lemma \ref{lemma: elementary mathematical analysis lemma}, we have $\lim_{k\rightarrow\infty} \boldsymbol{m}^{(k)} = \boldsymbol{m}^{(\infty)}$.
\end{proof}

\begin{lemma}\label{lemma: convergence of projection the fitting term of Q}
    Under the assumptions of Lemma \ref{lemma: convergence of projection indices}, for any $\Bg\in\mathscr{F}$, we have
\begin{align}\label{eq: convergence of the fitting error as the projection index converges}
    \lim_{k\rightarrow\infty} \mathbb{E}\left\{\left\Vert \boldsymbol{X} - \Bg\left(\boldsymbol{m}^{(k)}(\boldsymbol{X})\right) \right\Vert^2\right\} = \mathbb{E}\left\{\left\Vert \boldsymbol{X} - \Bg\left(\boldsymbol{m}^{(\infty)}(\boldsymbol{X})\right) \right\Vert^2\right\},
\end{align}
which implies the following limit for all $\Bg\in\mathscr{F}$
\begin{align*}
    \lim_{k\rightarrow\infty} \mathcal{Q}_\lambda(\Bg \,\vert\, \boldsymbol{h}^{(k)}) = \mathcal{Q}_\lambda(\Bg \,\vert\, \boldsymbol{h}^{(\infty)}).
\end{align*}
\end{lemma}
\begin{proof}[\textbf{Proof of Lemma \ref{lemma: convergence of projection the fitting term of Q}}]
It suffices to show \eqref{eq: convergence of the fitting error as the projection index converges}. 

The continuity of $\Bg$, together with Lemma~\ref{lemma: convergence of projection indices}, implies that
    \begin{align*}
        \lim_{k\rightarrow\infty} \left\Vert \boldsymbol{x} - \Bg\left(\boldsymbol{m}^{(k)}(\boldsymbol{x})\right) \right\Vert^2 = \left\Vert \boldsymbol{x} - \Bg\left(\boldsymbol{m}^{(\infty)}(\boldsymbol{x})\right) \right\Vert^2,
    \end{align*}
    for all $\boldsymbol{x}\in\mathrm{supp}(\mathbb{P})$. By the definition of $\mathscr{F}$, we have 
    \begin{align*}
        \left\Vert \boldsymbol{x} - \Bg\left(\boldsymbol{m}^{(k)}(\boldsymbol{x})\right) \right\Vert^2 \le \left( \norm{\boldsymbol{x}} + \left\Vert \Bg\left(\boldsymbol{m}^{(k)}(\boldsymbol{x})\right)\right\Vert \right)^2 \leq 4\Upsilon^2 <\infty.
    \end{align*}
    Then, the dominated convergence theorem, together with the compactness of $\mathrm{supp}(\mathbb{P})$, implies that
    \begin{align*}
        \mathbb{E}\left\{\left\Vert \boldsymbol{X} - \Bg\left(\boldsymbol{m}^{(\infty)}(\boldsymbol{X})\right) \right\Vert^2\right\} &= \int_{\mathbb{R}^D} \left\Vert \boldsymbol{x} - \Bg\left(\boldsymbol{m}^{(\infty)}(\boldsymbol{x})\right) \right\Vert^2 \, \mathbb{P}(\mathrm{d}\boldsymbol{x}) \\
        &= \int_{\mathrm{supp}(\mathbb{P})} \lim_{k\rightarrow\infty} \left\Vert \boldsymbol{x} - \Bg\left(\boldsymbol{m}^{(k)}(\boldsymbol{x})\right) \right\Vert^2 \, \mathbb{P}(\mathrm{d}\boldsymbol{x}) \\
        &= \lim_{k\rightarrow\infty} \int_{\mathrm{supp}(\mathbb{P})} \left\Vert \boldsymbol{x} - \Bg\left(\boldsymbol{m}^{(k)}(\boldsymbol{x})\right) \right\Vert^2 \, \mathbb{P}(\mathrm{d}\boldsymbol{x}) \\
        & = \lim_{k\rightarrow\infty}\mathbb{E}\left\{ \left\Vert \boldsymbol{X} - \Bg\left(\boldsymbol{m}^{(k)}(\boldsymbol{X})\right) \right\Vert^2\right\}.
    \end{align*}
\end{proof}

\begin{lemma}\label{lemma: properties of Q and T functionals}
For every fixed $\Bg\in\mathscr{F}$, the argmin set $\argmin_{\Bf\in\mathscr{F}}\mathcal{Q}_\lambda(\Bf\,\vert\,\Bg)$ is nonempty.
\end{lemma}
\begin{proof}[\textbf{Proof of Lemma \ref{lemma: properties of Q and T functionals}}] There exists a sequence $\{\Bf^{(n)}\}_{n\in\mathbb{N}}\subseteq\mathscr{F}$ such that
\begin{align*}
    \lim_{n\rightarrow\infty} \mathcal{Q}_\lambda(\Bf^{(n)} \,\vert\,\Bg) &= \inf_{\Bf\in\mathscr{F}} \mathcal{Q}_\lambda(\Bf \,\vert\,\Bg) \\
    &= \inf_{\Bf\in\mathscr{F}}\left[ \mathbb{E} \left\{\Vert \boldsymbol{X}-\Bf\left(\Bpi_{\Bg}(\boldsymbol{X})\right)\Vert^2\right\} + \lambda\cdot \Vert \nabla{}^2 \boldsymbol{f} \Vert^2_{L^2(\M)}\right]<\infty.
\end{align*}
Using the same argument in Part 1 of the proof of Theorem \ref{thm: existence of f star}, we have the following: there exists $\Bf^*\in\mathscr{F}$ and a subsequence $\{\Bf^{(n,k)}\}_{k\in\mathbb{N}}$ of $\{\Bf^{(n)}\}_{n\in\mathbb{N}}$ such that
\begin{align*}
    & \Bf^{(n,k)} \overset{w}{\rightharpoonup} \Bf^*\in H^2(\mathfrak{M})\quad \text{ as }k\rightarrow\infty,\ \text{ and} \\
    & \lim_{k\rightarrow\infty}\Vert \Bf^{(n,k)}-\Bf^*\Vert_{C(\mathfrak{M})}=\lim_{k\rightarrow\infty}\max_{\boldsymbol{m}\in\mathfrak{M}}\Vert \Bf^{(n,k)}(\boldsymbol{m})-\Bf^*(\boldsymbol{m})\Vert=0.
\end{align*}
For every $\boldsymbol{x}\in\mathrm{supp}(\mathbb{P})$, we have the following upper bound for all $k\in\mathbb{N}$
\begin{align*}
    \Vert \boldsymbol{x} - \fnk\left(\Bpi_{\Bg}(\boldsymbol{x})\right) \Vert \le \Vert \boldsymbol{x} \Vert + \max_{\boldsymbol{m}\in\mathfrak{M}}\Vert \fnk(\boldsymbol{m})\Vert \le 2\Upsilon.
\end{align*}
Then, the dominated convergence theorem implies that
\begin{align*}
    \lim_{k\rightarrow\infty} \mathbb{E} \left\{\Vert \boldsymbol{X}-\fnk\left(\Bpi_{\Bg}(\boldsymbol{X})\right)\Vert^2\right\} &= \lim_{k\rightarrow\infty} \int_{\mathrm{supp}(\mathbb{P})} \Vert \boldsymbol{x}-\fnk\left(\Bpi_{\Bg}(\boldsymbol{x})\right)\Vert^2 \, \mathbb{P}(d\boldsymbol{x}) \\
    &= \int_{\mathrm{supp}(\mathbb{P})} \lim_{k\rightarrow\infty}\Vert \boldsymbol{x}-\fnk\left(\Bpi_{\Bg}(\boldsymbol{x})\right)\Vert^2 \, \mathbb{P}(d\boldsymbol{x}) \\
    &= \int_{\mathrm{supp}(\mathbb{P})} \Vert \boldsymbol{x}-\fstar\left(\Bpi_{\Bg}(\boldsymbol{x})\right)\Vert^2 \, \mathbb{P}(d\boldsymbol{x}) \\
    &= \mathbb{E} \left\{\Vert \boldsymbol{X}-\fstar\left(\Bpi_{\Bg}(\boldsymbol{X})\right)\Vert^2\right\}.
\end{align*}
Using the same argument in Part 3 of the proof of Theorem \ref{thm: existence of f star}, we have $\Vert \nabla{}^2 \boldsymbol{f}^* \Vert^2_{L^2(\M)} \le \liminf_{k\rightarrow\infty} \Vert \nabla{}^2 \boldsymbol{f}^{(n,k)} \Vert^2_{L^2(\M)}$. Then,
\begin{align*}
    \mathcal{Q}_\lambda (\fstar\,\vert\,\Bg) &= \mathbb{E} \left\{\Vert \boldsymbol{X}-\fstar\left(\Bpi_{\Bg}(\boldsymbol{X})\right)\Vert^2\right\} + \lambda\cdot \Vert \nabla{}^2 \boldsymbol{f}^* \Vert^2_{L^2(\M)}\\
    &\leq \lim_{k\rightarrow\infty} \mathbb{E} \left\{\Vert \boldsymbol{X}-\fnk\left(\Bpi_{\Bg}(\boldsymbol{X})\right)\Vert^2\right\} + \lambda\cdot  \liminf_{k \to \infty} \Vert \nabla{}^2 \boldsymbol{f}^{(n,k)} \Vert^2_{L^2(\M)}\\
    &\le \liminf_{k\to \infty} \Big[  \mathbb{E} \left\{\Vert \boldsymbol{X}-\fnk\left(\Bpi_{\Bg}(\boldsymbol{X})\right)\Vert^2\right\} +\lambda\cdot \Vert \nabla{}^2 \boldsymbol{f}^{(n,k)} \Vert^2_{L^2(\M)} \Big] \\
    &= \liminf_{k \to \infty} \mathcal{Q}_\lambda(\fnk \,\vert\,\Bg) \\
    &= \lim_{n\rightarrow\infty} \mathcal{Q}_\lambda(\Bf^{(n)} \,\vert\,\Bg) \\ 
    &= \inf_{\Bf\in\mathscr{F}} \mathcal{Q}_\lambda(\Bf \,\vert\,\Bg).
\end{align*}
The other inequality direction holds from the definition of infimum. This then implies that $\fstar$ is a minimizer. 
\end{proof}

\begin{lemma}\label{lemma: continuity of T}
Fix $\lambda>0$. Let $\{\boldsymbol{h}^{(k)}\}_{k\in\mathbb{N}}\subseteq\mathscr{F}$ converge uniformly to $\boldsymbol{h}^{(\infty)}\in\mathscr{F}$, under the setup and conditions of Lemma \ref{lemma: convergence of projection indices}. For any sequence $\boldsymbol{u}^{(k)}\in\mathcal{T}_\lambda(\boldsymbol{h}^{(k)})$, every accumulation point $\boldsymbol{\bar{f}}$ under the supremum norm satisfies
\begin{align*}
    \boldsymbol{\bar{f}}\in\mathcal{T}_\lambda(\boldsymbol{h}^{(\infty)}).
\end{align*}
Equivalently, the graph of $\mathcal{T}_\lambda$ is sequentially closed along such sequences.
\end{lemma}
\begin{proof}[\textbf{Proof of Lemma \ref{lemma: continuity of T}}]
Choose
\begin{align*}
    \boldsymbol{u}^{(k)}\in\mathcal{T}_\lambda(\boldsymbol{h}^{(k)})\ \ \text{ and }\ \ \boldsymbol{u}^{(\infty)}\in\mathcal{T}_\lambda(\boldsymbol{h}^{(\infty)}).
\end{align*}
We show that every uniformly convergent subsequence of $\{\boldsymbol{u}^{(k)}\}$ has its limit in $\mathcal{T}_\lambda(\boldsymbol{h}^{(\infty)})$. The compactness argument below also shows that every subsequence has a further uniformly convergent subsequence.

Let $\boldsymbol{\bar{f}}$ be an arbitrary accumulation point of $\{\boldsymbol{u}^{(k)}\}$, and let $\{\boldsymbol{u}^{(k,l)}\}_{l\in\mathbb{N}}$ be a subsequence converging uniformly to $\boldsymbol{\bar{f}}$. Pick a fixed $\boldsymbol{v}_\star\in\mathscr{F}$. By the optimality of $\boldsymbol{u}^{(k,l)}\in\mathcal{T}_\lambda(\boldsymbol{h}^{(k,l)})$, we have
\begin{equation*}
\lambda\|\nabla^2 \boldsymbol{u}^{(k,l)}\|^2_{L^2(\mathfrak{M})}
\;\le\;\mathcal{Q}_\lambda(\boldsymbol{u}^{(k,l)}\mid \boldsymbol{h}^{(k,l)})\;\le\;\mathcal{Q}_\lambda(\boldsymbol{v}_\star\mid \boldsymbol{h}^{(k,l)})\qquad\text{for all }l\in\mathbb{N}.
\end{equation*}
Lemma \ref{lemma: convergence of projection the fitting term of Q} shows that the limit $\lim_{l\rightarrow\infty}\mathcal{Q}_\lambda(\boldsymbol{v}_\star\mid \boldsymbol{h}^{(k,l)}) = \mathcal{Q}_\lambda(\boldsymbol{v}_\star\mid \boldsymbol{h}^{(\infty)})$ exists, which implies that
\begin{align*}
    \lambda \|\nabla^2 \boldsymbol{u}^{(k,l)}\|^2_{L^2(\mathfrak{M})}
\;\le\; \sup_{l'\in\mathbb{N}} \mathcal{Q}_\lambda(\boldsymbol{v}_\star\mid \boldsymbol{h}^{(k,l')}) < \infty \ \ \text{ for all }l\in\mathbb{N}.
\end{align*}
In addition, by the definition of $\mathscr{F}$, we have
\begin{align*}
    \Vert \boldsymbol{u}^{(k,l)} \Vert^2_{L^2(\mathfrak{M})}
    = \int_{\mathfrak{M}} \Vert \boldsymbol{u}^{(k,l)}(\boldsymbol{m}) \Vert^2 \, d \vol_g(\boldsymbol{m})
    \le \Upsilon^2\vol_g(\mathfrak{M})
    \qquad\text{for all }l\in\mathbb{N}.
\end{align*}
Therefore, $\{\boldsymbol{u}^{(k,l)}\}_{l\in\mathbb{N}}$ is bounded in $H^2(\mathfrak{M};\,\mathbb{R}^D)$. Theorem 3.18 of \cite{brezis2011functional} implies that there exists a further subsequence $\{\boldsymbol{u}^{(k,l,m)}\}_{m\in\mathbb{N}}$ that converges weakly in $H^2(\mathfrak{M};\,\mathbb{R}^D)$. Because the original subsequence converges uniformly to $\boldsymbol{\bar{f}}$, compactness of the embedding identifies its weak limit with $\boldsymbol{\bar{f}}$. Thus,
\[
\boldsymbol{u}^{(k,l,m)} \overset{w}{\rightharpoonup}  \boldsymbol{\bar{f}} \quad\text{in }H^2(\mathfrak{M};\,\mathbb{R}^D) \text{ as }m\rightarrow\infty.
\]
Theorem \ref{thm: Rellich–Kondrachov embedding}, together with Remark 2 of Chapter 6 of \cite{brezis2011functional}, implies that
\begin{equation}\label{eq:fk_uniform_conv_to_fbar}
\lim_{m\rightarrow\infty}\|\boldsymbol{u}^{(k,l,m)}- \boldsymbol{\bar{f}}\|_{C(\mathfrak{M};\,\mathbb{R}^D)} = \lim_{m\rightarrow\infty}\max_{\boldsymbol{m}\in\mathfrak{M}}\|\boldsymbol{u}^{(k,l,m)}(\boldsymbol{m})- \boldsymbol{\bar{f}}(\boldsymbol{m})\| = 0.
\end{equation}

We claim that $ \boldsymbol{\bar{f}}$ minimizes $\mathcal{Q}_\lambda(\,\cdot\,\mid \boldsymbol{h}^{(\infty)})$ over $\mathscr{F}$.
Fix any $\Bf\in\mathscr{F}$, by the optimality of $\boldsymbol{u}^{(k,l,m)}\in\mathcal{T}_\lambda(\boldsymbol{h}^{(k,l,m)})$, we have
\begin{equation*}
\mathcal{Q}_\lambda(\boldsymbol{u}^{(k,l,m)}\mid \boldsymbol{h}^{(k,l,m)})\;\le\;\mathcal{Q}_\lambda(\Bf\mid \boldsymbol{h}^{(k,l,m)})\qquad\text{for all }m\in\mathbb{N}.
\end{equation*}
We first show a lower bound for the left-hand side limit, i.e., the following several steps are for showing \eqref{eq:liminf_Q}.
For $\boldsymbol{x}\in\mathrm{supp}(\mathbb{P})$, by Lemma \ref{lemma: convergence of projection indices} and \eqref{eq:fk_uniform_conv_to_fbar}, we have
\begin{align*}
&\Vert \boldsymbol{u}^{(k,l,m)}(\Bpi_{\boldsymbol{h}^{(k,l,m)}}(\boldsymbol{x})) - \boldsymbol{\bar{f}}(\Bpi_{\boldsymbol{h}^{(\infty)}}(\boldsymbol{x})) \Vert \\
    &\le \Vert \boldsymbol{u}^{(k,l,m)}(\Bpi_{\boldsymbol{h}^{(k,l,m)}}(\boldsymbol{x})) - \boldsymbol{\bar{f}}(\Bpi_{\boldsymbol{h}^{(k,l,m)}}(\boldsymbol{x})) \Vert + \Vert \boldsymbol{\bar{f}}(\Bpi_{\boldsymbol{h}^{(k,l,m)}}(\boldsymbol{x})) - \boldsymbol{\bar{f}}(\Bpi_{\boldsymbol{h}^{(\infty)}}(\boldsymbol{x})) \Vert \\
    & \le \max_{\boldsymbol{m}\in\mathfrak{M}} \Vert \boldsymbol{u}^{(k,l,m)}(\boldsymbol{m}) - \boldsymbol{\bar{f}}(\boldsymbol{m}) \Vert + \Vert \boldsymbol{\bar{f}}(\Bpi_{\boldsymbol{h}^{(k,l,m)}}(\boldsymbol{x})) - \boldsymbol{\bar{f}}(\Bpi_{\boldsymbol{h}^{(\infty)}}(\boldsymbol{x})) \Vert \\
    & \rightarrow 0, \text{ as }m\rightarrow\infty.
\end{align*}
Thus, we have the following pointwise convergence
\[
\lim_{m\rightarrow\infty}\|\boldsymbol{x}-\boldsymbol{u}^{(k,l,m)}(\Bpi_{\boldsymbol{h}^{(k,l,m)}}(\boldsymbol{x}))\|^2\;=\;\|\boldsymbol{x}- \boldsymbol{\bar{f}}(\Bpi_{\boldsymbol{h}^{(\infty)}}(\boldsymbol{x}))\|^2.
\]
Additionally, the definition of $\mathscr{F}$ implies
\begin{align*}
\|\boldsymbol{x}-\boldsymbol{u}^{(k,l,m)}(\Bpi_{\boldsymbol{h}^{(k,l,m)}}(\boldsymbol{x}))\|^2
\le (2\Upsilon)^2=4\Upsilon^2
\qquad\text{for all }m\in\mathbb{N}.
\end{align*}
Then, the dominated convergence theorem implies that
\begin{equation}\label{eq:data_term_conv_with_moving_fk}
\lim_{m\rightarrow\infty}\mathbb{E}\left\{\|\boldsymbol{X}-\boldsymbol{u}^{(k,l,m)}(\Bpi_{\boldsymbol{h}^{(k,l,m)}}(\boldsymbol{X}))\|^2 \right\}\;=\;
\mathbb{E}\left\{\|\boldsymbol{X}- \boldsymbol{\bar{f}}(\Bpi_{\boldsymbol{h}^{(\infty)}}(\boldsymbol{X}))\|^2\right\}.
\end{equation}
For the penalty term, weak lower semicontinuity in $H^2(\mathfrak{M};\,\mathbb{R}^D)$ yields
\begin{equation}\label{eq:lsc_penalty 1}
\|\nabla^2  \boldsymbol{\bar{f}}\|^2_{L^2(\M)}
\;\le\;\liminf_{m\to\infty} \|\nabla^2 \boldsymbol{u}^{(k,l,m)}\|^2_{L^2(\mathfrak{M})}.
\end{equation}
Combining \eqref{eq:data_term_conv_with_moving_fk} and \eqref{eq:lsc_penalty 1}, we obtain
\begin{equation}\label{eq:liminf_Q}
\mathcal{Q}_\lambda( \boldsymbol{\bar{f}}\mid \boldsymbol{h}^{(\infty)})\;\le\;\liminf_{m\to\infty} \mathcal{Q}_\lambda(\boldsymbol{u}^{(k,l,m)}\mid \boldsymbol{h}^{(k,l,m)}).
\end{equation}
Next, we apply Lemma \ref{lemma: convergence of projection the fitting term of Q} to get
\begin{equation}\label{eq:Q_fstar_conv}
\lim_{m\rightarrow\infty}\mathcal{Q}_\lambda(\boldsymbol{u}^{(\infty)}\mid \boldsymbol{h}^{(k,l,m)})\;=\;\mathcal{Q}_\lambda(\boldsymbol{u}^{(\infty)}\mid \boldsymbol{h}^{(\infty)}).
\end{equation}
Since $\boldsymbol{u}^{(k,l,m)}\in\mathcal{T}_\lambda(\boldsymbol{h}^{(k,l,m)})$, it minimizes $\mathcal{Q}_\lambda(\cdot\mid \boldsymbol{h}^{(k,l,m)})$,
\begin{equation}\label{eq:Q_fk_le_Q_fstar}
\mathcal{Q}_\lambda(\boldsymbol{u}^{(k,l,m)}\mid \boldsymbol{h}^{(k,l,m)})\;\le\;\mathcal{Q}_\lambda(\boldsymbol{u}^{(\infty)}\mid \boldsymbol{h}^{(k,l,m)}).
\end{equation}
Taking $\limsup$ in \eqref{eq:Q_fk_le_Q_fstar} and using \eqref{eq:Q_fstar_conv} yields
\begin{equation}\label{eq:limsup_Q}
\limsup_{m\to\infty}\mathcal{Q}_\lambda(\boldsymbol{u}^{(k,l,m)}\mid \boldsymbol{h}^{(k,l,m)})\;\le\;\mathcal{Q}_\lambda(\boldsymbol{u}^{(\infty)}\mid \boldsymbol{h}^{(\infty)}).
\end{equation}
Combining \eqref{eq:liminf_Q} and \eqref{eq:limsup_Q} gives
\begin{equation*}
\mathcal{Q}_\lambda( \boldsymbol{\bar{f}}\mid \boldsymbol{h}^{(\infty)})\;\le\;\mathcal{Q}_\lambda(\boldsymbol{u}^{(\infty)}\mid \boldsymbol{h}^{(\infty)}).
\end{equation*}
Therefore, $ \boldsymbol{\bar{f}}$ is a minimizer of $\mathcal{Q}_\lambda(\cdot\mid \boldsymbol{h}^{(\infty)})$ over $\mathscr{F}$.
Note $\boldsymbol{u}^{(\infty)}\in\mathcal{T}_\lambda(\boldsymbol{h}^{(\infty)})$ is also a minimizer. The preceding argument already proves that $\boldsymbol{\bar{f}}$ minimizes $\mathcal{Q}_\lambda(\cdot\mid\boldsymbol{h}^{(\infty)})$. Hence,
\begin{align*}
    \boldsymbol{\bar{f}}\in\mathcal{T}_\lambda(\boldsymbol{h}^{(\infty)}),
\end{align*}
which is the desired set-valued conclusion.
\end{proof}

\begin{lemma}\label{thm: properties of the iterative algorithm}
For every $\lambda>0$, let $\{\Bf^{(n)}_\lambda\}_{n\in\mathbb{N}}$ be a sequence of functions generated by the iterative algorithm in \eqref{eq: the core iterative algorithm}. If the conditions of Theorem \ref{thm: the convergence theorem of the core iterative algorithm} hold, we have
\begin{enumerate}
    \item There exists a subsequence $\{\Bf^{(n_k)}_\lambda\}_{k\in\mathbb{N}} \subseteq \{\Bf^{(n)}_\lambda\}_{n\in\mathbb{N}}$ and $\Bf_\lambda^{(\infty)}\in\mU$ such that
    \begin{align}\label{eq: result 1 of the convergence lemma}
        \lim_{k\rightarrow\infty} \max_{\boldsymbol{m}\in\mathfrak{M}} \Vert \Bf^{(n_k)}_\lambda(\boldsymbol{m}) - \Bf_\lambda^{(\infty)}(\boldsymbol{m}) \Vert = 0 \ \ \text{ and }\ \ \mathcal{L}_\lambda(\Bf_\lambda^{(\infty)}) \le \lim_{n\rightarrow\infty} \mathcal{L}_\lambda(\Bf^{(n)}).
    \end{align}

    \item If $\argmin_{\boldsymbol{m}\in\mathfrak{M}} \Vert \boldsymbol{x} - \Bf_\lambda^{(\infty)}(\boldsymbol{m})\Vert$ is a singleton for every $\boldsymbol{x}\in\mathrm{supp}(\mathbb{P})$, we have
    \begin{align*}
\mathcal{L}_\lambda(\Bf_\lambda^{(\infty)}) = \lim_{n\rightarrow\infty} \mathcal{L}_\lambda(\Bf^{(n)}).
    \end{align*}
\end{enumerate}
\end{lemma}
\begin{proof}[\textbf{Proof of Lemma \ref{thm: properties of the iterative algorithm}}]
\underline{Part {\textit{i)}}.} Lemma~\ref{lemma: the connection between L and Q} implies
\begin{align*}
    0 \le \lambda\cdot\Vert \nabla^2\Bf_\lambda^{(n)} \Vert^2_{L^2(\M)}\le \mathcal{L}_\lambda(\Bf_\lambda^{(n)}) \le \mathcal{L}_\lambda(\Bf_\lambda^{(1)}) \quad \text{for all } n\ge 1.
\end{align*}
Then, the same argument in Part 1 of the proof of Theorem \ref{thm: existence of f star} implies that there exists a subsequence $\{\Bf^{(n_k)}_\lambda\}_{k\in\mathbb{N}} \subseteq \{\Bf^{(n)}_\lambda\}_{n\in\mathbb{N}}$ and a function $\Bf_\lambda^{(\infty)} \in\mathscr{F}$ such that
    \begin{align*}
    & \Bf^{(n_k)} \overset{w}{\rightharpoonup} \Bf^{(\infty)}\in H^2(\mathfrak{M})\ \ \text{ as }k\rightarrow\infty, \text{ and } \\
    & \lim_{k\rightarrow\infty} \max_{\boldsymbol{m}\in\mathfrak{M}} \Vert \Bf^{(n_k)}_\lambda(\boldsymbol{m}) - \Bf_\lambda^{(\infty)}(\boldsymbol{m}) \Vert = 0.
    \end{align*}
Due to the closedness of $\mU$ under the supremum norm, $\Bf_\lambda^{(\infty)}\in\mU$. Applying the lower semicontinuity argument utilized in \eqref{eq: the lower semicontinuity tool} and \eqref{eq: lower semicontinuity argument}, we have
\begin{align*}
\mathcal{L}_\lambda(\Bf_\lambda^{(\infty)}) \le \liminf_{k\rightarrow\infty} \mathcal{L}_\lambda(\Bf^{(n_k)}_\lambda)=  \lim_{n\rightarrow\infty} \mathcal{L}_\lambda(\Bf^{(n)}).
\end{align*}
\underline{Part {\textit{ii)}}.} Lemma \ref{lemma: convergence of projection the fitting term of Q}, together with Assumption~\ref{assumption: singleton assumption for the sequence}, implies that
\begin{align*}
    \lim_{k\rightarrow\infty} \mathcal{Q}_\lambda ( \Bf_\lambda^{(\infty)} \,\vert\, \Bf^{(n_k)}_\lambda) = \mathcal{Q}_\lambda ( \Bf_\lambda^{(\infty)} \,\vert\, \Bf_\lambda^{(\infty)}  )=\mathcal{L}_\lambda(\Bf^{(\infty)}).
\end{align*}
Using Lemma \ref{lemma: the connection between L and Q} and the definition of $\mathcal{Q}_\lambda(\cdot\,\vert\,\cdot)$, we have\footnote{Note that the superscript $n_k+1$ means $(n_k)+1$, rather than $n_{(k+1)}$. This follows from iteration $\boldsymbol{f}^{(n_k+1)}\in\mathcal{T}_\lambda(\boldsymbol{f}^{(n_k)})$.} 
\begin{align*}
    \mathcal{L}_\lambda(\Bf^{(n_k +1 )}) \le \mathcal{Q}_\lambda ( \Bf^{(n_k +1 )} \,\vert\, \Bf^{(n_k)}_\lambda) = \min_{\Bf\in\mathscr{F}} \mathcal{Q}_\lambda ( \Bf \,\vert\, \Bf^{(n_k)}_\lambda) \le \mathcal{Q}_\lambda ( \Bf_\lambda^{(\infty)} \,\vert\, \Bf^{(n_k)}_\lambda).
\end{align*}
Then, we have
\begin{align*}
    \mathcal{L}_\lambda(\Bf_\lambda^{(\infty)}) \le \lim_{n\rightarrow\infty} \mathcal{L}_\lambda(\Bf^{(n)}) =  \lim_{k\rightarrow\infty} \mathcal{L}_\lambda(\Bf^{(n_k +1 )}) \le \lim_{k\rightarrow\infty} \mathcal{Q}_\lambda ( \Bf_\lambda^{(\infty)} \,\vert\, \Bf^{(n_k)}_\lambda) = \mathcal{L}_\lambda(\Bf_\lambda^{(\infty)}).
\end{align*}
Therefore, $\lim_{n\rightarrow\infty} \mathcal{L}_\lambda(\Bf^{(n)}) = \mathcal{L}_\lambda(\Bf_\lambda^{(\infty)})$.
\end{proof}

\begin{proof}[\textbf{Proof of Theorem \ref{thm: the convergence theorem of the core iterative algorithm}}]
By Lemma~\ref{lemma: the connection between L and Q}, we have that $\mathcal{L}_\lambda(\Bf_\lambda^{(n+1)}) \le \mathcal{L}_\lambda(\Bf_\lambda^{(n)})$ for all $n\in\mathbb{N}$, so the sequence $\{\mathcal{L}_\lambda(\Bf_\lambda^{(n)})\}_{n\in\mathbb{N}}$ is nonincreasing and bounded below by $\inf_{\Bf \in \mathscr{F}} \mathcal{L}_\lambda(\Bf)$. Hence, there exists a finite limit
\[
  \ell := \lim_{n \to \infty} \mathcal{L}_\lambda(\Bf_\lambda^{(n)}) \in \mathbb{R}.
\]
Choose the global minimizer $\Bf_\lambda^*\in\mU$ supplied by Assumption~\ref{assumption: Assumption for the convergence of the iterative algorithm}; its orbit is the unique minimizing orbit in $\mU$. Denote
\[
  L^* := \mathcal{L}_\lambda(\Bf_\lambda^*) = \inf_{f \in \mathscr{F}} \mathcal{L}_\lambda(f).
\]
Clearly, $L^* \le \ell$. Let $\{\Bf_\lambda^{(m_k)}\}_{k\in\mathbb N}$ be an arbitrary subsequence of the PA iterates. The uniform $H^2$ bound established in Part~\textit{(i)} of the proof of Lemma~\ref{thm: properties of the iterative algorithm}, together with the compact embedding $H^2(\M)\hookrightarrow C(\M)$ and the closedness of $\mU$, yields a further subsequence $\{\Bf_\lambda^{(n_k)}\}_{k\in\mathbb N}\subseteq\{\Bf_\lambda^{(m_k)}\}_{k\in\mathbb N}$ and a function $\Bf_\lambda^{(\infty)}\in\mU$ satisfying \eqref{eq: result 1 of the convergence lemma}. Assumption~\ref{assumption: singleton assumption for the sequence} applies to this accumulation point, and the argument in Part~\textit{(ii)} of the proof of Lemma~\ref{thm: properties of the iterative algorithm} gives
\begin{equation}\label{eq:limit-l}
  \mathcal{L}_\lambda(\Bf_\lambda^{(\infty)}) = \lim_{n \to \infty} \mathcal{L}_\lambda(\Bf_\lambda^{(n)}) = \ell.
\end{equation}

Recall the iteration $\Bf_\lambda^{(n+1)} \in \mathcal{T}_\lambda(\Bf_\lambda^{(n)})$ for all $n\in\mathbb{N}$. Let $\{\Bf_\lambda^{(n_k)}\}_{k\in\mathbb{N}}$ be the convergent subsequence with limit $\Bf_\lambda^{(\infty)}$ as above. Consider the shifted subsequence $\{\Bf_\lambda^{(n_k+1)}\}$, where $\Bf_\lambda^{(n_k+1)}\in\mathcal{T}_\lambda(\Bf_\lambda^{(n_k)})$. The sequence $\mathcal{L}_\lambda(\Bf_\lambda^{(n_k+1)})$ also converges to $\ell$, since it is a subsequence of the same nonincreasing sequence $\{\mathcal{L}_\lambda(\Bf_\lambda^{(n)})\}$. Using the same argument as in the proof of Lemma \ref{thm: properties of the iterative algorithm}, there exists a further subsequence $\{\Bf_\lambda^{(n_{k,l}+1)}\}_{l\in\mathbb{N}} \subseteq \{\Bf_\lambda^{(n_k+1)}\}_{k\in\mathbb{N}}$ and a map $\tilde{\Bf}_\lambda^{(\infty)}\in\mU$ such that $\Bf_\lambda^{(n_{k,l}+1)} \overset{w}{\rightharpoonup} \tilde{\Bf}_\lambda^{(\infty)}$ in $H^2(\mathfrak{M})$ as $l\rightarrow\infty$,
    \begin{align*}
    & \lim_{l\rightarrow\infty} \max_{\boldsymbol{m}\in\mathfrak{M}} \Vert \Bf^{(n_{k,l}+1)}_\lambda(\boldsymbol{m}) - \tilde{\Bf}_\lambda^{(\infty)}(\boldsymbol{m}) \Vert = 0.
    \end{align*}

Therefore, $\tilde{\Bf}_\lambda^{(\infty)}$ is also an accumulation point under the supremum norm. Assumption~\ref{assumption: singleton assumption for the sequence}, together with the same argument used in the proof of Lemma \ref{thm: properties of the iterative algorithm}, implies
\begin{equation}\label{eq:limit-l2}
  \mathcal{L}_\lambda\left( \tilde{\Bf}_\lambda^{(\infty)} \right) = \lim_{l\rightarrow\infty} \mathcal{L}_\lambda\left( \Bf_\lambda^{(n_{k,l}+1)} \right) = \ell.
\end{equation}
On the other hand, by definition of the iteration,
\begin{align*}
    \Bf_\lambda^{(n_{k,l}+1)} \in \mathcal{T}_\lambda(\Bf_\lambda^{(n_{k,l})}).
\end{align*}
Lemma \ref{lemma: continuity of T} and the uniform convergence of the two subsequences imply
\begin{align}\label{eq:Tinfty}
    \tilde{\Bf}_\lambda^{(\infty)} \in \mathcal{T}_\lambda(\Bf_\lambda^{(\infty)}).
\end{align}
Combining \eqref{eq:limit-l}, \eqref{eq:limit-l2}, and \eqref{eq:Tinfty}, we obtain
\begin{align}\label{eq: equality "obtained above"}
    \mathcal{L}_\lambda(\Bf_\lambda^{(\infty)}) = \ell = \mathcal{L}_\lambda(\tilde{\Bf}_\lambda^{(\infty)}).
\end{align}
Lemma \ref{lemma: the connection between L and Q}, together with \eqref{eq:Tinfty}, implies that
\begin{equation}\label{eq:majorization}
  \mathcal{L}_\lambda(\tilde{\Bf}_\lambda^{(\infty)}) \le \mathcal{Q}_\lambda(\tilde{\Bf}_\lambda^{(\infty)} \mid \Bf_\lambda^{(\infty)}) = \min_{\Bf\in\mathscr{F}}\mathcal{Q}_\lambda(\Bf \mid \Bf_\lambda^{(\infty)}) \le \mathcal{Q}_\lambda(\Bf_\lambda^{(\infty)} \mid \Bf_\lambda^{(\infty)}) = \mathcal{L}_\lambda(\Bf_\lambda^{(\infty)}).
\end{equation}
Together with \eqref{eq: equality "obtained above"}, we conclude that both inequalities in \eqref{eq:majorization} are equalities. In particular,
\[
  \mathcal{Q}_\lambda(\Bf_\lambda^{(\infty)} \mid \Bf_\lambda^{(\infty)}) = \min_{\Bf\in\mathscr{F}}\mathcal{Q}_\lambda(\Bf \mid \Bf_\lambda^{(\infty)}).
\]
Therefore, $\Bf_\lambda^{(\infty)}$ is itself a minimizer of $\Bf\mapsto \mathcal{Q}_\lambda(\Bf \mid \Bf_\lambda^{(\infty)})$, and hence
\[
  \Bf_\lambda^{(\infty)} \in \mathcal{T}_\lambda(\Bf_\lambda^{(\infty)}).
\]
That is, $\Bf_\lambda^{(\infty)}$ is a fixed point of the set-valued map $\mathcal{T}_\lambda$. Assumption \ref{assumption: Assumption for the convergence of the iterative algorithm} and isometry invariance imply
\begin{align}\label{eq: f^infty orbit}
    d_{\mathrm{iso}}([\Bf_\lambda^{(\infty)}],[\Bf_\lambda^*])=0,
    \qquad
    \mathcal L_\lambda(\Bf_\lambda^{(\infty)})=\mathcal L_\lambda(\Bf_\lambda^*).
\end{align}
Lemma \ref{thm: properties of the iterative algorithm}, together with \eqref{eq: f^infty orbit}, implies that
\begin{align*}
    \mathcal{L}_\lambda(\Bf_\lambda^*) = \lim_{n\rightarrow\infty} \mathcal{L}_\lambda(\Bf^{(n)}).
\end{align*}
Finally, apply Lemma~\ref{lemma: elementary mathematical analysis lemma} to the nonnegative sequence $d_{\mathrm{iso}}([\Bf_\lambda^{(n)}],[\Bf_\lambda^*])$. Because $\{m_k\}$ was arbitrary, while $\{n_k\}\subseteq\{m_k\}$ indexes a further subsequence converging uniformly to $\Bf_\lambda^{(\infty)}$ and \eqref{eq: f^infty orbit} holds, the definition of $d_{\mathrm{iso}}$ gives
\[
 d_{\mathrm{iso}}([\Bf_\lambda^{(n_k)}],[\Bf_\lambda^*])
 \le \Vert\Bf_\lambda^{(n_k)}-\Bf_\lambda^{(\infty)}\Vert_{C(\M)}
 \rightarrow0.
\]
Thus, every subsequence of the displayed nonnegative sequence has a further subsequence converging to zero. The lemma therefore gives
\[
    d_{\mathrm{iso}}([\Bf_\lambda^{(n)}],[\Bf_\lambda^*])\rightarrow0.
\]
This completes the proof.
\end{proof}

\subsection{Proofs of the Results in Section \ref{section: Consistency of the Empirical Minimizer}}

\begin{lemma}\label{lemma:bounded penalty}
For fixed $\lambda > 0$, let $\fstar_{\lambda}$ be a minimizer as given in \eqref{eq: the core min problem} and let $\fstar_{N,\lambda}$ be the empirical minimizer of \eqref{eq: core iterative algorithm, empirical}. Let $\Bg \in \mathscr{F}$ be arbitrary. Then, 
\begin{align}\label{eq: def of R in mathscrF R}
    \begin{aligned}
        & \norm{\fstar_{\lambda}}_{H^2(\M)} \leq R \quad \text{ and } \quad \norm{\fstar_{N,\lambda}}_{H^2(\M)} \leq R, \\
    & \text{where } R = \sqrt{ \Upsilon^2\cdot\vol_g(\M) + \frac{4\Upsilon^2}{\lambda}+  \norm{\nabla^2 \Bg}_{L^2(\M)}^2}.
    \end{aligned}
\end{align}
\end{lemma}

\begin{proof}[\textbf{Proof of Lemma \ref{lemma:bounded penalty}}]
For any $\boldsymbol{x}\in\operatorname{supp}(\mathbb{P})$ and $\Bg\in\mathscr F$,
\[
\norm{\boldsymbol{x} - \Bg(\Bpi_{\Bg}(\boldsymbol{x}))}
\leq \norm{\boldsymbol{x}} + \norm{\Bg(\Bpi_{\Bg}(\boldsymbol{x}))}
\leq 2\Upsilon.
\]
The same bound applies almost surely to each observed $\boldsymbol X_i$.

We bound the penalty terms first. By the definition of $\fstar_\lambda$,
\begin{align*}
    \lambda \norm{\nabla^2 \Bf^*_{\lambda}}^2_{L^2(\M)}
    \leq \Lfunc(\fstar_\lambda)
    &\leq \Lfunc(\Bg) \\
    &= \mathbb{E}\left\{ \norm{\boldsymbol{X}-\Bg(\Bpi_{\Bg}(\boldsymbol{X}))}^2 \right\} + \lambda \norm{\nabla^2 \Bg}_{L^2(\M)}^2 \\
    &\le 4\Upsilon^2 + \lambda \norm{\nabla^2 \Bg}_{L^2(\M)}^2,
\end{align*}
and therefore $\norm{\nabla^2 \Bf^*_{\lambda}}^2_{L^2(\M)}
 \leq \frac{4\Upsilon^2}{\lambda} +  \norm{\nabla^2 \Bg}_{L^2(\M)}^2$. The empirical minimizer satisfies the same estimate:
\begin{align*}
\lambda \norm{\nabla^2 \Bf^*_{N,\lambda}}^2_{L^2(\M)}
\leq \Lnfunc(\fstar_{N,\lambda})
&\leq \Lnfunc(\Bg)\\
&=\frac1N \sum_{i=1}^N \norm{\boldsymbol{X}_i-\Bg(\Bpi_{\Bg}(\boldsymbol{X}_i))}^2+\lambda \norm{\nabla^2 \Bg}_{L^2(\M)}^2\\ &\leq  4\Upsilon^2 + \lambda \norm{\nabla^2 \Bg}_{L^2(\M)}^2,
\end{align*}
which implies $\norm{\nabla^2 \Bf^*_{N,\lambda}}^2_{L^2(\M)}
\leq \frac{4\Upsilon^2}{\lambda} +  \norm{\nabla^2 \Bg}_{L^2(\M)}^2$. Now, bounding the $L^2$ norms of the original functions, we have
\[
\norm{\fstar_\lambda}_{L^2(\M)}^2
= \int_{\M} \norm{\fstar_\lambda(\boldsymbol{m})}^2 \,d \vol_g(\boldsymbol{m})
\leq \norm{\fstar_\lambda}_{C(\M)}^2\cdot\vol_g(\M)
\leq \Upsilon^2\cdot \vol_g(\M),
\]
and similarly $\norm{\fstar_{N,\lambda}}_{L^2(\M)}^2 \leq \Upsilon^2\cdot \vol_g(\M)$. The definition of $R$ in \eqref{eq: def of R in mathscrF R} gives
\begin{align*}
    & \norm{\fstar_{\lambda}}_{H^2(\M)}^2
    = \norm{\fstar_{\lambda}}_{L^2(\M)}^2 + \norm{\nabla^2 \Bf^*_{\lambda}}^2_{L^2(\M)} \leq R^2, \\
    & \norm{\fstar_{N,\lambda}}_{H^2(\M)}^2
    = \norm{\fstar_{N,\lambda}}_{L^2(\M)}^2 + \norm{\nabla^2 \Bf^*_{N,\lambda}}^2_{L^2(\M)} \leq R^2.
\end{align*}
Taking square roots concludes the proof.
\end{proof}

\newcommand{\Bx}{\boldsymbol{x}}
\newcommand{\Bm}{\boldsymbol{m}}
\newcommand{\infm}{\inf\limits_{m \in \M}}
\begin{lemma}\label{lemma:infimum inequality} Let $U = \{u_m\}_{m \in \M}$ and $V = \{v_m\}_{m \in \M}$ be subsets of $\mbR$ indexed by $m \in \M$. Assume that $\inf_{m\in\M}u_m$ and $\inf_{m\in\M}v_m$ are finite. Suppose $\delta > 0$ is some constant. If $| u_m - v_m| \leq \delta$ for all $m \in \M$, it follows that
$$\left|\inf_{m \in \M} u_m - \inf_{m \in \M}  v_m \right| \leq \delta.$$
\end{lemma}

\begin{proof}[\textbf{Proof of Lemma \ref{lemma:infimum inequality}}]
By hypothesis, for all $m \in \M$,
$$-\delta \leq u_m - v_m \leq\delta.$$
Since $u_m - v_m \leq \delta$, it follows that $\inf_{m'\in\M} u_{m'}  \leq  u_m \leq  \delta + v_m$ for all $m\in\M$. This implies that for all $m \in \M$, $\inf_{m'\in\M} u_{m'}$ is a lower bound for $\delta + v_m$. But since the infimum has the property of being the greatest lower bound, it must be that
$$\infm u_m  \leq \infm(\delta + v_m) = \delta + \infm v_m,$$
implying
$$\infm u_m-\infm v_m  \leq \delta.$$

Since we also have that $v_m \leq u_m + \delta$, an analogous argument swapping the roles of $u_m$ and $v_m$ yields
$$\infm v_m-\infm u_m  \leq \delta.$$
Combining results, we obtain
$$\left|\inf_{m \in \M} u_m - \inf_{m \in \M}  v_m \right| \leq \delta.$$

\end{proof}

\begin{proof}[\textbf{Proof of Theorem \ref{thm: argmin}}] 
We use the standard argmin argument together with the preceding lemmas. To this end, fix $\lambda>0$ and choose an arbitrary $\Bg \in \mF$. Define
\begin{align*}
    \mF_R := \left\{ \Bf \in \mF :  \norm{\Bf}_{H^2(\M)} \leq R \right\},
\end{align*}
where $R$ is given by \eqref{eq: def of R in mathscrF R}. By Lemma~\ref{lemma:bounded penalty}, we have
$\Bf^*_\lambda, \Bf^*_{N,\lambda} \in \mF_R$ for all $N\in\mathbb{N}$. Henceforth, we consider the empirical processes
$\{\mcL_{N,\lambda}(\Bf)\}_{\Bf\in\mF_R}$ indexed by $\mF_R$, together with the deterministic functional
$\mcL_{\lambda}: \Bf\mapsto \mcL_{\lambda}(\Bf)$ defined on $\mF_R$. It remains to verify the following three conditions:
\begin{enumerate}
    \item $T_N:=\sup_{\Bf \in \mF_R} |\mcL_{N,\lambda}(\Bf) - \mcL_{\lambda}(\Bf)|$ converges in outer probability to zero as $N\rightarrow\infty$. Here, $T_N$ is measurable, as shown below.

    \item For every open, isometry-saturated set $\mG\subseteq C(\M;\mathbb R^D)$ containing $[\fstar_\lambda]$,
    \[
        \mcL_\lambda(\fstar_\lambda) < \inf_{\Bf\in\mathscr F\setminus\mG} \mcL_\lambda(\Bf).
    \]

    \item For every $N \in \mbN$, the minimizer $\Bf^*_{N,\lambda}$ of \eqref{eq: core iterative algorithm, empirical} satisfies
    \[
        \mcL_{N,\lambda}(\Bf^*_{N,\lambda}) \leq \inf_{\Bf \in \mF_R} \mcL_{N,\lambda}(\Bf).
    \]
\end{enumerate}
Condition~(ii) is assumed in the statement of the theorem, while condition~(iii) is immediate from the definition of $\Bf^*_{N,\lambda}$. Thus, it remains only to prove condition~(i).

As in Lemma~\ref{lemma: continuity of Delta wrt f}, for $\boldsymbol{x} \in \mbR^D$ and $\Bf \in \mF$, define
\[
    \Deltaxf{x}{f} := \min_{\boldsymbol{m} \in \M} \norm{\boldsymbol{x} - \Bf(\boldsymbol{m})}^2
    = \norm{\boldsymbol{x} - \Bf(\Bpi_{\Bf}(\boldsymbol{x}))}^2.
\]
Following the standard notation of empirical process theory \citep[e.g.,][]{van1996weak}, write
\begin{align*}
    & \mathbb{P}_N\Deltadotf{f} := \frac1N \sum_{i=1}^N \varrho(\boldsymbol{X}_i, \, \boldsymbol{f}), \\
    & \mathbb{P}\Deltadotf{f} := \mbE \varrho(\boldsymbol{X}_1, \, \boldsymbol{f}), \\
    & (\mathbb{P}_N-\mathbb{P})\Deltadotf{f} := \mathbb{P}_N\Deltadotf{f} - \mathbb{P}\Deltadotf{f}.
\end{align*}
Consequently,
\begin{align*}
    T_N
    = \sup_{\Bf \in \mF_R} |\mcL_{N,\lambda}(\Bf) - \mcL_{\lambda}(\Bf)|
    = \sup_{\Bf \in \mF_R} |(\mathbb{P}_N-\mathbb{P})\Deltadotf{f}|.
\end{align*}
The proof proceeds by constructing an $\varepsilon$-net for $\mF_R$ and using it to obtain a uniform bound.

By Theorem~\ref{thm: Rellich–Kondrachov embedding}, the embedding
$\iota: H^2(\M) \hookrightarrow C(\M)$ is compact. Since $\mF_R$ is bounded in $H^2(\M)$, it follows that $\mF_R$ is totally bounded in $C(\M)$ \citep[e.g.,][Chapter 0, §2 “Topological Spaces,” p.13]{Yosida1965}. Hence, for every $\varepsilon>0$, there exist finitely many functions
$\Bf^{(1)},\ldots,\Bf^{(K)} \in \mF_R$ such that, for every $\Bf \in \mF_R$, there is some $j\in\{1,\ldots,K\}$ satisfying
\begin{align}\label{net on f}
     \norm{\Bf - \Bf^{(j)}}_{C(\M)} < \varepsilon .
\end{align}
Taking the union of such nets for $\varepsilon=1/k$, $k\in\mathbb N$, gives a fixed countable dense subset $\mathscr D\subseteq\mF_R$. For all $\boldsymbol x\in\mathbb R^D$ and $\Bf,\Bg\in\mF_R$, the argument of Lemma~\ref{lemma: continuity of Delta wrt f} gives
\[
    |\varrho(\boldsymbol x,\Bf)-\varrho(\boldsymbol x,\Bg)|
    \leq (2\norm{\boldsymbol x}+2\Upsilon)\norm{\Bf-\Bg}_{C(\M)}.
\]
Together with the bounded support of $\mathbb P$, this shows that $\Bf\mapsto |(\mathbb P_N-\mathbb P)\varrho(\cdot,\Bf)|$ is continuous in the supremum norm for every fixed sample. For each fixed $\Bf$, the function $\boldsymbol x\mapsto\varrho(\boldsymbol x,\Bf)$ is continuous, being the squared distance to the compact set $\Bf(\M)$. Therefore,
\[
    T_N=\sup_{\Bf\in\mathscr D}|(\mathbb P_N-\mathbb P)\varrho(\cdot,\Bf)|
\]
is a countable supremum of measurable random variables and is measurable.

The definition of $\mathscr{F}$ implies that
$\mathbb{P}\varrho(\cdot,\Bf^{(j)}) \le 4\Upsilon^2$. Therefore, by the weak law of large numbers \citep[][Section~5.3]{klenke2008probability},
\begin{align*}
    \left|(\mathbb{P}_N - \mathbb{P})\varrho(\cdot, \Bf^{(j)})\right|
    \overset{P}{\longrightarrow}0
    \quad \text{as } N\rightarrow\infty,
\end{align*}
where $\overset{P}{\longrightarrow}$ denotes convergence in probability. Since the collection
$\{\Bf^{(1)},\ldots,\Bf^{(K)}\}$ is finite, we also have
\begin{align}\label{net lln}
\max _{1 \leq j \leq K}
\left|(\mathbb{P}_N - \mathbb{P})\varrho(\cdot, \Bf^{(j)})\right|
\overset{P}{\longrightarrow} 0
\quad \text{as } N\rightarrow\infty.
\end{align}

We next establish a Lipschitz-type bound. Let $\Bf,\Bg \in \mF_R$, $\boldsymbol{x}\in \operatorname{supp}(\mathbb P)$, and $\Bm \in \M$. Then
\begin{align*}
 \Big| \norm{\Bx - \Bf(\Bm)}^2 - \norm{\Bx - \Bg(\Bm)}^2  \Big|
 &=
 \Big( \norm{\Bx - \Bf(\Bm)} +  \norm{\Bx - \Bg(\Bm)} \Big)
 \Big| \norm{\Bx - \Bf(\Bm)} - \norm{\Bx - \Bg(\Bm)} \Big|  \\
 &\leq
 \Big(\norm{\Bx - \Bf(\Bm)} +  \norm{\Bx - \Bg(\Bm)} \Big)
 \norm{\Bf(\Bm) -\Bg(\Bm)} \\
 &\leq
 \Big(2 \norm{\Bx} + \norm{\Bf(\Bm)} + \norm{\Bg(\Bm)}\Big)
 \norm{\Bf(\Bm) -\Bg(\Bm)}\\
 &\leq
 4\Upsilon \max_{\Bm \in \M} \norm{\Bf(\boldsymbol{m}) -\Bg(\boldsymbol{m})}.
\end{align*}
The two squared-distance functions are continuous on the compact template $\M$, so their infima are finite minima. Since $\Bm\in\M$ was arbitrary, Lemma~\ref{lemma:infimum inequality} yields
\begin{align}\label{delta f delta g}
    \begin{aligned}
        | \Deltaxf{x}{f} - \Deltaxf{x}{g}|
        &=
        \Big| \inf_{\Bm \in \M} \norm{\Bx - \Bf(\Bm)}^2
        - \inf_{\Bm \in \M}   \norm{\Bx - \Bg(\Bm)}^2  \Big| \\
        &\leq
        4\Upsilon \max_{\Bm \in \M} \norm{\Bf(\boldsymbol{m}) -\Bg(\boldsymbol{m})}.
    \end{aligned}
\end{align}

We now return to \eqref{net lln}. Fix $\Bf\in\mF_R$, and choose
$j\in\{1,\ldots,K\}$ such that \eqref{net on f} holds. Then
\begin{align*}
    |\mathbb{P}_N \Deltadotf{f} - \mathbb{P} \Deltadotf{f}|
    &=
    |(\mathbb{P}_N- \mathbb{P})\Deltadotf{f}|  \\
    &=
    \Big|(\mathbb{P}_N- \mathbb{P})
    \big(\Deltadotf{f} - \varrho(\cdot, \Bf^{(j)}) + \varrho(\cdot, \Bf^{(j)})\big)\Big| \\
    &\leq
    \Big|(\mathbb{P}_N- \mathbb{P})\varrho(\cdot, \Bf^{(j)})\Big|
    + \mathbb{P}_N\Big| \Deltadotf{f} - \varrho(\cdot, \Bf^{(j)}) \Big|
    + \mathbb{P}\Big| \Deltadotf{f} - \varrho(\cdot, \Bf^{(j)})\Big|.
\end{align*}
By \eqref{delta f delta g}, the third term on the right-hand side is bounded by
$4\Upsilon\norm{\Bf-\Bf^{(j)}}_{C(\M)}$. Since each $\boldsymbol X_i$ has law $\mathbb P$ and $\mathbb P(\operatorname{supp}(\mathbb P))=1$, almost surely $\boldsymbol X_i\in\operatorname{supp}(\mathbb P)$ for every $i\in\mathbb N$; on this event, the second term satisfies the same bound. Hence, almost surely,
\begin{align*}
    |\mathbb{P}_N \Deltadotf{f} - \mathbb{P} \Deltadotf{f}|
    &\leq
    \Big|(\mathbb{P}_N- \mathbb{P})\varrho(\cdot, \Bf^{(j)})\Big|
    + 8\Upsilon \norm{\Bf - \Bf^{(j)}}_{C(\M)}\\
    &\leq
    \Big|(\mathbb{P}_N- \mathbb{P})\varrho(\cdot, \Bf^{(j)})\Big|
    + 8\Upsilon \varepsilon.
\end{align*}
Taking the supremum over $\Bf\in\mF_R$ and the maximum over $1\le j\le K$, we obtain, almost surely, for each fixed $\varepsilon>0$,
\begin{align*}
    0 \leq T_N
    \leq
    \max_{1 \leq j \leq K}
    \Big|(\mathbb{P}_N- \mathbb{P})\varrho(\cdot, \Bf^{(j)})\Big|
    + 8\Upsilon\varepsilon.
\end{align*}

Since $T_N$ is measurable, we may take its measurable cover to be $T_N^*=T_N$. The preceding bound gives
\begin{align*}
    T_N^*
    \leq
    \max_{1 \leq j \leq K}
    \Big|(\mathbb{P}_N- \mathbb{P})\varrho(\cdot, \Bf^{(j)})\Big|
    + 8\Upsilon\varepsilon
    \quad\text{almost surely}.
\end{align*}
Now let $\eta>0$ and choose $\varepsilon>0$ such that $8\Upsilon\varepsilon<\eta$. Using \eqref{net lln}, we obtain
\[
    \mathbb{P}\big(T_N^*>\eta\big)
    \leq
    \mathbb{P}\left(
    \max_{1 \leq j \leq K}
    \Big|(\mathbb{P}_N- \mathbb{P})\varrho(\cdot, \Bf^{(j)})\Big|
    > \eta-8\Upsilon\varepsilon
    \right)
    \rightarrow 0 .
\]
Thus, $T_N^*\overset{P}{\longrightarrow}0$. Since $T_N^*=T_N$, this gives convergence of $T_N$ in ordinary probability, and hence in outer probability. We retain the latter formulation for the theorem because no measurable selection of the empirical minimizer is assumed. Hence condition~(i) is proved.

Fix $\epsilon>0$ and let $\mathscr G_\epsilon :=\{\Bf\in C(\M;\mathbb R^D): d_{\mathrm{iso}}([\Bf],[\Bf_\lambda^*])<\epsilon\}$. This set is open and isometry-saturated. By the separation assumption,
\[
    \delta_\epsilon
    :=\inf_{\Bf\in\mathscr F\setminus\mathscr G_\epsilon}
      \mcL_\lambda(\Bf)-\mcL_\lambda(\Bf_\lambda^*)>0.
\]
On the event $2T_N^*<\delta_\epsilon$, empirical optimality precludes $\Bf^*_{N,\lambda}\notin\mathscr G_\epsilon$: otherwise,
\[
    \mcL_{N,\lambda}(\Bf^*_{N,\lambda})
    \geq \mcL_\lambda(\Bf^*_{N,\lambda})-T_N^*
    > \mcL_\lambda(\Bf_\lambda^*)+T_N^*
    \geq \mcL_{N,\lambda}(\Bf_\lambda^*),
\]
a contradiction. Therefore,
\[
    \mathbb P^*\!\left\{
    d_{\mathrm{iso}}([\Bf^*_{N,\lambda}],[\Bf_\lambda^*])
    \geq\epsilon\right\}
    \leq \mathbb P\{2T_N^*\geq\delta_\epsilon\}\rightarrow0,
\]
which proves the theorem.
\end{proof}

\subsection{Proofs of the Results in Section~\ref{section: Bayesian Viewpoint and Tuning Parameter Selection}}

\begin{proof}[\textbf{Proof of Theorem~\ref{thm: diffuse prior limit for an empirical A-step}}]
The objective in \eqref{eq: argmin via RKHS} and the GPR model in
\eqref{eq: GPR interpretation of a PA iteration} both separate over
the ambient coordinates. It therefore suffices to fix \(j\in[D]\) and
prove the result for the \(j\)th coordinate.

Let \(\boldsymbol T\in\mathbb R^{N\times d_0}\) and
\(\boldsymbol K\in\mathbb R^{N\times N}\) be the matrices defined in
\eqref{eq: notations for the representer theorem}, and write
\[
    \boldsymbol x_j:=(X_{1j},\ldots,X_{Nj})^\T,\qquad
    \boldsymbol\phi(\boldsymbol m)
    :=\big(\phi_1(\boldsymbol m),\ldots,
    \phi_{d_0}(\boldsymbol m)\big)^\T,
\]
and $\boldsymbol r_n(\boldsymbol m)
    :=\big(R_1(\boldsymbol m,\boldsymbol m_1),\ldots,
    R_1(\boldsymbol m,\boldsymbol m_N)\big)^\T$. We first justify the finite-dimensional representation of the rough
process. Define
\[
    \mathcal W_n
    :=\operatorname{span}
    \{R_1(\cdot,\boldsymbol m_i):i\in[N]\}
    \subseteq\mathcal H_1,
\]
and let \(E_n:\mathcal W_n\to\mathbb R^N\) be the evaluation map $E_nh :=\big(h(\boldsymbol m_1),\ldots, h(\boldsymbol m_N)\big)^\T$. The map \(E_n\) is injective. Indeed, if
\(h=\sum_{i=1}^Na_iR_1(\cdot,\boldsymbol m_i)\in\mathcal W_n\)
and \(E_nh=\boldsymbol0\), then the reproducing property gives
\[
    \Vert h\Vert_{R_1}^2
    =\sum_{i=1}^Na_i h(\boldsymbol m_i)=0,
\]
and hence \(h=0\). Thus, every vector in
\(\operatorname{range}(E_n)\) has a unique interpolant in
\(\mathcal W_n\), even when the kernel sections are linearly dependent
and their coefficient vector is not unique.

To describe the induced Gaussian law on \(\mathcal W_n\), set
\(q_n:=\dim\mathcal W_n\), choose an \(R_1\)-orthonormal basis
\(\{e_a\}_{a=1}^{q_n}\) of \(\mathcal W_n\), and define
\[
    (\boldsymbol L_n)_{ia}:=e_a(\boldsymbol m_i),
    \qquad 1\le i\le N,\quad 1\le a\le q_n.
\]
Since the kernel sections belong to \(\mathcal W_n\), the reproducing
property and Parseval's identity imply
\[
    (\boldsymbol K)_{i\ell}
    =R_1(\boldsymbol m_i,\boldsymbol m_\ell)
    =\sum_{a=1}^{q_n}
    e_a(\boldsymbol m_i)e_a(\boldsymbol m_\ell).
\]
Consequently,
\(\boldsymbol K=\boldsymbol L_n\boldsymbol L_n^\T\).
The injectivity of \(E_n\) also implies that
\(\boldsymbol L_n\) has full column rank.

Now let
\[
    \boldsymbol h_j
    :=\big(h_j(\boldsymbol m_1),\ldots,
    h_j(\boldsymbol m_N)\big)^\T.
\]
Under the GP prior,
\(\boldsymbol h_j\sim
\mathcal N(\boldsymbol0,\omega^2\boldsymbol K)\), so
\(\boldsymbol h_j\in\operatorname{range}(\boldsymbol L_n)
=\operatorname{range}(E_n)\) almost surely. Let
\(\boldsymbol L_n^+\) denote the Moore--Penrose inverse and define
\[
    \boldsymbol\xi_j:=\boldsymbol L_n^+\boldsymbol h_j,
    \qquad
    h_{j,n}:=\sum_{a=1}^{q_n}\xi_{aj}e_a.
\]
Then \(E_nh_{j,n}=\boldsymbol h_j\) almost surely, and the injectivity
of \(E_n\) shows that \(h_{j,n}\) is the unique
\(\mathcal W_n\)-interpolant of the GP values. Moreover,
\[
    \operatorname{Cov}(\boldsymbol\xi_j)
    =\omega^2\boldsymbol L_n^+
      \boldsymbol L_n\boldsymbol L_n^\T
      (\boldsymbol L_n^+)^\T
    =\omega^2\boldsymbol I_{q_n}.
\]
Thus, in the orthonormal coordinates
\((\xi_{1j},\ldots,\xi_{q_nj})\), the negative log prior density of
\(h_{j,n}\), up to an additive constant, is
\(\Vert h_{j,n}\Vert_{R_1}^2/(2\omega^2)\).

We also verify that this finite representative gives the posterior mean
of the full rough process. For every \(\boldsymbol m\in\M\),
\[
    \sum_{a=1}^{q_n}
    e_a(\boldsymbol m)e_a(\boldsymbol m_i)
    =
    \left\langle
    P_{\mathcal W_n}R_1(\cdot,\boldsymbol m),
    R_1(\cdot,\boldsymbol m_i)
    \right\rangle_{R_1}
    =
    R_1(\boldsymbol m,\boldsymbol m_i),
\]
where \(P_{\mathcal W_n}\) is the \(R_1\)-orthogonal projection onto
\(\mathcal W_n\). It follows that
\[
    \operatorname{Cov}
    \{h_{j,n}(\boldsymbol m),\boldsymbol h_j\}
    =\omega^2\boldsymbol r_n(\boldsymbol m)^\T
    =\operatorname{Cov}
    \{h_j(\boldsymbol m),\boldsymbol h_j\}.
\]
Hence the centered Gaussian residual
\(h_j(\boldsymbol m)-h_{j,n}(\boldsymbol m)\) is independent of
\(\boldsymbol h_j\). It is also independent of
\(\boldsymbol\beta_j\) and the observation errors, so it is independent
of \(\boldsymbol x_j\) and has conditional mean zero. Therefore,
\[
    \mathbb E\{p_{\boldsymbol\beta_j}(\boldsymbol m)
    +h_j(\boldsymbol m)\mid\boldsymbol x_j\}
    =
    \mathbb E\{p_{\boldsymbol\beta_j}(\boldsymbol m)
    +h_{j,n}(\boldsymbol m)\mid\boldsymbol x_j\}.
\]

Set \(\mathcal V_n:=\mathcal H_0\oplus\mathcal W_n\). Because the
priors on \(\boldsymbol\beta_j\) and \(h_{j,n}\) are nondegenerate on
their respective finite-dimensional spaces, the induced posterior on
\(\mathcal V_n\) is a proper finite-dimensional Gaussian distribution
for every finite \(\rho>0\). More explicitly, the likelihood and the
two independent prior densities are proportional to
\begin{align*}
&\exp\left[
    -\frac{1}{2\sigma^2}\sum_{i=1}^N
    \left\{X_{ij}
    -p_{\boldsymbol\beta_j}(\boldsymbol m_i)
    -h_{j,n}(\boldsymbol m_i)\right\}^2
    \right],\\
&\exp\left[
    -\frac{1}{2\rho^2}
    \Vert\boldsymbol\beta_j\Vert_2^2\right],\\
&\exp\left[
    -\frac{1}{2\omega^2}\Vert h_{j,n}\Vert_{R_1}^2
    \right],
\end{align*}
respectively. After multiplying the negative log posterior density by
\(2\sigma^2\) and omitting terms independent of
\((\boldsymbol\beta_j,h_{j,n})\), we obtain the criterion
\begin{equation}\label{eq: posterior criterion for a PA iteration}
    \sum_{i=1}^N
    \left\{X_{ij}-p_{\boldsymbol\beta_j}(\boldsymbol m_i)
    -h_{j,n}(\boldsymbol m_i)\right\}^2
    {}+{} \frac{\sigma^2}{\rho^2}
    \Vert\boldsymbol\beta_j\Vert_2^2
    {}+{} \frac{\sigma^2}{\omega^2}\Vert h_{j,n}\Vert_{R_1}^2.
\end{equation}
Because the posterior is a nondegenerate Gaussian distribution, its
mean is its unique mode, which is the unique minimizer of
\eqref{eq: posterior criterion for a PA iteration}.

As \(\rho\to\infty\), the coefficient
\(\sigma^2/\rho^2\) of the null-space penalty converges to zero. On the
other hand, if \(f=p_{\boldsymbol\beta_j}+h\) with
\(p_{\boldsymbol\beta_j}\in\mathcal H_0\) and \(h\in\mathcal H_1\),
then \(\nabla^2p_{\boldsymbol\beta_j}=0\) and
\[
    \Vert\nabla^2f\Vert_{L^2(\M)}^2
    =\Vert\nabla^2h\Vert_{L^2(\M)}^2
    =\Vert h\Vert_{R_1}^2.
\]
Multiplying the \(j\)th coordinate objective in
\eqref{eq: argmin via RKHS} by \(N\) therefore gives a roughness
coefficient \(N\lambda\). It agrees with the limiting posterior
criterion precisely when
\begin{align}\label{eq: Bayesian variance ratio for lambda}
    \lambda=\frac{\sigma^2}{N\omega^2}.
\end{align}
This comparison identifies the candidate limiting criterion. We next
show directly that the posterior mean converges to its unique
minimizer.

Under the proper null-space prior, the observation vector has covariance
\begin{align}\label{eq: proper GPR marginal covariance for a PA iteration}
    \boldsymbol C_n
    :=\rho^2\boldsymbol T\boldsymbol T^\T
    +\omega^2\boldsymbol K+\sigma^2\boldsymbol I_N.
\end{align}
Indeed, the three summands are the covariances of
\(\boldsymbol T\boldsymbol\beta_j\), \(\boldsymbol h_j\), and the
error vector, respectively. Similarly,
\[
    \operatorname{Cov}
    \{p_{\boldsymbol\beta_j}(\boldsymbol m)+h_j(\boldsymbol m),
    \boldsymbol x_j\}
    =
    \rho^2\boldsymbol\phi(\boldsymbol m)^\T
    \boldsymbol T^\T
    +\omega^2\boldsymbol r_n(\boldsymbol m)^\T.
\]
Since \(\sigma^2>0\), \(\boldsymbol C_n\) is positive definite.
The conditional-mean formula for jointly Gaussian random vectors now
gives
\begin{align}\label{eq: proper GPR posterior mean for a PA iteration}
    \widehat f_{j,\rho}^{(n+1)}(\boldsymbol m)
    =
    \left\{
    \rho^2\boldsymbol\phi(\boldsymbol m)^\T
    \boldsymbol T^\T
    +\omega^2\boldsymbol r_n(\boldsymbol m)^\T
    \right\}
    \boldsymbol C_n^{-1}\boldsymbol x_j.
\end{align}

Although \eqref{eq: proper GPR posterior mean for a PA iteration} is a
closed-form expression, its limit is more transparent under the
hierarchical representation obtained by conditioning on
\(\boldsymbol\beta_j\) and integrating out the rough process:
\[
    \boldsymbol V_n:=\omega^2\boldsymbol K+\sigma^2\boldsymbol I_N,
    \qquad
    \boldsymbol x_j\mid\boldsymbol\beta_j
    \sim\mathcal N(\boldsymbol T\boldsymbol\beta_j,\boldsymbol V_n).
\]
Up to an additive constant and a positive multiplicative factor, the
negative log posterior of \(\boldsymbol\beta_j\) is
\[
    \left(\boldsymbol x_j-\boldsymbol T\boldsymbol\beta_j\right)^\T
    \boldsymbol V_n^{-1}
    \left(\boldsymbol x_j-\boldsymbol T\boldsymbol\beta_j\right)
    +\rho^{-2}\Vert\boldsymbol\beta_j\Vert_2^2.
\]
Thus, if $\boldsymbol A_{n,\rho}
    :=\boldsymbol T^\T\boldsymbol V_n^{-1}\boldsymbol T
    +\rho^{-2}\boldsymbol I_{d_0}$, then completing the square shows that the posterior covariance and
mean of \(\boldsymbol\beta_j\) are
\[
    \operatorname{Cov}(\boldsymbol\beta_j\mid\boldsymbol x_j)
    =\boldsymbol A_{n,\rho}^{-1},
    \qquad
    \widehat{\boldsymbol\beta}_{j,\rho}
    :=
    \mathbb E(\boldsymbol\beta_j\mid\boldsymbol x_j)
    =
    \left(
        \boldsymbol T^\T\boldsymbol V_n^{-1}\boldsymbol T
        +\rho^{-2}\boldsymbol I_{d_0}
    \right)^{-1}
    \boldsymbol T^\T\boldsymbol V_n^{-1}\boldsymbol x_j.
\]
The matrix \(\boldsymbol V_n\) is positive definite. Moreover, for every
nonzero \(\boldsymbol c\in\mathbb R^{d_0}\),
\[
    \boldsymbol c^\T\boldsymbol T^\T\boldsymbol V_n^{-1}
    \boldsymbol T\boldsymbol c
    =
    (\boldsymbol T\boldsymbol c)^\T\boldsymbol V_n^{-1}
    (\boldsymbol T\boldsymbol c)>0,
\]
because \(\boldsymbol T\) has full column rank. Thus,
\(\boldsymbol T^\T\boldsymbol V_n^{-1}\boldsymbol T\) is positive
definite. Since \(\rho^{-2}\boldsymbol I_{d_0}\to\boldsymbol 0\)
and matrix inversion is continuous on the positive-definite matrices,
\[
    \widehat{\boldsymbol\beta}_{j,\rho}
    \rightarrow
    \boldsymbol\theta_j
    :=
    \left(
        \boldsymbol T^\T\boldsymbol V_n^{-1}\boldsymbol T
    \right)^{-1}
    \boldsymbol T^\T\boldsymbol V_n^{-1}\boldsymbol x_j.
\]
Define $\boldsymbol M :=\boldsymbol K+\frac{\sigma^2}{\omega^2}\boldsymbol I_N =\boldsymbol K+N\lambda\boldsymbol I_N$.
Because \(\boldsymbol V_n=\omega^2\boldsymbol M\), the common scale \(\omega^2\) cancels from the preceding generalized least-squares formula, and hence
\[
    \boldsymbol\theta_j
    =
    \left(
        \boldsymbol T^\T\boldsymbol M^{-1}\boldsymbol T
    \right)^{-1}
    \boldsymbol T^\T\boldsymbol M^{-1}\boldsymbol x_j.
\]

We next identify the posterior mean of the rough component. Conditional
on \(\boldsymbol\beta_j\), the vector
\((h_j(\boldsymbol m),\boldsymbol x_j^\T)^\T\) is jointly Gaussian,
and its relevant covariance blocks are
\[
    \operatorname{Cov}(\boldsymbol x_j\mid\boldsymbol\beta_j)
    =\boldsymbol V_n,
    \qquad
    \operatorname{Cov}\{h_j(\boldsymbol m),\boldsymbol x_j
    \mid\boldsymbol\beta_j\}
    =\omega^2\boldsymbol r_n(\boldsymbol m)^\T.
\]
Gaussian conditioning therefore gives
\[
    \mathbb E\{h_j(\boldsymbol m)\mid
    \boldsymbol\beta_j,\boldsymbol x_j\}
    =
    \omega^2\boldsymbol r_n(\boldsymbol m)^\T
    \boldsymbol V_n^{-1}
    \left(\boldsymbol x_j-\boldsymbol T\boldsymbol\beta_j\right).
\]
Taking the conditional expectation of this affine expression over
\(\boldsymbol\beta_j\mid\boldsymbol x_j\) and applying the tower
property yields the following equivalent form of
\eqref{eq: proper GPR posterior mean for a PA iteration}:
\[
    \widehat f_{j,\rho}^{(n+1)}(\boldsymbol m)
    =
    \boldsymbol\phi(\boldsymbol m)^\T
    \widehat{\boldsymbol\beta}_{j,\rho}
    +\omega^2\boldsymbol r_n(\boldsymbol m)^\T
    \boldsymbol V_n^{-1}
    \left(
    \boldsymbol x_j-\boldsymbol T
    \widehat{\boldsymbol\beta}_{j,\rho}
    \right).
\]
Equivalently, define
\[
    \boldsymbol\alpha_{j,\rho}
    :=
    \boldsymbol M^{-1}
    \left(
        \boldsymbol x_j-\boldsymbol T
        \widehat{\boldsymbol\beta}_{j,\rho}
    \right).
\]
Because
\(\omega^2\boldsymbol V_n^{-1}=\boldsymbol M^{-1}\), we have
\[
    \widehat f_{j,\rho}^{(n+1)}(\boldsymbol m)
    =
    \boldsymbol\phi(\boldsymbol m)^\T
    \widehat{\boldsymbol\beta}_{j,\rho}
    +\boldsymbol r_n(\boldsymbol m)^\T
    \boldsymbol\alpha_{j,\rho}.
\]
The convergence of \(\widehat{\boldsymbol\beta}_{j,\rho}\) and the
definition of \(\boldsymbol\alpha_{j,\rho}\) imply
\begin{align*}
    \boldsymbol\alpha_{j,\rho}
    \rightarrow
    \boldsymbol\alpha_j
    :=
    \boldsymbol M^{-1}
    \left(\boldsymbol x_j-\boldsymbol T\boldsymbol\theta_j\right) =
    \boldsymbol M^{-1}
    \left\{
    \boldsymbol I_N
    -\boldsymbol T
    \left(
    \boldsymbol T^\T\boldsymbol M^{-1}\boldsymbol T
    \right)^{-1}
    \boldsymbol T^\T\boldsymbol M^{-1}
    \right\}\boldsymbol x_j.
\end{align*}

The vectors \(\boldsymbol\theta_j\) and \(\boldsymbol\alpha_j\) are
exactly the \(j\)th columns of the coefficient matrices in
Theorem~\ref{thm: wahba representation theorem}, because
\(\boldsymbol M=\boldsymbol K+N\lambda\boldsymbol I_N\).
That theorem therefore identifies
\[
    f_{N,\lambda,j}^{(n+1)}(\boldsymbol m)
    =
    \boldsymbol\phi(\boldsymbol m)^\T\boldsymbol\theta_j
    +\boldsymbol r_n(\boldsymbol m)^\T\boldsymbol\alpha_j,
\]
as a minimizer of the \(j\)th coordinate objective in
\eqref{eq: argmin via RKHS}.

It remains to verify that this minimizer is unique. Up to an irrelevant
positive factor, the quadratic form associated with the Hessian of the
coordinate objective in a direction \(g\in H^2(\M)\) is
\[
    \frac1N\sum_{i=1}^N g(\boldsymbol m_i)^2
    +\lambda\Vert\nabla^2g\Vert_{L^2(\M)}^2.
\]
Both terms are nonnegative. If their sum is zero, then
\(\nabla^2g=0\), so
\(g\in\mathcal H_0\) and
\(g(\boldsymbol m)=\boldsymbol\phi(\boldsymbol m)^\T\boldsymbol c\)
for some \(\boldsymbol c\in\mathbb R^{d_0}\). Its values at the
projected inputs then satisfy \(\boldsymbol T\boldsymbol c=\boldsymbol
0\), and the full column rank of \(\boldsymbol T\) forces
\(\boldsymbol c=\boldsymbol 0\). Hence the coordinate objective is
strictly convex. Since the vector objective is the sum of the
coordinate objectives, \(\Bf_{N,\lambda}^{(n+1)}\) is also unique.

It remains to establish the asserted mode of convergence. Let
\[
    \boldsymbol G_0
    :=
    \left(
        \langle\phi_r,\phi_s\rangle_R
    \right)_{1\le r,s\le d_0}.
\]
The orthogonality of \(\mathcal H_0\) and \(\mathcal H_1\), together
with the reproducing-kernel identity, gives
\begin{align*}
    \left\Vert
        \widehat f_{j,\rho}^{(n+1)}
        -f_{N,\lambda,j}^{(n+1)}
    \right\Vert_R^2
    =
    \left(
        \widehat{\boldsymbol\beta}_{j,\rho}
        -\boldsymbol\theta_j
    \right)^\T
    \boldsymbol G_0
    \left(
        \widehat{\boldsymbol\beta}_{j,\rho}
        -\boldsymbol\theta_j
    \right) +
    \left(
        \boldsymbol\alpha_{j,\rho}-\boldsymbol\alpha_j
    \right)^\T
    \boldsymbol K
    \left(
        \boldsymbol\alpha_{j,\rho}-\boldsymbol\alpha_j
    \right)
    \rightarrow0.
\end{align*}
Thus, every coordinate converges in the RKHS norm. Since there are only
finitely many coordinates, the embedding bound
\eqref{eq: reasoning of H2 equipped with R} gives
\begin{align*}
\left\Vert
    \boldsymbol{\widehat{f}}_{\rho}^{(n+1)}
    -\Bf_{N,\lambda}^{(n+1)}
\right\Vert_{C(\M)}^2
&\quad
=\max_{\boldsymbol m\in\M}
\sum_{j=1}^D
\left|
    \widehat f_{j,\rho}^{(n+1)}(\boldsymbol m)
    -f_{N,\lambda,j}^{(n+1)}(\boldsymbol m)
\right|^2\\
&\quad
\le (C_H\cdot B)^2\sum_{j=1}^D
\left\Vert
    \widehat f_{j,\rho}^{(n+1)}
    -f_{N,\lambda,j}^{(n+1)}
\right\Vert_R^2
\rightarrow0.
\end{align*}
This is the uniform convergence stated in the theorem.
\end{proof}

\subsection{Proofs of the Results in Appendix~\ref{section: Geometric Interpretation of the Penalty}}

The following lemma is about the Sobolev space $W^{1,1}(I)=\{u\in L^{1}(I): u'\in L^{1}(I)\}$, where $I\subset\mathbb R$ is a bounded interval, and $u'$ is the \textit{weak derivative} of $u$ \citep[][Section 5.2.1]{evans1998pde}. This lemma is commonly known as \textit{Poincaré–Wirtinger’s inequality} in the literature on partial differential equations \citep[e.g.,][]{brezis2011functional}. For the reader’s convenience, we include a brief proof.
\begin{lemma}[Poincaré–Wirtinger's inequality]\label{lemma: Poincaré-Wirtinger's inequality}
Let $I\subset\mathbb R$ be a bounded interval and let $u\in W^{1,1}(I)$. Define the mean value of $u$ on $I$ by $\bar u := \frac{1}{|I|}\int_I u(x)\,dx$, where $|I|$ denotes the length of $I$. Then, $\|u-\bar u\|_{L^\infty(I)} \le \|u'\|_{L^1(I)}$.
\end{lemma}

\begin{proof}[\textbf{Proof of Lemma }\ref{lemma: Poincaré-Wirtinger's inequality}]
Write $I=(a,b)$, and $|I|=b-a$. Since $u\in W^{1,1}(a,b)$, $u$ admits an absolutely continuous representative, still denoted by $u$, such that for all $x,y\in(a,b)$,
\[
u(x)-u(y)=\int_y^x u'(t)\,dt,
\]
where $u'$ is the weak derivative of $u$ \citep[][Theorem 8.2]{brezis2011functional}. Fix $x\in(a,b)$. By the definition of $\bar u$,
\[
u(x)-\bar u
= u(x)-\frac{1}{|I|}\int_a^b u(y)\,dy
= \frac{1}{|I|}\int_a^b \big(u(x)-u(y)\big)\,dy.
\]
Taking absolute values and using the triangle inequality yields
\[
|u(x)-\bar u|
\le \frac{1}{|I|}\int_a^b |u(x)-u(y)|\,dy.
\]
For each $y\in(a,b)$, the fundamental theorem of calculus for the absolutely continuous representative gives
\[
|u(x)-u(y)|
= \left|\int_y^x u'(t)\,dt\right|
\le \int_{\min\{x,y\}}^{\max\{x,y\}} |u'(t)|\,dt
\le \int_a^b |u'(t)|\,dt
= \|u'\|_{L^1(I)}.
\]
Substituting this bound into the previous inequality, we obtain
\[
|u(x)-\bar u|
\le \frac{1}{|I|}\int_a^b \|u'\|_{L^1(I)}\,dy
= \|u'\|_{L^1(I)}.
\]
Since $x\in(a,b)$ was arbitrary, taking the supremum over $x$ gives Poincaré--Wirtinger's inequality, $\|u-\bar u\|_{L^\infty(I)} \le \|u'\|_{L^1(I)}$.
\end{proof}

\begin{proof}[\textbf{Proof of Equation~\eqref{eq: the dominating inequality related to Kegl et al.}}]
Let $\gamma(t):=(\cos t,\sin t)$ for $t\in[0,2\pi)$ and 
\begin{align*}
    \boldsymbol g(t)=\left(g_1(t),\ldots,g_D(t) \right):=\Bf(\gamma(t))\in\mathbb R^D.
\end{align*}
Then, $t\mapsto \boldsymbol g(t)$ is a parametrized differentiable curve \citep[e.g.,][Chapter 1]{doCarmo1976DifferentialGeometry}, and
\begin{align*}
\text{the length of the curve }\Bf(\mathbb{S}^1)\le\int_0^{2\pi}\|\boldsymbol g'(t)\|\,dt
&\le 2\pi\cdot \sup_{t\in[0,2\pi)}\|\boldsymbol g'(t)\|\\
&= 2\pi\cdot \sup_{t\in[0,2\pi)}\left(\sum_{j=1}^D \left|g_j'(t)\right|^2\right)^{1/2}\\
&\le 2\pi\cdot \left(\sum_{j=1}^D \sup_{t\in[0,2\pi)}\left|g_j'(t)\right|^2\right)^{1/2}.
\end{align*}
Now, fix $j\in\{1,\dots,D\}$. Since $g_j$ is $2\pi$--periodic, we have
\[
\overline{g_j'}:=\frac1{2\pi}\int_0^{2\pi}g_j'(t)\,dt
=\frac{g_j(2\pi)-g_j(0)}{2\pi}=0.
\]
Applying Poincaré--Wirtinger's inequality (Lemma \ref{lemma: Poincaré-Wirtinger's inequality}) to $g_j'$ yields
\[
\sup_{t\in[0,2\pi)}|g_j'(t)|
=\|g_j'-\overline{g_j'}\|_{L^\infty(0,2\pi)}
\le \int_0^{2\pi}|g_j''(t)|\,dt,
\]
which implies that
\[
\sup_{t\in[0,2\pi)}|g_j'(t)|^2
\le \left(\int_0^{2\pi}|g_j''(t)|\,dt\right)^2.
\]
Therefore, applying the Cauchy--Schwarz inequality, we have
\begin{align*}
\text{the length of the curve }\Bf(\mathbb{S}^1)
&\le 2\pi\cdot \left\{\sum_{j=1}^D \left(\int_0^{2\pi}\left|g_j''(t)\right|\,dt\right)^2\right\}^{1/2}\\
&\le 2\pi\cdot \left(\sum_{j=1}^D 2\pi\cdot\int_0^{2\pi}\left|g_j''(t)\right|^2\,dt\right)^{1/2} \\
&= (2\pi)^{3/2}\left(\int_0^{2\pi}\|\boldsymbol g''(t)\|^2\,dt\right)^{1/2}.
\end{align*}
Since $\mathbb S^1$ is endowed with the metric $g$ induced from $\mathbb R^2$, its Riemannian volume form $d\vol_g$ is the arc-length measure. Under the unit-speed parametrization $\gamma(t)=(\cos t,\sin t)$, we have
$d\vol_g=dt$. Moreover, viewing $\Bf=(f_1,\dots,f_D):\mathbb S^1\to\mathbb R^D$ componentwise, the (scalar) Hessian
$\nabla^2 f_j=\nabla(df_j)$ is a symmetric $(0,2)$-tensor, and the vector-valued Hessian is $\nabla^2\Bf=(\nabla^2 f_1,\dots,\nabla^2 f_D)=:\nabla(d\Bf)$. In the arc-length coordinate $t$, the Levi--Civita connection on $\mathbb S^1$ satisfies $\nabla_{\partial_t}\partial_t=0$. Hence, by \eqref{eq: def of Hess}, we have that
\begin{align*}
    (\nabla^2 f_j)(\partial_t,\partial_t)
=\partial_t(\partial_t f_j)-(\nabla_{\partial_t}\partial_t)f_j
=\frac{d^2}{dt^2}\big(f_j\circ\gamma\big)(t) = g_j''(t), \ \ \text{ for each }j.
\end{align*}
Therefore, $\nabla^2\Bf(\gamma(t))(\partial_t,\partial_t)=\boldsymbol g''(t)$, where $\boldsymbol g(t)=\Bf(\gamma(t))$. Consequently, using \eqref{eq: vert g norm using bases} and \eqref{eq: L2 norm of Hf}, we have that
\begin{align*}
    \int_0^{2\pi}\|\boldsymbol g''(t)\|^2\,dt
&=\sum_{j=1}^D\int_0^{2\pi}\left|(\nabla^2 f_j)(\partial_t,\partial_t)\right|^2\,dt \\
&=\sum_{j=1}^D\int_{\mathbb S^1}\big|\nabla^2 f_j\big|_g^2\,d\vol_g 
=\int_{\mathbb S^1}\big|\nabla^2\Bf\big|_g^2\,d\vol_g 
=\|\nabla^2\Bf\|_{L^2(\mathbb S^1)}^2.
\end{align*}
Consequently, $\text{the length of the curve }\Bf(\mathbb{S}^1) \le (2\pi)^{3/2}\,\|\nabla^2\Bf\|_{L^2(\mathbb S^1)}$.
\end{proof}

\begin{lemma}\label{lemma: a lemma for the pointwise orthogonal decomposition of the Hessian}
We adopt the notations defined and utilized in Section \ref{section: Geometric Interpretation of the Penalty}. In addition, $\overline{\nabla}$ denotes the Levi--Civita connection induced by the Euclidean metric $\delta$ on $\mathbb{R}^D$. Then, we have
\begin{align}\label{eq: a key identity for the Hessian decomposition}
    (\nabla d\boldsymbol{f})(u,v)
= \overline{\nabla}_{d\boldsymbol{f}(u)}(d\boldsymbol{f}(v)) - d\boldsymbol{f}(\nabla_u v).
\end{align}
\end{lemma}

\begin{proof}[\textbf{Proof of Lemma \ref{lemma: a lemma for the pointwise orthogonal decomposition of the Hessian}}]
Let $(x^1,\dots,x^d)$ be local coordinates on $\M$, and let
$(y^1,\dots,y^D)$ be the standard coordinates on $\mathbb{R}^D$.
Write vector fields on $\M$ as
\[
u = u^j \,\partial_j, \qquad v = v^i \,\partial_i,
\]
where $\partial_j := \frac{\partial}{\partial x^j}$ and we use the
Einstein summation convention.

Let $\boldsymbol{f} = (f^1,\dots,f^D):\M\to\mathbb{R}^D$. In these
coordinates, the differential of $\boldsymbol{f}$ is
\[
d\boldsymbol{f}(\partial_i)
  = \partial_i f^\alpha \,\partial_{y^\alpha},
\]
so for a general vector field $v = v^i\partial_i$, we have $d\boldsymbol{f}(v)
  = v^i \,\partial_i f^\alpha \,\partial_{y^\alpha}$. Let $\Gamma^k_{ij}$ be the Christoffel symbols of the Levi--Civita
connection $\nabla$ on $(\M,g)$. The Euclidean connection
$\overline{\nabla}$ on $\mathbb{R}^D$ is flat, so in the coordinates
$(y^\alpha)$ it is simply given by
\begin{align}\label{eq: Euclidean covariant derivative}
    \overline{\nabla}_W Z
 = W^\beta \,\partial_{y^\beta} Z^\alpha \,\partial_{y^\alpha}
\end{align}
for vector fields $W = W^\beta \partial_{y^\beta}$ and
$Z = Z^\alpha \partial_{y^\alpha}$ on $\mathbb{R}^D$.

\noindent\underline{Part 1: Computing $\overline{\nabla}_{d\boldsymbol{f}(u)}(d\boldsymbol{f}(v))$.}
Along the embedded submanifold $\boldsymbol{f}(\M)$, we have
\begin{align*}
    & d\boldsymbol{f}(u)
 = u^j \,\partial_j f^\beta \,\partial_{y^\beta}, \\
 & d\boldsymbol{f}(v)
 = v^i \,\partial_i f^\alpha \,\partial_{y^\alpha}.
\end{align*}
Hence, using \eqref{eq: Euclidean covariant derivative}, we have that
\begin{align*}
\overline{\nabla}_{d\boldsymbol{f}(u)}\big(d\boldsymbol{f}(v)\big)
 &= \Big(u^j \,\partial_j f^\beta\Big)
    \partial_{y^\beta}\!\Big( v^i \,\partial_i f^\alpha \Big)
    \,\partial_{y^\alpha}.
\end{align*}
Now note that the scalar function $\Phi^\alpha(x) := v^i(x)\,\partial_i f^\alpha(x)$ is defined on $\M$, and the vector field $d\boldsymbol{f}(u)$ is the
pushforward of $u$. By the chain rule, differentiating $\Phi^\alpha$
along $d\boldsymbol{f}(u) = u^j \,\partial_j f^\beta \,\partial_{y^\beta}$ on $\boldsymbol{f}(\M)$ is the same as
differentiating $\Phi^\alpha$ along $u = u^j \,\partial_j$ on $\M$, i.e.
\[
\Big(u^j \partial_j f^\beta\Big)
\partial_{y^\beta}\!\Big( v^i \partial_i f^\alpha \Big)
\Big|_{y = \boldsymbol{f}(x)}
 = u^j \,\partial_j \big( v^i \partial_i f^\alpha \big)(x).
\]
Therefore, we have that
\begin{align*}
\overline{\nabla}_{d\boldsymbol{f}(u)}\big(d\boldsymbol{f}(v)\big)
 = u^j\,\partial_j\big( v^i \,\partial_i f^\alpha \big)\,
    \partial_{y^\alpha} = u^j\Big( (\partial_j v^i)\,\partial_i f^\alpha
           + v^i\,\partial_j\partial_i f^\alpha \Big)\partial_{y^\alpha}.
\end{align*}

\medskip
\noindent\underline{Part 2: Computing $d\boldsymbol{f}(\nabla_u v)$.}
The covariant derivative of $v$ has components $(\nabla_u v)^k
 = u^j \,\partial_j v^k + \Gamma^k_{j\ell} u^j v^\ell$. Thus, we have that
\begin{align*}
d\boldsymbol{f}(\nabla_u v)
 = (\nabla_u v)^k \,\partial_k f^\alpha \,\partial_{y^\alpha} = \Big(u^j \,\partial_j v^k
         + \Gamma^k_{j\ell} u^j v^\ell \Big)
     \partial_k f^\alpha \,\partial_{y^\alpha}.
\end{align*}

\medskip
\noindent\underline{Part 3: Subtracting.}
We now compute $\overline{\nabla}_{d\boldsymbol{f}(u)}\big(d\boldsymbol{f}(v)\big)
 - d\boldsymbol{f}(\nabla_u v)$. Using the two expressions above, we obtain
\begin{align*}
\quad \overline{\nabla}_{d\boldsymbol{f}(u)}\big(d\boldsymbol{f}(v)\big)
 - d\boldsymbol{f}(\nabla_u v) = u^j\Big( (\partial_j v^i)\,\partial_i f^\alpha
           + v^i\,\partial_j\partial_i f^\alpha \Big)\partial_{y^\alpha}
 - \Big(u^j \,\partial_j v^k
         + \Gamma^k_{j\ell} u^j v^\ell \Big)
   \partial_k f^\alpha \,\partial_{y^\alpha}.
\end{align*}
The terms containing $\partial_j v^i\,\partial_i f^\alpha$ cancel, and
we are left with
\[
\overline{\nabla}_{d\boldsymbol{f}(u)}\big(d\boldsymbol{f}(v)\big)
 - d\boldsymbol{f}(\nabla_u v)
 = u^j v^i \Big( \partial_j\partial_i f^\alpha
                 - \Gamma^k_{ji}\,\partial_k f^\alpha \Big)
   \partial_{y^\alpha}.
\]
By definition, the covariant derivative $\nabla d\boldsymbol{f}$ is a
$(0,2)$-tensor with values in $\mathbb{R}^D$. In these coordinates, its
components are $(\nabla d\boldsymbol{f})^\alpha_{ji}
 = \partial_j\partial_i f^\alpha
    - \Gamma^k_{ji}\,\partial_k f^\alpha$. For vector fields $u = u^j\partial_j$ and $v = v^i\partial_i$, we
have
\[
(\nabla d\boldsymbol{f})(u,v)
 = u^j v^i\,(\nabla d\boldsymbol{f})^\alpha_{ji}\,\partial_{y^\alpha}
 = \overline{\nabla}_{d\boldsymbol{f}(u)}\big(d\boldsymbol{f}(v)\big)
   - d\boldsymbol{f}(\nabla_u v).
\]
This proves the identity in \eqref{eq: a key identity for the Hessian decomposition}.
\end{proof}

\begin{proof}[\textbf{Proof of Equation \eqref{eq: the key decomposition of the Hessian}}]
Recall that $\overline{\nabla}$ denotes the Levi--Civita connection induced by the Euclidean metric $\delta$ on $\mathbb{R}^D$. For vector fields $u,v\in\Gamma(T\M)$, the Euclidean covariant derivative of $d\boldsymbol{f}(v)$ along $d\boldsymbol{f}(u)$ decomposes orthogonally into tangent and normal components:
\begin{equation}\label{eq:Gauss}
\overline{\nabla}_{d\boldsymbol{f}(u)}\big(d\boldsymbol{f}(v)\big)
\;=\;
d\boldsymbol{f}\big(\nabla^{\boldsymbol{f}}_u v\big)
\;+\;
\II(u,v),
\end{equation}
where $\II(u,v) \in N_{\boldsymbol{f}(\boldsymbol{m})}\boldsymbol{f}(\M)$ is the second fundamental form of the embedded submanifold $\boldsymbol{f}(\M)\subset\mathbb{R}^D$. The decomposition in \eqref{eq:Gauss} is the Gauss formula in Riemannian geometry \citep[][Theorem 8.2]{lee2018riemannian}.

Combining \eqref{eq:Gauss} with \eqref{eq: a key identity for the Hessian decomposition}, we have
\begin{align*}
\nabla^2\boldsymbol{f}(u,v)=(\nabla d\boldsymbol{f})(u,v)
&= \overline{\nabla}_{d\boldsymbol{f}(u)}(d\boldsymbol{f}(v)) - d\boldsymbol{f}(\nabla_u v) \\
&= d\boldsymbol{f}(\nablaf_u v) + \II(u,v) - d\boldsymbol{f}(\nabla_u v) \\
&= \II(u,v) + d\boldsymbol{f}\big(\nablaf_u v - \nabla_u v\big).
\end{align*}
This proves the identity in \eqref{eq: the key decomposition of the Hessian}.
\end{proof}

\begin{proof}[\textbf{Proof of Equation \eqref{eq: the key L2 norm decomposition of the penalty}}]
Fix $\boldsymbol{m}\in\M$ and choose a $g$-orthonormal basis $\{e_i\}_{i=1}^d$ of $T_{\boldsymbol{m}}\M$.
Then, using \eqref{eq: vert g norm using bases} and \eqref{eq: pointwise orthogonal decomposition of the penalty term}, we have
\begin{align*}
    |\nabla^2 \boldsymbol{f}(\boldsymbol{m})|_g^2
=
\sum_{i,k=1}^d \|(\nabla^2 \boldsymbol{f})(e_i,e_k)\|_{\R^D}^2
&=
\sum_{i,k=1}^d \left\{\| \II(e_i,e_k) \|^2_{\mathbb{R}^D} + \| d\boldsymbol{f}\big(\nablaf_{e_i} e_k - \nabla_{e_i} e_k\big) \|^2_{\mathbb{R}^D}\right\}  \\
&=| \II(\boldsymbol{m}) |^2_g + \sum_{i,k=1}^d\left\| d\boldsymbol{f}\big(\nablaf_{e_i} e_k - \nabla_{e_i} e_k\big) \big\vert_{\boldsymbol{f}(\boldsymbol{m})} \right\|^2_{\mathbb{R}^D}
\end{align*}
Then, Equation \eqref{eq: the key L2 norm decomposition of the penalty} is obtained by applying the integral $\int_{\M} \;(\cdot)\; d\vol_g(\boldsymbol{m})$ to the preceding identity.
\end{proof}

\subsection{Proofs of Results in Appendix \ref{appendix: Reproducing Kernel Hilbert Spaces}}

\begin{proof}[\textbf{Proof of Lemma \ref{lemma: finite dimensionality of the kernel space}}]
Obviously, $\mathcal{H}_0$ is a linear subspace of $H^{2}(\M)=H^2(\M;\mathbb{R}^1)$.

\noindent\underline{Part 1: from $\nabla^{2}f=0$ to $\Delta f=0$ and interior smoothness.}
Let $f\in \mathcal{H}_0$. Since $f\in H^{2}(\M)$, its Hessian $\nabla^{2}f$ is well-defined as an $L^{2}$-tensor
(in the weak derivative sense). Taking the trace $\mathrm{tr}_g(\cdot)$ defined in \eqref{eq: def of Laplace-Beltrami operator}, we obtain
\[
\Delta f \;=\; \operatorname{tr}_{g}(\nabla^{2}f) \;=\; 0
\]
in the weak derivative sense on $\M$, where $\Delta$ is the Laplace--Beltrami operator.

By standard interior elliptic regularity for second-order elliptic equations \citep[][Chapter 6, Theorem 3, ``Infinite differentiability in the interior'']{evans1998pde}, the $H^2$ equivalence class $f$ has a representative that is smooth on $\mathrm{int}(\M)$; equivalently, it belongs to $C^\infty(U)$ for every open set $U\Subset\mathrm{int}(\M)$. On the interior, the classical Hessian of this representative agrees with its weak Hessian. Since the latter vanishes almost everywhere and the former is continuous, $\nabla^2f=0$ pointwise on $\mathrm{int}(\M)$.

\noindent\underline{Part 2: $\nabla^{2}f=0$ implies $\nabla f$ is parallel on $\mathrm{int}(\M)$.}
Let $X,Y$ be smooth vector fields on $\mathrm{int}(\M)$.
Recall:

\begin{itemize}
\item The gradient $\nabla f$ is characterized by
\[
df(Y)=Y(f)=\langle \nabla f,\,Y\rangle_g
\qquad \text{for all vector fields }Y.
\]

\item The Hessian is the covariant derivative of the $1$-form $df$:
\[
(\nabla^2 f)(X,Y):=(\nabla df)(X,Y)
= X\bigl(df(Y)\bigr) - df(\nabla_X Y).
\]
\end{itemize}
Therefore, $(\nabla^2 f)(X,Y)
= X\bigl(df(Y)\bigr) - df(\nabla_X Y) = X\bigl(\langle \nabla f,\,Y\rangle_g\bigr) - \langle \nabla f,\,\nabla_X Y\rangle_g$. Use the metric compatibility of the Levi--Civita connection, i.e., $\nabla g=0$, which yields the product rule $X\langle U,V\rangle_g
= \langle \nabla_X U,\,V\rangle_g + \langle U,\,\nabla_X V\rangle_g$ for all vector fields $U$ and $V$. Applying this with $U=\nabla f$ and $V=Y$ gives $X\bigl(\langle \nabla f,\,Y\rangle_g\bigr)
= \langle \nabla_X(\nabla f),\,Y\rangle_g + \langle \nabla f,\,\nabla_X Y\rangle_g$. Substituting back, we obtain
\begin{align*}
(\nabla^2 f)(X,Y)
= \Bigl(\langle \nabla_X(\nabla f),\,Y\rangle_g + \langle \nabla f,\,\nabla_X Y\rangle_g\Bigr)
 - \langle \nabla f,\,\nabla_X Y\rangle_g
= \langle \nabla_X(\nabla f),\,Y\rangle_g.
\end{align*}
Since $\nabla^{2}f\equiv 0$, it follows that $\langle \nabla_{X}(\nabla f),\,Y\rangle_{g} = 0$ for any $X, Y$. Hence,
\begin{align}\label{eq: nabla f is parallel}
    \nabla_{X}(\nabla f)=0\qquad \text{ for all }\,X \text{ on } \mathrm{int}(\M).
\end{align}

\noindent\underline{Part 3: injectivity of evaluation at one point.} Fix a point $p\in \mathrm{int}(\M)$ and define a map
\begin{align*}
    T:\ \ & \mathcal{H}_0 \rightarrow \mathbb{R}\times T_{p}\M,\\
& f \mapsto T(f):=\bigl(f(p),\,\nabla f(p)\bigr).
\end{align*}
Obviously, $T$ is linear. We claim that $T$ is injective.

Indeed, suppose $f\in \mathcal{H}_0$ such that $T(f)=0$, i.e., $f(p)=0$ and $\nabla f(p)=0$.
Let $q\in \mathrm{int}(\M)$ be arbitrary. Since $\M$ is connected, there exists a smooth curve
$\gamma:[0,1]\to \mathrm{int}(\M)$ with $\gamma(0)=p$ and $\gamma(1)=q$.
Along $\gamma$, \eqref{eq: nabla f is parallel} implies $\nabla_{\dot\gamma(t)}(\nabla f)=0$, i.e., $\nabla f(\gamma(t))$ is obtained by parallel transport of $\nabla f(p)=0$; hence, $\nabla f(\gamma(t)) = 0$ for all $t\in[0,1]$. Therefore, $\frac{d}{dt} f(\gamma(t))
= df_{\gamma(t)}(\dot\gamma(t))
= \langle \nabla f(\gamma(t)),\,\dot\gamma(t)\rangle_{g}
=0$. Thus, $f(\gamma(t))$ is constant in $t$, and since $f(\gamma(0))=f(p)=0$, we conclude $f(q)=0$. Because $q\in \mathrm{int}(\M)$ was arbitrary, the smooth interior representative is identically zero on $\mathrm{int}(\M)$. The smooth boundary $\partial\M$ has zero $d$-dimensional Riemannian volume, so $f=0$ almost everywhere on $\M$ and hence $f$ is the zero element of $H^2(\M)$. In the standing regime $d<4$, Theorem~\ref{thm: Rellich–Kondrachov embedding} identifies this class with its unique continuous representative, which is therefore zero everywhere on $\M$. Thus, $\ker T=\{0\}$, and $T$ is injective.

\noindent\underline{Part 4: dimension bound.}
Since $T$ is an injective linear map into a finite-dimensional space,
\[
\dim \mathcal{H}_0 \;\le\; \dim\bigl(\mathbb{R}\times T_{p}\M\bigr)
\;=\; 1+\dim \M.
\]
In particular, $\mathcal{H}_0$ is finite-dimensional.
\end{proof}

\begin{proof}[\textbf{Proof of the inequalities in \eqref{eq: Sobolev norm equivalence}}]
We first show that there exists \(C>0\) such that
\begin{equation}\label{eq: coercive estimate on H1 in proof}
\|u\|_{L^2(\M)}\le C\|\nabla^2 u\|_{L^2(\M)}
\qquad\text{for all }u\in\mathcal H_1.
\end{equation}
Suppose, by contradiction, that \eqref{eq: coercive estimate on H1 in proof} fails. Then there exists a sequence
\(\{\tilde{u}_n\}_{n\in\mathbb{N}}\subset\mathcal H_1\) such that $\Vert \tilde{u}_n\Vert_{L^2(\M)}> n\cdot\Vert \nabla^2\tilde{u}_n\Vert_{L^2(\M)}$. Define $\{u_n:=\tilde{u}_n/\Vert \tilde{u}_n\Vert_{L^2(\M)}\}_{n\in\mathbb{N}} \subseteq \mathcal{H}_1$, then
\begin{align}\label{eq: norm 1 hessian 0}
    \|u_n\|_{L^2(\M)}=1
\qquad\text{and}\qquad
\|\nabla^2 u_n\|_{L^2(\M)} < \frac{1}{n}\to 0.
\end{align}
Hence, \(\{u_n\}_{n\in\mathbb{N}}\) is bounded in \(H^2(\M)\). Theorem 3.18 of \cite{brezis2011functional} implies that there exist a subsequence $\{u_{n,k}\}_{k\in\mathbb{N}}$ and $u^*\in H^2(\M)$ such that $u_{n,k} \overset{w}{\rightharpoonup} u^*$ in $H^2(\M)$ as $k\rightarrow\infty$. Since $\mathcal{H}_1$ is assumed to be a closed linear subspace, Proposition 1.9 of \cite{brezis2011functional} implies
    \begin{align}\label{eq:iterated perp representation of H1}
        \mathcal{H}_1 = (\mathcal{H}_1^{\perp})^\perp.
    \end{align}
Since \(\{u_{n,k}\}_{k\in\mathbb{N}}\subseteq\mathcal{H}_1 =(\mathcal{H}_1^{\perp})^\perp \), for every \(v \in \mathcal H_1^\perp\), we have $\langle u_{n,k}, v\rangle_{H^2(\M)}=0$ for all $k\in\mathbb{N}$. Passing to the weak limit gives $\langle u^*, v\rangle_{H^2(\M)}=0$. Hence, \eqref{eq:iterated perp representation of H1} implies $u^* \in (\mathcal H_1^\perp)^\perp = \mathcal H_1$.

Theorem \ref{thm: Rellich–Kondrachov embedding} (specifically, the compactness of the embedding $\iota: H^2(\M)\rightarrow C(\M)$) implies $\lim_{k\rightarrow\infty} \Vert u_{n,k}-u^*\Vert_{C(\M)} = 0$, which further implies
\begin{align}\label{eq: L2 convergence of u nk}
    \lim_{k\rightarrow\infty} \Vert u_{n,k}-u^*\Vert_{L^2(\M)}=0.
\end{align}
The weak lower semicontinuity of the norm $\Vert \cdot \Vert_{H^2(\mathfrak{M})}$ \citep[][Proposition 3.5(iii)]{brezis2011functional} implies that
\begin{align*}
    \Vert u^*\Vert_{L^2(\M)}^2+\Vert \nabla^2 u^*\Vert_{L^2(\M)}^2 =\Vert u^*\Vert_{H^2(\M)}^2 &\le \liminf_{k\rightarrow\infty} \Vert u_{n,k}\Vert_{H^2(\M)}^2 \\
    & =\liminf_{k\rightarrow\infty} \left( \Vert u_{n,k}\Vert_{L^2(\M)}^2+\Vert \nabla^2 u_{n,k}\Vert_{L^2(\M)}^2 \right) \\
    & = \Vert u^*\Vert_{L^2(\M)}^2 + \liminf_{k\rightarrow\infty} \Vert \nabla^2 u_{n,k}\Vert_{L^2(\M)}^2.
\end{align*}
Then, we have $\Vert \nabla^2 u^*\Vert_{L^2(\M)}^2 \le \liminf_{k\rightarrow\infty} \Vert \nabla^2 u_{n,k}\Vert_{L^2(\M)}^2$. Using \eqref{eq: norm 1 hessian 0}, we have $\Vert \nabla^2 u^*\Vert_{L^2(\M)}=0$, which implies $\nabla^2 u^*=0$ almost everywhere on $\M$, i.e., $u^*\in\mathcal{H}_0$. Therefore,
\[
u^*\in\mathcal H_0\cap\mathcal H_1=\{0\}.
\]
However, \eqref{eq: norm 1 hessian 0} and \eqref{eq: L2 convergence of u nk} imply that 
\begin{align*}
    0=\Vert u^*\Vert_{L^2(\M)}=\lim_{k\rightarrow\infty} \Vert u_{n,k}\Vert_{L^2(\M)}=1,
\end{align*}
which is a contradiction. Hence \eqref{eq: coercive estimate on H1 in proof} holds.

Let \(f\in H^2(\M)\). By the direct-sum decomposition $H^2(\M)=\mathcal H_0\oplus \mathcal H_1$, write \(f=f_0+f_1\) with \(f_0\in\mathcal H_0\) and \(f_1\in\mathcal H_1\). Because \(\nabla^2 f_0=0\) almost everywhere, we have $\nabla^2 f=\nabla^2 f_1$ almost everywhere. By \eqref{eq: coercive estimate on H1 in proof}, $\|f_1\|_{L^2(\M)}\le C\|\nabla^2 f_1\|_{L^2(\M)}$. Hence,
\begin{align*}
\|f\|_{H^2(\M)} & \le \|f_0\|_{H^2(\M)} + \|f_1\|_{H^2(\M)} \\
&= \|f_0\|_{L^2(\M)} + \left( \|f_1\|_{L^2(\M)}^2 + \|\nabla^2f_1\|_{L^2(\M)}^2 \right)^{1/2}\\
&\le \|f_0\|_{L^2(\M)}+\|f_1\|_{L^2(\M)}+\|\nabla^2 f_1\|_{L^2(\M)} \\
&\le \|f_0\|_{L^2(\M)} + (C+1)\|\nabla^2 f_1\|_{L^2(\M)}  \\
&\le (C+1)\left( \|f_0\|_{L^2(\M)}+ \|\nabla^2 f_1\|_{L^2(\M)} \right) \\
&\le \sqrt{2} (C+1) \left( \|f_0\|_{L^2(\M)}^2 + \|\nabla^2 f_1\|_{L^2(\M)}^2 \right)^{1/2} \\
&= B\|f\|_R 
\end{align*}
where $B = \sqrt{2} (C+1) $. Conversely, since $\mathcal H_0$ and $\mathcal H_1$ are closed complementary subspaces of the Hilbert space $H^2(\M)$, the projections $P_0f=f_0$ and $P_1f=f_1$ are bounded. Therefore,
\begin{align*}
\|f\|_R
&=\bigl(\|f_0\|_{L^2(\M)}^2+\|\nabla^2 f_1\|_{L^2(\M)}^2\bigr)^{1/2} \\
& \le \bigl(\|f_0\|_{H^2(\M)}^2+\|f_1\|_{H^2(\M)}^2\bigr)^{1/2} \\
&\le \|P_0 f\|_{H^2(\M)} +\|P_1 f\|_{H^2(\M)} \\
&\le \bigl(\|P_0\|_{\mathrm{op}}+\|P_1\|_{\mathrm{op}}\bigr)\|f\|_{H^2(\M)} \\
&= A^{-1}\|f\|_{H^2(\M)}
\end{align*}
Here, $\|P_j\|_{\mathrm{op}}$ denotes the operator norm of $P_j$ on $H^2(\M)$, and 
\begin{align*}
    A:=\bigl(\|P_0\|_{\mathrm{op}}+\|P_1\|_{\mathrm{op}}\bigr)^{-1}>0.
\end{align*}
This proves the norm equivalence.
\end{proof}

\begin{proof}[\textbf{Proof of Corollary \ref{cor: minimizerkernels}}]
It suffices to show only the first case. For the template-adapted complement $\mathcal H_1=\{u\in H^2([0,1]):u(0)=u'(0)=0\}$, Section 1.2 of \cite{wahba1990spline} implies the following
\begin{itemize}
    \item $\mathcal{H}_0$ is spanned by the functions $\phi_1(t)=1, \phi_2(t) = t$.
    \item \begingroup\emergencystretch=2em
    The reproducing kernel $R_1$ is given by $R_1(s,t)= \int_0^1 G_2(t,u)G_2(s,u)du$,  where $G_2(t,u) = (t-u)_+$ and $t_+ := \max\{t,0\} = t\,\mathbbm{1}(t\ge0)$.\par\endgroup
\end{itemize}
We now compute a closed form of $R_1$ by evaluating the integral:
    \begin{align*}
        R_1(s,t) &= \int_0^1 (t-u)_+ (s-u)_+ du \\ 
        & =\int_0^{\min(s,t)}(t-u) (s-u)du \\ 
        & = \int_0^{\min(s,t)}( u^2 - u(s+t) + st )du\\
        &= \frac13 \min(s,t)^3 - \frac12\min(s,t)^2(s+t) + \min(s,t)(st).
    \end{align*}
\end{proof}

\newcommand{\Stwo}{\mbS^2}
\newcommand{\Ltwo}{L^2(\Stwo)}
\newcommand{\Cinftwo}{C^{\infty}(\Stwo)}
\newcommand{\Htwo}{H^2(\Stwo)}
\begin{proof}[\textbf{Proof of Lemma \ref{lemma: hessian laplace equivalence}}]
By the componentwise definitions of the vector-valued norms,
\begin{align*}
\norm{\nabla^2\Bf}_{L^2(\mbS^2)}
=\left(\sum_{j=1}^D\norm{\nabla^2f_j}_{L^2(\mbS^2)}^2\right)^{1/2},\qquad
\norm{\Delta\Bf}_{L^2(\mbS^2)}
=\left(\sum_{j=1}^D\norm{\Delta f_j}_{L^2(\mbS^2)}^2\right)^{1/2}.
\end{align*}
Therefore, it suffices to establish the following inequality for each component of $\Bf=(f_1,\ldots,f_D)$:
    \begin{align}\label{eq: goal of hessian laplace equivalence}
        \norm{\nabla^2 f}_{L^2(\mbS^2)} \leq \norm{\Delta f}_{L^2(\mbS^2)} \leq \sqrt 2 \norm{\nabla^2 f}_{L^2(\mbS^2)}\quad \text{ for all }f\in H^2(\mathbb{S}^2).
    \end{align}
    We first prove the result for smooth functions $f \in C^\infty(\mbS^2)$ and then show it holds for $H^2(\mbS^2)$. Let $g$ be the Riemannian metric associated with $\mbS^2$, induced by the Euclidean metric on $\mathbb{R}^3$. For notational purposes, denote $g(u,v)=\inner{u}{v}_g$ and $\lvert u\rvert_g = \sqrt{g(u,u)}$.
    
    For any $f \in C^{\infty}(\mbS^2)$, Bochner's formula \citep[][Chapter 3]{li2012geometric} gives
\begin{align}\label{eq: Bochner's formula in a proof}
    \int_{\mbS^2} (\Delta f)^2\, d\mathrm{vol}_g
=
\int_{\mbS^2} \left( \lvert \nabla^2 f \rvert^2 + \mathrm{Ric}(\nabla f,\nabla f) \right)\, d\mathrm{vol}_g.
\end{align}
On a sphere of radius $1$, the Riemannian curvature can be expressed as
$$R(X,Y)Z = (g(Y,Z)X - g(X,Z)Y)$$
for tangent vectors $X,Y,Z \in \Gamma(T\mbS^2)$ \citep[][Section 4.2.1]{petersen2006riemannian}. Let $p \in \mbS^2$ be arbitrary. Let $v,w \in T_p\mbS^2$ be arbitrary tangent vectors and $e_1,e_2 \in T_p\mbS^2$ be an orthonormal basis.  Applying the definition of Ricci curvature \citep[][Section 3.1.4]{petersen2006riemannian}, we have
\begin{align*}
\Ric(v,w) &= \sum_{i=1}^2  g_p(R(e_i, w)v, e_i) \\
 &= \sum_{i=1}^2 g_p\big( g_p(w, v)e_i - g_p(e_i, v)w, e_i \big)\\
 &= \sum_{i=1}^2 g_p(w,v) g_p(e_i, e_i) - \sum_{i=1}^2 g_p(e_i,v)g_p(w,e_i)\\
 &= 2 g_p(w,v) -  g_p(w,v) \\
 &= g_p(w,v).
\end{align*}
This implies $\Ric(\nabla f, \nabla f) = | \nabla f|_g^2$. Hence, \eqref{eq: Bochner's formula in a proof} implies
 $$\intsph{|\nabla^2f|_g^2}  = \intsph{(\Delta f)^2} - \intsph{ | \nabla f|_g^2},$$
and by rewriting in terms of norms,
\begin{align}\label{eq: inequality source}
 \norm{\nabla^2 f}_{L^2(\mbS^2)}^2 = \norm{\Delta f}_{L^2(\mbS^2)}^2 - \norm{\nabla f}_{L^2(\mbS^2)}^2.  
\end{align}
It follows immediately that
    \begin{align}\label{eq: inequality A}
\norm{\nabla^2 f}_{L^2(\mbS^2)}  \leq \norm{\Delta f}_{L^2(\mbS^2)}.
    \end{align}
We now show the other direction. We begin with a spectral decomposition of the Laplace-Beltrami operator on $\mbS^2$. Letting $f \in C^{\infty}(\mbS^2),$ the decomposition is
$$f = \sum_{\ell \geq 0} \sum_{m=-\ell}^\ell a_{\ell m}Y_{\ell m}$$
where $Y_{\ell m}$ are the spherical harmonics given by $- \Delta Y_{\ell m} = \lambda_\ell Y_{\ell m}$ for $\lambda_\ell = \ell (\ell + 1) \geq 2, \,\,\,\ell \geq 1$ \citep[][Section 2.2]{wahba1990spline}. Now, we evaluate the following norm using the eigenbasis expansion:
\begin{align*}
  \norm{\nabla f}_{L^2(\mbS^2)}^2 = \intsph{\lvert\nabla f\rvert_g^2}  
    &= \intsph{\left\lvert \sum_{\ell \geq 1} \sum_{m=-\ell}^\ell a_{\ell m}  \nabla Y_{\ell m}\right\rvert_g^2}  \\
    &= \sum_{\ell \geq 1} \sum_{k \geq 1} \sum_{m=-\ell}^\ell \sum_{n=-k}^k  a_{\ell m}a_{kn} \intsph{\inner{\nabla Y_{\ell m}}{\nabla Y_{kn}}_g}.\\
\end{align*}
By integration by parts, we furthermore have
\begin{align*}
    \norm{\nabla f}_{L^2(\mbS^2)}^2 &= -\sum_{\ell \geq 1} \sum_{k \geq 1} \sum_{m=-\ell}^\ell \sum_{n=-k}^k  a_{\ell m}a_{kn} \intsph{ Y_{\ell m}\Delta Y_{kn}}\\
    &= \sum_{\ell \geq 1} \sum_{k \geq 1} \sum_{m=-\ell}^\ell \sum_{n=-k}^k  a_{\ell m}a_{kn} \lambda_k \intsph{ Y_{\ell m} Y_{kn}}\\
    &= \sum_{\ell \geq 1} \sum_{k \geq 1} \sum_{m=-\ell}^\ell \sum_{n=-k}^k  a_{\ell m}a_{kn} \lambda_k   \delta_{\ell k} \delta_{mn} \\
    &= \sum_{\ell \geq 1} \sum_{m = -\ell}^\ell a_{\ell m}^2 \lambda_{\ell}.
\end{align*}
We perform a similar calculation on the Laplace-Beltrami term:
\begin{align*}
  \norm{\Delta f}_{L^2(\mbS^2)}^2 = \intsph{\lvert\Delta f\rvert_g^2} 
    &= \intsph{\left\lvert \sum_{\ell \geq 1} \sum_{m=-\ell}^\ell a_{\ell m}  \Delta Y_{\ell m}\right\rvert_g^2}  \\
     &= \intsph{\left\lvert - \sum_{\ell \geq 1} \sum_{m=-\ell}^\ell a_{\ell m}   \lambda_{\ell} Y_{\ell m}\right\rvert_g^2}  \\
    &= \sum_{\ell \geq 1} \sum_{k \geq 1} \sum_{m=-\ell}^\ell \sum_{n=-k}^k  a_{\ell m}a_{kn} \lambda_k \lambda_\ell \intsph{ Y_{\ell m} Y_{kn}}\\
    &= \sum_{\ell \geq 1} \sum_{k \geq 1} \sum_{m=-\ell}^\ell \sum_{n=-k}^k  a_{\ell m}a_{kn} \lambda_k  \lambda_\ell \delta_{\ell k} \delta_{mn} \\
    &= \sum_{\ell \geq 1} \sum_{m = -\ell}^\ell a_{\ell m}^2 \lambda_{\ell}^2.
\end{align*}
We compare the two norm calculations above. Noting that for all $\ell \geq 1$, we have $\lambda_\ell\geq 2$, and thus
$$ \norm{\nabla f}_{L^2(\mbS^2)}^2 = \sum_{\ell \geq 1} \sum_{m = -\ell}^\ell a_{\ell m}^2 \lambda_{\ell} \leq \frac12 \sum_{\ell \geq 1} \sum_{m = -\ell}^\ell a_{\ell m}^2 \lambda_{\ell}^2 = \frac12 \norm{\Delta f}_{L^2(\mbS^2)}^2.$$
Substituting this back into \eqref{eq: inequality source},
$$\norm{\nabla^2 f}_{L^2(\mbS^2)}^2 = \norm{\Delta f}_{L^2(\mbS^2)}^2 - \norm{\nabla f}_{L^2(\mbS^2)}^2 \geq \frac12 \norm{\Delta f}_{L^2(\mbS^2)}^2.$$
Then,
\begin{align}\label{eq: inequality B}
\norm{\Delta f}_{L^2(\mbS^2)} \leq \sqrt 2 \norm{\nabla^2 f}_{L^2(\mbS^2)}.
\end{align}
Combining \eqref{eq: inequality A} and \eqref{eq: inequality B} gives the desired result on functions in $\Cinftwo$:
\begin{align*}
\norm{\nabla^2 f}_{\Ltwo}  \leq \norm{\Delta f}_{L^2(\mbS^2)} \leq \sqrt 2 \norm{\nabla^2 f}_{\Ltwo}.
\end{align*}
Since $C^{\infty}(\mathbb{S}^2)$ is dense in $H^2(\mathbb{S}^2)$, \eqref{eq: goal of hessian laplace equivalence} holds for every scalar $f\in H^2(\mathbb{S}^2)$. Applying it to each component $f_j$, squaring, summing over $j$, and taking square roots proves the asserted vector-valued inequalities.
\end{proof}

\begin{proof}[\textbf{Proof of Lemma \ref{lemma:  closed form penalty circle}}]

Recall the decomposition $H^2(\M) = \mcH_0 \oplus \mcH_1$ from Section \ref{appendix: An Orthogonal Decomposition of H2}. By the representer theorem (Theorem \ref{thm: wahba representation theorem}), the optimal $\Bf_{N,\lambda}^{(n+1)}=(f_{N,\lambda,1}^{(n+1)}, \ldots, f_{N,\lambda, D}^{(n+1)})$ at a given iteration $n$ can be written as
$$f_{N,\lambda,j}^{(n+1)} = q_j + \sum_{i=1}^N \alpha_{ij} R_1(\cdot, \boldsymbol{m}_i) \quad \text{ for all }j\in[D],$$
where $\{q_j\}_{j\in[D]} \subseteq \mcH_0$ and $\{\alpha_{ij}\}_{i \in [N],\, j \in [D]} \subseteq \mbR.$ Applying the Hessian operator, we have
$$\nabla^2 f_{N,\lambda,j}^{(n+1)} = \sum_{i=1}^N \alpha_{ij}  \nabla^2 R_1(\cdot, \boldsymbol{m}_i).$$
We then compute the norm of this term.
\begin{align*}
 \norm{\nabla^2 f_{N,\lambda,j}^{(n+1)} }_{L^2(\M)}^2 &= \inner{\nabla^2 f_{N,\lambda,j}^{(n+1)}}{\nabla^2 f_{N,\lambda,j}^{(n+1)} }_{L^2(\M)}\\
 &=\sum_{i=1}^N \sum_{\ell=1}^N \alpha_{ij} \alpha_{\ell,j} \inner{\nabla^2 R_1(\cdot, \boldsymbol{m}_i)}{\nabla^2 R_1(\cdot, \boldsymbol{m}_\ell)}_{L^2(\M)}\\
 &= \sum_{i=1}^N \sum_{\ell=1}^N \alpha_{ij} \alpha_{\ell,j}  \inner{R_1(\cdot, \boldsymbol{m}_i)}{R_1(\cdot, \boldsymbol{m}_\ell)}_{R_1}\\
 &= \sum_{i=1}^N \sum_{\ell=1}^N \alpha_{ij} \alpha_{\ell,j}   R_1(\boldsymbol{m}_i, \boldsymbol{m}_\ell),
\end{align*}
where the last equality follows from \eqref{eq: L2 representation of R1}. Since $(\boldsymbol{K})_{ij} = R_1(\boldsymbol{m}_i,\boldsymbol{m}_j)$, we have 
\begin{align*}
    \norm{\nabla^2 f_{N,\lambda,j}^{(n+1)} }_{L^2(\M)}^2 = \boldsymbol{\alpha}_{j}^\T\boldsymbol{K}\boldsymbol{\alpha}_j,
\end{align*}
which implies \eqref{eq: closed form penalty circle}.
\end{proof}

\begin{proof}[\textbf{Proof of Proposition~\ref{prop: empirical PA sup norm bound}}]
Fix an iteration \(n\), we write
\begin{align*}
\boldsymbol m_i:=\boldsymbol m_i^{(n)}
=\Bpi_{\Bf_{N,\lambda}^{(n)}}(\boldsymbol X_i), \qquad
 J_n(\Bf):=
\frac1N\sum_{i=1}^N\Vert\boldsymbol X_i-\Bf(\boldsymbol m_i)\Vert^2
+\lambda\Vert\nabla^2\Bf\Vert_{L^2(\M)}^2 .
\end{align*}
Here, \(J_n\) is the fixed-projection objective for iteration \(n\); the projection sites \(\{\boldsymbol m_i^{(n)}\}_{i=1}^N\) are held fixed in this update. Let \(\boldsymbol{\widehat{f}}:=\Bf_{N,\lambda}^{(n+1)}\). Since
\(\boldsymbol{\widehat{f}}\) is the minimizer in the relaxed update \eqref{eq: argmin via RKHS}, it minimizes \(J_n\) over \(H^2(\M;\mathbb R^D)\). The assumption in \eqref{eq: def of tau star} implies
\[
\lambda_{\min}\left(N^{-1}\boldsymbol T_n^\T\boldsymbol T_n\right)\ge \tau_*>0.
\]
Hence, \(\boldsymbol T_n^\T\boldsymbol T_n\) is positive definite and \(\boldsymbol T_n\) has full column rank. Equivalently, no nonzero function in \(\mathcal H_0\) can vanish at all projection sites \(\boldsymbol m_1,\ldots,\boldsymbol m_N\). This rules out the unpenalized null-space ambiguity: the fitted null-space values
\begin{align*}
    \left( g_{0,n+1,j}(\boldsymbol m_1),\ldots,g_{0,n+1,j}(\boldsymbol m_N) \right)^\T = \boldsymbol T_n\boldsymbol\theta_j
\end{align*}
determine the coefficient vector \(\boldsymbol\theta_j\) uniquely for each component \(j\). Thus, the null-space part of the representer decomposition is identifiable at this iteration. In the notation of Theorem~\ref{thm: wahba representation theorem}, the kernel
matrix \(\boldsymbol K\) is positive semidefinite, and \(\lambda>0\). Hence
\(\boldsymbol M=\boldsymbol K+N\lambda\boldsymbol I_N\) is positive definite,
so the full column rank of \(\boldsymbol T_n\) implies that
\(\boldsymbol T_n^\T\boldsymbol M^{-1}\boldsymbol T_n\) is nonsingular. Therefore, the rank condition in the representer theorem (Theorem~\ref{thm: wahba representation theorem})
holds, and the relaxed update decomposes as
\[
\boldsymbol{\widehat{f}}=\boldsymbol g_{0,n+1}+\boldsymbol h_{n+1},
\]
where \(\boldsymbol g_{0,n+1}=(g_{0,n+1,1},\ldots,g_{0,n+1,D})^\T\in\mathcal H_0^D\),
\(\boldsymbol h_{n+1}=(h_{n+1,1},\ldots,h_{n+1,D})^\T\in\mathcal H_1^D\), and,
componentwise for \(j=1,\ldots,D\),
\begin{align*}
     g_{0,n+1,j}(\boldsymbol m)=\sum_{r=1}^{d_0}\theta_{rj}\phi_r(\boldsymbol m), \qquad
 h_{n+1,j}(\boldsymbol m)=\sum_{i=1}^N\alpha_{ij}R_1(\boldsymbol m,\boldsymbol m_i).
\end{align*}
We first bound the null-space component pointwise. Let
\(\boldsymbol\Theta=(\theta_{rj})_{r,j}\in\mathbb R^{d_0\times D}\), and set
\[
\boldsymbol\phi(\boldsymbol m):=(\phi_1(\boldsymbol m),\ldots,\phi_{d_0}(\boldsymbol m))^\T .
\]
The componentwise formula for \(\boldsymbol g_{0,n+1}\) is equivalent to
\[
\boldsymbol g_{0,n+1}(\boldsymbol m)=\boldsymbol\Theta^\T\boldsymbol\phi(\boldsymbol m).
\]
Therefore, for every \(\boldsymbol m\in\M\), the definition of the matrix
operator norm gives
\[
\Vert\boldsymbol g_{0,n+1}(\boldsymbol m)\Vert
\allowbreak =
\Vert\boldsymbol\Theta^\T\boldsymbol\phi(\boldsymbol m)\Vert
\le \Vert\boldsymbol\Theta^\T\Vert_{\operatorname{op}}\,
\Vert\boldsymbol\phi(\boldsymbol m)\Vert .
\]
Here, \(\Vert\cdot\Vert_{\operatorname{op}}\) denotes the matrix operator norm induced by the Euclidean vector norm. Since this operator norm is bounded by the Frobenius norm \citep[e.g.,][Proposition 9.1]{fan2020statistical} and
\[
\Vert\boldsymbol\phi(\boldsymbol m)\Vert
=\left(\sum_{r=1}^{d_0}\phi_r(\boldsymbol m)^2\right)^{1/2},
\]
we obtain
\[
\Vert\boldsymbol g_{0,n+1}(\boldsymbol m)\Vert
\le
\left(\sum_{r=1}^{d_0}\phi_r(\boldsymbol m)^2\right)^{1/2}
\Vert\boldsymbol\Theta\Vert_F
\le b_0\Vert\boldsymbol\Theta\Vert_F .
\]
Taking the supremum over \(\boldsymbol m\in\M\) yields
\begin{equation}\label{eq:prop31-null-prelim}
\Vert\boldsymbol g_{0,n+1}\Vert_{C(\M)}
\le b_0\Vert\boldsymbol\Theta\Vert_F .
\end{equation}

We next bound the penalty-space component \(\boldsymbol h_{n+1}\). For fixed
\(\boldsymbol m\in\M\) and fixed component \(j\), the reproducing property in
the RKHS \(\mathcal H_1\) gives
\[
h_{n+1,j}(\boldsymbol m)
=\langle h_{n+1,j},R_1(\cdot,\boldsymbol m)\rangle_{R_1}.
\]
Applying Cauchy--Schwarz in \(\mathcal H_1\) gives
\[
|h_{n+1,j}(\boldsymbol m)|
\le
\Vert h_{n+1,j}\Vert_{R_1}\,
\Vert R_1(\cdot,\boldsymbol m)\Vert_{R_1}.
\]
Using the reproducing property once more,
\[
\Vert R_1(\cdot,\boldsymbol m)\Vert_{R_1}^2
=\langle R_1(\cdot,\boldsymbol m),R_1(\cdot,\boldsymbol m)\rangle_{R_1}
=R_1(\boldsymbol m,\boldsymbol m).
\]
Thus,
\[
|h_{n+1,j}(\boldsymbol m)|
\le \sqrt{R_1(\boldsymbol m,\boldsymbol m)}\,\Vert h_{n+1,j}\Vert_{R_1}.
\]
Summing over components and using the definition of \(\kappa\), we get
\[
\Vert\boldsymbol h_{n+1}(\boldsymbol m)\Vert^2
=\sum_{j=1}^D |h_{n+1,j}(\boldsymbol m)|^2
\le
R_1(\boldsymbol m,\boldsymbol m)\sum_{j=1}^D\Vert h_{n+1,j}\Vert_{R_1}^2
\le
\kappa^2\sum_{j=1}^D\Vert h_{n+1,j}\Vert_{R_1}^2.
\]
Taking square roots and then the supremum over \(\boldsymbol m\) gives
\begin{equation}\label{eq:prop31-h-rkhs-sup}
\Vert\boldsymbol h_{n+1}\Vert_{C(\M)}
\le
\kappa\left(\sum_{j=1}^D\Vert h_{n+1,j}\Vert_{R_1}^2\right)^{1/2}.
\end{equation}
The coefficient form of this RKHS norm is also explicit: if
\(\boldsymbol\alpha_j=(\alpha_{1j},\ldots,\alpha_{Nj})^\T\), then
\[
\Vert h_{n+1,j}\Vert_{R_1}^2
=
\left\langle \sum_{i=1}^N\alpha_{ij}R_1(\cdot,\boldsymbol m_i),
\sum_{\ell=1}^N\alpha_{\ell j}R_1(\cdot,\boldsymbol m_\ell)\right\rangle_{R_1}
=\boldsymbol\alpha_j^\T\boldsymbol K\boldsymbol\alpha_j.
\]

We now bound the penalty-space norm. Since the zero map belongs to
\(H^2(\M;\mathbb R^D)\), optimality of the relaxed update implies
\[
J_n(\boldsymbol{\widehat{f}})
\le J_n(\boldsymbol 0)
=\frac1N\sum_{i=1}^N\Vert\boldsymbol X_i\Vert^2
=R_N^2 .
\]
Both terms in \(J_n\) are nonnegative. Hence the data-fit term and the penalty
term are each bounded by \(R_N^2\):
\begin{equation}\label{eq:prop31-fit-penalty-bound}
\begin{aligned}
\frac1N\sum_{i=1}^N\Vert\boldsymbol X_i-\boldsymbol{\widehat{f}}(\boldsymbol m_i)\Vert^2
\le R_N^2,
\qquad
\lambda\Vert\nabla^2\boldsymbol{\widehat{f}}\Vert_{L^2(\M)}^2\le R_N^2.
\end{aligned}
\end{equation}
Because \(\boldsymbol g_{0,n+1}\in\mathcal H_0^D\), its Hessian vanishes
componentwise. Therefore
\(\nabla^2\boldsymbol{\widehat{f}}=\nabla^2\boldsymbol h_{n+1}\), and by the definition
of the \(R_1\)-inner product on \(\mathcal H_1\), together with the penalty
part of \eqref{eq:prop31-fit-penalty-bound},
\begin{equation}\label{eq:prop31-penalty-norm-bound}
\sum_{j=1}^D\Vert h_{n+1,j}\Vert_{R_1}^2
=\Vert\nabla^2\boldsymbol h_{n+1}\Vert_{L^2(\M)}^2
=\Vert\nabla^2\boldsymbol{\widehat{f}}\Vert_{L^2(\M)}^2
\le \frac{R_N^2}{\lambda}.
\end{equation}
Combining \eqref{eq:prop31-penalty-norm-bound} with the RKHS evaluation bound
\eqref{eq:prop31-h-rkhs-sup} yields
\begin{equation}\label{eq:prop31-h-sup-bound}
\Vert\boldsymbol h_{n+1}\Vert_{C(\M)}
\le \frac{\kappa R_N}{\sqrt{\lambda}}.
\end{equation}

It remains to control \(\Vert\boldsymbol\Theta\Vert_F\), which determines the
size of the null-space component. Let
\[
\boldsymbol z_i:=\boldsymbol{\widehat{f}}(\boldsymbol m_i),\qquad
\boldsymbol Z:=(\boldsymbol z_1,\ldots,\boldsymbol z_N)^\T\in\mathbb R^{N\times D}.
\]
For each \(i\), write
\(\boldsymbol z_i=\boldsymbol X_i-(\boldsymbol X_i-\boldsymbol z_i)\). The
elementary inequality \(\Vert u-v\Vert^2\le 2\Vert u\Vert^2+2\Vert v\Vert^2\)
and the data-fit part of \eqref{eq:prop31-fit-penalty-bound} therefore give
\begin{equation}\label{eq:prop31-z-average-bound}
\frac1N\sum_{i=1}^N\Vert\boldsymbol z_i\Vert^2
\le
\frac2N\sum_{i=1}^N\Vert\boldsymbol X_i\Vert^2
+\frac2N\sum_{i=1}^N\Vert\boldsymbol X_i-\boldsymbol z_i\Vert^2
\le 4R_N^2.
\end{equation}
Thus, \eqref{eq:prop31-z-average-bound} implies
\begin{equation}\label{eq:prop31-Z-fro-bound}
\Vert\boldsymbol Z\Vert_F^2=\sum_{i=1}^N\Vert\boldsymbol z_i\Vert^2
\le 4NR_N^2,
\qquad
\Vert\boldsymbol Z\Vert_F\le 2\sqrt N R_N.
\end{equation}
Let \(\boldsymbol H\in\mathbb R^{N\times D}\) be the matrix whose \(i\)th row is
\(\boldsymbol h_{n+1}(\boldsymbol m_i)^\T\), and let
\(\boldsymbol U:=\boldsymbol T_n\boldsymbol\Theta\). The \(i\)th row of
\(\boldsymbol U\) is \(\boldsymbol g_{0,n+1}(\boldsymbol m_i)^\T\), so
\(\boldsymbol Z=\boldsymbol U+\boldsymbol H\). Consequently, by the triangle
inequality for the Frobenius norm, the definition of the sup norm,
\eqref{eq:prop31-Z-fro-bound}, and \eqref{eq:prop31-h-sup-bound},
\begin{equation}\label{eq:prop31-U-bound}
\Vert\boldsymbol U\Vert_F
\le \Vert\boldsymbol Z\Vert_F+\Vert\boldsymbol H\Vert_F
\le
2\sqrt N R_N+\sqrt N\,\Vert\boldsymbol h_{n+1}\Vert_{C(\M)}
\le
\sqrt N R_N\left(2+\frac{\kappa}{\sqrt{\lambda}}\right).
\end{equation}
Finally, the inequality in the assumption \eqref{eq: def of tau star} controls
the coefficient matrix \(\boldsymbol\Theta\) through the empirical design matrix.
Indeed,
\begin{equation}\label{eq:prop31-design-lower-bound}
\Vert\boldsymbol T_n\boldsymbol\Theta\Vert_F^2
=\operatorname{tr}\bigl(\boldsymbol\Theta^\T\boldsymbol T_n^\T\boldsymbol T_n\boldsymbol\Theta\bigr)
\allowbreak =
\sum_{j=1}^D\boldsymbol\theta_j^\T\boldsymbol T_n^\T\boldsymbol T_n\boldsymbol\theta_j
\ge
N\tau_*\sum_{j=1}^D\Vert\boldsymbol\theta_j\Vert^2
=N\tau_*\Vert\boldsymbol\Theta\Vert_F^2,
\end{equation}
where \(\boldsymbol\theta_j\) denotes the \(j\)th column of \(\boldsymbol\Theta\).
Since \(\boldsymbol U=\boldsymbol T_n\boldsymbol\Theta\), combining
\eqref{eq:prop31-U-bound} and \eqref{eq:prop31-design-lower-bound} implies
\begin{equation}\label{eq:prop31-theta-bound}
\Vert\boldsymbol\Theta\Vert_F
\le
\frac{\Vert\boldsymbol U\Vert_F}{\sqrt{N\tau_*}}
\le
\tau_*^{-1/2}R_N\left(2+\frac{\kappa}{\sqrt{\lambda}}\right).
\end{equation}
Using \eqref{eq:prop31-null-prelim} and \eqref{eq:prop31-theta-bound}, we obtain
\begin{equation}\label{eq:prop31-g-sup-bound}
\Vert\boldsymbol g_{0,n+1}\Vert_{C(\M)}
\le
b_0\tau_*^{-1/2}R_N\left(2+\frac{\kappa}{\sqrt{\lambda}}\right).
\end{equation}
Consequently, \eqref{eq:prop31-g-sup-bound} and \eqref{eq:prop31-h-sup-bound}
give
\[
\Vert\boldsymbol{\widehat{f}}\Vert_{C(\M)}
\le
\Vert\boldsymbol g_{0,n+1}\Vert_{C(\M)}
+\Vert\boldsymbol h_{n+1}\Vert_{C(\M)}
\le
R_N\left[
\frac{\kappa}{\sqrt{\lambda}}
+b_0\tau_*^{-1/2}\left(2+\frac{\kappa}{\sqrt{\lambda}}\right)
\right] =: \Upsilon_\star.
\]
This is the asserted bound. If \(\Upsilon\ge\Upsilon_\star\), then
\(\Vert\boldsymbol{\widehat{f}}\Vert_{C(\M)}\le\Upsilon\), so the relaxed update belongs
to \(\mathscr F\). Since the relaxed problem minimizes over the larger space
\(H^2(\M;\mathbb R^D)\), a feasible relaxed minimizer is also a minimizer over
the constrained class \(\mathscr F\).
\end{proof}

\subsection{Proofs of the Results in Appendix~\ref{section: Latent-Variable Model and Tuning Parameter Selection via Oracle Criterion}}

\begin{proof}[\textbf{Proof of Theorem~\ref{thm: the theorem for lambda selection}}]
First consider a different representative $\widetilde{\Bf}_\lambda^*=\Bf_\lambda^*\circ\phi$ of the minimizing orbit. By the singleton projection assumption, and choosing the transported conditional-mean version,
\[
    \widetilde{\boldsymbol M}_\lambda=\phi^{-1}(\boldsymbol M_\lambda),
    \qquad
    \widetilde{\boldsymbol R}_\lambda=\boldsymbol R_\lambda,
    \qquad
    \widetilde\Phi_\lambda=\Phi_\lambda\circ\phi
    \quad \nu_g\text{-a.e.}
\]
Every $\phi\in\operatorname{Isom}(\M,g)$ preserves $\nu_g$, so $\widetilde{\overline\Phi}_\lambda=\overline\Phi_\lambda$ and the centered $L^2(\nu_g)$ norms agree. Hence $\mathcal C(\lambda)$ is independent of the orbit representative.

Let $\widetilde{\Phi}_\lambda$ be another Borel version of the conditional expectation in~\eqref{eq: def of residual function}, and set
\[
    A_\lambda
    :=\{\boldsymbol{m}\in\M:
    \Phi_\lambda(\boldsymbol{m})\neq
    \widetilde{\Phi}_\lambda(\boldsymbol{m})\}.
\]
The two versions agree $\mu_\lambda$-almost everywhere. By Assumption~\ref{assumption: the assumption for lambda selection}(i),
\[
    0=\mu_\lambda(A_\lambda)
    =\int_{A_\lambda}p_\lambda\,d\nu_g
    \geq c_M\nu_g(A_\lambda),
\]
so $\nu_g(A_\lambda)=0$. Hence, all $\nu_g$-integrals defining $\overline{\Phi}_\lambda$ and $\mathcal{C}(\lambda)$ are independent of the conditional-expectation version.

Assumption~\ref{assumption: the assumption for lambda selection}(ii) gives $\Phi_\lambda\in L^2(\nu_g)$ and $\overline{\Phi}_\lambda\geq b>0$. Thus, $\mathcal{C}(\lambda)$ is finite and well defined. Moreover, because $\nu_g$ is a probability measure,
\[
    |\overline{\Phi}_\lambda-\overline{\Phi}_{\lambda'}|
    \leq d_{2,\mathrm{iso}}([\Phi_\lambda],[\Phi_{\lambda'}]).
\]
The reverse triangle inequality then shows that
\[
    \left|
    \Vert\Phi_\lambda-\overline{\Phi}_\lambda\Vert_{L^2(\nu_g)}
    -\Vert\Phi_{\lambda'}-\overline{\Phi}_{\lambda'}\Vert_{L^2(\nu_g)}
    \right|
    \leq2d_{2,\mathrm{iso}}([\Phi_\lambda],[\Phi_{\lambda'}]).
\]
The quotient continuity in Assumption~\ref{assumption: the assumption for lambda selection}(ii), together with the positive denominator, therefore implies that $\mathcal{C}$ is continuous on $\Lambda$. To see uniqueness explicitly, if $\lambda\neq\lambda_0$, take $\epsilon=|\lambda-\lambda_0|$. Then Assumption~\ref{assumption: the assumption for lambda selection}(iii) gives $\mathcal{C}(\lambda)-\mathcal{C}(\lambda_0) \geq\kappa_\epsilon>0$. Thus, $\lambda_0$ is the unique minimizer, and~\eqref{eq: choice of the optimal lambda} gives $\lambda^*=\lambda_0$.

It remains to prove the conditional sample conclusion. Put
\[
    e_N:=\sup_{\lambda\in\Lambda}
    d_{2,\mathrm{iso}}([\widehat\Phi_{N,\lambda}],[\Phi_\lambda]).
\]
By the sample assumption in the theorem,
$e_N\overset{\mathbb{P}^*}{\longrightarrow}0$. Aligning the two profiles within their isometry orbits and using invariance of $\nu_g$ gives
\[
    \sup_{\lambda\in\Lambda}
    |\widehat{\overline{\Phi}}_{N,\lambda}-\overline{\Phi}_\lambda|
    \leq e_N.
\]
For all sufficiently large $N$, $\tau_N<b/2$. Whenever also $e_N<b/2$, every estimated mean is at least $b/2$, so the truncation in the denominator is inactive. The exceptional set $\{e_N\geq b/2\}$ has outer probability tending to zero because it is contained in $\{e_N>b/4\}$. The reverse triangle inequality also gives, uniformly in $\lambda$,
\[
    \left|
    \Vert\widehat{\Phi}_{N,\lambda}-\widehat{\overline{\Phi}}_{N,\lambda}\Vert_{L^2(\nu_g)}
    -\Vert\Phi_\lambda-\overline{\Phi}_\lambda\Vert_{L^2(\nu_g)}
    \right|
    \leq2e_N.
\]
Quotient continuity of $\lambda\mapsto[\Phi_\lambda]$ and compactness of $\Lambda$ imply that
\[
    S:=\sup_{\lambda\in\Lambda}
    \Vert\Phi_\lambda-\overline{\Phi}_\lambda\Vert_{L^2(\nu_g)}<\infty.
\]
Set
\[
    a_\lambda
    :=\Vert\Phi_\lambda-\overline{\Phi}_\lambda\Vert_{L^2(\nu_g)},
    \qquad
    \widehat{a}_{N,\lambda}
    :=\Vert\widehat{\Phi}_{N,\lambda}
    -\widehat{\overline{\Phi}}_{N,\lambda}\Vert_{L^2(\nu_g)}.
\]
Consequently, whenever $e_N<b/2$ and $N$ is sufficiently large, uniformly in $\lambda$,
\begin{align*}
    \big|\widehat{\mathcal{C}}_N(\lambda)-\mathcal{C}(\lambda)\big|
    \leq
    \frac{|\widehat{a}_{N,\lambda}-a_\lambda|}
    {\widehat{\overline{\Phi}}_{N,\lambda}}
    +\frac{a_\lambda
    |\widehat{\overline{\Phi}}_{N,\lambda}-\overline{\Phi}_\lambda|}
    {\widehat{\overline{\Phi}}_{N,\lambda}\overline{\Phi}_\lambda}\leq
    \left(\frac{4}{b}+\frac{2S}{b^2}\right)e_N.
\end{align*}
Hence, on the same set,
\[
    \Delta_N
    :=\sup_{\lambda\in\Lambda}
    |\widehat{\mathcal{C}}_N(\lambda)-\mathcal{C}(\lambda)|
    \leq K e_N,
    \qquad
    K:=\frac{4}{b}+\frac{2S}{b^2}.
\]
Thus, for every $\eta>0$ and all sufficiently large $N$,
\[
    \{\Delta_N>\eta\}
    \subseteq
    \{e_N\geq b/2\}\cup\{K e_N>\eta\}.
\]
By subadditivity of outer probability,
$\Delta_N\overset{\mathbb{P}^*}{\longrightarrow}0$.

Finally, let $r_N$ denote the optimization gap defined in the theorem, so $r_N=o_{\mathbb{P}^*}(1)$.
Fix $\epsilon>0$ for which
$\{\lambda\in\Lambda:|\lambda-\lambda_0|\geq\epsilon\}$ is nonempty. If $2\Delta_N+r_N<\kappa_\epsilon$ and $|\widehat{\lambda}_N-\lambda_0|\geq\epsilon$, then separation gives
\[
    \widehat{\mathcal{C}}_N(\widehat{\lambda}_N)
    \geq\mathcal{C}(\lambda_0)+\kappa_\epsilon-\Delta_N,
\]
whereas approximate optimality gives
\[
    \widehat{\mathcal{C}}_N(\widehat{\lambda}_N)
    \leq\widehat{\mathcal{C}}_N(\lambda_0)+r_N
    \leq\mathcal{C}(\lambda_0)+\Delta_N+r_N,
\]
a contradiction. Hence the following pointwise inclusion holds:
\[
    \{|\widehat{\lambda}_N-\lambda_0|\geq\epsilon\}
    \subseteq
    \{2\Delta_N+r_N\geq\kappa_\epsilon\}.
\]
Moreover,
\[
    \{2\Delta_N+r_N\geq\kappa_\epsilon\}
    \subseteq
    \{\Delta_N>\kappa_\epsilon/8\}
    \cup
    \{r_N>\kappa_\epsilon/2\}.
\]
Therefore, by subadditivity of outer probability,
\[
\begin{aligned}
    \mathbb{P}^*(|\widehat{\lambda}_N-\lambda_0|\geq\epsilon)
    \leq
    \mathbb{P}^*(\Delta_N>\kappa_\epsilon/8)
    +\mathbb{P}^*(r_N>\kappa_\epsilon/2)\rightarrow0.
\end{aligned}
\]
If the separated parameter set is empty, then
$\{|\widehat{\lambda}_N-\lambda_0|\geq\epsilon\}$ is empty and has outer probability zero. Hence
$\widehat{\lambda}_N\overset{\mathbb{P}^*}{\longrightarrow}\lambda_0$.
\end{proof}

\section{Data-Generating Mechanisms for Numerical Experiments}\label{appendix: Data-Generating Mechanisms for Numerical Experiments}

We provide detailed descriptions of the data-generating mechanisms used in the numerical experiments throughout the article. We refer to the point cloud as $\{\boldsymbol{X}_i\}_{i\in[N]}$ where $N$ is the sample size. When describing the dataset, $d$ specifies the intrinsic dimension of the manifold, and $D$ the ambient dimension for $\mbR^D$.

\subsection{Half Circle Point Cloud ($d=1$, $D=2$)}\label{appendix: Half Circle Point Cloud}

\begin{align*}
    \boldsymbol{X}_i=\big(\cos(\pi t_i),\, \sin(\pi t_i) \big) + (\epsilon_{i,1},\, \epsilon_{i,2}), \quad \text{ where}
\end{align*}
\begin{itemize}
    \item $t_1,\ldots,t_N \overset{\mathrm{iid}}{\sim}\mathrm{Unif}(0,1)$,
    \item $\epsilon_{1,1},\ldots,\epsilon_{N,1}, \epsilon_{1,2}, \ldots, \epsilon_{N,2} \overset{\mathrm{iid}}{\sim} \mathcal{N}(0,\sigma^2)$.
\end{itemize}

In Figure \ref{fig:interval_S1_full}, $N= 2500$ and $\sigma^2 = 4\times10^{-4}$.

\subsection{Boundary of a Flower/Star ($d=1$, $D=2$)}\label{appendix: Boundary of a Flower/Star}

\begin{align*}
    \boldsymbol{X}_i=\big(r(t_i) \cos(2\pi t_i),\, r(t_i) \sin(2\pi t_i) \big) + (\epsilon_{i,1},\, \epsilon_{i,2}), \quad \text{ where}
\end{align*}
\begin{itemize}
    \item $t_1,\ldots,t_N \overset{\mathrm{iid}}{\sim}\mathrm{Unif}(0,1)$,
    \item $r(t_i) = 1 + 0.3 \sin(2\pi p t_i)$ for $i=1,\ldots,N$,
    \item $p=5$ is the number of petals,
    \item $\epsilon_{1,1},\ldots,\epsilon_{N,1}, \epsilon_{1,2},\ldots, \epsilon_{N,2} \overset{\mathrm{iid}}{\sim} \mathcal{N}(0,\sigma^2)$.
\end{itemize}

Both figures \ref{fig:interval_S1_full} and \ref{fig:pme_demo_cv_lambda_selection} use $N=2500$ and $\sigma^2 = 4\times10^{-4}$, corresponding to a coordinatewise noise standard deviation of $0.02$. The Monte Carlo experiment in Figure~\ref{fig:PME_monte_carlo} instead uses standard deviation $0.04$ and sample sizes $N\in\{100,250,500,1000\}$.

\subsection{Surface of a Flower/Star ($d=2$, $D=3$)}\label{appendix: Surface of a Flower/Star (2d, 3D)}

\begin{align*}
    \boldsymbol{X}_i=\left( r(\theta_i) \cos(\theta_i) \sqrt{1-z_i^2},\,  r(\theta_i) \sin(\theta_i) \sqrt{1-z_i^2},\, 0.5z_i\right) + (\epsilon_{i,1},\, \epsilon_{i,2}, \epsilon_{i,3}), \quad \text{ where}
\end{align*}
\begin{itemize}
    \item $\theta_1 ,\ldots, \theta_N  \overset{\mathrm{iid}}{\sim} \mathrm{Unif}(0, 2 \pi)$,
    \item $z_1,\ldots, z_N  \overset{\mathrm{iid}}{\sim} \mathrm{Unif}(-1, 1)$,
    \item $r(\theta_i) =(1 + 0.3\cos(p \theta_i))$ for $i =1,\ldots,N$,
    \item $p$ is the number of petals.
    \item $\epsilon_{1,1},\ldots,\epsilon_{N,1}, \epsilon_{1,2},\ldots,\epsilon_{N,2}, \epsilon_{1,3}, \ldots, \epsilon_{N,3} \overset{\mathrm{iid}}{\sim} \mathcal{N}(0,\sigma^2)$.
\end{itemize}

 Figure \ref{fig:2d3D_star_moon} uses $N=300$, $p=6$, and $\sigma^2 = 10^{-4}$. In the flower/star example of Figure \ref{fig:sph_results_panel_8}, we use $N=2500$, $p=5$, and $\sigma^2 = 4\times10^{-4}$, but $\theta_i$ and $z_i$ are generated differently. We construct $N$ deterministic Fibonacci-sphere points in Euclidean coordinates $(x_i,y_i,z_i)$, take $\theta_i$ as the azimuthal angle of the $i$th point, and retain its Euclidean $z$-coordinate.

\subsection{Surface of a Moon/Cashew ($d=2$, $D=3$)}\label{section: Surface of a Moon/Cashew (2d, 3D)}

Define the bending map
\[
\Psi_{b,\rho,R}(x,y,z)
=
\left(
\cos\!\left(\frac{bx}{2R}\right)(\rho R+y),
\sin\!\left(\frac{bx}{2R}\right)(\rho R+y),
z
\right).
\]
The implementation first generates $5N$ independent sphere candidates
\[
\boldsymbol S_j
=
R\left(
\sqrt{1-u_j^2}\cos\phi_j,\,
\sqrt{1-u_j^2}\sin\phi_j,\,
u_j
\right)^{\T},
\qquad j=1,\ldots,5N,
\]
where $u_j\overset{\mathrm{iid}}{\sim}\operatorname{Unif}(-1,1)$ and
$\phi_j\overset{\mathrm{iid}}{\sim}\operatorname{Unif}(0,2\pi)$, independently.
Set $\boldsymbol Y_j=\Psi_{b,\rho,R}(\boldsymbol S_j)$. To obtain approximately
uniform coverage of the bent surface, let $d_{j,10}$ denote the Euclidean
distance from $\boldsymbol Y_j$ to its tenth nearest neighbor among the other
candidates. Select $N$ distinct candidate indices $J_1,\ldots,J_N$ without
replacement, using sampling weights proportional to $d_{j,10}^{\,3}$.
The observed point cloud is
\[
\boldsymbol X_i=\boldsymbol Y_{J_i}+\boldsymbol\epsilon_i,
\qquad
\boldsymbol\epsilon_i\overset{\mathrm{iid}}{\sim}
\mathcal N_3(\boldsymbol 0,\sigma^2\boldsymbol I_3),
\qquad i=1,\ldots,N,
\]
with the noise independent of the candidate generation and selection.

Figure~\ref{fig:2d3D_star_moon} uses $N=300$, $b=1.2\pi$, $\rho=2$,
$R=1$, and $\sigma^2=10^{-4}$. The moon/cashew example in
Figure~\ref{fig:sph_results_panel_8} uses $N=2500$, $b=\pi$, $\rho=1.2$,
$R=1$, and $\sigma^2=4\times10^{-4}$. The Poisson comparison in
Figure~\ref{fig:cashew_hole} instead uses $N=1000$, $b=1.5\pi$,
$\rho=2$, $R=1$, and $\sigma^2=0.05^2$.

\subsection{Downstream Regression with Planar-Flower Responses
($d=1,D=2$)}\label{section: Downstream Regression with Planar-Flower Responses, D=2}

For $i=1,\ldots,N$, let
\[
    \boldsymbol{X}_i
    =
    \begin{pmatrix}
        r_i\cos(\theta_i)\\
        r_i\sin(\theta_i)
    \end{pmatrix}
    +
    \begin{pmatrix}
        \epsilon_{i,1}\\
        \epsilon_{i,2}
    \end{pmatrix},
\]
where
\begin{itemize}
    \item $t_1,\ldots,t_N
    \overset{\mathrm{iid}}{\sim}\operatorname{Unif}(-1,1)$,
    \item
    \[
        u_i
        =
        \left\{
            \frac{1+\sin(\pi t_i)}{2}
        \right\}
        \bmod 1,
    \]
    \item $\theta_i=2\pi u_i$,
    \item $r_i=1+0.3\cos(5\theta_i)$,
    \item
    $\epsilon_{1,1},\ldots,\epsilon_{N,1},
      \epsilon_{1,2},\ldots,\epsilon_{N,2}
      \overset{\mathrm{iid}}{\sim}\mathcal{N}(0,\sigma^2)$.
\end{itemize}
The map $t_i\mapsto u_i$ is deterministic and nonmonotone. Consequently,
the observed predictor $t_i$ is not the latent circular ordering of the
flower. In Figure~5.2(a)--(b), $N=600$ and
$\sigma^2=0.04^2$. The predictors are generated once and are subsequently held fixed.
 
\subsection{Downstream Regression with Flower-Surface Responses
($d=2,D=3$)}\label{section: Downstream Regression with Planar-Flower Responses, D=3}

For $i=1,\ldots,N$, generate
\[
    t_i=(t_{i1},t_{i2})^\T
    \overset{\mathrm{iid}}{\sim}
    \operatorname{Unif}([-1,1]^2)
\]
and define
\[
    u_{ij}:=\frac{t_{ij}+1}{2},
    \qquad j=1,2.
\]
The predictor is mapped noninjectively to the spherical template through
\[
    q_i
    =
    \left\{
        2u_{i1}+0.22\sin(2\pi u_{i2})
    \right\}
    \bmod 1,
    \qquad
    \zeta_i
    =
    \frac{2}{\pi}
    \arcsin\{\sin(2\pi u_{i2})\},
\]
and
\[
    \psi_i
    =
    \begin{pmatrix}
        \sqrt{1-\zeta_i^2}\cos(2\pi q_i)\\
        \sqrt{1-\zeta_i^2}\sin(2\pi q_i)\\
        \zeta_i
    \end{pmatrix}
    \in\mathbb{S}^2.
\]
The observed response is
\[
    \boldsymbol X_i
    =
    \begin{pmatrix}
        r_i\sqrt{1-\zeta_i^2}\cos(\theta_i)\\
        r_i\sqrt{1-\zeta_i^2}\sin(\theta_i)\\
        0.5\zeta_i
    \end{pmatrix}
    +
    \begin{pmatrix}
        \epsilon_{i,1}\\
        \epsilon_{i,2}\\
        \epsilon_{i,3}
    \end{pmatrix},
\]
where
\begin{itemize}
    \item $\theta_i=2\pi q_i$,
    \item $r_i=1+0.3\cos(5\theta_i)$,
    \item the flower has five petals, radial amplitude $0.3$, and vertical
    scale $0.5$,
    \item
    $\epsilon_{1,1},\ldots,\epsilon_{N,1},
      \epsilon_{1,2},\ldots,\epsilon_{N,2},
      \epsilon_{1,3},\ldots,\epsilon_{N,3}
      \overset{\mathrm{iid}}{\sim}\mathcal{N}(0,\sigma^2)$.
\end{itemize}
In Figure~\ref{fig:extrinsic-regression-lambda-selection}(c)--(e), $N=800$ and $\sigma^2=0.04^2$. The predictors are
generated once and are subsequently held fixed.

\stopcontents[appendix]

\bibliography{References}

\end{document}